\documentclass[12pt,a4paper,french,english,BCOR8mm, most, usenames, dvipsnames]{article}
\usepackage[T1]{fontenc}
\usepackage[latin9]{inputenc}
\usepackage{color}
\usepackage{babel}
\makeatletter
\addto\extrasfrench{%
   \providecommand{\fg}{\ifdim\lastskip>\z@\unskip\fi~\frqq}%
}

\makeatother
\usepackage{float}
\usepackage{mathrsfs}
\usepackage{amscd}
\usepackage{tcolorbox}
\usepackage{amsmath}
\usepackage{amsthm}
\usepackage{amssymb}
\usepackage{graphicx}
\usepackage{wasysym}
\usepackage[unicode=true,pdfusetitle,
 bookmarks=true,bookmarksnumbered=false,bookmarksopen=false,
 breaklinks=false,pdfborder={0 0 1},backref=false,colorlinks=false]
 {hyperref}

\makeatletter

\pdfpageheight\paperheight
\pdfpagewidth\paperwidth

\newcommand{\noun}[1]{\textsc{#1}}
\providecommand{\tabularnewline}{\\}

\numberwithin{equation}{section}
\numberwithin{figure}{section}
\numberwithin{table}{section}
\theoremstyle{plain}
\newtheorem{thm}{\protect\theoremname}[section]
\theoremstyle{remark}
\newtheorem{rem}[thm]{\protect\remarkname}
\theoremstyle{plain}
\newtheorem{cor}[thm]{\protect\corollaryname}
\theoremstyle{remark}
\newtheorem{acknowledgement}[thm]{\protect\acknowledgementname}
\theoremstyle{definition}
\newtheorem{defn}[thm]{\protect\definitionname}
\theoremstyle{plain}
\newtheorem{lem}[thm]{\protect\lemmaname}
\theoremstyle{plain}
\newtheorem{prop}[thm]{\protect\propositionname}

\usepackage{amsmath, amsthm, amssymb}

\let\emph\relax 
\DeclareTextFontCommand{\emph}{\bfseries\em}

\tcbuselibrary{breakable}

\newcommand{\real}{\mathbb{R}}

\usepackage[a4paper, total={17cm, 27cm}]{geometry}
\usepackage{color}
\usepackage{graphics}

\usepackage{stmaryrd}

\makeatother

\addto\captionsenglish{\renewcommand{\acknowledgementname}{Acknowledgement}}
\addto\captionsenglish{\renewcommand{\corollaryname}{Corollary}}
\addto\captionsenglish{\renewcommand{\definitionname}{Definition}}
\addto\captionsenglish{\renewcommand{\lemmaname}{Lemma}}
\addto\captionsenglish{\renewcommand{\propositionname}{Proposition}}
\addto\captionsenglish{\renewcommand{\remarkname}{Remark}}
\addto\captionsenglish{\renewcommand{\theoremname}{Theorem}}
\addto\captionsfrench{\renewcommand{\acknowledgementname}{Remerciements}}
\addto\captionsfrench{\renewcommand{\corollaryname}{Corollaire}}
\addto\captionsfrench{\renewcommand{\definitionname}{Définition}}
\addto\captionsfrench{\renewcommand{\lemmaname}{Lemme}}
\addto\captionsfrench{\renewcommand{\propositionname}{Proposition}}
\addto\captionsfrench{\renewcommand{\remarkname}{Remarque}}
\addto\captionsfrench{\renewcommand{\theoremname}{Théorème}}
\providecommand{\acknowledgementname}{Acknowledgement}
\providecommand{\corollaryname}{Corollary}
\providecommand{\definitionname}{Definition}
\providecommand{\lemmaname}{Lemma}
\providecommand{\propositionname}{Proposition}
\providecommand{\remarkname}{Remark}
\providecommand{\theoremname}{Theorem}

\begin{document}
\selectlanguage{english}
\title{Fractal Weyl law for the Ruelle spectrum of Anosov flows}
\author{\href{https://www-fourier.ujf-grenoble.fr/~faure/}{Frédéric Faure}{\small{}}\\
{\small{}Institut Fourier, UMR 5582, Laboratoire de Mathématiques}\\
{\small{}Université Grenoble Alpes, CS 40700, 38058 Grenoble cedex
9, France }\\
{\small{}\href{mailto:frederic.faure@univ-grenoble-alpes.fr}{frederic.faure@univ-grenoble-alpes.fr}}\\
\and\\
\href{https://tsujiimasato.wordpress.com/}{Masato Tsujii}{\small{}}\\
{\small{}Department of Mathematics, Kyushu University,}\\
{\small{}Moto-oka 744, Nishi-ku, Fukuoka, 819-0395, JAPAN }\\
{\small{}\href{mailto:tsujii@math.kyushu-u.ac.jp}{tsujii@math.kyushu-u.ac.jp}}\\
}

\maketitle
\selectlanguage{english}
\begin{abstract}
On a closed manifold $M$, we consider a smooth vector field $X$
that generates an Anosov flow. Let $V\in C^{\infty}\left(M;\mathbb{R}\right)$
be a smooth function called potential. It is known that for any $C>0$,
there exists some anisotropic Sobolev space $\mathcal{H}_{C}$ such
that the operator $A=-X+V$ has intrinsic discrete spectrum on $\mathrm{Re}\left(z\right)>-C$
called Ruelle resonances. In this paper, we show a ``Fractal Weyl
law``: the density of resonances is bounded by $O\left(\left\langle \omega\right\rangle ^{\frac{n}{1+\beta_{0}}}\right)$
where $\omega=\mathrm{Im}\left(z\right)$, $n=\mathrm{dim}M-1$ and
$0<\beta_{0}\leq1$ is the Hölder exponent of the distribution $E_{u}\oplus E_{s}$
(strong stable and unstable). We also obtain some more precise results
concerning the wave front set of the resonances and the invertibility
of the transfer operator. Since the dynamical distributions $E_{u},E_{s}$
are non smooth, we use some semi-classical analysis based on wave
packet transform associated to an adapted metric $g$ on $T^{*}M$
and construct some specific anisotropic Sobolev spaces.
\end{abstract}
\footnote{2010 Mathematics Subject Classification:37D20 Uniformly hyperbolic
systems (expanding, Anosov, Axiom A, etc.) 37D35 Thermodynamic formalism,
variational principles, equilibrium states 37C30 Zeta functions, (Ruelle-Frobenius)
transfer operators, and other functional analytic techniques in dynamical
systems 81Q20 Semiclassical techniques, including WKB and Maslov methods
81Q50 Quantum chaos

Keywords: Transfer operator; Ruelle resonances; decay of correlations;
Semi-classical analysis. }

\selectlanguage{french}%

\newtcolorbox{cBoxA}[1][]{enhanced, frame style={purple!80}, interior style={red!0}, #1}

\newtcolorbox{cBoxB}[2][]{enhanced, frame style={teal!80}, interior style={cyan!0}, #2}

\global\long\def\eq#1{\underset{(#1)}{=}}%
\global\long\def\ineq#1{\underset{(#1)}{\leq}}%
\global\long\def\ineqs#1{\underset{(#1)}{\geq}}%

\newpage{}

\tableofcontents{}

\newpage{}

\section{Introduction}

\subsection*{Anosov flow and transfer operator}

In this paper we consider an Anosov flow $\phi^{t}$ on a compact
smooth manifold $M$ generated by a smooth vector field $X$. An Anosov
flow exhibits sensitivity to initial conditions (or hyperbolicity)
and manifests deterministic chaotic behavior. In the 1970's, Rufus
Bowen, David Ruelle and Yakov Sinai have constructed the ergodic theory
of hyperbolic dynamical systems succeeding the pioneering works of
Smale and Anosov. In particular, a functional and spectral approach
has been pursued by David Ruelle. This approach consists in describing
the evolution, not of individual trajectories which appear unpredictable,
but the evolution of functions $u\in C^{\infty}\left(M\right)$ under
the \emph{transfer operator} 
\[
\mathcal{L}^{t}:u\mapsto e^{\int_{0}^{t}V\circ\phi^{-s}ds}\cdot u\circ\phi^{-t}
\]
where $V\in C^{\infty}\left(M\right)$ is a potential function that
changes the amplitude along the transport (i.e. push forward) of $u$.
This evolution of functions appears to be predictable. In particular
it converges towards an equilibrium state in the space of distributions.
This approach has progressed from the 70's. It has been shown \cite{liverani_02,liverani_04,baladi_05,Baladi05,liverani_butterley_07},
\cite{fred-RP-06,fred-roy-sjostrand-07,fred_flow_09}, \cite{dyatlov_guillarmou_2014},
that the generator $A=-X+V$ of the evolution operator $\mathcal{L}^{t}=e^{tA}$
has a discrete spectrum, eigenvalues are called\emph{ Ruelle resonances},
which describes the effective convergence and fluctuations towards
the equilibrium state.

\subsection*{Semi-classical analysis with wave packets}

Due to hyperbolicity of the Anosov flow, the transfer operator $\mathcal{L}^{t}$
sends the information towards small scales and technically it is natural
to use semi-classical analysis which concerns the large frequency
components of distributions. Following the idea of semi-classical
analysis, we consider the flow $\phi^{t}$ lifted to the cotangent
space $\tilde{\phi}^{t}=\left(d\phi^{-t}\right)^{*}:T^{*}M\rightarrow T^{*}M$
that encodes both the localization of a function and its internal
frequency. This lifted flow $\tilde{\phi}^{t}$ is a Hamiltonian flow
and preserves the level sets of the frequency along the flow direction
which are co-dimension one affine sub-bundles of $T^{*}M$. The assumption
of hyperbolicity implies that $\tilde{\phi}^{t}$ restricted to such
a sub-bundle has a compact trapped set (or, non-wandering set) and
that the dynamics scatters on this trapped set \cite{fred-roy-sjostrand-07,fred_flow_09}.
The existence and properties of the discrete Ruelle spectrum follows
from this observation and the uncertainty principle (i.e. effective
discreteness of $T^{*}M$ into symplectic boxes) and rejoins a more
general theory of semi-classical analysis developed in the 1980's
by B. Helffer and J. Sjöstrand called quantum scattering on phase
space \cite{sjostrand_87}.

In the papers \cite{fred-roy-sjostrand-07,fred_flow_09}, the semi-classical
analysis is performed with the Hörmander's theory of pseudo-differential
operators, considering the generator $A=-X+V$ of the transfer operator
$\mathcal{L}^{t}=e^{tA}$. In this paper we adopt a slightly different
approach: in Section \ref{sec:Semiclassical-analysis-with}, we develop
an analysis using wave packets $\Phi_{\rho}$ that are parameterized
by points $\rho\in T^{*}M$ on the cotangent space. The precise structure
of these wave packet is determined by a metric $g$ on $T^{*}M$ that
is compatible with the symplectic structure and adapted to the dynamics.
In other words the metric measures the uncertainty principle. From
these wave packets we define a wave-packet transform $\mathcal{T}:C^{\infty}\left(M\right)\rightarrow\mathcal{S}\left(T^{*}M\right)$
by $\left(\mathcal{T}u\right)\left(\rho\right)=\langle\Phi_{\rho}|u\rangle_{L^{2}\left(M\right)}$.
Then the semi-classical analysis is performed on $T^{*}M$, considering
the transfer operator $\mathcal{L}^{t}$ for some range of time $t\in\left[0,T\right]$
and analyzing the Schwartz kernel of the equivalent lifted operator
on phase space $T^{*}M$ given by $\tilde{\mathcal{L}}_{W}^{t}:=W\mathcal{T}\mathcal{L}^{t}\mathcal{T}^{\dagger}W^{-1}$,
conjugated by some suitable weight $W$. From this analysis we deduce
properties of the resolvent operator $\left(z-A\right)^{-1}$ and
then properties of the spectrum of the generator $A$.

This approach using phase space representation with wave-packet transform
$\mathcal{T}$\footnote{The wave packet transform that we consider in this paper is related
to Anti-Wick quantization, Berezin quantization, FBI transforms, Bargmann-Segal
transforms, Gabor frames and Toeplitz operators \cite[chap.3]{martinez-01},\cite{wunsh-zworski_01},\cite[chap.13]{zworski_book_2012},\cite{sjostrand_density_resonances_96,sjostrand_hitrick_07}.} and a metric $g$ on the phase space $T^{*}M$ is similar to the
Weyl-Hörmander calculus \cite{hormander1979weyl}\cite[Section 2.2]{lerner2011metrics}
and is also similar to the approach taken in \cite{tsujii_08,tsujii_FBI_10,faure-tsujii_anosov_flows_13,faure-tsujii_prequantum_maps_12}
for dynamical systems. It will provide a new proof of Theorem \ref{thm:Discrete_spectrum}
below that shows the existence of a discrete spectrum for $A$ and,
further, enables us to give some new results, Theorems \ref{thm:Weyl law},~\ref{thm:grey-band}
and \ref{thm:WF} in this paper.

Technically one advantage of using micro-local analysis with wave
packets (i.e. Toeplitz quantization) instead of the more usual Weyl
quantization is that it allows to consider symbols on $T^{*}M$ that
are not necessarily smooth functions. This is particularly interesting
in the context of Anosov flows where the stable/unstable foliations
are Hölder continuous. However it may be possible to develop an analysis
similar to the one that we develop here, but using Weyl quantization
and techniques called ``second micro-localization'' and ``exotic
calculus''. We think that another advantage of using the metric $g$
and wave-packets quantization is to provide a better geometric insight
and meaning to these difficult techniques.

One purpose of this paper is to put the basis of this micro-local
analysis using wave-packets in preparation for a more refined analysis
in case of contact Anosov flows. This second step is done in the more
recent paper \cite{faure-tsujii_anosov_flows_16}.

In a recent paper \cite{bonthonneau2020fbi} Yannick Guedes Bonthonneau
and Malo Jézéquel develop FBI transform in Gevrey classes for Anosov
flows in order to analyze the internal Ruelle spectrum. In Gevrey
classes, we can use a stronger escape function than in smooth classes
and this permits to reveal the whole Ruelle spectrum at once.

\subsection*{Discrete spectrum of the generator.}

Using the semi-classical analysis depicted previously, for any $C>0$,
we can design a positive function $W$ on $T^{*}M$ and some \emph{anisotropic
Sobolev space} $\mathcal{H}_{W}$ in which the transfer operator $\mathcal{L}^{t}$
acts as a strongly continuous semi-group and its generator $A=-X+V$
has discrete spectrum on the spectral domain $\mathrm{Re}\left(z\right)>-C$
\cite{liverani_butterley_07,fred_flow_09}. This discrete spectrum
is intrinsic to the flow, \textit{i.e. }does not depend on the choice
of the space $\mathcal{H}_{W}$ (but of course the norm of the resolvent
depends on $\mathcal{H}_{W}$). In this paper our semi-classical approach
using wave packets permits to design anisotropic Sobolev spaces $\mathcal{H}_{W}$
with more accurate properties than previously constructed ones \cite{fred_flow_09}
and this leads to new refined results on the Ruelle spectrum. For
example we show in Theorem \ref{thm:grey-band} that $\mathcal{L}^{t}$
for $t\in\mathbb{R}$ form a strongly continuous group on $\mathcal{H}_{W}$
and not only a semi-group. We also obtain refined estimate on the
density of the eigenvalues in Theorem \ref{thm:Weyl law} and more
precise description of the eigendistributions in terms of their wave
front set in Theorem \ref{thm:WF}.

The general ideas that sustain the analysis performed in this paper,
Theorem \ref{thm:Discrete_spectrum} on the discrete spectrum and
related properties, are summarized as follows: 
\begin{enumerate}
\item Any distribution in $\mathcal{D}'\left(M\right)$ can be decomposed
as a superposition of wave-packets $\Phi_{\rho}$ parameterized by
$\rho\in T^{*}M$ on cotangent space. This family of wave-packets
is determined by an admissible metric $g$ on $T^{*}M$, which is
compatible with the symplectic form asymptotically flat and also adapted
to the lifted dynamics on $T^{*}M$. We define and use the wave-packet
transform $\mathcal{T}:C^{\infty}\left(M\right)\rightarrow\mathcal{S}\left(T^{*}M\right)$
which expresses a function in $C^{\infty}\left(M\right)$ as superposition
of the wave-packets $\Phi_{\rho}$ in Section \ref{subsec:Resolution-of-identity}.
\item The transfer operator $\mathcal{L}^{t}$ transforms a wave packet
$\Phi_{\rho}$, $\rho\in T^{*}M$, into another (deformed) wave packet
at position $\tilde{\phi}^{t}\left(\rho\right)\in T^{*}M$ where $\tilde{\phi}^{t}:T^{*}M\rightarrow T^{*}M$
is the canonical lift of the flow $\phi^{t}$. Precisely the Schwartz
kernel of the lifted operator $\mathcal{T}\mathcal{L}^{t}\mathcal{T}^{\dagger}:\mathcal{S}\left(T^{*}M\right)\rightarrow\mathcal{S}\left(T^{*}M\right)$
decays very fast on the outside of the graph of $\tilde{\phi}^{t}$.
In other terms, $\mathcal{L}^{t}$ is a Fourier integral operator
whose associated canonical map is $\tilde{\phi}^{t}$. This property
is expressed in Theorem \ref{thm:Microlocality-of-the_TO} usually
called ``propagation of singularities''.
\item For an Anosov flow $\phi^{t}$, the trajectories of the lifted flow
$\tilde{\phi}^{t}$ in $T^{*}M$ escape to infinity for $t\mapsto+\infty$
or $t\rightarrow-\infty$, except for points on a trapped set (or
non-wandering set), which is compact for each frequency $\omega$
along the flow direction. See Figure \ref{fig:The-Anosov-flow-on_T*M}.
One can then find an admissible positive escape function (or Lyapunov
function) $W$ on $T^{*}M$ that is ``temperate and moderately varying
with respect to the metric $g$'' and decreases exponentially along
the flow $\tilde{\phi}^{t}$ on the outside of the trapped set. By
considering $W$ as a $L^{2}-$weight on $T^{*}M$, we define the
anisotropic Sobolev space $\mathcal{H}_{W}\left(M\right)$ and show
that the generator $A:\mathcal{H}_{W}\left(M\right)\rightarrow\mathcal{H}_{W}\left(M\right)$
has a discrete spectrum. Appendix \ref{sec:A-simple-model-of resonances}
illustrates the choice of $W$ and the appearance of discrete spectrum
in $\mathcal{H}_{W}$ with a very simple matrix model.
\end{enumerate}

\section{\label{sec:Results}Results}

In this section we present the main results that we obtain in this
paper concerning the Ruelle spectrum of Anosov flows.

We first review the following theorem that defines the discrete spectrum
of Ruelle resonances. We write $\mathcal{D}'\left(M\right)$ for the
space of distributions on $M$ and $H^{r}\left(M\right)\subset\mathcal{D}'\left(M\right)$
for the Sobolev space of order $r\in\mathbb{R}$. We refer to \cite[def. 2.1 page 4]{pazy_semigroups_83},\cite[p.  79]{engel_1999}
for generalities about semi-groups of operators. $\lambda_{\mathrm{min}}>0$
is the exponent of hyperbolicity defined in (\ref{eq:hyperbolicity}).

\begin{cBoxB}{}
\begin{thm}[Discrete spectrum]
\label{thm:Discrete_spectrum} \cite{liverani_butterley_07,fred_flow_09}.
Let $X$ be a smooth Anosov vector field on a closed manifold $M$
(considered as a differential operator). Let $V\in C^{\infty}\left(M;\mathbb{C}\right)$
be a smooth function and let $A:=-X+V$. For any $r\geq0$ there exists
a Hilbert space $\mathcal{H}_{W}\left(M\right)$, called an anisotropic
Sobolev space, satisfying 
\begin{equation}
H^{r}\left(M\right)\subset\mathcal{H}_{W}\left(M\right)\subset H^{-r}\left(M\right)\label{eq:H_W}
\end{equation}
such that the transfer operator $\mathcal{L}^{t}=\exp\left(tA\right)$
for $t\geq0$ extends to a \textbf{strongly continuous semi-group}
on $\mathcal{H}_{W}\left(M\right)$. The generator $A$ has discrete
spectrum (discrete eigenvalues with finite multiplicities) on the
spectral domain $D_{W}=\left\{ z\in\mathbb{C}\mid\mathrm{Re}\left(z\right)>C_{X,V}-r\lambda_{\mathrm{min}}\right\} $,
with $C_{X,V}\in\mathbb{R}$ given in (\ref{eq:C_XV}). These discrete
eigenvalues are called \textbf{Ruelle resonances}. The Ruelle resonances
and the corresponding generalized eigenspaces are intrinsic, i.e.
they do not depend on the choice of the space $\mathcal{H}_{W}\left(M\right)$,
see \cite[Thm 1.5]{fred_flow_09} or \cite[Lemma B.3]{jezequel2020spectral_thesis}.
See Figure \ref{fig:discrete_spectrum} (a).
\end{thm}

\end{cBoxB}

Theorem \ref{thm:Discrete_spectrum}, giving discrete spectrum for
Anosov flows, has been obtained first by O. Butterley and C. Liverani
(for some Banach spaces) in \cite{liverani_butterley_07}. A proof
using semi-classical analysis and anisotropic Sobolev spaces has been
obtained in \cite{fred_flow_09}. A generalization to Axiom A flows
(and some open uniformly hyperbolic dynamics) has been obtained in
\cite{dyatlov_guillarmou_2014},\cite{https://doi.org/10.48550/arxiv.2107.08875}.
\begin{rem}
Notice that in Theorem \ref{thm:Discrete_spectrum}, the set of operators
$\mathcal{L}^{t}=\exp\left(tA\right)$ for $t\geq0$ in $\mathcal{H}_{W}\left(M\right)$
form a semi-group and not a group. Indeed, in the space $\mathcal{H}_{W}\left(M\right)$
proposed in the papers \cite{liverani_butterley_07,fred_flow_09},
the operator $\mathcal{L}^{t}$ is not invertible. For negative time,
we need to construct a different space to get a semi-group. In this
paper, we somehow improve this aspect, in Theorem \ref{thm:grey-band}
below, where we propose a space $\mathcal{H}_{W}\left(M\right)$ in
which the set of operators $\mathcal{L}^{t}=\exp\left(tA\right)$
for $t\in\mathbb{R}$ form a group. This group property (that provides
invertibility) has been used for example in \cite{faure-tsujii_anosov_flows_16}.
\end{rem}

\subsection{Upper bound for density of eigenvalues}
\begin{flushleft}
In the next theorem we obtain an upper bound for the density of resonances
in the limit of high frequencies $\omega=\mathrm{Im}\left(z\right)$.
This bound depends on the Hölder exponent $0<\beta_{0}\leq1$ of the
distribution $E_{u}\oplus E_{s}$ (strong stable and unstable) defined
in (\ref{eq:Holder_exp}). Recall that $n=\dim M-1$. For $s\in\mathbb{R}$,
we set 
\begin{equation}
\langle s\rangle:=\left(1+s^{2}\right)^{1/2}\underset{\left|s\right|\gg1}{\sim}\left|s\right|.\label{eq:def_Japonese_bracket}
\end{equation}
\par\end{flushleft}

\begin{cBoxB}{}
\begin{thm}[Fractal Weyl law: upper bound for density of eigenvalues]
\begin{flushleft}
\label{thm:Weyl law} Let $X$ be a smooth Anosov vector field on
a closed manifold $M$ and $V\in C^{\infty}\left(M;\mathbb{C}\right)$.
Let $\sigma_{+}\left(A\right)\subset\mathbb{C}$ be the discrete Ruelle
spectrum defined in Theorem \ref{thm:Discrete_spectrum} for the operator
$A=-X+V$. Then, for any $\gamma\in\mathbb{R}$, there exists $C>0$
such that
\begin{align}
\sharp\left\{ z\in\sigma_{+}\left(A\right);\quad\mathrm{Re}\left(z\right)>\gamma,\,\mathrm{Im}\left(z\right)\in\left[\omega,\omega+1\right]\right\} \leq C\left\langle \omega\right\rangle ^{\frac{n}{1+\beta_{0}}} & \quad\text{for any}\:\omega\in\mathbb{R}.\label{eq:weyl_upper_bound}
\end{align}
\par\end{flushleft}
\end{thm}

\end{cBoxB}

The proof of Theorem \ref{thm:Weyl law} is given in Section \ref{subsec:Dynamics-lifted-in}
and relies on an adapted phase space metric $g$ and escape function
$W$ (that defines the Hilbert space $\mathcal{H}_{W}\left(M\right)$)
to the non-smooth trapped set. Note that the exponent $\frac{n}{1+\beta_{0}}$
in the upper bound (\ref{eq:weyl_upper_bound}) depends on the Hölder
exponent $0<\beta_{0}<1$ in (\ref{eq:Holder_exp}). This kind of
upper bound has been called fractal Weyl law after the work of J.
Sjöstrand\cite{sjostrand_90}, see also \cite{Nonnenmacher_Zworski_14}. 

Concerning the upper bound (\ref{eq:weyl_upper_bound}), there are
a few preceding results:
\begin{itemize}
\item For general Anosov flows, i.e. without assumption on $\beta_{0}$,
in \cite[Th 1.8]{fred_flow_09}, the density upper bound $o\left(\left\langle \omega\right\rangle ^{n}\right)$
has been obtained for intervals in $\mathrm{Im}\left(z\right)$ of
width $\left\langle \omega\right\rangle ^{\frac{1}{2}}$.
\item Under the assumption that $m\mapsto E_{u}\left(m\right)\oplus E_{s}\left(m\right)$
is smooth (and therefore $\beta_{0}=1$), in \cite{dyatlov_Ruelle_resonances_2012},
the density upper bound \emph{$O\left(\left\langle \omega\right\rangle ^{\frac{n}{2}}\right)$}
is obtained.
\item For contact Anosov flows (where $E_{u}\oplus E_{s}$ is smooth and
$\beta_{0}=1$), in \cite{faure_tsujii_band_CRAS_2013,faure-tsujii_anosov_flows_13,faure-tsujii_anosov_flows_16},
we obtained the density lower bound $C^{-1}\left\langle \omega\right\rangle ^{\frac{n}{2}}$
under some conditions that guarantee the band structure of the spectrum.
In a more recent work \cite{faure-tsujii_anosov_flows_16} we obtain
the precise asymptotic expression for the density under some pinching
conditions. 
\end{itemize}

\subsubsection{\label{subsec:Heuristic-explanation-of}Heuristic explanation of
the fractal exponent $\frac{n}{1+\beta_{0}}$}

The presence of $\beta_{0}$ in the denominator of the exponent $\frac{n}{1+\beta_{0}}$
in (\ref{eq:weyl_upper_bound}) may look strange at first sight since
a usual treatment of the upper bound in the Weyl law considers minimal
coverings of the trapped set by boxes of size $\delta x=\omega^{-1/2}$,
$\delta\xi=\omega^{1/2}$ for large $\omega\gg1$ and, as explained
below, would give the weaker upper bound $O\left(\omega^{\frac{\mathrm{dim}_{B}\mathscr{A}-1}{2}}\right)=O\left(\omega^{n\left(1-\frac{\beta_{0}}{2}\right)}\right)$
where $\mathrm{dim}_{B}\mathscr{A}=\underbrace{\left(n+1\right)}_{\mathrm{dim}M}+n\left(1-\beta_{0}\right)$
is the fractal box dimension of the graph of the Anosov one form $\mathscr{A}$,
Eq.(\ref{eq:one_form}), see \cite[chap.11]{falconer_03_book}\cite{Nonnenmacher_Zworski_14}.
To obtain the better bound $O\left(\omega^{\frac{n}{1+\beta_{0}}}\right)$,
we consider coverings by boxes of size $\delta x=\omega^{-\alpha}$,
$\delta\xi=\omega^{\alpha}$ where $\frac{1}{2}\leq\alpha<1$ is an
arbitrary parameter (that enters in the metric (\ref{eq:metric_g_in_coordinates})
under the name $\alpha^{\perp}$) and we set $\alpha=\frac{1}{\beta_{0}+1}$
at the end in order to optimize the result. Below we explain this
argument in more detail.

We consider some fixed frequency $\omega$ along the flow and assume
that $\omega$ is large. We will see\footnote{In flow box coordinates we have $-X=\frac{\partial}{\partial z}$,
hence a function that has frequency $\omega$ along the flow writes
$u\left(x,z\right)=u_{0}\left(x\right)e^{i\omega z}$ giving $-Xu=i\omega u$.} that it corresponds with $\omega=\mathrm{Im}\left(z\right)$ on the
spectral domain. It will appear in the proof that, after damped by
a weight function $W$ on $T^{*}M$, Ruelle eigenfunctions at frequency
$\omega$ are micro-locally supported in a vicinity of the graph of
the map $m\in M\mapsto\omega\mathscr{A}\left(m\right)\in T^{*}M$
where $\mathscr{A}$ is the Anosov one form (\ref{eq:one_form}).
We call this graph the \emph{trapped set}. In general it is a fractal
set \cite[Chap.11]{falconer_03_book} because $\mathscr{A}$ is not
smooth. See Figure \ref{fig:fractal_bound}. In order to describe
all the set of resonant states (or eigenfunctions) near frequency
$\omega$, we consider a covering of this graph by symplectic boxes\footnote{In coordinates $\left(y_{j}\right)_{j}$ on $M$ and dual coordinates
$\left(\eta_{j}\right)_{j}$ on $T_{y}^{*}M$, a symplectic box in
$T^{*}M$ has size $\Delta y^{j}=\delta$ and $\Delta\eta_{j}=\delta^{-1}$
for some $\delta>0$, hence symplectic volume 1.} of unit size (corresponding to wave packets) and count how many boxes
$\mathcal{N}\left(\omega\right)$ we need. This number of boxes $\mathcal{N}\left(\omega\right)$
will give an upper bound for the number of eigenvalues. Assume that
each symplectic box has size $\delta x\sim\omega^{-\alpha}$ on the
manifold $M$, transversely to the flow, with some exponent $\frac{1}{2}\leq\alpha<1$.
The symplectic condition (or uncertainty principle) imposes $\delta x.\delta\xi=1$,
i.e. the size $\delta\xi\sim\delta x^{-1}\sim\omega^{\alpha}$ in
the fibers of $T^{*}M$ transversely to the trapped set $\mathbb{R}\mathscr{A}$,
as on Figure \ref{fig:Unity_ball}. Due to its Hölder exponent $\beta_{0}$,
the graph of $\omega\mathscr{A}$ spreads over a range of frequencies
of size {\small{}$\delta\left(\omega\mathscr{A}\right)\sim\omega\left(\delta x\right)^{\beta_{0}}=\omega^{1-\alpha\beta_{0}}$}.
Then there are two cases to consider, see Figure \ref{fig:fractal_bound}:
\begin{enumerate}
\item If $\alpha\leq\frac{1}{1+\beta_{0}}\Leftrightarrow1-\alpha\beta_{0}\geq\alpha$,
then for large frequencies $\omega$, we have $\omega^{1-\alpha\beta_{0}}\geq\omega^{\alpha}\Leftrightarrow\delta\left(\omega\mathcal{A}\right)\geq\delta\xi$,
i.e. the variance $\delta\left(\omega\mathcal{A}\right)$ of $\omega\mathcal{A}$
is larger than the size $\delta\xi$ of the box in the frequency space.
The symplectic volume to be covered by the boxes will be proportional
to $\left(\delta\left(\omega\mathcal{A}\right)\right)^{n}$. This
gives the estimate $\mathcal{N}\left(\omega\right)\asymp\left(\delta\left(\omega\mathcal{A}\right)\right)^{n}=\omega^{n\left(1-\beta_{0}\alpha\right)}$.
\item On the contrary, if $\alpha\geq\frac{1}{1+\beta_{0}}$, the variance
$\delta\left(\omega\mathcal{A}\right)$ of $\omega\mathcal{A}$ is
smaller than the size $\delta\xi$ of the box in the frequency space.
In this case, the symplectic volume to be covered by the boxes will
be proportional to $\left(\delta\xi\right)^{n}$. This gives the estimate
$\mathcal{N}\left(\omega\right)\asymp\left(\delta\xi\right)^{n}=\omega^{n\alpha}$.
\end{enumerate}
The value of $\alpha$ that minimizes $\mathcal{N}\left(\omega\right)$
is $\alpha=\frac{1}{1+\beta_{0}}$, giving $\mathcal{N}\left(\omega\right)\asymp\omega^{\frac{n}{1+\beta_{0}}}$,
that is the upper bound (\ref{eq:weyl_upper_bound}) in Theorem \ref{thm:Weyl law}.

{\footnotesize{}}
\begin{figure}[h]
\begin{centering}
{\footnotesize{}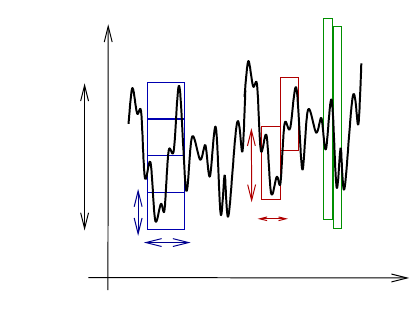$\qquad$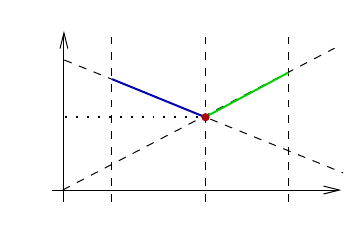}{\footnotesize\par}
\par\end{centering}
{\footnotesize{}\caption{\label{fig:fractal_bound}To estimate from above the density of resonances
at frequency $\omega$, we cover the trapped set, i.e. the graph of
the map $m=\left(x,z\right)\in M\protect\mapsto\omega\mathscr{A}\left(m\right)\in T^{*}M$
by symplectic boxes, where $x$ is a coordinate transverse to the
flow direction. This graph is Hölder continuous with exponent $\beta_{0}$.
With the choice $\delta x=\omega^{-\alpha}$ we obtain that the symplectic
volume of this cover is $\omega^{E\left(\alpha\right)}$ with some
exponent $E\left(\alpha\right)$ that is minimum for the choice $\alpha=\frac{1}{1+\beta_{0}}$.
On the picture, the cover in the middle (red boxes) is more efficient
than the left one (blue boxes) or the right one (green boxes).}
}{\footnotesize\par}
\end{figure}
{\footnotesize\par}

\subsection{Parabolic wave front set of the Ruelle eigenfunctions}

Concerning the eigendistributions associated to the Ruelle spectrum,
we obtain in Theorem \ref{thm:WF} and Corollary \ref{cor:Let-us-choose}
below a precise description of their semi-classical wave front set,
that is, the region in the phase space $T^{*}M$ where the Ruelle
eigenfunctions are non negligible. In these results, the wave front
sets are contained in a parabolic vicinity of the unstable affine
distribution $E_{u}^{*}+\omega\mathcal{A}$ uniformly in $\omega$
(where $\mathcal{A}\in\left(E_{u}\oplus E_{s}\right)^{\perp}$ is
the Anosov one form and $\omega$ is the imaginary part of the corresponding
eigenvalue) and this improves the previous results \cite{fred_flow_09}
which claims that the wave front set is contained in an arbitrary
conical vicinity of $E_{u}^{*}$. See Figure \ref{fig:Wave-front-set}
that compares these two results.

We will use the decomposition of $T^{*}M=E_{u}^{*}\oplus E_{s}^{*}\oplus E_{0}^{*}$
dual to (\ref{eq:decomp_TM}), where 
\[
E_{u}^{*}=\left(E_{u}\oplus E_{0}\right)^{\perp},\quad E_{s}^{*}=\left(E_{s}\oplus E_{0}\right)^{\perp},\quad E_{0}^{*}=\left(E_{u}\oplus E_{s}\right)^{\perp}.
\]
We first recall the definition of the wave front set $WF\left(u\right)$
of a distribution $u\in\mathcal{D}'\left(M\right)$ from \cite[p.254]{hormander_1},\cite[p.77]{grigis_sjostrand},\cite[p.27]{taylor_tome2}.
A point $\left(m,\eta\right)\in T^{*}M$ does not belong to $WF\left(u\right)$
if and only if there exist a smooth function $\chi\in C^{\infty}\left(M;\mathbb{R}^{+}\right)$
with $\chi\left(m\right)=1$ and an open cone $\mathbf{C}\subset\mathbb{R}^{n+1}$
with $\eta\in\mathbf{C}$ such that, for any $N>1$, we have
\[
\left|\left(\mathcal{F}\left(\chi u\right)\right)\left(\eta'\right)\right|\leq\frac{C_{N}}{\left|\eta'\right|^{N}}\quad\text{for all }\eta'\in\mathbf{C},
\]
with some constant $C_{N}$, where $\mathcal{F}$ is the Fourier transform
(in a local chart).

The following theorem describes what is known from \cite{fred_flow_09}
concerning the wave front set $WF\left(u\right)$ of a Ruelle generalized
eigenfunction\footnote{A generalized eigenfunction of a linear operator $A$ for an eigenvalue
$z\in\mathbb{{C}}$ is a distribution $u\in\mathcal{D}'\left(M\right)$
such that $(z-A)^{n}u=0$ for some $n\geq1$.}.

\begin{cBoxB}{}
\begin{thm}[Conical wave front set]
\label{thm:conical-WF}\cite{fred_flow_09} Assume that $u\in\mathcal{H}_{W}\left(M\right)$
is a generalized eigenfunction of the generator $A$ for a Ruelle
eigenvalue $z\in\mathbb{C}$. Then we have
\[
WF\left(u\right)\subset E_{u}^{*}=\left(E_{u}\oplus E_{0}\right)^{\perp}.
\]
\end{thm}

\end{cBoxB}

The claim of Theorem \ref{thm:conical-WF} is illustrated on Figure
\ref{fig:Wave-front-set}(a). In Theorem \ref{thm:WF} below, we give
a more precise description of the singularities of Ruelle eigenfunctions.
For this, we first introduce (a simplified version of) the wave packet
transform.

\begin{figure}[h]
\begin{centering}
\scalebox{0.6}[0.6]{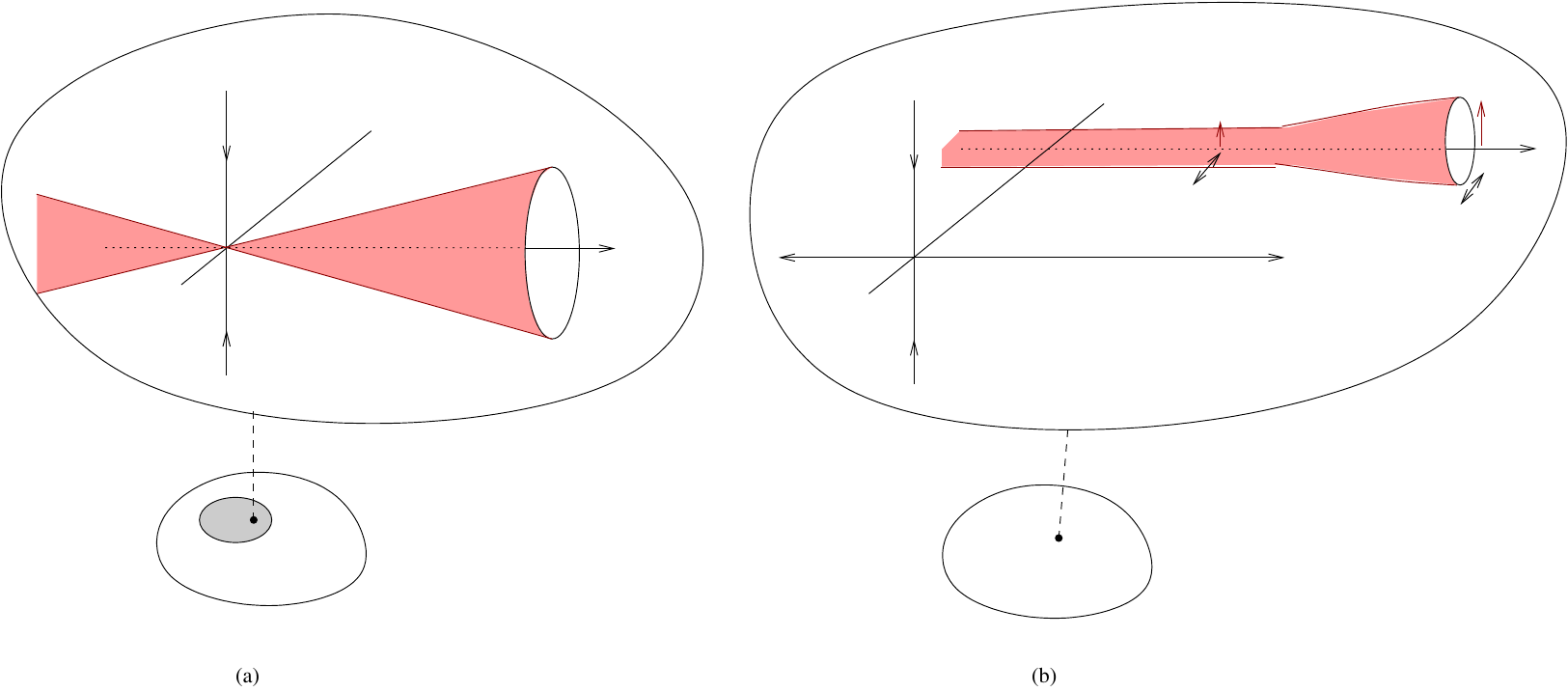}
\par\end{centering}
\caption{\label{fig:Wave-front-set}\textbf{(a)} Theorem \ref{thm:conical-WF}
shows that for high frequencies a Ruelle eigendistribution represented
in phase space $T^{*}M$ is negligible outside any conical vicinity
of the linear sub-bundle $E_{u}^{*}\subset T^{*}M$.\textbf{ (b)}
Theorem \ref{thm:WF} improves this description and shows that a Ruelle
eigendistribution with eigenvalue $z=a+i\omega_{0}$, is negligible
outside a \textquotedblleft parabolic domain\textquotedblright{} \emph{$\mathcal{V}_{\omega_{0},\epsilon}$}
of the affine sub-bundle $\omega_{0}\mathscr{A}+E_{u}^{*}\subset T^{*}M$,
uniformly with respect to $\omega_{0}\in\mathbb{R}$. Here $\alpha^{\perp}=\frac{1}{1+\min\left(\beta_{u},\beta_{s}\right)}\in[\frac{1}{2};1[$
and $\epsilon>0$ is arbitrary small. Observe that the cone $C$ contains
the domain $\mathcal{V}_{\omega_{0},\epsilon}$, except for a compact
part of it.}
\end{figure}

\subsubsection{Metric $g$ on $T^{*}M$}

The analysis made in this paper relies on the use of a specific metric
on $T^{*}M$. From this metric $g$ we will define later the wave
packet transform. We consider local flow box coordinates $y=\left(x,z\right)\in\mathbb{R}^{n}\times\mathbb{R}$
for the vector field $X$ on an open set $U\subset M$, which is by
definition a local chart map $\kappa:m\in U\subset M\mapsto y=\left(x,z\right)\in\mathbb{R}^{n}\times\mathbb{R}$,
such that 
\begin{equation}
\kappa_{*}\left(-X\right)=\frac{\partial}{\partial z}.\label{eq:X_ddz}
\end{equation}
We write $\eta=\left(\xi,\omega\right)\in\mathbb{R}^{n}\times\mathbb{R}$
for the dual coordinates on $T_{\left(x,z\right)}^{*}\mathbb{R}^{n+1}\equiv\mathbb{R}^{n}\times\mathbb{R}$.
Let $\tilde{\kappa}:\rho\in T^{*}U\mapsto\left(y,\eta\right)\in T^{*}\mathbb{R}^{n+1}$
be the canonical extension of the local chart map $\kappa$ to the
cotangent bundle.

Let us consider parameters  $\alpha^{\perp},\alpha^{\parallel}\in[0,1[$
such that
\begin{equation}
0\leq\alpha^{\parallel}<\alpha^{\perp}<1,\quad\frac{1}{2}\leq\alpha^{\perp}<1.\label{eq:conditions}
\end{equation}
Let
\begin{equation}
\delta^{\perp}\left(\eta\right):=\left\langle \left|\eta\right|\right\rangle ^{-\alpha^{\perp}},\quad\delta^{\parallel}\left(\eta\right):=\left\langle \left|\eta\right|\right\rangle ^{-\alpha^{\parallel}},\label{eq:def_delta}
\end{equation}
where $\left|\eta\right|$ denotes the Euclidean norm in $\mathbb{R}^{n+1}$
and consider the following metric $g$ on $T^{*}\mathbb{R}^{n+1}$
given at each point 
\[
\varrho=\left(y,\eta\right)=\left(\left(x,z\right),\left(\xi,\omega\right)\right)\in T^{*}\mathbb{R}^{n+1}
\]
by 
\begin{equation}
g_{\varrho}:=\left(\frac{dx}{\delta^{\perp}\left(\eta\right)}\right)^{2}+\left(\delta^{\perp}\left(\eta\right)d\xi\right)^{2}+\left(\frac{dz}{\delta^{\parallel}\left(\eta\right)}\right)^{2}+\left(\delta^{\parallel}\left(\eta\right)d\omega\right)^{2}.\label{eq:metric_g_in_coordinates}
\end{equation}
The metric $g$ is compatible\footnote{Indeed we have $g\left(u,v\right)=\Omega\left(u,Jv\right)$ if we
define an almost complex structure $J$ by

\begin{align*}
J\left(\delta^{\perp}\left(\eta\right)\partial_{x}\right) & =\frac{1}{\delta^{\perp}\left(\eta\right)}\partial_{\xi},\quad J\left(\frac{1}{\delta^{\perp}\left(\eta\right)}\partial_{\xi}\right)=-\delta^{\perp}\left(\eta\right)\partial_{x},\\
J\left(\delta^{\parallel}\left(\eta\right)\partial_{z}\right) & =\frac{1}{\delta^{\parallel}\left(\eta\right)}\partial_{\omega},\qquad J\left(\frac{1}{\delta^{\parallel}\left(\eta\right)}\partial_{\omega}\right)=-\delta^{\parallel}\left(\eta\right)\partial_{z}.
\end{align*}
} \cite{mac_duff_98,da_silva_01} with the canonical symplectic form
on $T^{*}M$: $\Omega=\sum_{k=1}^{n+1}dy_{k}\wedge d\eta_{k}$. The
norm of a vector $v\in\mathbb{R}^{2\left(n+1\right)}$ with respect
to $g$ is denoted by
\begin{equation}
\left\Vert v\right\Vert _{g_{\varrho}}:=\left(g_{\varrho}\left(v,v\right)\right)^{1/2}.\label{eq:def_norm_g}
\end{equation}
The unit ball for the metric is illustrated on Figure \ref{fig:Unity_ball}.

As we will see in Section \ref{subsec:Lipschitz-property-of}, the
conditions (\ref{eq:conditions}) ensure that different choices of
flow box charts give a uniformly equivalent metric on the intersection
and consequently defines an equivalence class of metric on the cotangent
bundle $T^{*}M$. 

\begin{figure}[H]
\centering{}\scalebox{0.8}[0.8]{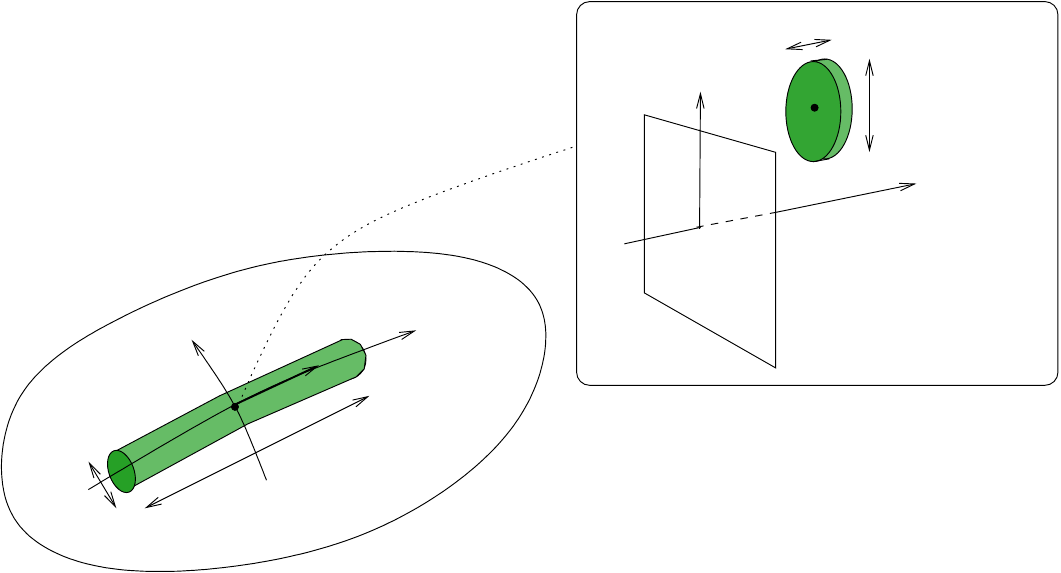}\caption{\label{fig:Unity_ball}We use flow box coordinates $y=\left(x,z\right)$
on $M$ i.e. such that the vector field is $\left(-X\right)=\frac{\partial}{\partial z}$
and dual coordinates $\eta=\left(\xi,\omega\right)$ in $T_{m}^{*}M$.
$g$ is a (class of) metric on $T^{*}M$ and the filled cylinders
(in green) represent the unit ball for this metric $g$ at a point
$\rho\in T^{*}M$, projected on the base $M$ and on the fiber $T_{m}^{*}M$
with $m=\pi\left(\rho\right)$. This unit ball has size $\delta^{\parallel}\left(\eta\right)$
along the flow $z$, transverse size $\delta^{\perp}\left(\eta\right)$
along $x$ and sizes $\left(\delta^{\perp}\left(\eta\right)\right)^{-1},\left(\delta^{\parallel}\left(\eta\right)\right)^{-1}$
along dual coordinates $\xi,\omega$ in $T_{m}^{*}M$. These sizes
$\delta^{\perp}\left(\eta\right)\protect\leq\delta^{\parallel}\left(\eta\right)\protect\leq1$
are functions on $T^{*}M$, defined in (\ref{eq:def_delta}), and
decay at infinity.}
\end{figure}

\subsubsection{Wave packet transform}

For a given $\varrho=\left(y,\eta\right)\in T^{*}\mathbb{R}^{n+1}$
we define the Gaussian function $\Phi_{\varrho}\in C_{0}^{\infty}\left(\mathbb{R}^{n+1};\mathbb{C}\right)$
called wave packet:
\begin{align}
\Phi_{\varrho}\left(y'\right) & =a_{\varrho}\chi\left(y'-y\right)\exp\left(i\eta.y'-\left\Vert y'-y\right\Vert _{g_{\varrho}}^{2}\right)\label{eq:wave_packet_1}\\
 & =a_{\varrho}\chi\left(y'-y\right)\exp\left(i\eta.y'-\left|\frac{\left(z'-z\right)}{\delta^{\parallel}\left(\eta\right)}\right|^{2}-\left|\frac{\left(x'-x\right)}{\delta^{\perp}\left(\eta\right)}\right|^{2}\right)\nonumber 
\end{align}
with $y'=\left(x',z'\right)$ and where $\chi\in C_{0}^{\infty}\left(\mathbb{R}^{n+1}\right)$
is some cut-off function with $\chi\equiv1$ near the origin and $a_{\varrho}>0$
is such that $\left\Vert \Phi_{\varrho}\right\Vert _{L^{2}\left(\mathbb{R}^{n+1}\right)}=1$.
We define the wave-packet transform of a distribution $u\in\mathcal{D}'\left(M\right)$
on $T^{*}U$ by\footnote{The wave-packet transform $\mathcal{T}$ defined here is similar to
\href{https://en.wikipedia.org/wiki/Segal\%E2\%80\%93Bargmann_space}{Bargman transform}.}
\begin{equation}
\left(\mathcal{T}_{g}u\right)\left(\rho\right):=\langle\Phi_{\tilde{\kappa}\left(\rho\right)}\circ\kappa|u\rangle_{L^{2}\left(M\right)},\quad\text{with }\rho\in T^{*}U.\label{eq:def_Bg_u}
\end{equation}

\begin{rem}
Later we will introduce in (\ref{eq:wave_packet}) an expression for
wave packets that is slightly different but converges to (\ref{eq:wave_packet_1})
in the high frequency limit. The expression will be more complicated
but will have the advantage of giving an exact resolution of identity.
Since both expressions become essentially equivalent in the high frequency
limits, the definitions and properties given in this section are not
affected by the difference.
\end{rem}

\subsubsection{Conical wave front set}

We first reformulate Theorem \ref{thm:conical-WF} using the wave
packet transform $\mathcal{T}_{g}u$ using the metric $g$ as follows.
Assume that $u\in\mathcal{H}_{W}\left(M\right)$ is a generalized
eigenfunction of the generator $A$ for a Ruelle eigenvalue $z\in\mathbb{C}$
given in Theorem \ref{thm:grey-band}. Then, for any continuous field
of open positive cones $\mathbf{C}:m\in M\mapsto\mathbf{C}\left(m\right)\subset T_{m}^{*}M$
with $\mathbf{C}\left(m\right)\supset E_{u}^{*}\left(m\right)$ and
for any $N>1$, there exists $C_{N}>0$ such that
\begin{equation}
\left|\left(\mathcal{T}_{g}u\right)\left(\rho\right)\right|\leq\frac{C_{N}}{\left|\rho\right|^{N}}\quad\text{for }\rho\in T^{*}M\setminus\mathbf{C}.\label{eq:WF}
\end{equation}
This means that $\mathcal{T}_{g}u$ is negligible on the outside of
any conical vicinity of the sub-bundle $E_{u}^{*}$. See Figure \ref{fig:Wave-front-set}(a).
Theorem \ref{thm:WF} below improves this result.

\subsubsection{Parabolic wave front set}

For a point $m\in M$, a cotangent vector $\rho\in T_{m}^{*}M$ has
a unique decomposition
\[
\rho=\omega\mathcal{A}\left(m\right)+\rho_{u}+\rho_{s}\quad\text{with }\omega\in\mathbb{R},\,\rho_{u}\in E_{u}^{*},\,\rho_{s}\in E_{s}^{*}.
\]
We introduce a function $W\in C\left(T^{*}M;\mathbb{R}^{+}\right)$,
called an escape function, defined for $\rho\in T^{*}M$ by
\begin{equation}
W\left(\rho\right)=\frac{\left\langle \left\Vert \rho_{s}\right\Vert _{g_{\rho}}\right\rangle ^{R_{s}}}{\left\langle \left\Vert \rho_{u}\right\Vert _{g_{\rho}}\right\rangle ^{R_{u}}}\underset{(\ref{eq:metric_g_in_coordinates})}{\asymp}\frac{\left\langle \left\langle \left|\rho\right|\right\rangle ^{-\alpha^{\perp}}\left|\rho_{s}\right|\right\rangle ^{R_{s}}}{\left\langle \left\langle \left|\rho\right|\right\rangle ^{-\alpha^{\perp}}\left|\rho_{u}\right|\right\rangle ^{R_{u}}}\label{eq:W2-2-1}
\end{equation}
where $R_{u},R_{s}>0$ are arbitrary parameters which will be taken
large enough.

\begin{cBoxB}{}
\begin{thm}[Parabolic wave front set of the Ruelle eigenfunction]
\label{thm:WF}Assume that $u\in\mathcal{H}_{W}\left(M\right)$ is
a generalized eigenfunction of the generator $A$ for a Ruelle eigenvalue
$z\in\mathbb{C}$ with $\mathrm{Re}\left(z\right)\geq-C$. Let $\omega_{0}=\mathrm{Im}\left(z\right)$.
Then any $N>1$ there exists a constant $C_{N}>0$ such that for any
$\rho\in T^{*}M$,
\begin{equation}
\left|\left(\mathcal{T}_{g}u\right)\left(\rho\right)\right|\leq\frac{C_{N}}{W\left(\rho\right)\left(\frac{\left\langle \omega-\omega_{0}\right\rangle }{\left\langle \left|\rho\right|\right\rangle ^{\alpha^{\parallel}}}\right)^{N}}\left\Vert u\right\Vert _{\mathcal{H}_{W}\left(M\right)}.\label{eq:WF-1-1}
\end{equation}
\end{thm}

\end{cBoxB}

The claim (\ref{eq:WF-1-1}) implies that a Ruelle eigenfunction (represented
in phase space $T^{*}M$) is micro-localized near the frequency $\omega=\omega_{0}=\mathrm{Im}\left(z\right)$
and bounded by $1/W\left(\rho\right)$ up to some multiplicative constant.
In order to explain what we mean by ''parabolic'' wave front set,
we present the following corollary that is obtained from the theorem
above with a choice of the metric $g$ and the escape function $W\left(\rho\right)$.
See Figure \ref{fig:Wave-front-set}(b).

\begin{cBoxB}{}
\begin{cor}
\label{cor:Let-us-choose} Set $\alpha^{\perp}=\frac{1}{1+\min\left(\beta_{u},\beta_{s}\right)}$
and $\alpha^{\parallel}=0$ in the definition of the metric $g$ in
(\ref{eq:metric_g_in_coordinates}). Then, for any (large) $N>0$
and any (small) $\epsilon>0$, the Ruelle  eigenfunction $u\in\mathcal{H}_{W}\left(M\right)$
for the eigenvalue $z\in\mathbb{C}$, in Theorem \ref{thm:WF}, satisfies
\begin{equation}
\left|\left(\mathcal{T}_{g}u\right)\left(\rho\right)\right|\leq\frac{C_{N,\epsilon}}{\left\langle \rho\right\rangle ^{N}}\left\Vert u\right\Vert _{\mathcal{H}_{W}\left(M\right)}\label{eq:WF-1}
\end{equation}
on the outside of the parabolic vicinity 
\begin{align}
\mathcal{V}_{\omega_{0},\epsilon}: & =\left\{ \rho\in T^{*}M\;\left.\;\left\langle \omega-\omega_{0}\right\rangle \leq\left|\rho\right|^{\epsilon}\text{ and }\left\langle \left|\rho\right|^{-\alpha^{\perp}}\left|\rho_{s}\right|\right\rangle \leq\left|\rho\right|^{\epsilon}\right.\right\} \label{eq:Vu}
\end{align}
 of the affine bundle $E_{u}^{*}+\omega_{0}\mathscr{A}=\left\{ \rho\in T^{*}M\mid\rho_{s}=0,\omega=\omega_{0}\right\} $,
where the constant $C_{N,\epsilon}>0$ depends on $N$ and $\epsilon$
but not on $u$.
\end{cor}

\end{cBoxB}

The proofs of Theorem \ref{thm:WF} and Corollary \ref{cor:Let-us-choose}
are given in Section \ref{sec:Proof-of-Theorem_WF}.
\begin{rem}
We can write (\ref{eq:WF-1}) as
\begin{equation}
\left|\left(\mathcal{T}_{g}u\right)\left(\rho\right)\right|\leq\frac{C_{N,\epsilon}}{\left\langle \left|\rho\right|^{-\epsilon}\mathrm{dist}_{g}\left(\rho,E_{u}^{*}+\omega_{0}\mathscr{A}\right)\right\rangle ^{N}}\left\Vert u\right\Vert _{\mathcal{H}_{W}\left(M\right)}\label{eq:parabolic_WF}
\end{equation}
with 
\[
\mathrm{dist}_{g}\left(\rho,E_{u}^{*}+\omega_{0}\mathscr{A}\right):=\min_{\rho'\in E_{u}^{*}+\omega_{0}\mathscr{A}}\mathrm{dist}_{g}\left(\rho,\rho'\right)\asymp\max\left(\left\Vert \rho_{s}\right\Vert _{g_{\rho}},\left\langle \omega-\omega_{0}\right\rangle \right)
\]
and because $\left\Vert \rho_{s}\right\Vert _{g_{\rho}}=\left|\rho\right|^{-\alpha^{\perp}}\left|\rho_{s}\right|$.
Further, since the Ruelle eigenvalues and the corresponding generalized
eigenspaces are intrinsic to the transfer operators $\mathcal{L}^{t}$
and do not depend on the choice of the weight function $W$, it is
possible (and might be better) to replace (\ref{eq:parabolic_WF})
by an expression that does not depend on $W$:
\[
\left|\left(\mathcal{T}_{g}u\right)\left(\rho\right)\right|\leq\frac{C_{N,\epsilon}}{\left\langle \left|\rho\right|^{-\epsilon}\mathrm{dist}_{g}\left(\rho,E_{u}^{*}+\omega_{0}\mathscr{A}\right)\right\rangle ^{N}}\left(\int_{\mathrm{dist}_{g}\left(\rho,\omega_{0}\mathscr{A}\right)\leq\left|\rho\right|^{\epsilon}}\left|\left(\mathcal{T}_{g}u\right)\left(\rho\right)\right|^{2}\frac{d\rho}{\left(2\pi\right)^{n+1}}\right)^{1/2}.
\]
\end{rem}

~
\begin{rem}
Theorem \ref{thm:WF} is more precise than Theorem \ref{thm:conical-WF}
because the parabolic domain $\mathcal{V}_{\omega_{0},\epsilon}$
is contained in any conical vicinity $\mathbf{C}$ of $E_{u}^{*}$
at high frequencies (compare Figures \ref{fig:Wave-front-set}(a)
and (b)) and because the constant $C_{N,\epsilon}$ in (\ref{eq:WF-1})
and (\ref{eq:parabolic_WF}) does not depend on $\omega_{0}\in\mathbb{R}$
(hence on $u$), whereas the constant $C_{N}$ in (\ref{eq:WF}) depends
on $u$.
\end{rem}

~
\begin{rem}
(Technical) We expect that the exponent $\alpha^{\perp}$ in Corollary
\ref{cor:Let-us-choose} should be $\alpha^{\perp}=\frac{1}{1+\beta_{u}}$
instead of $\alpha^{\perp}=\frac{1}{1+\min\left(\beta_{u},\beta_{s}\right)}$.
The reason that we have $\min\left(\beta_{u},\beta_{s}\right)$ instead
of $\beta_{u}$ is due to the construction of the metric $g$.
\end{rem}

\subsection{Past and future Ruelle spectrum}

We first present some notations. We take a constant $\lambda_{\mathrm{max}}\geq\lambda_{\mathrm{min}}>0$
such that
\begin{equation}
\frac{1}{C}e^{-\lambda_{\mathrm{max}}t}\left\Vert v\right\Vert _{g_{M}}\le\left\Vert d\phi^{t}v\right\Vert _{g_{M}}\leq Ce^{\lambda_{\mathrm{max}}t}\left\Vert v\right\Vert _{g_{M}}\label{eq:l_max}
\end{equation}
for all $v\in TM$ and $t\ge0$ where $g_{M}$ is a smooth metric
on $M$.  For a function $\varphi\in C\left(M;\mathbb{C}\right)$,
we set
\begin{equation}
\overline{\max}\left(\varphi\right):=\lim_{t\rightarrow+\infty}\max_{m\in M}\frac{1}{t}\int_{0}^{t}\varphi\left(\phi^{s}\left(m\right)\right)ds,\label{eq:def_tmax}
\end{equation}
\begin{equation}
\overline{\min}\left(\varphi\right):=\lim_{t\rightarrow+\infty}\min_{x\in M}\frac{1}{t}\int_{0}^{t}\varphi\left(\phi^{s}\left(x\right)\right)ds\label{eq:def_tmin}
\end{equation}
which are called\footnote{Here are other equivalent expressions (or definitions) of $\overline{\max}\left(\varphi\right)$:
\[
\overline{\max}\left(\varphi\right)=\max_{\mathrm{inv.\,prob.\,measure\,}\mu}\int\varphi d\mu=\lim_{\beta\rightarrow+\infty}\frac{1}{\beta}\mathrm{Pr}\left(\beta\varphi\right),
\]
where $\mathrm{Pr}\left(.\right)$ denotes the topological pressure.
We have $\overline{\min}\left(\varphi\right)=-\overline{\max}\left(-\varphi\right)$.} the maximal and minimal ergodic average of $\varphi$ respectively
\cite{jenkinson_2018}.

We fix an (arbitrary) smooth density $dm$ on $M$, consider the space
$L^{2}\left(M,dm\right)$ and the divergence of the vector field $X$
with respect to $dm$ is denoted $\mathrm{div}X$. Using these definitions
we set
\begin{equation}
C_{X,V}:=\overline{\max}\left(\frac{1}{2}\mathrm{div}X+\mathrm{Re}\left(V\right)\right),\label{eq:C_XV}
\end{equation}
\begin{equation}
C'_{X,V}:=\overline{\min}\left(\frac{1}{2}\mathrm{div}X+\mathrm{Re}\left(V\right)\right),\label{eq:C_XV-1}
\end{equation}
so that $C'_{X,V}\leq C_{X,V}$.

In the next theorem we obtain that the transfer operator $\mathcal{L}^{t}:\mathcal{H}_{W}\left(M\right)\rightarrow\mathcal{H}_{W}\left(M\right)$
for $t\in\mathbb{R}$ forms a group of operators, whose generator
$A=-X+V$ has some intrinsic discrete spectrum formed by two separated
sets, which are called the \emph{future and past spectrum} respectively
and represented on Figure \ref{fig:discrete_spectrum}.

\begin{cBoxB}{}
\begin{thm}[Past and future spectrum]
\label{thm:grey-band} For any $\epsilon>0$ and $r\in\mathbb{R}$,
there exists a Hilbert space $\mathcal{H}_{W}\left(M\right)$ with
\[
H^{|r|}\left(M\right)\subset\mathcal{H}_{W}\left(M\right)\subset H^{-|r|}\left(M\right)
\]
such that $\mathcal{L}^{t}=\exp\left(tA\right)$ for $t\in\mathbb{R}$
is a \textbf{strongly continuous group} and the essential spectrum
$\sigma_{ess}\left(A\right)$ of the generator $A:\mathcal{H}_{W}\left(M\right)\rightarrow\mathcal{H}_{W}\left(M\right)$
is contained in the vertical band
\begin{equation}
\begin{cases}
C'_{X,V}-2r\lambda_{\mathrm{max}}-\epsilon\leq\mathrm{Re}\left(z\right)\leq C_{X,V}-r\lambda_{\mathrm{min}}+\epsilon & \text{if }r\ge0\\
C'_{X,V}-r\lambda_{\mathrm{min}}-\epsilon\leq\mathrm{Re}\left(z\right)\leq C_{X,V}-2r\lambda_{\mathrm{max}}+\epsilon & \text{if }r\le0
\end{cases}\label{eq:bound_for_ess_spec}
\end{equation}
Further we have that
\begin{itemize}
\item if $r\geq0$, then $A$ has a uniformly bounded resolvent on 
\[
\left\{ \mathrm{Re}\left(z\right)\leq C'_{X,V}-2r\lambda_{\mathrm{max}}-\epsilon\right\} \cup\left\{ \mathrm{Re}\left(z\right)\geq C_{X,V}+\epsilon\right\} 
\]
 and discrete spectrum $\sigma_{+}\left(A\right)$ (Ruelle resonances
for the future) on the domain 
\[
\{C_{X,V}-r\lambda_{\mathrm{min}}+\epsilon\leq\mathrm{Re}\left(z\right)\leq C_{X,V}+\epsilon\}.
\]
\item if $r\leq0$, then $A$ has a uniformly bounded resolvent on 
\[
\left\{ \mathrm{Re}\left(z\right)\leq C'_{X,V}-\epsilon\right\} \cup\left\{ \mathrm{Re}\left(z\right)\geq C_{X,V}-2r\lambda_{\mathrm{max}}+\epsilon\right\} 
\]
and has discrete spectrum $\sigma_{-}\left(A\right)$ (Ruelle resonances
for the past) on the domain 
\[
\{C'_{X,V}-\epsilon\leq\mathrm{Re}\left(z\right)\leq C'_{X,V}-r\lambda_{\mathrm{min}}+\epsilon\}.
\]
\item On these corresponding domains, the bound of the resolvent is uniform
with respect to $r$.
\end{itemize}
\end{thm}

\end{cBoxB}

The proof of Theorem \ref{thm:grey-band} is given in Section \ref{sec:Proof-of-Theorem_discrete_spectrum}.
In particular the result that the resolvent is uniformly bounded (implies
no spectrum) on the outside domains is a consequence of Lemma \ref{lem:For-the-weight}
that gives bounds for the transfer operator like $\left\Vert \mathcal{L}^{t}\right\Vert _{\mathcal{H}_{W}\left(M\right)}\leq Ce^{\left(C_{X,V}+\epsilon\right)t}$
for $t\geq0$.
\begin{rem}
In Theorem \ref{thm:grey-band}, due to the margin $\epsilon>0$,
we can replace $\lambda_{\mathrm{min}},\lambda_{\mathrm{max}}$ by
respectively the minimal/maximal Lyapunov exponents that are the sup/inf
of $\lambda_{\mathrm{min}},\lambda_{\mathrm{max}}$ defined in eq.(\ref{eq:hyperbolicity}),
eq.(\ref{eq:l_max}).
\end{rem}

~
\begin{rem}
The upper estimate (\ref{eq:bound_for_ess_spec}) on the vertical
band of essential spectrum has been improved by Alexander Adam and
Viviane Baladi \cite[eq.(52)]{https://doi.org/10.48550/arxiv.1809.04062}
and by Semyon Dyatlov \cite[Theorem 2]{dyatlov_2021_pollicott_ruelle_sobolev}.
\end{rem}

~
\begin{rem}
If we set $\mathcal{H}_{W}\left(M\right)=L^{2}\left(M\right)$ in
the theorem above, the spectral set $\sigma\left(A\right)$ of the
generator $A$ is contained in the vertical band 
\begin{equation}
\{C'_{X,V}\le\mathrm{Re}\left(z\right)\le C_{X,V}\}.\label{eq:L2-spectrum}
\end{equation}
This conclusion for the special case $r=0$ can be deduced directly
from (\ref{eq:L*t*Lt}). If we move the parameter $r\rightarrow+\infty$,
then the vertical band \eqref{eq:bound_for_ess_spec} moves to the
left like a ``theater blackout curtain'' revealing the resonances
for the future, if we move the parameter $r\rightarrow-\infty$ then
the vertical band moves to the right revealing the resonances for
the past. See Figure \ref{fig:discrete_spectrum}.
\end{rem}

{\footnotesize{}}
\begin{figure}[H]
\begin{centering}
{\footnotesize{}\scalebox{0.9}[0.9]{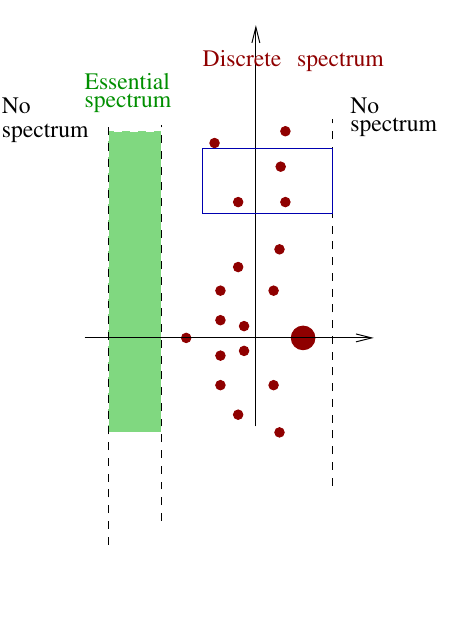}$\qquad$\scalebox{0.9}[0.9]{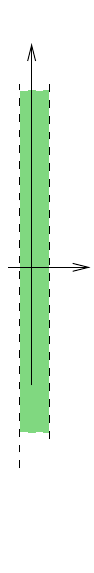}$\qquad\qquad$\scalebox{0.9}[0.9]{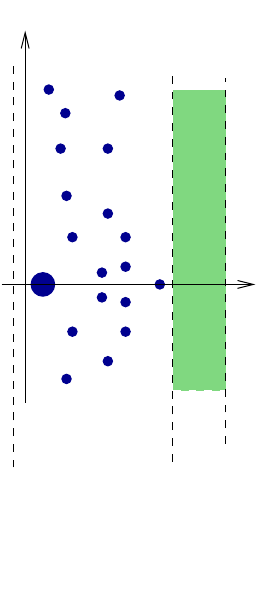}}{\footnotesize\par}
\par\end{centering}
{\footnotesize{}\caption{\label{fig:discrete_spectrum}Spectrum of $A=-X+V$ in $\mathcal{H}_{W}\left(M\right)$,
depending on $r\in\mathbb{R}$, with Anosov vector field $X$ and
potential $V\in C\left(M;\mathbb{C}\right)$.\textbf{ (a):} For $r\protect\geq0$.
$A$ has intrinsic discrete spectrum on $\mathrm{Re}\left(z\right)>C_{X,V}-r\lambda_{\mathrm{min}}$
and a bounded resolvent on $\mathrm{Re}\left(z\right)>C_{X,V}$ and
$\mathrm{Re}\left(z\right)<C'_{X,V}-2r\lambda_{\mathrm{max}}$. Theorem
\ref{thm:Weyl law} gives an upper bound for the number of resonances
in the dashed rectangle, for $\omega\gg1$. As $r\rightarrow\infty$,
the vertical band containing essential spectrum (in green) moves to
the left and may reveal some new resonances. If $V$ is real valued
then the rightmost (leading) eigenvalue is given by a topological
pressure: $\gamma_{\mathrm{Gibbs}}^{+}=\mathrm{Pr}\left(V+\mathrm{div}X_{/E_{s}}\right)\in\mathbb{R}$.\textbf{
(b):} For $r=0$. Then $\mathcal{H}_{W}\left(M\right)=L^{2}\left(M\right)$.
The spectrum is contained in a vertical band (\ref{eq:L2-spectrum}).\textbf{
(c):} For $r\protect\leq0$. Letting $r\rightarrow-\infty$ pushes
the band containing the essential spectrum to the right, revealing
Ruelle resonances for the past dynamics. If $V$ is real valued then
the leftmost (leading) eigenvalue is $-\gamma_{\mathrm{Gibbs}}^{-}=-\mathrm{Pr}\left(-V-\mathrm{div}X_{/E_{u}}\right)$.}
}{\footnotesize\par}
\end{figure}
{\footnotesize\par}
\begin{rem}
If $V$ is real-valued, the operator $A$ in (\ref{eq:generator_A})
commutes with the conjugation operator, hence the Ruelle spectrum
is symmetric with respect to the real axis, as shown in Figure \ref{fig:discrete_spectrum}.
In this case, there is a leading real eigenvalue $\gamma_{\mathrm{Gibbs}}^{+}=\mathrm{Pr}\left(V+\mathrm{div}X_{/_{E_{s}}}\right)$
called Perron eigenvalue, where $\mathrm{div}X_{/_{E_{s}}}<0$ is
the divergence of $X$ measured on $E_{s}$ with respect to some (arbitrary)
volume form. We observe this eigenvalue if we choose sufficiently
large $r>0$ in the theorem above. Similarly, by considering the time-reversed
system, we also find the leading real eigenvalue $-\gamma_{\mathrm{Gibbs}}^{-}$
with $\gamma_{\mathrm{Gibbs}}^{-}=\mathrm{Pr}\left(-V-\mathrm{div}X_{/E_{u}}\right)$,
which we may call the Perron eigenvalue for the past and observe by
letting $r<0$ be large. See Figure \ref{fig:discrete_spectrum}.
\end{rem}

~
\begin{rem}
\textbf{``Relation between future and past spectrum''}. From the
relation $\langle v|\mathcal{L}^{t}u\rangle_{L^{2}\left(M\right)}=\langle\left(\mathcal{L}^{t}\right)^{\dagger}v|u\rangle_{L^{2}\left(M\right)}$
between 
\[
\mathcal{L}^{t}=e^{t\left(-X+V\right)}\quad\text{and}\quad\left(\mathcal{L}^{t}\right)^{\dagger}=e^{t\left(X+\mathrm{div}X+\overline{V}\right)},
\]
we can make the following observation. Assume that two potential functions
$V,V'\in C^{\infty}\left(M;\mathbb{C}\right)$ satisfy the relation
$V'+\overline{V}+\mathrm{div}X=0$. Then the future spectrum $\sigma_{+}\left(A\right)$
of $A=-X+V$ is in one-to-one correspondence with the past spectrum
$\sigma_{-}\left(A'\right)$ of $A'=-X+V'$ by the relation that $\lambda\in\sigma_{+}\left(A\right)$
if and only if $\lambda':=-\bar{\lambda}\in\sigma_{-}\left(A'\right)$.
Further the respective eigenprojectors, $\Pi_{\lambda}^{+}$ for $A$
and $\Pi_{\lambda'}^{-}$ for $A'$, are related by $\Pi_{\lambda'}^{-}=(\Pi_{\lambda}^{+})^{\dagger}$.
In particular, if a potential function $V$ satisfies the condition
$\mathrm{Re}\left(V\right)=-\frac{1}{2}\mathrm{div}X$, called ``half-density
correction'', equivalently if $\mathcal{L}^{t}:L^{2}\left(M\right)\rightarrow L^{2}\left(M\right)$
is unitary from (\ref{eq:L*t*Lt}), then we have $A=A'=-X-\frac{1}{2}\mathrm{div}X+i\mathrm{Im}\left(V\right)$
and hence $\sigma_{+}\left(A\right)=-\overline{\sigma_{-}\left(A\right)}$.
\end{rem}

\begin{acknowledgement}
This work has been supported by ANR-13-BS01-0007-01, JSPS KAKENHI
JP 15H03627 and PICS n 7475.
\end{acknowledgement}

\section{\label{subsec:Anosov_def}Anosov vector field}

\subsection{Definition of Anosov vector field}

Let $M$ be a $C^{\infty}$ compact connected manifold without boundary
and let $n=\dim M-1$. Let $X$ be a $C^{\infty}$ non-vanishing vector
field on $M$. The flow on $M$ generated by the vector field $X$
is denoted by 
\begin{equation}
\phi^{t}:=\exp\left(tX\right):\quad M\rightarrow M,\quad t\in\mathbb{R}.\label{eq:flow_map}
\end{equation}
We make the assumption that $X$ is an Anosov vector field on $M$.
This means that we have a continuous splitting of the tangent bundle
\begin{equation}
TM=E_{u}\oplus E_{s}\oplus\underbrace{\mathbb{R}X}_{E_{0}},\label{eq:decomp_TM}
\end{equation}
that is invariant by the flow $\phi^{t}$ and there exist $\lambda_{\mathrm{min}}>0$,
$C>0$ and a smooth metric $g_{M}$ on $M$ such that

\begin{equation}
\left\Vert d\phi_{/E_{u}\left(m\right)}^{-t}\right\Vert _{g_{M}}\leq Ce^{-\lambda_{\mathrm{min}}t}\quad\mbox{and}\quad\left\Vert d\phi_{/E_{s}\left(m\right)}^{t}\right\Vert _{g_{M}}\leq Ce^{-\lambda_{\mathrm{min}}t}\quad\mbox{for any }t\geq0,m\in M.\label{eq:hyperbolicity}
\end{equation}
See Figure \ref{fig:Anosov-flow.}. The linear subspace $E_{u}\left(m\right),E_{s}\left(m\right)\subset T_{m}M$
are unique and called the unstable and stable space respectively.
We set $E_{0}\left(m\right):=\mathbb{R}X\left(m\right)$ and call
it the neutral direction or flow direction. In general, the maps $m\mapsto E_{u}\left(m\right)$,
$m\mapsto E_{s}\left(m\right)$ and $m\mapsto E_{u}\left(m\right)\oplus E_{s}\left(m\right)$
are only Hölder continuous. We will write 
\begin{equation}
\beta_{u},\beta_{s},\beta_{0}\in]0,1]\label{eq:Holder_exp}
\end{equation}
for the respective Hölder exponent. We have\footnote{We may expect that $\beta_{0}=\min\left(\beta_{u},\beta_{s}\right)$
for generic Anosov flows. But this equality is not true in general.
For instance, we have $\beta_{0}=1$ for contact Anosov flows, but
$\min\left(\beta_{u},\beta_{s}\right)$ will be smaller than $1$
in most of the cases because the (un)stable subspaces $E_{u}\left(m\right)$
and $E_{s}\left(m\right)$ will not be smooth.}
\begin{equation}
\beta_{0}\geq\min\left(\beta_{u},\beta_{s}\right).\label{eq:def_beta_*}
\end{equation}
See \cite{hurder-90,hasselblatt1994regularity} for estimates on $\beta_{0},\beta_{u},\beta_{s}$.

\begin{figure}[h]
\centering{}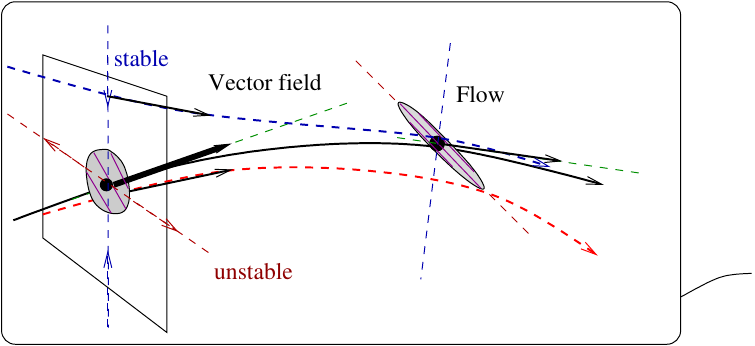\caption{\label{fig:Anosov-flow.}Anosov flow $\phi^{t}$ generated by a vector
field $X$ on a compact manifold $M$.}
\end{figure}

Let $\mathscr{A}\in C^{0}\left(M;T^{*}M\right)$ be the continuous
one form on $M$ called Anosov one form, which is defined for each
$m\in M$ by the conditions
\begin{equation}
\mathscr{A}\left(m\right)\left(X\left(m\right)\right)=1\quad\text{and}\quad\mathrm{Ker}\left(\mathscr{A}\left(m\right)\right)=E_{u}\left(m\right)\oplus E_{s}\left(m\right).\label{eq:one_form}
\end{equation}
By definition the Anosov one form $\mathscr{A}$ is preserved by the
flow $\phi^{t}$ and the map $m\in M\mapsto\mathscr{A}\left(m\right)\in T^{*}M$
is Hölder continuous with exponent $\beta_{0}$.

\subsection{\label{subsec:Transfer-operator}Transfer operator}

Let $V\in C^{\infty}\left(M;\mathbb{C}\right)$ be an arbitrary smooth
function, which is called a potential function. For $t\in\mathbb{R}$,
let us denote the time integral of $V$ along the trajectory of $m\in M$
by 
\begin{equation}
V_{\left[-t,0\right]}\left(m\right):=\int_{-t}^{0}V\left(\phi^{s}\left(m\right)\right)ds.\label{eq:time_averaged}
\end{equation}
For a given function $u\in C^{\infty}\left(M\right)$, $t\in\mathbb{R}$,
$m\in M$, we consider the forward transported and amplitude modulated
function along the trajectory:
\[
u_{t}\left(m\right):=\underbrace{e^{V_{\left[-t,0\right]}\left(m\right)}}_{\mathrm{amplitude}}\cdot\underbrace{\left(u\left(\phi^{-t}\left(m\right)\right)\right)}_{\mathrm{transport}}.
\]

\begin{center}
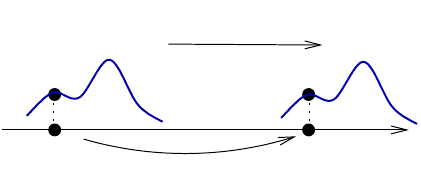
\par\end{center}

Then we have $\frac{du_{t}}{dt}=\left(-X+V\right)u_{t}$ where the
generating vector field $X$ is regarded as a first order differential
operator\footnote{In local coordinates $x=\left(x_{1},\ldots x_{n}\right)$ on $M$
we write $X=\sum_{j=1}^{\mathrm{dim}M}X_{j}\left(x\right)\frac{\partial}{\partial x^{j}}$}. In other words, we have $u_{t}=\mathcal{L}^{t}u$ for the one-parameter
group of operators $\mathcal{L}^{t}=\exp\left(tA\right)$ with generator
$A=-X+V$. This gives the following definition.

\begin{cBoxA}{}
\begin{defn}[Ruelle transfer operator]
\label{def:transfer_operator} The one-parameter group of operators
\begin{equation}
\mathcal{L}^{t}:\begin{cases}
C^{\infty}\left(M\right) & \rightarrow C^{\infty}\left(M\right)\\
u & \mapsto e^{tA}u=e^{V_{\left[-t,0\right]}}\cdot\left(u\circ\phi^{-t}\right)
\end{cases}\label{eq:def_Transfer_operator}
\end{equation}
is called Ruelle transfer operators. It is generated by the first
order differential operator on $C^{\infty}\left(M\right)$,
\begin{equation}
A:=-X+V.\label{eq:generator_A}
\end{equation}
\end{defn}

\end{cBoxA}

\begin{rem}
The results presented in this paper can be generalized for transfer
operators acting on sections of a general (complex) vector bundle
$E\rightarrow M$. We consider a linear operator acting on sections,
$A:C^{\infty}\left(M;E\right)\rightarrow C^{\infty}\left(M;E\right)$,
satisfying the Leibniz condition that for any function $f\in C^{\infty}\left(M;\mathbb{C}\right)$
and any section $s\in C^{\infty}\left(M;E\right)$, we have
\[
A\left(fs\right)=-X\left(f\right)s+fA\left(s\right).
\]
The group of transfer operators is defined by $\mathcal{L}^{t}:=e^{tA}$
for $t\in\mathbb{R}$. With respect to a local frame $\left(e_{1},\ldots e_{m}\right)$
of the vector bundle $E$ of rank $m$, a section is expressed as
$s\left(x\right)=\sum_{j=1}^{m}u^{j}\left(x\right)e_{j}\left(x\right)$
with components $u^{j}\left(x\right)\in\mathbb{C}$. Then the operator
$A$ is expressed as 
\begin{equation}
\left(As\right)\left(x\right)=\sum_{j=1}^{m}\left(-\left(Xu^{j}\right)\left(x\right)+\sum_{k=1}^{m}V_{k}^{j}\left(x\right)u^{k}\left(x\right)\right)e_{j}\left(x\right)\label{eq:A_bundle}
\end{equation}
with a matrix of potential functions $V_{k}^{j}\left(x\right)\in\mathbb{C}$
defined by $\left(Ae_{k}\right)\left(x\right)=\sum_{j}V_{k}^{j}\left(x\right)e_{j}\left(x\right)$.
The expression (\ref{eq:A_bundle}) generalizes (\ref{eq:generator_A}).
\end{rem}

\subsubsection*{Inverse and $L^{2}$-adjoint transfer operators}

Recall the time-averaged notation (\ref{eq:time_averaged}). The inverse
operator $\mathcal{L}^{-t}$ that satisfies $\mathcal{L}^{-t}\circ\mathcal{L}^{t}=\mathrm{Id}_{/C^{\infty}\left(M\right)}$
is given by 
\[
\mathcal{L}^{-t}v\underset{(\ref{eq:def_Transfer_operator})}{=}e^{-V_{\left[-t,0\right]}\circ\phi^{t}}\left(v\circ\phi^{t}\right)=e^{-V_{\left[0,t\right]}}\left(v\circ\phi^{t}\right).
\]

Using an arbitrary smooth measure $dm$ on $M$ we define the $L^{2}$
scalar product: $\langle u|v\rangle_{L^{2}\left(M,dm\right)}:=\int_{M}\overline{u}vdm$
for $u,v\in C^{\infty}\left(M\right)$ and completion gives the space
$L^{2}\left(M,dm\right)$. We define the formal adjoint operator $\left(\mathcal{L}^{t}\right)^{\dagger}:C^{\infty}\left(M\right)\rightarrow C^{\infty}\left(M\right)$
by 
\[
\langle u|\left(\mathcal{L}^{t}\right)^{\dagger}v\rangle_{L^{2}}=\langle\mathcal{L}^{t}u|v\rangle_{L^{2}},\forall u,v\in C^{\infty}\left(M\right).
\]
We deduce that\footnote{With the change of variables $m'=\phi^{-t}\left(m\right)$ we get
\begin{align*}
\langle u|\left(\mathcal{L}^{t}\right)^{\dagger}v\rangle_{L^{2}} & =\langle\mathcal{L}^{t}u|v\rangle_{L^{2}}=\int e^{\overline{V_{\left[-t,0\right]}\left(m\right)}}\overline{u\left(\phi^{-t}\left(m\right)\right)}v\left(m\right)dm\\
 & =\int e^{\overline{V_{\left[-t,0\right]}\left(\phi^{t}\left(m'\right)\right)}}\overline{u\left(m'\right)}v\left(\phi^{t}\left(m'\right)\right)\left|\mathrm{det}d\phi^{t}\left(m'\right)\right|dm'.
\end{align*}
}
\begin{align}
\left(\mathcal{L}^{t}\right)^{\dagger}v & =e^{\overline{V_{\left[0,t\right]}}}\left|\mathrm{det}d\phi^{t}\right|\left(v\circ\phi^{t}\right)\label{eq:L2-adjoint_Lt}\\
 & =e^{\left(\overline{V}+\mathrm{div}X\right)_{\left[0,t\right]}}\left(v\circ\phi^{t}\right)
\end{align}
where $\mathrm{div}X$ is the divergence of the vector field $X$
with respect to $dm$. So 
\[
\left(\mathcal{L}^{t}\right)^{\dagger}=e^{\left(2\mathrm{Re}\left(V\right)+\mathrm{div}X\right)_{\left[0,t\right]}}\mathcal{L}^{-t}
\]
and
\begin{equation}
\left(\mathcal{L}^{t}\right)^{\dagger}\circ\mathcal{L}^{t}=\mathcal{M}_{\exp\left(2\varphi_{\left[0,t\right]}\right)}\label{eq:L*t*Lt}
\end{equation}
is a multiplication operator by the function $\exp\left(2\varphi_{\left[0,t\right]}\right)\in C^{\infty}\left(M;\mathbb{C}\right)$
with 
\begin{equation}
\varphi=\frac{1}{2}\mathrm{div}X+\mathrm{Re}\left(V\right).\label{eq:def_At}
\end{equation}
This expression explains the result (\ref{eq:L2-spectrum}) that expresses
the spectrum of $A$ in $L^{2}\left(M\right)$ and will be used again
in the paper (in the proof of Lemma \ref{lem:For-the-weight} and
Theorem \ref{thm:decay_outside_trapped_set}).

\section{\label{sec:Semiclassical-analysis-with}Semi-classical analysis with
wave-packets}

In this section, we develop general tools and lemmas for a flow generated
by a \textbf{non vanishing smooth vector field $X$ }on a compact
manifold $M$. So we do not assume that $X$ is Anosov. In Subsection
\ref{subsec:Resolution-of-identity}, we introduce a wave-packet transform
that gives a representation of distributions $u\in\mathcal{D}'\left(M\right)$
as smooth functions on the cotangent space $T^{*}M$. The definition
of our wave-packet transform is based on a metric $g$ on $T^{*}M$
that has a nice property called ``slowly varying''. We will introduce
the metric $g$ in Subsection \ref{subsec:A-metric}. In Subsection
\ref{subsec:Pseudo-differential-operator}, we give the definition
of pseudo-differential operators (PDO) using the wave-packet transform
and prove a useful theorem on compositions of PDO. In Subsection \ref{subsec:Sobolev-space},
we define the Sobolev space $\mathcal{H}_{W}\left(M\right)$ associated
to a weight function $W$ on $T^{*}M$. In Section \ref{subsec:A-matrix-elements},
we give a fundamental ``micro-local property'' of the transfer operator
$\mathcal{L}^{t}$ and prove a version of Egorov's theorem. The former
micro-local property shows that the kernel of the operator induced
on the phase space by $\mathcal{L}^{t}$ decays very fast on the outside
of the graph of the lifted flow $\tilde{\phi}^{t}:T^{*}M\rightarrow T^{*}M$.
In Subsection \ref{subsec:Strongly-continuous-semi-group}, we show
that the transfer operators $\mathcal{L}^{t}$ form a strongly continuous
(semi-)group on $\mathcal{H}_{W}\left(M\right)$ if the weight function
$W$ satisfies some reasonable conditions with respect to the lifted
flow $\tilde{\phi}^{t}:T^{*}M\rightarrow T^{*}M$ .

The key results of this section are the resolution of identity in
$C^{\infty}\left(M\right)$ given in Lemma \ref{prop:resolution_identity_CM}
and micro-locality of the transfer operator given in Lemma \ref{thm:Microlocality-of-the_TO}.

\subsection{\label{subsec:Resolution-of-identity}Wave packets transform and
resolution of identity in $T^{*}M$}

\subsubsection{Flow box coordinates}

In this section we first introduce charts on the manifold $M$ and
a partition of unity in order to decompose each function on $M$ into
those supported on a single chart. Below we write $\mathbb{R}_{x}^{n}$
to indicate that we use the variable name $x$ for points on $\mathbb{R}^{n}$
and write $\mathbb{B}_{x}^{n}\left(c\right):=\left\{ x\in\mathbb{R}^{n},\left|x\right|<c\right\} $
for the open ball of radius $c>0$ in $\mathbb{R}_{x}^{n}$.
\begin{lem}[Flow box coordinates]
\label{lem:charts}\cite[p.33]{taylor_tome1}. For $c>0$ and $l>0$
small enough, there exist open subsets $U_{j}\subset M$ such that
$M=\bigcup_{j=1}^{J}U_{j}$ and $C^{\infty}$ local charts diffeomorphism
\begin{equation}
\kappa_{j}:\begin{cases}
U_{j}\subset M & \rightarrow V_{j}=\mathbb{B}_{x}^{n}\left(c\right)\times\mathbb{B}_{z}^{1}\left(l\right)\subset\mathbb{R}_{x}^{n}\times\mathbb{R}_{z}\\
m & \to y=\left(x,z\right)
\end{cases}\label{eq:def_kappa_j}
\end{equation}
such that 
\begin{align}
\left(\kappa_{j}\right)_{*}\left(-X\right) & =\frac{\partial}{\partial z}.\label{eq:X}
\end{align}
and the $\kappa_{j}$ can be extended to a small neighborhood of $U_{j}$
(such coordinates are said to be admissible). 
\end{lem}

The next lemma introduces the operators $I$ and $I^{*}$ that decompose
and reconstruct functions with respect to the charts (defined in Lemma
\ref{lem:charts}).
\begin{lem}[Quadratic partition of unity]
\label{lem:I-I*}For the local charts $\kappa_{j}:U_{j}\to V_{j}$,
$1\le j\le J$, given in Lemma \ref{lem:charts}, there exist functions
$\chi_{j}\in C_{0}^{\infty}\left(V_{j};\mathbb{R}^{+}\right)$ for
$1\le j\le J$, which give a quadratic partition of unity in the sense
that
\begin{equation}
\sum_{j:\,m\in U_{j}}\left(\left(\chi_{j}\circ\kappa_{j}\right)\left(m\right)\right)^{2}\left|\mathrm{det}d\kappa_{j}\left(m\right)\right|=1\quad\text{for all }m\in M\label{eq:partition_quadratic_unity-1}
\end{equation}
where $\left|\mathrm{det}d\kappa_{j}\left(m\right)\right|=d\kappa_{j}^{*}\left(\mathrm{Leb}\right)/dm$.
For every $j\in\left\{ 1,\ldots J\right\} $, let
\begin{equation}
I_{j}:\begin{cases}
C^{\infty}\left(M\right) & \rightarrow C_{0}^{\infty}\left(V_{j}\right)\\
u & \mapsto v_{j}\left(y\right):=\chi_{j}\left(y\right)\cdot\left(u\circ\kappa_{j}^{-1}\right)\left(y\right),
\end{cases}\label{eq:def_I-2}
\end{equation}
and
\begin{equation}
I:=\left(I_{j}\right)_{j}\quad:C^{\infty}\left(M\right)\rightarrow\bigoplus_{j=1}^{J}C_{0}^{\infty}\left(V_{j}\right).\label{eq:def_I_1}
\end{equation}
Then the $L^{2}$-adjoint of $I:L^{2}\left(M,dm\right)\rightarrow\bigoplus_{j}L^{2}\left(V_{j},dy\right)$
is given by
\begin{equation}
I^{\dagger}:\begin{cases}
\bigoplus_{j}C_{0}^{\infty}\left(V_{j}\right) & \rightarrow C^{\infty}\left(M\right)\\
v=\left(v_{j}\right)_{j} & \mapsto u\left(m\right)=\sum_{j}\left(\chi_{j}\circ\kappa_{j}\right)\left(m\right)\cdot\left(v_{j}\circ\kappa_{j}\right)\left(m\right)\left|\mathrm{det}d\kappa_{j}\left(m\right)\right|
\end{cases}\label{eq:def_I*-2}
\end{equation}
and we have
\begin{equation}
I^{\dagger}\circ I=\mathrm{Id}_{/C^{\infty}\left(M\right)}.\label{eq:I*_I}
\end{equation}
\end{lem}

\begin{proof}
Let us consider functions $\chi_{j}^{\left(0\right)}\in C_{0}^{\infty}\left(V_{j};\mathbb{R}^{+}\right)$
for $1\leq j\leq J$ such that 
\[
S\left(m\right):=\sum_{j:\,m\in U_{j}}\left(\left(\chi_{j}^{\left(0\right)}\circ\kappa_{j}\right)\left(m\right)\right)^{2}\left|\mathrm{det}d\kappa_{j}\left(m\right)\right|>0
\]
 for every $m\in M$. Then $\chi_{j}\left(x\right):=\chi_{j}^{\left(0\right)}\left(x\right)/\sqrt{S\left(\kappa_{j}\left(x\right)\right)}$
satisfies (\ref{eq:partition_quadratic_unity-1}). For $u\in C^{\infty}\left(M\right)$
and $v=(v_{j})\in\bigoplus_{j}C_{0}^{\infty}\left(V_{j}\right)$,
we have
\begin{align*}
\left\langle Iu|v\right\rangle _{\bigoplus_{j}L^{2}\left(V_{j}\right)} & \underset{(\ref{eq:def_I-2})}{=}\sum_{j}\left(\chi_{j}\cdot\left(u\circ\kappa_{j}^{-1}\right),v_{j}\right)_{L^{2}\left(V_{j}\right)}=\sum_{j}\int_{V_{j}}\chi_{j}\cdot\overline{\left(u\circ\kappa_{j}^{-1}\right)}\cdot v_{j}dy\\
 & =\int_{M}\sum_{j}\left(\chi_{j}\circ\kappa_{j}\right)\cdot\overline{u}\cdot\left(v_{j}\circ\kappa_{j}\right)\left|\mathrm{det}d\kappa_{j}\left(m\right)\right|dm\underset{(\ref{eq:def_I*-2})}{=}\left\langle u|I^{\dagger}v\right\rangle _{L^{2}\left(M\right)}
\end{align*}
and
\begin{align*}
\left(I^{\dagger}\circ I\right)u & \underset{(\ref{eq:def_I-2})(\ref{eq:def_I*-2})}{=}\sum_{j}\left(\chi_{j}\circ\kappa_{j}\right)\cdot\left(\left(\chi_{j}\circ\kappa_{j}\right)\cdot u\right)\left|\mathrm{det}d\kappa_{j}\right|\\
 & =u\sum_{j}\left(\chi_{j}\circ\kappa_{j}\right)^{2}\left|\mathrm{det}d\kappa_{j}\right|\underset{(\ref{eq:partition_quadratic_unity-1})}{=}u.
\end{align*}
This completes the proof.
\end{proof}
In the following, we will consider the flow box coordinates $\kappa_{j}$
on the manifold $M$ and the associated quadratic partition of unity
$\chi_{j}$ given by the two lemmas above.
\begin{rem}
\label{rem:Projection}We will find relations similar to (\ref{eq:I*_I})
a few times in the course of the argument below. Here we note that
the relation (\ref{eq:I*_I}) can be interpreted as follows. The operator
$I:L^{2}\left(M,dm\right)\rightarrow\bigoplus_{j}L^{2}\left(V_{j},dy\right)$
is an isometry onto this image because $\left\Vert Iu\right\Vert ^{2}=\langle Iu|Iu\rangle=\langle u|I^{\dagger}Iu\rangle\underset{(\ref{eq:I*_I})}{=}\left\Vert u\right\Vert ^{2}$.
If we define 
\[
P:=I\circ I^{\dagger}\quad:\bigoplus_{j}L^{2}\left(V_{j},dy\right)\rightarrow\bigoplus_{j}L^{2}\left(V_{j},dy\right),
\]
we have $P^{\dagger}=P$ and $P^{2}=P$ meaning that $P$ is the orthogonal
projector on $\mathrm{Im}\left(I\right)$. For the transfer operator
$\mathcal{L}^{t}:C^{\infty}\left(M\right)\rightarrow C^{\infty}\left(M\right)$
(or more generally for any linear operator), we consider the operator
\[
\tilde{\mathcal{L}}^{t}:=I\circ\mathcal{L}^{t}\circ I^{\dagger}.
\]
This is the simplest extension of $\mathcal{L}^{t}$ in the sense
that $\tilde{\mathcal{L}}^{t}$ preserves the decomposition $\oplus_{j}C_{0}^{\infty}\left(V_{j}\right)=\mathrm{Im}\left(I\right)\oplus\mathrm{Ker}\left(P\right)$,
its restriction $\tilde{\mathcal{L}^{t}}_{/\mathrm{Im}\left(I\right)}$
is conjugated to $\mathcal{L}^{t}$ and the restriction $\tilde{\mathcal{L}^{t}}_{/\mathrm{Ker}\left(P\right)}$
is the null operator. In other terms, $\tilde{\mathcal{L}^{t}}=\left(I\circ\mathcal{L}^{t}\circ I^{-1}\right)\oplus0$
where $I^{-1}=I_{/\mathrm{Im}\left(I\right)}^{\dagger}$. We may therefore
regard the operator $\tilde{\mathcal{L}}^{t}$ as a lifted representative
(or a trivial extension) of $\mathcal{L}^{t}$.
\end{rem}

\begin{defn}[Local coordinates on $T^{*}M$]
 \label{def:local_charts_on_T*M}Let $\kappa_{j}:m\in U_{j}\mapsto y=\left(x,z\right)\in V_{j}\subset\mathbb{R}^{n}\times\mathbb{R}$
be the local flow-box coordinates in Lemma \ref{lem:charts}. We write
$\eta=\left(\xi,\omega\right)\in\mathbb{R}^{n}\times\mathbb{R}$ for
the dual coordinates of $y=\left(x,z\right)\in\mathbb{R}^{n}\times\mathbb{R}$.
Then the map $\kappa_{j}$ in (\ref{eq:def_kappa_j}) has a canonical
extension to the cotangent bundle: 
\begin{equation}
\tilde{\kappa_{j}}:\begin{cases}
T^{*}U_{j} & \rightarrow T^{*}V_{j}=\mathbb{B}_{x}^{n}\left(c\right)\times\mathbb{B}_{z}^{1}\left(l\right)\times\mathbb{R}_{\xi}^{n}\times\mathbb{R}_{\omega}\\
\rho & \mapsto\varrho=\left(y,\eta\right)=\left(\left(x,z\right),\left(\xi,\omega\right)\right)
\end{cases}.\label{eq:def_k_tilde_j}
\end{equation}
We will regard these as local coordinates on $T^{*}M$. We will henceforth
use the notation $\rho\in T^{*}U_{j}$ for a point on $T^{*}M$ and
$\varrho=\left(y,\eta\right)\in T^{*}V_{j}$ for the corresponding
point in a local chart as above. The canonical volume form on the
cotangent bundle $T^{*}M$ will be denoted by 
\begin{equation}
d\rho:=\left(\tilde{\kappa_{j}}\right)^{*}\left(d\varrho\right)\qquad\text{with }d\varrho:=\left(\Pi_{k=1}^{n}dx_{k}\wedge d\xi_{k}\right)\wedge dz\wedge d\omega.\label{eq:def_d_rho}
\end{equation}
\end{defn}

\subsubsection{\label{subsec:A-metric}A global metric $g$ on $T^{*}M$}

For a given chart index $j$, using the local charts (\ref{eq:def_k_tilde_j}),
we have already defined the metric $g$ on $T^{*}\mathbb{R}^{n+1}=\mathbb{R}_{y}^{n+1}\times\mathbb{R}_{\eta}^{n+1}$
in (\ref{eq:metric_g_in_coordinates}). We write $g_{j}:=\kappa_{j}^{*}\left(g\right)$
for the metric induced on $T^{*}U_{j}$, i.e. at a given point $\rho\in T^{*}U_{j}$,
it is given by
\[
g_{j,\rho}\left(u_{1},u_{2}\right):=g_{\varrho}\left(d\tilde{\kappa_{j}}\left(u_{1}\right),d\tilde{\kappa_{j}}\left(u_{2}\right)\right)\quad\text{for }u_{1},u_{2}\in T_{\rho}\left(T^{*}M\right).
\]
We will prove below in (\ref{eq:equiv_g_tilde}) that the metric $g_{j}$
and $g_{j'}$ are uniformly equivalent on $U_{j}\cap U_{j'}$. Then
we define the following global metric $g$ on $T^{*}M$ (abusively
we use the same letter $g$)
\begin{equation}
g:=\sum_{j=1}^{J}\left(\chi_{j}^{2}\circ\kappa_{j}\circ\pi\right)\cdot g_{j}\label{eq:global_metric}
\end{equation}
where $\pi:T^{*}M\to M$ denotes the bundle projection. The global
metric $g$ in (\ref{eq:global_metric}) will provide the size of
wave packets to our analysis of the transfer operator over the manifold
$M$. We will write $\mathrm{dist}_{g}\left(\rho,\rho'\right)$ for
the geodesic distance between $\rho,\rho'\in T^{*}M$ defined from
this metric $g$.

We will give now few essential properties of the metric $g$.

\begin{cBoxB}{}
\begin{lem}
\label{thm:The-metric-}The Riemann metric $g$ on $T^{*}M$ is \textbf{geodesically
complete}.
\end{lem}

\end{cBoxB}

\begin{proof}
Let $g_{M}$ be a Riemann metric on $M$. Consider the function $\Gamma:T^{*}M\rightarrow\mathbb{R}^{+}$
given by 
\begin{equation}
\Gamma\left(\rho\right)=\left\Vert \rho\right\Vert _{g_{M}}^{1-\alpha^{\perp}}\asymp\left|\eta\right|^{1-\alpha^{\perp}},\label{eq:def_Gamma}
\end{equation}
using local coordinates $\rho=\left(y,\eta\right)$ with $\left|\eta\right|>1$.
From (\ref{eq:conditions}), we have $1-\alpha^{\perp}>0$ so that
$\Gamma$ increases with $\left\Vert \rho\right\Vert _{g_{M}}$ .
There exists $C,C'>0$ such that for any $v\in TT^{*}M$, with local
components $\left(v_{y},v_{\eta}\right)$, we have, using also $1-\alpha^{\perp}\leq\alpha^{\perp}$
from (\ref{eq:conditions}),
\begin{equation}
\left|d\Gamma\left(v\right)\right|\underset{(\ref{eq:def_Gamma})}{\leq}C'\left(\left|\eta\right|^{1-\alpha^{\perp}}\left|v_{y}\right|+\left|\eta\right|^{-\alpha^{\perp}}\left|v_{\eta}\right|\right)\underset{(\ref{eq:conditions})}{\leq}C'\left(\left|\eta\right|^{\alpha^{\perp}}\left|v_{y}\right|+\left|\eta\right|^{-\alpha^{\perp}}\left|v_{\eta}\right|\right)\underset{(\ref{eq:metric_g_in_coordinates})}{\leq}C\left\Vert v\right\Vert _{g}.\label{eq:bound-3}
\end{equation}
For any $A>0$, the set $R_{A}:=\left\{ \rho\in T^{*}M,\Gamma\left(\rho\right)\in\left[0,A\right]\right\} $
is compact. Take any point $\rho\in T^{*}M$ and consider $\gamma:t\mapsto\gamma\left(t\right)\in T^{*}M$
a unit speed geodesic starting at $\rho$, i.e. $\left\Vert \dot{\gamma}\right\Vert _{g}=1$.
Then $\left|d\Gamma\left(\dot{\gamma}\right)\right|\underset{(\ref{eq:bound-3})}{\leq}C\left\Vert \dot{\gamma}\right\Vert _{g}\leq C$
so for any $t\geq0$, 
\[
\left|\Gamma\left(\gamma\left(t\right)\right)-\Gamma\left(\gamma\left(0\right)\right)\right|=\left|\int_{0}^{t}d\Gamma\left(\dot{\gamma}\right)dt'\right|<Ct
\]
i.e. $\gamma\left(t\right)\in R_{A}$ with $A=\Gamma\left(\gamma\left(0\right)\right)+Ct$.
This shows that the time to reach the boundary of $R_{A}$ tends to
$\infty$ as $A\rightarrow\infty$. This implies that the geodesic
is defined for every $t\in\mathbb{R}$, i.e. the metric $g$ is geodesically
complete.
\end{proof}

\subsubsection{\label{subsec:Lipschitz-property-of}Lipschitz property of $g$ with
respect to the flow map and change of charts.}

The push-forward action of the flow $\phi^{t}=e^{tX}:M\rightarrow M$
on the cotangent space $T^{*}M$ is denoted by
\begin{equation}
\tilde{\phi}^{t}:\begin{cases}
T^{*}M & \rightarrow T^{*}M,\\
\rho & \mapsto\rho\left(t\right):=\left(\left(d_{m}\phi^{t}\right)^{*}\right)^{-1}\rho.
\end{cases}\label{eq:lifted_flow}
\end{equation}
with $m=\pi\left(\rho\right)$. In this section we will prove the
following Lemma.

\begin{cBoxB}{}
\begin{lem}[Lipschitz property of $g$]
\label{lem:Lipschitz_property}From the conditions (\ref{eq:conditions})
we have that for any $t\in\mathbb{R}$, there exists a constant $C_{t}>0$
such that for any $\rho,\rho'\in T^{*}M$,
\begin{equation}
\mathrm{dist}_{g}\left(\tilde{\phi}^{t}(\rho),\tilde{\phi}^{t}(\rho')\right)\le C_{t}\mathrm{dist}_{g}\left(\rho,\rho'\right).\label{eq:invariance_g}
\end{equation}
\end{lem}

\end{cBoxB}

\begin{rem}
Property (\ref{eq:invariance_g}) will be crucial to get Lemma \ref{lem:Description-of-evolving}
below. This property means that the pulled back metric $\left(\tilde{\phi}^{t}\right)^{\circ}g$
is equivalent to the metric $g$ uniformly with respect to $\rho\in T^{*}M$.
Equivalently one can write
\[
\forall t\in\mathbb{R},\exists C_{t}>0,\forall\rho\in T^{*}M,\quad\left\Vert d\tilde{\phi}^{t}\left(\rho\right)\right\Vert _{g}\leq C_{t},
\]
or infinitesimally with the vector field $\tilde{X}$ that is $X$
lifted on $T^{*}M$ (i.e. the generator of $\tilde{\phi}^{t}$) acting
as the Lie derivative $\mathcal{L}_{\tilde{X}}$ on $g$,
\[
\exists C>0,\forall\rho\in T^{*}M,\quad\left\Vert \left(\mathcal{L}_{\tilde{X}}g\right)\left(\rho\right)\right\Vert _{g}\leq C,
\]
and this is true more generally for any local change of flow box coordinate
as given below in (\ref{eq:coordinate_change}) or for a local diffeomorphism
on $M$ that preserves the vector field $X$.
\end{rem}

\begin{proof}
In order to prove Lemma \ref{lem:Lipschitz_property}, remark first
that for any chart indices $j,j'\in\left\{ 1,\ldots J\right\} $ with
$U_{j}\cap U_{j'}\neq\emptyset$, the expression of the metric in
(\ref{eq:metric_g_in_coordinates}) gives different metrics $g_{j}=\kappa_{j}^{*}\left(g\right)$,
$g_{j'}=\kappa_{j'}^{*}\left(g\right)$ on different charts. We will
show that conditions (\ref{eq:conditions}) guaranty that these metrics
are uniformly equivalent in the sense that $\exists C>0,\forall\rho\in T^{*}\left(U_{j}\cap U_{j'}\right)$,
$\forall u\in T_{\rho}\left(T^{*}M\right)$,
\begin{equation}
\frac{1}{C}\left\Vert u\right\Vert _{g_{j}}\leq\left\Vert u\right\Vert _{g_{j'}'}\leq C\left\Vert u\right\Vert _{g_{j}}.\label{eq:equiv_g_tilde}
\end{equation}
\begin{rem}
\label{rem:asympt}We will express the relation (\ref{eq:equiv_g_tilde})
as $\left\Vert u\right\Vert _{g_{j}}\asymp\left\Vert u\right\Vert _{g_{j'}}$
by using the notation $\asymp$.
\end{rem}

Let us consider a local change of coordinates $\left(x,z\right)\mapsto\left(x',z'\right)$
between the flow box coordinates in Lemma \ref{lem:charts}. It is
written in the form
\begin{equation}
x'=f\left(x\right),\quad z'=z+h\left(x\right)\label{eq:coordinate_change}
\end{equation}
using a smooth diffeomorphism $f:\mathbb{R}^{n}\rightarrow\mathbb{R}^{n}$
and a smooth function $h:\mathbb{R}^{n}\rightarrow\mathbb{R}$. By
compactness of $M$, we may and do assume that $f$ and $h$ are bounded
in the $C^{2}$ sense. From \eqref{eq:coordinate_change}, we have
\begin{equation}
dx'=\left(d_{x}f\right)dx,\quad dz'=dz+\left(d_{x}h\right)dx.\label{eq:relation_coordinates}
\end{equation}
Let us write $(x,z,\xi,\omega)\to(x',z',\xi',\omega')$ for the induced
coordinate change on the cotangent bundle. Then the relation
\[
\xi dx+\omega dz=\xi'dx'+\omega'dz'=\left(\left(d_{x}f\right)^{T}\xi'+\left(d_{x}h\right)\omega'\right)dx+\omega'dz
\]
gives
\begin{equation}
\xi=\left(d_{x}f\right)^{T}\xi'+\left(d_{x}h\right)\omega'\quad\text{and}\quad\omega=\omega'.\label{eq:dd}
\end{equation}
Hence we find
\begin{align}
d\xi & =\left(\left(d_{x'}\left(d_{x}f\right)^{T}\right)\xi'+\left(d_{x'}\left(d_{x}h\right)\right)\omega'\right)dx'+\left(d_{x}f\right)^{T}d\xi'+\left(d_{x}h\right)d\omega',\label{eq:relation_coordinates2}\\
d\omega & =d\omega'.\nonumber 
\end{align}
From $C^{2}$ boundedness of $f$ and $h$, we have $\left\langle \eta\right\rangle \asymp\left\langle \eta'\right\rangle $
for $\eta=(\xi,\omega)$ and $\eta'=(\xi',\omega')$ and hence 
\begin{equation}
\delta^{\perp}\left(\eta\right)\asymp\delta^{\perp}\left(\eta'\right),\quad\delta^{\parallel}\left(\eta\right)\asymp\delta^{\parallel}\left(\eta'\right).\label{eq:equivalence_of_delta}
\end{equation}
To compare $g$ and $g'$, we rewrite the relations (\ref{eq:relation_coordinates})
and (\ref{eq:relation_coordinates2}) as
\begin{align*}
\frac{dx'}{\delta^{\perp}\left(\eta\right)} & =d_{x}f\left(\frac{dx}{\delta^{\perp}\left(\eta\right)}\right)\asymp\frac{dx}{\delta^{\perp}\left(\eta\right)},\\
\frac{dz'}{\delta^{\parallel}\left(\eta\right)} & =\frac{dz}{\delta^{\parallel}\left(\eta\right)}+\frac{\delta^{\perp}\left(\eta\right)}{\delta^{\parallel}\left(\eta\right)}\cdot d_{x}h\left(\frac{dx}{\delta^{\perp}\left(\eta\right)}\right),\\
\delta^{\perp}\left(\eta\right)d\xi & =\left((\delta^{\perp}\left(\eta\right))^{2}\left(d_{x'}\left(d_{x}f\right)^{T}(\xi')+d_{x'}\left(d_{x}h(\omega')\right)\right)\right)\cdot\frac{dx'}{\delta^{\perp}\left(\eta\right)}\\
 & \qquad\qquad+\delta^{\perp}\left(\eta\right)\left(d_{x}f\right)^{T}\left(d\xi'\right)+\frac{\delta^{\perp}\left(\eta\right)}{\delta^{\parallel}\left(\eta\right)}\left(d_{x}h\right)\left(\delta^{\parallel}\left(\eta\right)d\omega'\right).
\end{align*}
Therefore, in order that $g_{\varrho}\preceq g'_{\varrho'}$, a necessary
and sufficient condition is that 
\begin{equation}
(\delta^{\perp}\left(\eta\right))^{2}\left|\xi'\right|\leq C\quad\text{and}\quad(\delta^{\perp}\left(\eta\right))^{2}\left|\omega'\right|\leq C\quad\text{and}\quad\frac{\delta^{\perp}\left(\eta\right)}{\delta^{\parallel}\left(\eta\right)}\leq C.\label{eq:requirement}
\end{equation}
Since the former two inequalities are written 
\[
\delta^{\perp}\left(\eta\right)\leq C\min\left\{ \left|\xi'\right|^{-1/2},\left|\omega'\right|^{-1/2}\right\} 
\]
we see that the last condition \eqref{eq:requirement} is equivalent
to the condition
\[
1/2\leq\alpha^{\perp}\quad\text{and}\quad0\leq\alpha^{\parallel}\leq\alpha^{\perp}.
\]
that was assumed in (\ref{eq:conditions}). We have obtained (\ref{eq:equiv_g_tilde}).
Since the flow $\phi^{t}$ viewed in the local charts $\kappa_{j}$
is also written in the form \eqref{eq:coordinate_change} we get Lemma
\ref{lem:Lipschitz_property}.
\end{proof}

\subsubsection{Moderate and temperate properties of the metric $g$}

Now we will discuss so called moderate and temperate properties of
the metric $g$.
\begin{rem}
The metric $g$ that we have introduced is similar to the so-called
symplectic metric $g^{\natural}$ introduced in semi-classical analysis
by Hörmander, see \cite[Def 2.2.19 page 78]{lerner2011metrics}\cite[chap XVIII.]{hormander_3}\cite{nicola_rodino_livre_11}.
Temperate, slowly varying and moderate properties of the metric are
discussed in these books for the purpose of semi-classical analysis.

\begin{cBoxA}{}
\begin{defn}[Distortion function $\Delta:T^{*}M\rightarrow\mathbb{R}^{+}$]
For $\rho\in T^{*}M$, set
\begin{equation}
\Delta\left(\rho\right):=\left\langle \left\Vert \rho\right\Vert _{g_{M}}\right\rangle ^{-\left(1-\alpha^{\perp}\right)}.\label{eq:def_distortion_function_D}
\end{equation}
where $g_{M}$ is a Riemannian metric on $M$.
\end{defn}

\end{cBoxA}
\end{rem}

Observe that $\Delta\left(\rho\right)\rightarrow0$ when $\left\Vert \rho\right\Vert _{g_{M}}\rightarrow\infty$.
For practical purpose we can also define a distortion function in
local chart using the same expression:
\begin{defn}[Distortion function $\Delta:\mathbb{R}^{2\left(n+1\right)}\rightarrow\mathbb{R}^{+}$
in local chart]
For $\varrho=\left(y,\eta\right)\in\mathbb{R}^{2\left(n+1\right)}$,
we set
\begin{equation}
\Delta\left(\varrho\right):=\delta^{\perp}\left(\eta\right)^{\left(1/\alpha^{\perp}\right)-1}\eq{\ref{eq:def_delta}}\left\langle \left|\eta\right|\right\rangle ^{-\left(1-\alpha^{\perp}\right)}.\label{eq:def_distortion_function_D-1}
\end{equation}

From (\ref{eq:equivalence_of_delta}), we have that the equivalence
$\Delta\left(\rho\right)\asymp\Delta\left(\varrho\right)$ for $\varrho=\left(y,\eta\right)=\tilde{\kappa}_{j}\left(\rho\right)$
with any chart index $j$.
\end{defn}

The next lemma shows that the distortion function $\Delta\left(\varrho\right)$
is related to the variation of the metric $g$ on $T^{*}\mathbb{R}^{n+1}=\mathbb{R}^{2\left(n+1\right)}$
with respect to itself. Recall the notation $\left\langle s\right\rangle :=\left(1+s^{2}\right)^{1/2}$
for $s\in\mathbb{R}$ introduced in (\ref{eq:def_Japonese_bracket}).

\begin{cBoxB}{}
\begin{lem}[The metric $g$ is $\Delta^{\gamma}$-moderate and temperate]
\label{lem:temperate_metric} For any $0\leq\gamma<1$, there exist
$N>0$ and $C>0$ such that
\begin{align}
\max\left\{ \frac{\|v\|_{g_{\varrho'}}}{\|v\|_{g_{\varrho}}},\frac{\|v\|_{g_{\varrho}}}{\|v\|_{g_{\varrho'}}}\right\} \le1+C\Delta(\varrho)^{1-\gamma}\left\langle \left(\Delta\left(\varrho\right)\right)^{\gamma}\left\Vert \varrho'-\varrho\right\Vert _{g_{\varrho}}\right\rangle ^{N}\label{eq:g_moderate and temperate}
\end{align}
for any $\varrho,\varrho'\in\mathbb{R}^{2\left(n+1\right)}$ and $v\in\mathbb{R}^{2\left(n+1\right)}$.
Consequently we have
\begin{align}
\frac{1}{C}\left\langle \left\Vert \varrho'-\varrho\right\Vert _{g_{\varrho}}\right\rangle ^{1/N}\leq & \left\langle \left\Vert \varrho'-\varrho\right\Vert _{g_{\varrho'}}\right\rangle \le C\left\langle \left\Vert \varrho'-\varrho\right\Vert _{g_{\varrho}}\right\rangle ^{N}.\label{eq:equivalence_distance}
\end{align}
\end{lem}

\end{cBoxB}

\begin{rem}
The inequality (\ref{eq:g_moderate and temperate}) expresses two
properties about the variation of $g$:
\begin{enumerate}
\item the moderate property (also called slowly varying property): if two
points $\varrho,\varrho'$ are in a distance $\left\Vert \varrho'-\varrho\right\Vert _{g_{\varrho}}\leq\Delta^{\gamma}\left(\varrho\right)^{-1}$
then the ratio between $g_{\varrho},g_{\varrho'}$ is bounded and
close to one at high frequencies. In other words the metric $g$ varies
slowly at the scale of the metric itself. 
\item the temperate property (at infinity) : the metric $g_{\varrho'}$
grows no faster than a power of $\left\Vert \varrho'-\varrho\right\Vert _{g_{\varrho}}$
(modified by the factor $\Delta^{\gamma}\left(\varrho\right)$) when
$\varrho'\to\infty$.
\end{enumerate}
\end{rem}

\begin{proof}[Proof of Lemma \ref{lem:temperate_metric}]
Let $v=\left(\left(v_{x},v_{z}\right),\left(v_{\xi},v_{\omega}\right)\right)\in\mathbb{R}^{2\left(n+1\right)}$
and $\varrho=\left(y,\eta\right),$ $\varrho'=\left(y',\eta'\right)\in\mathbb{R}^{2(n+1)}$.
Recall that $0\leq\alpha^{\parallel}\leq\alpha^{\perp}$ in (\ref{eq:conditions}).
We have
\begin{align*}
\|v\|_{g_{\varrho}}^{2}\underset{(\ref{eq:metric_g_in_coordinates})}{=} & \left(\frac{\left|v_{x}\right|}{\left\langle \left|\eta\right|\right\rangle ^{-\alpha^{\perp}}}\right)^{2}+\left(\left\langle \left|\eta\right|\right\rangle ^{-\alpha^{\perp}}\left|v_{\xi}\right|\right)^{2}+\left(\frac{\left|v_{z}\right|}{\left\langle \left|\eta\right|\right\rangle ^{-\alpha^{\parallel}}}\right)^{2}+\left(\left\langle \left|\eta\right|\right\rangle ^{-\alpha^{\parallel}}\left|v_{\omega}\right|\right)^{2}\\
= & \left(\frac{\left\langle \left|\eta\right|\right\rangle }{\left\langle \left|\eta'\right|\right\rangle }\right)^{2\alpha^{\perp}}\left(\frac{\left|v_{x}\right|}{\left\langle \left|\eta'\right|\right\rangle ^{-\alpha^{\perp}}}\right)^{2}+\left(\frac{\left\langle \left|\eta'\right|\right\rangle }{\left\langle \left|\eta\right|\right\rangle }\right)^{2\alpha^{\perp}}\left(\left\langle \left|\eta'\right|\right\rangle ^{-\alpha^{\perp}}\left|v_{\xi}\right|\right)^{2}\\
 & +\left(\frac{\left\langle \left|\eta\right|\right\rangle }{\left\langle \left|\eta'\right|\right\rangle }\right)^{2\alpha^{\parallel}}\left(\frac{\left|v_{z}\right|}{\left\langle \left|\eta'\right|\right\rangle ^{-\alpha^{\parallel}}}\right)^{2}+\left(\frac{\left\langle \left|\eta'\right|\right\rangle }{\left\langle \left|\eta\right|\right\rangle }\right)^{2\alpha^{\parallel}}\left(\left\langle \left|\eta'\right|\right\rangle ^{-\alpha^{\parallel}}\left|v_{\omega}\right|\right)^{2}\\
\underset{(\ref{eq:conditions})}{\leq} & \max\left\{ \left(\frac{\langle|\eta'|\rangle}{\langle|\eta|\rangle}\right)^{2\alpha^{\perp}},\left(\frac{\langle|\eta|\rangle}{\langle|\eta'|\rangle}\right)^{2\alpha^{\perp}}\right\} \|v\|_{g_{\varrho'}}^{2}.
\end{align*}
We deduce that
\begin{equation}
\max\left\{ \frac{\|v\|_{g_{\varrho'}}}{\|v\|_{g_{\varrho}}},\frac{\|v\|_{g_{\varrho}}}{\|v\|_{g_{\varrho'}}}\right\} \underset{(\ref{eq:metric_g_in_coordinates})}{\leq}\max\left\{ \left(\frac{\langle|\eta'|\rangle}{\langle|\eta|\rangle}\right)^{\alpha^{\perp}},\left(\frac{\langle|\eta|\rangle}{\langle|\eta'|\rangle}\right)^{\alpha^{\perp}}\right\} .\label{eq:comparison}
\end{equation}
Below we estimate the ratios $\langle|\eta'|\rangle/\langle|\eta|\rangle$
and $\langle|\eta|\rangle/\langle|\eta'|\rangle$ to get the required
estimate (\ref{eq:g_moderate and temperate}). For the former, we
have
\begin{align}
\frac{\langle|\eta'|\rangle}{\langle|\eta|\rangle}\underset{(\ref{eq:D1p})}{\leq}1+\frac{\left|\eta'-\eta\right|}{\langle|\eta|\rangle} & \underset{(\ref{eq:metric_g_in_coordinates})}{\leq}1+\frac{\langle\left|\eta\right|\rangle^{\alpha^{\perp}}\|\varrho'-\varrho\|_{g_{\varrho}}}{\left\langle \left|\eta\right|\right\rangle }\eq{\ref{eq:def_distortion_function_D-1}}1+\Delta\left(\varrho\right)\|\varrho'-\varrho\|_{g_{\varrho}}\label{eq:res1-1}\\
 & \ineq{\ref{eq:prod}}1+\Delta(\varrho)^{1-\gamma}\cdot\langle\Delta(\varrho)^{\gamma}\|\varrho'-\varrho\|_{g_{\varrho}}\rangle.
\end{align}
For the latter, if $|\eta'|\ge\frac{1}{2}|\eta'-\eta|$ then $\langle|\eta|\rangle\leq\langle|\eta'|\rangle+|\eta'-\eta|\le3\langle|\eta'|\rangle$
and we can proceed similarly to the previous case:
\[
\frac{\langle|\eta|\rangle}{\langle|\eta'|\rangle}\leq1+\frac{\left|\eta'-\eta\right|}{\langle|\eta'|\rangle}\leq1+3\frac{\left|\eta'-\eta\right|}{\langle|\eta|\rangle}=1+3\Delta\left(\varrho\right)\|\varrho'-\varrho\|_{g_{\varrho}}\le1+3\Delta(\varrho)^{1-\gamma}\cdot\langle\Delta(\varrho)^{\gamma}\|\varrho'-\varrho\|_{g_{\varrho}}\rangle.
\]
Otherwise if $\left|\eta'\right|<\frac{1}{2}\left|\eta'-\eta\right|$
then $|\eta'-\eta|\leq\left|\eta'\right|+\left|\eta\right|\leq\frac{1}{2}\left|\eta'-\eta\right|+\left|\eta\right|$
and
\[
\frac{1}{2}|\eta'-\eta|\leq\left|\eta\right|=|\eta-\eta'+\eta'|\leq|\eta'-\eta|+|\eta'|\leq\frac{3}{2}|\eta'-\eta|,
\]
i.e. $|\eta|$ is comparable with $|\eta'-\eta|$. Hence, for any
$N>0$, there exists a constant $C_{N}>0$ such that 
\[
\begin{aligned}\frac{\langle|\eta|\rangle}{\langle|\eta'|\rangle} & \leq1+|\eta'-\eta|\leq1+C_{N}|\eta|^{-\mu_{N}}\cdot\left\langle \langle|\eta|\rangle^{-(1-\alpha^{\perp})\gamma}\langle|\eta|\rangle^{-\alpha^{\perp}}|\eta'-\eta|\right\rangle ^{N}\\
 & \leq1+C_{N}|\eta|^{-\mu_{N}}\cdot\left\langle \Delta\left(\varrho\right)^{\gamma}\|\varrho'-\varrho\|_{g_{\varrho}}\right\rangle ^{N}
\end{aligned}
\]
where $\mu_{N}=N\left(1-\alpha^{\perp}\right)\left(1-\gamma\right)-1.$
Letting $N$ be sufficiently large so that $\left\langle \left|\eta\right|\right\rangle ^{-\mu_{N}}\leq\Delta(\varrho)^{1-\gamma}$
we obtain 
\[
\frac{\langle|\eta|\rangle}{\langle|\eta'|\rangle}\le1+C_{N}\Delta(\varrho)^{1-\gamma}\langle\Delta\left(\varrho\right)^{\gamma}\|\varrho'-\varrho\|_{g_{\varrho}}\rangle^{N}.
\]
Summarizing the estimates above, we obtain (\ref{eq:g_moderate and temperate}).
In order to get \eqref{eq:equivalence_distance}, we set $v=\varrho'-\varrho$
and $\gamma=0$ in (\ref{eq:g_moderate and temperate}).
\end{proof}
In the previous lemma we have used $\left\Vert \varrho'-\varrho\right\Vert _{g_{\varrho}}$
that makes sense in a local chart. This quantity will appear later
in the proof of Theorem \ref{thm:Microlocality-of-the_TO} from some
integration by parts. However to express the results, we would like
to use the more geometrical quantity that is the geodesic distance
$\mathrm{dist}_{g}\left(\rho',\rho\right)$ for the global metric
$g$ in (\ref{eq:global_metric}). The next lemma shows that both
quantities are equivalent.

\begin{cBoxB}{}
\begin{lem}
\label{lem:There-exist-}There exist $N>0$ and $C>0$ such that for
every $\varrho=\tilde{\kappa_{j}}\left(\rho\right),\varrho'=\tilde{\kappa}_{j}\left(\rho'\right)$,
\begin{equation}
\frac{1}{C}\left\langle \left\Vert \varrho'-\varrho\right\Vert _{g_{\varrho}}\right\rangle ^{1/N}\leq\left\langle \mathrm{dist}_{g}\left(\rho',\rho\right)\right\rangle \leq C\left\langle \left\Vert \varrho'-\varrho\right\Vert _{g_{\varrho}}\right\rangle ^{N}.\label{eq:log_log}
\end{equation}
\end{lem}

\end{cBoxB}

\begin{proof}
To show the second inequality in (\ref{eq:log_log}), we consider
the straight path $t\in\left[0,1\right]\mapsto\varrho\left(t\right)=\left(1-t\right)\varrho+t\varrho'$
and get that $\left\langle \mathrm{dist}_{g}\left(\rho',\rho\right)\right\rangle \leq C\left\langle \left\Vert \varrho'-\varrho\right\Vert _{g_{\varrho}}\right\rangle ^{N}$
for some $C>0,N>0$. To show the first inequality in (\ref{eq:log_log}),
let us assume that $\rho:t\in\left[0,1\right]\mapsto\rho\left(t\right)$
is a geodesic from $\rho=\rho\left(0\right)$ to $\rho'=\rho\left(1\right)$.
In a local chart $\varrho\left(t\right)=\tilde{\kappa_{j}}\left(\rho\left(t\right)\right)=\left(y\left(t\right),\eta\left(t\right)\right)\in\mathbb{R}^{n+1}\times\mathbb{R}^{n+1}$.
We will use that for some $C>0$,
\begin{align}
\mathrm{dist}_{g}\left(\rho',\rho\right) & =\int_{0}^{1}\left\Vert \dot{\rho}\left(t\right)\right\Vert _{g_{\rho\left(t\right)}}dt\nonumber \\
 & \geq C\left(\int_{0}^{1}\left\Vert \dot{y}\left(t\right)\right\Vert _{g_{\varrho\left(t\right)}}dt+\int_{0}^{1}\left\Vert \dot{\eta}\left(t\right)\right\Vert _{g_{\varrho\left(t\right)}}dt\right).\label{eq:bound-2}
\end{align}
 We may assume without loss of generality that $\left|\eta'\right|=\left|\eta\left(1\right)\right|\geq\left|\eta\left(0\right)\right|=\left|\eta\right|$
. Let us consider a few complementary cases.
\begin{enumerate}
\item If $\exists t'\in\left[0,1\right]$ such that $\left|\eta\left(t'\right)\right|<\frac{1}{2}\left|\eta\right|$,
then using the second integral in (\ref{eq:bound-2}), we get that
for some $C,N,C',N'>0$ we have
\[
\left\langle \mathrm{dist}_{g}\left(\rho',\rho\right)\right\rangle \ineqs{\ref{eq:equivalence_distance}}C\left\langle \frac{1}{2}\left\Vert \eta\right\Vert _{g_{\varrho}}\right\rangle ^{1/N}+C\left\langle \frac{1}{2}\left\Vert \eta'\right\Vert _{g_{\varrho}}\right\rangle ^{1/N}\geq C'\left\langle \left\Vert \varrho'-\varrho\right\Vert _{g_{\varrho}}\right\rangle ^{1/N'}.
\]
\item Otherwise $\forall t\in\left[0,1\right],$ $\left|\eta\left(t\right)\right|\geq\frac{1}{2}\left|\eta\right|$
and we consider two subcases:
\begin{enumerate}
\item If $\left\Vert y'-y\right\Vert _{g_{\varrho}}\geq\left\Vert \eta'-\eta\right\Vert _{g_{\varrho}}$
then $\left\Vert y'-y\right\Vert _{g_{\varrho}}\geq\frac{1}{2}\left\Vert \varrho'-\varrho\right\Vert _{g_{\varrho}}$.
Using the first integral in (\ref{eq:bound-2}), then for some $C,N>0$
we have
\begin{align*}
\left\langle \mathrm{dist}_{g}\left(\rho',\rho\right)\right\rangle  & \geq\left\langle \left\Vert y'-y\right\Vert _{g_{\left(y,\frac{1}{2}\eta\right)}}\right\rangle \ineqs{\ref{eq:equivalence_distance}}\frac{1}{C}\left\langle \left\Vert y'-y\right\Vert _{g_{\varrho}}\right\rangle ^{1/N}\\
 & \geq\frac{1}{C}\left\langle \left\Vert \varrho'-\varrho\right\Vert _{g_{\varrho}}\right\rangle ^{1/N}.
\end{align*}
\item Otherwise $\left\Vert \eta'-\eta\right\Vert _{g_{\varrho}}\geq\left\Vert y'-y\right\Vert _{g_{\varrho}}$
then $\left\Vert \eta'-\eta\right\Vert _{g_{\varrho}}\geq\frac{1}{2}\left\Vert \varrho'-\varrho\right\Vert _{g_{\varrho}}$
and for some $C,N>0$ we have
\begin{align*}
\left\langle \mathrm{dist}_{g}\left(\rho',\rho\right)\right\rangle  & \geq\frac{1}{C}\left\langle \left\Vert \eta'-\eta\right\Vert _{g_{\varrho'}}\right\rangle ^{1/N}\geq\frac{1}{C}\left\langle \left\Vert \varrho'-\varrho\right\Vert _{g_{\varrho}}\right\rangle ^{1/N}.
\end{align*}
\end{enumerate}
\end{enumerate}
We have finished the proof of Lemma \ref{lem:There-exist-}.
\end{proof}

\subsubsection{\label{subsec:Bargmann-transform-on}Wave packet transform in local
charts.}

In this section, from the given metric $g$ in (\ref{eq:metric_g_in_coordinates}),
we construct a family of wave packets functions $\varphi_{y,\eta}\left(.\right)$
and define a wave packet transform on $\mathbb{R}^{n+1}$ that gives
an exact resolution of identity. As before, we write $y=\left(x,z\right)\in\mathbb{R}^{n+1}$
for the coordinates on the local charts, we write $\eta=\left(\xi,\omega\right)\in\mathbb{R}^{n+1}$
for the dual coordinates and $\varrho=\left(y,\eta\right)$.

\paragraph{Definition of wave packets $\varphi_{\varrho}$}

We begin with considering the Gaussian function $\check{\varphi}_{\eta}^{\left(0\right)}:\mathbb{R}^{n+1}\to\mathbb{R}$
defined for each $\eta\in\mathbb{R}^{n+1}$ by
\begin{align}
\check{\varphi}_{\eta}^{\left(0\right)}\left(\eta'\right): & =\exp\left(-\frac{1}{2}\left\Vert \eta'-\eta\right\Vert _{g_{\eta}}^{2}\right)\nonumber \\
 & \underset{\eqref{eq:metric_g_in_coordinates}}{=}\exp\left(-\frac{1}{2}\left|\delta^{\perp}\left(\eta\right)\left(\xi'-\xi\right)\right|^{2}-\frac{1}{2}\left|\delta^{\parallel}\left(\eta\right)\left(\omega'-\omega\right)\right|^{2}\right).\label{eq:def_phi_0}
\end{align}
Then we set
\begin{equation}
\check{\varphi}_{\eta}\left(\eta'\right):=\left(m\left(\eta'\right)\right)^{-1/2}\check{\varphi}_{\eta}^{\left(0\right)}\left(\eta'\right)\label{eq:def_phi_xi}
\end{equation}
where
\begin{equation}
m\left(\eta'\right):=\int\left|\check{\varphi}_{\eta}^{\left(0\right)}\left(\eta'\right)\right|^{2}d\eta>0,\label{eq:def_m}
\end{equation}
that will play an essential role in the proof of the ``resolution
of identity'' Lemma \ref{lem:Wave_packets} below. For $\varrho=\left(y,\eta\right)\in T^{*}\mathbb{R}_{y}^{n+1}=\mathbb{R}_{\varrho}^{2n+2}$,
we define the wave packet $\varphi_{\varrho}\in\mathcal{S}\left(\mathbb{R}_{y'}^{n+1}\right)$
as the inverse Fourier transform of $\check{\varphi}_{\eta}$ shifted
by $y$:
\begin{eqnarray}
\varphi_{\varrho}\left(y'\right) & := & \left(\mathcal{F}^{-1}\check{\varphi}_{\eta}\right)\left(y'-y\right):=\frac{1}{\left(2\pi\right)^{\left(n+1\right)/2}}\int_{\mathbb{R}^{n+1}}e^{i\eta'\left(y'-y\right)}\check{\varphi}_{\eta}\left(\eta'\right)d\eta'.\label{eq:wave_packet}
\end{eqnarray}
Here are some uniform estimates on $\left(\varphi_{\varrho}\right)_{\varrho}$
for later use.

\begin{cBoxB}{}
\begin{lem}[Norm of wave packets]
\textbf{\label{lem:norm_wave_packets}}We have that $\forall\epsilon>0,\exists C_{\epsilon}>0$,
$\forall\varrho=\left(y,\eta\right)\in\mathbb{R}^{2n+2}$,
\begin{equation}
\left|\left\Vert \varphi_{\varrho}\right\Vert _{L^{2}\left(\mathbb{R}^{n+1}\right)}^{2}-1\right|\leq C_{\epsilon}\Delta\left(\varrho\right)^{1-\epsilon},\label{eq:norme_wave_packet}
\end{equation}
with the distortion function $\Delta\left(\varrho\right)$ defined
in (\ref{eq:def_distortion_function_D}).
\end{lem}

\end{cBoxB}

\begin{proof}
Let $0<\epsilon<1-\alpha^{\perp}$. We denote $B^{n}\left(\xi',r\right)$
the ball of center $\xi'$ and radius $r>0$ in $\mathbb{R}^{n}$.
For $\left|\eta'\right|\gg1$, writing $\eta=\left(\xi,\omega\right)\in\mathbb{R}^{n+1}$,
we have that $\forall N>0$,
\begin{align}
m\left(\eta'\right) & \eq{\ref{eq:def_m},\ref{eq:def_phi_0}}\int\exp\left(-\left\langle \eta\right\rangle ^{-2\alpha^{\perp}}\left|\xi'-\xi\right|^{2}-\left\langle \eta\right\rangle ^{-2\alpha^{\parallel}}\left|\omega'-\omega\right|^{2}\right)d\eta\label{eq:expression_m}\\
 & =\left(\int_{\xi\in B^{n}\left(\xi',\left|\eta'\right|^{\alpha^{\perp}+\epsilon}\right)}\exp\left(-\left|\eta\right|^{-2\alpha^{\perp}}\left|\xi'-\xi\right|^{2}\right)d\xi\right)\nonumber \\
 & \left(\int_{\omega\in B^{1}\left(\omega',\left|\eta'\right|^{\alpha^{\parallel}+\epsilon}\right)}\exp\left(-\left|\eta\right|^{-2\alpha^{\parallel}}\left|\omega'-\omega\right|^{2}\right)d\omega\right)+O_{N}\left(\left|\eta'\right|^{-N}\right).
\end{align}
For $\xi\in B^{n}\left(\xi',\left|\eta'\right|^{\alpha^{\perp}+\epsilon}\right)$
and $\omega\in B^{1}\left(\omega',\left|\eta'\right|^{\alpha^{\parallel}+\epsilon}\right)$,
we have $\left|\eta-\eta'\right|\leq C\left|\eta'\right|^{\alpha^{\perp}+\epsilon}$
and
\[
\left|\left|\eta\right|^{-2\alpha^{\perp}}-\left|\eta'\right|^{-2\alpha^{\perp}}\right|=O\left(\left|\eta'\right|^{-2\alpha^{\perp}-1}\left|\eta-\eta'\right|\right)=O\left(\left|\eta'\right|^{-\alpha^{\perp}-1+\epsilon}\right)
\]
hence
\begin{align*}
\left|\left|\eta\right|^{-2\alpha^{\perp}}\left|\xi'-\xi\right|^{2}-\left|\eta'\right|^{-2\alpha^{\perp}}\left|\xi'-\xi\right|^{2}\right| & =O\left(\left|\eta'\right|^{-1+\alpha^{\perp}+3\epsilon}\right).
\end{align*}
Similarly
\begin{align*}
\left|\left|\eta\right|^{-2\alpha^{\parallel}}\left|\omega'-\omega\right|^{2}-\left|\eta'\right|^{-2\alpha^{\parallel}}\left|\omega'-\omega\right|^{2}\right| & =O\left(\left|\eta'\right|^{-1+\alpha^{\perp}+3\epsilon}\right).
\end{align*}
On the other hand
\[
\int_{\mathbb{R}^{n+1}}\exp\left(-\left\langle \eta'\right\rangle ^{-2\alpha^{\perp}}\left|\xi'-\xi\right|^{2}-\left\langle \eta'\right\rangle ^{-2\alpha^{\parallel}}\left|\omega'-\omega\right|^{2}\right)d\eta\eq{\ref{eq:def_phi_0}}\left\Vert \check{\varphi}_{\eta'}^{\left(0\right)}\right\Vert _{L^{2}}^{2}=\pi^{\frac{n+1}{2}}\left(\delta^{\perp}\left(\eta'\right)\right)^{n}\delta^{\parallel}\left(\eta'\right).
\]
Hence
\[
m\left(\eta'\right)=\left\Vert \check{\varphi}_{\eta'}^{\left(0\right)}\right\Vert _{L^{2}}^{2}\left(1+O\left(\left|\eta'\right|^{-1+\alpha^{\perp}+3\epsilon}\right)\right).
\]
We deduce that
\[
\left\Vert \varphi_{\varrho}\right\Vert _{L^{2}\left(\mathbb{R}^{n+1}\right)}^{2}\eq{\ref{eq:wave_packet}}\left\Vert \check{\varphi}_{\eta}\right\Vert _{L^{2}}^{2}\eq{\ref{eq:def_phi_xi}}1+O\left(\left|\eta\right|^{-1+\alpha^{\perp}+3\epsilon}\right).
\]
Using that $\Delta\left(\eta\right)\eq{\ref{eq:def_distortion_function_D-1}}\left\langle \left|\eta\right|\right\rangle ^{-1+\alpha^{\perp}}$,
this gives (\ref{eq:norme_wave_packet}).
\end{proof}

\paragraph{Wave packet transform}

~

\begin{cBoxA}{}
\begin{defn}[Wave packet transform on local coordinates]
We define the wave packet transform on $\mathbb{R}^{n+1}$ (or on
local coordinates) by:
\begin{equation}
\mathcal{B}:\begin{cases}
\mathcal{S}\left(\mathbb{R}_{y}^{n+1}\right) & \rightarrow\mathcal{S}\left(\mathbb{R}_{\varrho}^{2\left(n+1\right)}\right)\\
u & \mapsto v\left(\varrho\right)=\left\langle \varphi_{\varrho}|u\right\rangle _{L^{2}\left(\mathbb{R}^{n+1}\right)}
\end{cases}.\label{eq:def_Bargamnn_B}
\end{equation}
\end{defn}

\end{cBoxA}

Its formal adjoint is

\begin{equation}
\mathcal{B}^{\dagger}:\begin{cases}
\mathcal{S}\left(\mathbb{R}_{\varrho}^{2\left(n+1\right)}\right) & \rightarrow\mathcal{S}\left(\mathbb{R}_{y'}^{n+1}\right)\\
v & \mapsto u\left(y'\right)=\int_{\mathbb{R}^{2\left(n+1\right)}}v\left(\varrho\right)\varphi_{\varrho}\left(y'\right)\frac{d\varrho}{\left(2\pi\right)^{n+1}}
\end{cases}.\label{eq:def_Bargamnn_*}
\end{equation}

\begin{rem}
In the special case of $\alpha^{\perp}=0$ and $\alpha^{\parallel}=0$
where $\delta^{\perp}\left(\eta\right)$ and $\delta^{\parallel}\left(\eta\right)$
are constant, the wave packet transform (\ref{eq:def_Bargamnn_B})
corresponds to the well known Fock-Bargman representation \cite[chap.1]{folland-88}.
However note that the condition $\alpha^{\perp}\geq1/2$ in (\ref{eq:conditions})
is not satisfied in such a case.
\end{rem}

~
\begin{rem}
\label{rem:As-in-}As in \cite[Prop. 3.1.6 page 76]{martinez-01}
one can show that $\mathcal{B}$ maps continuously $\mathcal{S}\left(\mathbb{R}_{y}^{n+1}\right)$
to $\mathcal{S}\left(\mathbb{R}_{\varrho}^{2\left(n+1\right)}\right)$.
\end{rem}

The next lemma is fundamental in our analysis, since it will permit
to perform the analysis in cotangent space (here $T^{*}\mathbb{R}^{n+1}$)
instead on the manifold itself (here $\mathbb{R}^{n+1}$).

\begin{cBoxB}{}
\begin{lem}[Resolution of identity on $C^{\infty}\left(\mathbb{R}^{n+1}\right)$]
\label{lem:Wave_packets} We have the resolution of identity on $C^{\infty}\left(\mathbb{R}^{n+1}\right)$:
\begin{equation}
\mathrm{Id}_{/C^{\infty}\left(\mathbb{R}^{n+1}\right)}=\mathcal{B}^{\dagger}\mathcal{B}=\int_{\mathbb{R}^{2\left(n+1\right)}}\pi\left(\varrho\right)\frac{d\varrho}{\left(2\pi\right)^{n+1}}\label{eq:resol_identity_B}
\end{equation}
where $\pi\left(\varrho\right)$ denotes the rank one self-adjoint
operator 
\begin{equation}
\pi\left(\varrho\right):L^{2}\left(\mathbb{R}^{n+1}\right)\rightarrow L^{2}\left(\mathbb{R}^{n+1}\right),\quad\quad\pi\left(\varrho\right)u:=\varphi_{\varrho}\langle\varphi_{\varrho}|u\rangle_{L^{2}\left(\mathbb{R}^{n+1}\right)},\label{eq:def_pi_wave_packet}
\end{equation}
that satisfies $\forall\epsilon>0$ 
\begin{equation}
\mathrm{Tr}\left(\pi\left(\varrho\right)\right)=\left\Vert \varphi_{\varrho}\right\Vert _{L^{2}\left(\mathbb{R}^{n+1}\right)}^{2}\underset{(\ref{eq:norme_wave_packet})}{=}1+O\left(\Delta\left(\varrho\right)^{1-\epsilon}\right).\label{eq:Tr_pi_rho}
\end{equation}
\end{lem}

\end{cBoxB}

\begin{rem}
From (\ref{eq:def_pi_wave_packet}),(\ref{eq:Tr_pi_rho}),(\ref{eq:def_distortion_function_D-1}),
in the limit $\left|\varrho\right|\rightarrow+\infty$ the operator
$\pi\left(\varrho\right)$ tends to be an orthogonal projector of
rank one onto the complex line $\mathbb{C}\varphi_{\rho}$.
\end{rem}

\begin{proof}
Note that
\begin{equation}
\int\left|\check{\varphi}_{\eta}\left(\eta'\right)\right|^{2}d\eta=\int\frac{\left|\check{\varphi}_{\eta}^{\left(0\right)}\left(\eta'\right)\right|^{2}}{m\left(\eta'\right)}d\eta=1.\label{eq:phi_eta_1}
\end{equation}
We write the operator $\mathcal{B}$ and and $\mathcal{B}^{\dagger}$
respectively as
\[
\left(\mathcal{B}u\right)\left(y,\eta\right)\underset{(\ref{eq:def_Bargamnn_B})}{=}\left\langle \varphi_{\varrho}|u\right\rangle _{L^{2}\left(\mathbb{R}^{n+1}\right)}=\int_{\mathbb{R}^{n+1}}\overline{\left(\mathcal{F}^{-1}\check{\varphi}_{\eta}\right)\left(y'-y\right)}u\left(y'\right)dy'=\left(\mathcal{F}^{-1}\check{\varphi}_{\eta}\right)\ast u
\]
and 
\begin{align*}
\begin{aligned}\left(\mathcal{B}^{\dagger}v\right)\left(y\right)\end{aligned}
 & =\left(2\pi\right)^{-(n+1)}\int\int\varphi_{y',\eta}\left(y\right)v\left(y',\eta\right)dy'd\eta\\
 & =\int\int\left(\mathcal{F}^{-1}\check{\varphi}_{\eta}\right)\left(y-y'\right)v\left(y',\eta\right)dy'd\eta=\int\left(\left(\mathcal{F}^{-1}\check{\varphi}_{\eta}\right)\ast_{y}v\right)\left(y,\eta\right)d\eta
\end{align*}
where we write $\ast_{y}$ for the convolution operator in the variable
$y$ with $\eta$ fixed. Hence

\begin{eqnarray*}
\left(\mathcal{B}^{\dagger}\mathcal{B}u\right)\left(y\right) & = & \frac{1}{\left(2\pi\right)^{n+1}}\int\left(\left(\mathcal{F}^{-1}\check{\varphi}_{\eta}\right)\ast\left(\mathcal{F}^{-1}\check{\varphi}_{\eta}\right)\ast u\right)\left(y\right)d\eta\\
 & = & \frac{1}{\left(2\pi\right)^{\left(n+1\right)/2}}\int\left(\left(\mathcal{F}^{-1}\left|\check{\varphi}_{\eta}\right|^{2}\right)\ast u\right)\left(y\right)d\eta\\
 & = & \frac{1}{\left(2\pi\right)^{\left(n+1\right)/2}}\left(\left(\mathcal{F}^{-1}\int\left|\check{\varphi}_{\eta}\right|^{2}d\eta\right)\ast u\right)\left(y\right)\\
 & \underset{(\ref{eq:phi_eta_1})}{=} & \frac{1}{\left(2\pi\right)^{\left(n+1\right)/2}}\left(\left(\mathcal{F}^{-1}1\right)\ast u\right)\left(y\right)=u\left(y\right),
\end{eqnarray*}
giving (\ref{eq:resol_identity_B}). Eq.(\ref{eq:Tr_pi_rho}) is a
consequence of (\ref{eq:norme_wave_packet}).
\end{proof}

\paragraph{Evolution of wave packets}

In this section, we describe the evolution of a wave packet $\varphi_{\varrho}$
under the push-forward operator by the flow $\left(\phi^{t}\right)_{t\in\mathbb{R}}$.
The next lemma is fundamental in our analysis, in particular the fact
that estimate (\ref{eq:size_wave_packet}) and (\ref{eq:F_phi_t})
are uniform w.r.t. $\varrho$.

\begin{cBoxB}[\foreignlanguage{french}{breakable}]{}

\begin{lem}[Description of evolving wave packets]
\textbf{\label{lem:Description-of-evolving}}Let $t\in\mathbb{R}$
and chart indices $j,j'\in\left\{ 0,1,\ldots J\right\} $. We assume
that $\phi^{t}\left(U_{j}\right)\cap U_{j'}\neq\emptyset$ and denote
\[
V:=\kappa_{j}\left(U_{j}\cap\phi^{-t}\left(U_{j'}\right)\right),\quad V':=\kappa_{j'}\left(\phi^{t}\left(U_{j}\right)\cap U_{j'}\right)\subset\mathbb{R}^{n+1},
\]
\[
\phi:=\kappa_{j'}\circ\phi^{t}\circ\kappa_{j}^{-1}\quad:V\rightarrow V'
\]
and similarly
\[
\tilde{\phi}:=\tilde{\kappa}_{j'}\circ\tilde{\phi}^{t}\circ\tilde{\kappa}_{j}^{-1}\quad:T^{*}V\rightarrow T^{*}V'.
\]
Then $\forall N>0,\exists C_{N,t}>0,\forall\varrho=\left(y,\eta\right)\in T^{*}V,\forall y''\in V',$
\begin{equation}
\left|\left(\varphi_{\varrho}\circ\phi^{-1}\right)\left(y''\right)\right|\leq\left(\mathrm{det}\delta\left(\eta\right)\right)^{-1/2}C_{N,t}\left\langle \left\Vert \left(y',\eta'\right)-\left(y'',\eta'\right)\right\Vert _{g_{\varrho}}\right\rangle ^{-N}\label{eq:size_wave_packet}
\end{equation}
with $\left(y',\eta'\right):=\tilde{\phi}\left(\varrho\right)$ and
\begin{equation}
\delta\left(\eta\right):=\mathrm{Diag}\left(\delta^{\perp}\left(\eta\right),\delta^{\parallel}\left(\eta\right)\right).\label{eq:def_delta-1}
\end{equation}
For the Fourier transform, with the conditions (\ref{eq:conditions})
that give property (\ref{eq:invariance_g}), we have $\forall\eta''\in\mathbb{R}^{n+1}$,
\begin{equation}
\left|\left(\mathcal{F}\left(\varphi_{\varrho}\circ\phi^{-1}\right)\right)\left(\eta''\right)\right|\leq C_{N,t}\left(\mathrm{det}\delta\left(\eta\right)\right)^{1/2}\left\langle \left\Vert \left(y',\eta'\right)-\left(y',\eta''\right)\right\Vert _{g_{\varrho}}\right\rangle ^{-N}.\label{eq:F_phi_t}
\end{equation}
This lemma is true if one replaces $\phi^{t}$ by any local diffeomorphism
that preserves the vector field $X$.
\end{lem}

\end{cBoxB}

\begin{rem}
The statement of Lemma \ref{lem:Description-of-evolving} is illustrated
on Figure \ref{fig:wave_packet} and Figure \ref{fig:Unity_ball}.
It shows that the effective size of $\varphi_{\varrho}$ in phase
space $T\mathbb{R}^{n+1}$ is comparable to the unit ball for the
metric $g_{\varrho}$. For this reason we call $\varphi_{\varrho}$
a wave packet.
\end{rem}

{\footnotesize{}}
\begin{figure}[H]
\begin{centering}
{\footnotesize{}\scalebox{0.9}[0.9]{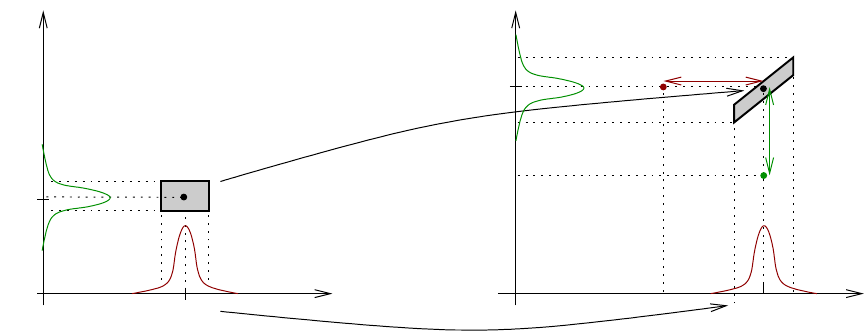}}{\footnotesize\par}
\par\end{centering}
{\footnotesize{}\caption{\label{fig:wave_packet}Evolution of a wave packet $\varphi_{\varrho}$
in space and Fourier space. In Lemma \ref{lem:Description-of-evolving},
the decay of the functions $\left(\varphi_{\varrho}\circ\phi^{-1}\right)\left(y''\right)$
and $\left(\mathcal{F}\left(\varphi_{\varrho}\circ\phi^{-1}\right)\right)\left(\eta''\right)$
is controlled respectively from the distances $\left\Vert \left(y',\eta'\right)-\left(y'',\eta'\right)\right\Vert _{g_{\varrho}}$
and $\left\Vert \left(y',\eta'\right)-\left(y',\eta''\right)\right\Vert _{g_{\varrho}}$.}
}{\footnotesize\par}
\end{figure}
{\footnotesize\par}
\begin{proof}
We use the notations introduced in Lemma \ref{lem:Description-of-evolving}.
We have
\begin{align*}
\left(\varphi_{\varrho}\circ\phi^{-1}\right)\left(y''\right)\eq{\ref{eq:wave_packet},\ref{eq:def_phi_xi},\ref{eq:def_phi_0}}\frac{1}{\left(2\pi\right)^{\left(n+1\right)/2}} & \int_{\mathbb{R}^{n+1}}e^{i\eta''\left(\phi^{-1}\left(y''\right)-y\right)}\left(m\left(\eta''\right)\right)^{-1/2}\\
 & \qquad\exp\left(-\frac{1}{2}\left|\delta\left(\eta\right)\left(\eta''-\eta\right)\right|^{2}\right)d\eta''.
\end{align*}
Let us consider the change of variable $\eta''\mapsto\tilde{\eta}''$
with
\[
\tilde{\eta}'':=\delta\left(\eta\right)\left(\eta''-\eta\right),
\]
giving
\begin{align}
\left(\varphi_{\varrho}\circ\phi^{-1}\right)\left(y''\right) & =\mathrm{det}\delta\left(\eta\right)^{-1}\left(2\pi\right)^{-\left(n+1\right)/2}e^{i\eta\left(\phi^{-1}\left(y''\right)-y\right)}\label{eq:wp-1}\\
 & \int_{\mathbb{R}^{n+1}}e^{i\tilde{\eta}''\delta\left(\eta\right)^{-1}\left(\phi^{-1}\left(y''\right)-y\right)}\left(m\left(\eta''\right)\right)^{-1/2}\exp\left(-\frac{1}{2}\left|\tilde{\eta}''\right|^{2}\right)d\tilde{\eta}''.\nonumber 
\end{align}
From the following estimates that follows from (\ref{eq:expression_m}),
$\forall\alpha,\exists C_{\alpha}>0,\forall\eta\in\mathbb{R}^{n+1},$
\[
\partial_{\tilde{\eta}''}^{\alpha}\left(\left(m\left(\eta''\right)\right)^{-1/2}\exp\left(-\frac{1}{2}\left|\tilde{\eta}''\right|^{2}\right)\right)\leq C_{\alpha}\mathrm{det}\delta\left(\eta\right)^{1/2},
\]
and noticing that (\ref{eq:wp-1}) is a usual Fourier transform of
a Schwartz function \cite[p.222]{taylor_tome1}, we deduce that $\forall N>0,\exists C_{N,t}>0,\forall\eta\in\mathbb{R}^{n+1}$,
\begin{align*}
\left|\left(\varphi_{\varrho}\circ\phi^{-1}\right)\left(y''\right)\right| & \leq\left(\mathrm{det}\delta\left(\eta\right)\right)^{-1/2}C_{N,t}\left\langle \left|\delta\left(\eta\right)^{-1}\left(\phi\left(y\right)-y''\right)\right|\right\rangle ^{-N}\\
 & =\left(\mathrm{det}\delta\left(\eta\right)\right)^{-1/2}C_{N,t}\left\langle \left\Vert \tilde{\phi}\left(\varrho\right)-\left(y'',\eta'\right)\right\Vert _{g_{\varrho}}\right\rangle ^{-N}.
\end{align*}
This gives (\ref{eq:size_wave_packet}). To show (\ref{eq:F_phi_t}),
we write
\begin{align}
\left(\mathcal{F}\left(\varphi_{\varrho}\circ\phi^{-1}\right)\right)\left(\eta''\right) & =\left(2\pi\right)^{-\left(n+1\right)/2}\int_{\mathbb{R}^{\left(n+1\right)}}e^{-i\eta''y''}\left(\varphi_{\varrho}\circ\phi^{-1}\right)\left(y''\right)dy''.\label{eq:exp1}
\end{align}
We consider the change of variable $y''\mapsto\tilde{y}''$ with
\begin{equation}
\tilde{y}'':=\delta\left(\eta\right)^{-1}\left(y''-\phi\left(y\right)\right),\label{eq:y_2tilde}
\end{equation}
giving
\begin{align*}
\left(\mathcal{F}\left(\varphi_{\varrho}\circ\phi^{-1}\right)\right)\left(\eta''\right) & \eq{\ref{eq:wp-1}}\left(2\pi\right)^{-\left(n+1\right)}\int_{\mathbb{R}^{2\left(n+1\right)}}e^{if\left(\tilde{\eta}''',\tilde{y}''\right)}\left(m\left(\eta'''\right)\right)^{-1/2}\exp\left(-\frac{1}{2}\left|\tilde{\eta}'''\right|^{2}\right)d\tilde{\eta}'''d\tilde{y}'',
\end{align*}
with phase function
\begin{align}
f\left(\tilde{\eta}''',\tilde{y}''\right) & =-\eta''y''+\eta'''\left(\phi^{-1}\left(y''\right)-y\right).\label{eq:f_function}
\end{align}
In order to integrate by parts later, we first extract the linear
part of the function $f$ at $\left(0,0\right)$, writing (beware
that in our notations, partial derivative have multi-components $\partial_{\tilde{y}''}\equiv\left(\partial_{\tilde{x}_{j}''},\partial_{\tilde{z}''}\right)$,
$\partial_{\tilde{\eta}'''}\equiv\left(\partial_{\tilde{\xi}_{j}'''},\partial_{\tilde{\omega}'''}\right)$
with $j=1\ldots n$)
\[
f\left(\tilde{\eta}''',\tilde{y}''\right)=f\left(0,0\right)+\tilde{\eta}'''.\partial_{\tilde{\eta}'''}f\left(0,0\right)+\tilde{y}''.\partial_{\tilde{y}''}f\left(0,0\right)+f_{\mathrm{non-lin}}\left(\tilde{\eta}''',\tilde{y}''\right),
\]
with

\[
\partial_{\tilde{\eta}'''}f\left(\tilde{\eta}''',\tilde{y}''\right)\eq{\ref{eq:f_function}}\left(\phi^{-1}\left(y''\right)-y\right)\delta\left(\eta\right)^{-1}.
\]
\[
\partial_{\tilde{y}''}f\left(\tilde{\eta}''',\tilde{y}''\right)\eq{\ref{eq:f_function}}\left(-\eta''+\eta'''\left(d\phi^{-1}\right)\left(y''\right)\right)\delta\left(\eta\right)
\]
hence
\[
\partial_{\tilde{\eta}'''}f\left(0,0\right)=0,\quad\partial_{\tilde{y}''}f\left(0,0\right)=\left(-\eta''+\eta\left(d\phi_{\phi\left(y\right)}^{-1}\right)\right)\delta\left(\eta\right).
\]
We get
\begin{equation}
\left(\mathcal{F}\left(\varphi_{\varrho}\circ\phi^{-1}\right)\right)\left(\eta''\right)=\left(2\pi\right)^{-\left(n+1\right)}\mathrm{det}\delta\left(\eta\right)^{1/2}e^{if\left(0,0\right)}\int_{\mathbb{R}^{2\left(n+1\right)}}e^{i\tilde{y}''\partial_{\tilde{y}''}f\left(0,0\right)}u_{\varrho}\left(\tilde{\eta}''',\tilde{y}''\right)d\tilde{\eta}'''d\tilde{y}''\label{eq:FexptX}
\end{equation}
with
\[
u_{\varrho}\left(\tilde{\eta}''',\tilde{y}''\right)=e^{if_{\mathrm{non-lin}}\left(\tilde{\eta}''',\tilde{y}''\right)}\mathrm{det}\delta\left(\eta\right)^{-1/2}\left(m\left(\eta'''\right)\right)^{-1/2}\exp\left(-\frac{1}{2}\left|\tilde{\eta}'''\right|^{2}\right).
\]
We have

\[
\partial_{\tilde{\eta}'''}^{2}f\left(\tilde{\eta}''',\tilde{y}''\right)=0,\quad\partial_{\tilde{\eta}'''}\partial_{\tilde{y}''}f\left(\tilde{\eta}''',\tilde{y}''\right)=\left(d\phi^{-1}\right)\left(y''\right),\quad\partial_{\tilde{y}''}^{2}f\left(\tilde{\eta}''',\tilde{y}''\right)=\left(\eta'''\left(dd\phi^{-1}\right)\left(y''\right)\right)\delta\left(\eta\right)^{2}.
\]
In the flow direction $d_{z}\phi^{-1}=1$ so $dd_{z}\phi^{-1}=0$.
Hence only $d_{x}d_{x}\phi^{-1}$ matters and we have
\[
\left\Vert \partial_{\tilde{y}''}^{2}f\left(0,\tilde{y}''\right)\right\Vert \eq{\ref{eq:def_delta-1}}\left\Vert \left(d_{x}d_{x}\phi^{-1}\right)\left(y''\right)\eta\left\langle \eta\right\rangle ^{-2\alpha^{\perp}}\right\Vert 
\]
that is uniformly bounded with respect to $\varrho=\left(y,\eta\right)$
if and only if $\alpha^{\perp}\geq1/2$, that is condition (\ref{eq:conditions}).
We check that higher derivatives of $f$ are always uniformly bounded
so we get that $u_{\varrho}\left(\tilde{\eta}''',\tilde{y}''\right)\in\mathcal{S}\left(\mathbb{R}^{2\left(n+1\right)}\right)$
is uniformly bounded with respect to $\varrho$. Since the Fourier
transform sends the Schwartz space to itself continuously \cite[p.222]{taylor_tome1}
we deduce that $\forall N>0,\exists C_{N,t}>0$,
\begin{align*}
\left|\mathcal{F}\left(\varphi_{\varrho}\circ\phi^{-1}\right)\left(\eta''\right)\right| & \ineq{\ref{eq:FexptX}}\left(\mathrm{det}\delta\left(\eta\right)\right)^{1/2}C_{N,t}\left\langle \left|\left(-\eta''+\eta\left(d\phi_{\phi\left(y\right)}^{-1}\right)\right)\delta\left(\eta\right)\right|\right\rangle ^{-N}\\
 & =C_{N,t}\left(\mathrm{det}\delta\left(\eta\right)\right)^{1/2}\left\langle \left\Vert \tilde{\phi}\left(\varrho\right)-\left(y',\eta''\right)\right\Vert _{g_{\varrho}}\right\rangle ^{-N}.
\end{align*}
This gives (\ref{eq:F_phi_t}).
\end{proof}

\subsubsection{\label{subsec:Resolution-of-identity-2}Global wave packet transform}

We have defined the wave packet transform in a local chart. We next
define a global wave packet transform on $C^{\infty}(M)$. Recall
the notation $y=\left(x,z\right)\in\mathbb{R}^{n+1}$, $\eta=\left(\xi,\omega\right)\in\mathbb{R}^{n+1}$
and $\varrho=\left(y,\eta\right)\in\mathbb{R}^{2\left(n+1\right)}$
in (\ref{eq:def_k_tilde_j}). For a chart index $j$, we write $\mathcal{B}_{j}$
for the wave packet transform $\mathcal{B}$ defined in (\ref{eq:def_Bargamnn_B})
acting on $C_{0}^{\infty}(V_{j})$ and set
\[
\mathcal{B}:=\oplus_{j=1}^{J}\mathcal{B}_{j}:\oplus_{j=1}^{J}C_{0}^{\infty}(V_{j})\rightarrow\oplus_{j=1}^{J}\mathcal{S}\left(\mathbb{R}_{\varrho}^{2\left(n+1\right)}\right).
\]
The global wave packet transform that we will consider is essentially
the composition
\begin{equation}
\mathcal{B}\circ I=\oplus_{j=1}^{J}\mathcal{B}_{j}\circ I_{j}:C^{\infty}\left(M\right)\rightarrow\oplus_{j=1}^{J}\mathcal{S}\left(\mathbb{R}_{\varrho}^{2\left(n+1\right)}\right).\label{eq:B-I}
\end{equation}
But, in order to get a more geometric expression, we would like that
each wave packet corresponds to a point on $T^{*}M$ (rather than
a point in $\mathbb{R}_{\varrho}^{2\left(n+1\right)}$) and that the
global wave packet transform sends a function on $M$ to a function
on $T^{*}M$. Simultaneously, we would like to have an exact resolution
of identity \eqref{eq:resol_ident_Pi_rho}. Thus we define the global
wave packet transform as follows.

Notice that for $u\in C^{\infty}\left(M\right)$, the component $\left(\mathcal{B}_{j}I_{j}\right)(u)$
will not be supported on a bounded subset in $\mathbb{R}_{\varrho}^{2\left(n+1\right)}$
though they decay rapidly on the outside of $\mathrm{supp}\left(\chi_{j}\right)\times\mathbb{R}^{n+1}$.
This is problematic for the purpose of getting the geometric expression
mentioned above. Our solution to this problem is to consider a bijection
$\tilde{\psi}$ from the space $T^{*}V_{j}$ to $\mathbb{R}_{\varrho}^{2\left(n+1\right)}=T^{*}\mathbb{R}^{n+1}$
so that it restricts to the identity map on a small neighborhood of
$\mathrm{supp}\left(\chi_{j}\right)\Subset V_{j}$. Recall the notation
for local charts in (\ref{eq:def_kappa_j}) with the constant $c>0$.
Let $\epsilon>0$ and 
\[
V_{j,\epsilon}:=\text{\ensuremath{\mathbb{B}_{x}^{n}(c-\epsilon)}\ensuremath{\times}}\ensuremath{\mathbb{B}_{z}^{n}(l-\epsilon)}=\left\{ \left(x,z\right)\in\mathbb{R}^{n}\times\mathbb{R};\,\left|x\right|<c-\epsilon\text{ and }\left|z\right|<l-\epsilon\right\} \subset V_{j}.
\]
We assume $\epsilon$ small enough so that $\mathrm{supp}\left(\chi_{j}\right)\subset V_{j,\epsilon}$
for any $j$. We take a (surjective) diffeomorphism
\begin{equation}
\psi:V_{j}\rightarrow\mathbb{R}_{y}^{n+1}\label{eq:def_psi}
\end{equation}
such that
\begin{equation}
\psi\left(y\right)=y\qquad\text{for }y\in V_{j,\epsilon/2}.\label{eq:psi_y_y}
\end{equation}
We may and will assume that the Jacobian of $\psi$ has temperate
growth, that is, for some $N_{0}>0$ and $C>0$, 
\[
\left|\det d_{y}\psi\right|\leq C\left|\psi\left(y\right)\right|^{N_{0}}\quad\text{for all }y\in V_{j}
\]
and also that $\psi$ is expanding, that is, 
\begin{equation}
\left\Vert \psi\left(y\right)-\psi\left(y'\right)\right\Vert \geq\left\Vert y-y'\right\Vert \quad\text{for all }y,y'\in V_{j}.\label{eq:expand}
\end{equation}
Then we consider the trivial extension of the map $\psi$ to $T^{*}V_{j}$,
which is defined by
\begin{equation}
\tilde{\psi}:\varrho=\left(y,\eta\right)\in T^{*}V_{j}\mapsto\varrho'=\left(\psi\left(y\right),\eta\right)\in T^{*}\mathbb{R}^{n+1}.\label{eq:def_psi_tilde}
\end{equation}
Beware that this map $\tilde{\psi}$ differs from the canonical extension
$d^{*}\psi^{-1}$ of $\psi$ to $T^{*}V_{j}$.

\begin{cBoxA}{}
\begin{defn}[Global wave packet transform $\mathcal{T}$]
\label{def:Phi_j_rho}For $1\leq j\leq J$ and $\rho\in T^{*}U_{j}$,
we define a wave packet function $\Phi_{\rho,j}\in C^{\infty}\left(U_{j}\right)\subset C^{\infty}\left(M\right)$
on the local chart $U_{j}$ by
\begin{equation}
\Phi_{\rho,j}(m):=\left|\mathrm{det}\left(d\psi\right)(\kappa_{j}(m))\right|^{1/2}\cdot\left(I_{j}^{\dagger}\varphi_{\tilde{\psi}\left(\tilde{\kappa}_{j}\left(\rho\right)\right)}\right)(m).\label{eq:def_wave_packet_Phi_j_rho}
\end{equation}
Then we define the wave packet transform
\begin{equation}
\mathcal{T}:\begin{cases}
C^{\infty}\left(M;\mathbb{C}\right) & \rightarrow\mathcal{S}\left(T^{*}M;\mathbb{C}^{J}\right)\\
u\left(m\right) & \mapsto v\left(\rho\right)=\left(\langle\Phi_{\rho,j}|u\rangle_{L^{2}\left(M\right)}\right)_{j\in\left\{ 1,2,\ldots J\right\} }
\end{cases}\label{eq:def_T_in-1}
\end{equation}
where we set $\Phi_{\rho,j}=0$ for $\rho\notin T^{*}U_{j}$.
\end{defn}

\end{cBoxA}

With the definitions above, we obtain the next proposition.

\begin{cBoxB}{}
\begin{prop}[Resolution of identity on $C^{\infty}\left(M\right)$]
\label{prop:resolution_identity_CM} The $L^{2}$-adjoint operator
of $\mathcal{T}$ is given by
\begin{equation}
\mathcal{T}^{\dagger}:\begin{cases}
\mathcal{S}\left(T^{*}M;\mathbb{C}^{J}\right) & \rightarrow C^{\infty}\left(M;\mathbb{C}\right)\\
v\left(\rho\right) & \mapsto u\left(m\right)=\int_{\rho\in T^{*}M}\sum_{j}v_{j}\left(\rho\right)\Phi_{\rho,j}\left(m\right)\frac{d\rho}{\left(2\pi\right)^{n+1}}.
\end{cases}\label{eq:def_T_*}
\end{equation}
We have the following resolution of identity on $C^{\infty}\left(M\right)$:
\begin{align}
\mathrm{Id}_{/C^{\infty}\left(M\right)} & =\mathcal{T}^{\dagger}\mathcal{T}=\int_{T^{*}M}\Pi\left(\rho\right)\frac{d\rho}{\left(2\pi\right)^{n+1}}\label{eq:resol_ident_Pi_rho}
\end{align}
where
\begin{equation}
\Pi\left(\rho\right):=\sum_{j=1}^{J}\Phi_{\rho,j}\langle\Phi_{\rho,j}|.\rangle_{L^{2}}.\label{eq:Pi_(rho)}
\end{equation}
The last operator $\Pi\left(\rho\right)$ is a finite rank operator,
self-adjoint and non-negative on $L^{2}\left(M,dm\right)$ and satisfies,
$\forall\epsilon>0$,
\begin{equation}
\left|\mathrm{Tr}\left(\Pi\left(\rho\right)\right)-1\right|=O\left(\Delta^{1-\epsilon}\left(\rho\right)+\delta^{\parallel}\left(\eta\right)\right).\label{eq:Tr(Pi_rho)}
\end{equation}
and $\left\Vert \Pi\left(\rho\right)\right\Vert _{\mathrm{Tr}_{L^{2}}}=\mathrm{Tr}\left(\Pi\left(\rho\right)\right)$.
\end{prop}

\end{cBoxB}

\begin{rem}
Eq.(\ref{eq:resol_ident_Pi_rho}) shows that $\mathcal{T}:L^{2}\left(M;\mathbb{C}\right)\rightarrow L^{2}\left(T^{*}M;\mathbb{C}^{J}\right)$
is an isometry, see remark \ref{rem:Projection}. A drawback of the
previous construction is that $\Pi\left(\rho\right)$ is not rank
one and $\Pi\left(\rho\right)^{2}\neq\Pi\left(\rho\right)$. This
is inevitable because the wave packets for $\rho\in T^{*}M$ on different
charts are not equal and the differences are not negligible.
\end{rem}

~
\begin{rem}
If $\alpha^{\parallel}<1-\alpha^{^{\perp}}$ then the right hand side
of (\ref{eq:Tr(Pi_rho)}) is simply $O\left(\delta^{\parallel}\left(\eta\right)\right)$,
because we have $\Delta\left(\rho\right)\asymp\left\langle \left|\eta\right|\right\rangle ^{-\left(1-\alpha^{\perp}\right)}$
and $\delta^{\parallel}\left(\eta\right)\asymp\left\langle \left|\eta\right|\right\rangle ^{-\alpha^{\parallel}}$.
\end{rem}

\begin{proof}
We have
\begin{align}
\mathrm{Id}_{/C^{\infty}\left(M\right)} & \underset{(\ref{eq:I*_I}),(\ref{eq:resol_identity_B})}{=}\left(\mathcal{B}\circ I\right)^{\dagger}\circ\left(\mathcal{B}\circ I\right)\label{eq:res_id-1}\\
 & \underset{(\ref{eq:resol_identity_B})}{=}\sum_{j}\int_{\mathbb{R}^{2n+2}}I_{j}^{\dagger}\circ\pi\left(y',\eta'\right)\circ I_{j}\frac{dy'd\eta'}{\left(2\pi\right)^{n+1}}\nonumber \\
 & \underset{(\ref{eq:def_psi})}{=}\sum_{j}\int_{\left(y,\eta\right)\in T^{*}V_{j}}I_{j}^{\dagger}\circ\pi\left(\psi\left(y\right),\eta\right)\circ I_{j}\left|\mathrm{det}\left(d\psi\left(y\right)\right)\right|\frac{dyd\eta}{\left(2\pi\right)^{n+1}}\nonumber \\
 & =\int_{T^{*}M}\Pi\left(\rho\right)\frac{d\rho}{\left(2\pi\right)^{n+1}}\eq{\ref{eq:Pi_(rho)},\ref{eq:def_T_in-1},\ref{eq:def_T_*}}\mathcal{T}^{\dagger}\mathcal{T}\nonumber 
\end{align}
where, for $\rho\in T^{*}M$, $m=\pi\left(\rho\right)\in M$, we set
\begin{equation}
\Pi\left(\rho\right):=\sum_{j;y\in U_{j}}\left|\mathrm{det}\left(d\psi\left(\kappa_{j}\left(m\right)\right)\right)\right|\cdot I_{j}^{\dagger}\circ\pi\left(\tilde{\psi}\left(\tilde{\kappa}_{j}\left(\rho\right)\right)\right)\circ I_{j}\underset{(\ref{eq:def_wave_packet_Phi_j_rho})}{=}\sum_{j}|\Phi_{\rho,j}\rangle\langle\Phi_{\rho,j}|\cdot\rangle.\label{eq:Pi_rho}
\end{equation}
The last operator $\Pi\left(\rho\right)$ is of finite rank and $L^{2}$-self-adjoint.
Its trace is
\begin{equation}
\mathrm{Tr}\left(\Pi\left(\rho\right)\right)\eq{\ref{eq:Pi_rho},\ref{eq:def_pi_wave_packet}}\sum_{j;y\in U_{j}}\left|\mathrm{det}\left(d\psi\left(\kappa_{j}\left(m\right)\right)\right)\right|\langle\varphi_{\tilde{\psi}\left(\tilde{\kappa}_{j}\left(\rho\right)\right)}|I_{j}I_{j}^{\dagger}\varphi_{\tilde{\psi}\left(\tilde{\kappa}_{j}\left(\rho\right)\right)}\rangle.\label{eq:eq}
\end{equation}
On the right hand side, we have
\begin{align*}
\langle\varphi_{\tilde{\psi}\left(\tilde{\kappa}_{j}\left(\rho\right)\right)}|I_{j}I_{j}^{\dagger}\varphi_{\tilde{\psi}\left(\tilde{\kappa}_{j}\left(\rho\right)\right)}\rangle & \underset{(\ref{eq:def_I*-2})}{=}\int_{y'\in\mathbb{R}^{n+1}}\chi_{j}^{2}\left(y'\right)\left|\mathrm{det}d\kappa_{j}\left(\kappa_{j}^{-1}\left(y'\right)\right)\right|\left|\varphi_{\tilde{\psi}\left(\tilde{\kappa}_{j}\left(\rho\right)\right)}\left(y'\right)\right|^{2}dy',\\
 & =\chi_{j}^{2}\left(\kappa_{j}\left(m\right)\right)\left|\mathrm{det}d\kappa_{j}\left(m\right)\right|\left\Vert \varphi_{\tilde{\psi}\left(\tilde{\kappa}_{j}\left(\rho\right)\right)}\right\Vert ^{2}+O\left(\delta^{\parallel}\left(\eta\right)\right)\\
 & \eq{\ref{eq:partition_quadratic_unity-1},\eqref{eq:norme_wave_packet}}1+O\left(\Delta^{1-\epsilon}\left(\rho\right)+\delta^{\parallel}\left(\eta\right)\right).
\end{align*}
The remainder $O\left(\delta^{\parallel}\left(\eta\right)\right)$
in the second line comes from the size of wave packets along the flow
direction and Taylor expansion of the smooth functions in the integral.
We deduce \eqref{eq:Tr(Pi_rho)}.
\end{proof}
\begin{cBoxB}{}
\begin{prop}[Wave packet projector]
\label{lem:Bargmann projector}The operator
\begin{equation}
\mathcal{P}:=\mathcal{T}\mathcal{T}^{\dagger}:L^{2}\left(T^{*}M;\mathbb{C}^{J},\frac{d\rho}{\left(2\pi\right)^{n+1}}\right)\rightarrow L^{2}\left(T^{*}M;\mathbb{C}^{J},\frac{d\rho}{\left(2\pi\right)^{n+1}}\right)\label{eq:BargmannProjector}
\end{equation}
is the orthogonal projector on the image of $\mathcal{T}$. It is
called \textbf{wave packet projector}. Its Schwartz kernel 
\begin{equation}
\langle\delta_{\rho',j'},|\mathcal{P}\delta_{\rho,j}\rangle_{L^{2}\left(T^{*}M\right)}=\langle\Phi_{j',\rho'}|\Phi_{j,\rho}\rangle_{L^{2}\left(M\right)}\label{eq:kernel_T_T*}
\end{equation}
decays rapidly on the outside the diagonal: for any $N\text{\ensuremath{\ge}0}$,
there exists $C_{N}>0$ such that 
\begin{equation}
\left|\langle\delta_{\rho',j'}|\mathcal{P}\delta_{\rho,j}\rangle_{L^{2}\left(T^{*}M\right)}\right|=\left|\langle\Phi_{j',\rho'}|\Phi_{j,\rho}\rangle_{L^{2}\left(M\right)}\right|\leq C_{N}\left\langle \mathrm{dist}_{g}\left(\rho',\rho\right)\right\rangle ^{-N}\label{eq:estimate_Bergman_kernel}
\end{equation}
for any $\rho,\rho'\in T^{*}M$.
\end{prop}

\end{cBoxB}

\begin{proof}
For the former claim, we refer to Remark \ref{rem:Projection}. The
estimate (\ref{eq:estimate_Bergman_kernel}) is a consequence of \eqref{eq:size_wave_packet}.
(We will show a more general statement in Lemma \ref{thm:Microlocality-of-the_TO}.)
\end{proof}

\subsection{\label{subsec:Pseudo-differential-operator}Pseudo-differential operators}

\subsubsection{Pseudo-differential operators (PDO)}

In this section we use the family of operators $\Pi\left(\rho\right)$
in (\ref{eq:Pi_(rho)}) to define pseudo-differential operators on
$C^{\infty}\left(M\right)$.

\begin{cBoxA}{}
\begin{defn}[Pseudo-differential operator $\mathrm{Op}\left(a\right)$]
\label{def:Op}Let $a\in L_{\mathrm{loc}}^{\infty}\left(T^{*}M\right)$
whose growth at infinity is temperate in the sense that, for some
constant $C>0$ and $N\ge0$, we have
\begin{equation}
\left|a\left(\rho\right)\right|\le C\langle\left|\rho\right|\rangle^{N}\quad\text{for all }\rho\in T^{*}M.\label{eq:symbol_temperate}
\end{equation}
For such a function $a$, we define the pseudo-differential operator
(PDO) with the symbol $a$
\[
\mathrm{Op}\left(a\right):C^{\infty}\left(M\right)\to C^{\infty}\left(M\right)
\]
 by 
\begin{align}
\mathrm{Op}\left(a\right): & =\int_{T^{*}M}a\left(\rho\right)\Pi\left(\rho\right)\frac{d\rho}{\left(2\pi\right)^{n+1}}\eq{\ref{eq:Pi_(rho)},\ref{eq:def_T_in-1},\ref{eq:def_T_*}}\mathcal{T}^{\dagger}\mathcal{M}_{a}\mathcal{T}\label{eq:def_operator}
\end{align}
where $\mathcal{M}_{a}:\mathcal{S}\left(T^{*}M;\mathbb{C}^{J}\right)\rightarrow\mathcal{S}\left(T^{*}M;\mathbb{C}^{J}\right)$
denotes the component-wise multiplication by $a$, that is, $\mathcal{M}_{a}v\left(\rho\right)=a\left(\rho\right)v\left(\rho\right)$.
We have $\mathrm{Op}\left(1\right)\underset{(\ref{eq:resol_ident_Pi_rho})}{=}\mathrm{Id}$.
\end{defn}

\end{cBoxA}

\begin{rem}
The quantization formula (\ref{eq:def_operator}) is usually called
anti-Wick quantization \cite[Def. 1.7.3]{nicola_rodino_livre_11}
or Toeplitz quantization \cite[chap 13.4]{zworski_book_2012} (or
Wick quantization in \cite[chap 2.4.1]{lerner2011metrics}) or coherent
states quantization. The function $a$ is called the (anti-Wick) symbol
of the operator $\mathrm{Op}\left(a\right)$.
\end{rem}

~
\begin{rem}
From Remark \ref{rem:As-in-}, for $u,v\in C^{\infty}\left(M\right)$,
we have $\mathcal{T}u,\mathcal{T}v\in\mathcal{S}\left(T^{*}M\right)$
and
\[
\langle u|\mathrm{Op}\left(a\right)v\rangle_{L^{2}\left(M\right)}\eq{\ref{eq:def_operator}}\langle\left(\mathcal{T}u\right)\overline{\left(\mathcal{T}v\right)}|a\rangle_{L^{2}\left(T^{*}M\right)},
\]
we deduce that the map $\mathrm{Op}:a\in L_{\mathrm{loc}}^{\infty}\left(T^{*}M\right)\mapsto\mathrm{Op}\left(a\right)$
defined in (\ref{eq:def_operator}) can be extended to symbols $a$
that are a tempered distribution:
\[
\mathrm{Op}:\begin{cases}
\mathcal{S}'\left(T^{*}M\right) & \rightarrow L\left(\mathcal{S}\left(M\right),\mathcal{S}'\left(M\right)\right)\\
a & \mapsto\mathrm{Op}\left(a\right)
\end{cases},
\]
where $L\left(\mathcal{S}\left(M\right),\mathcal{S}'\left(M\right)\right)$
stands for linear operators from $\mathcal{S}\left(M\right)$ to $\mathcal{S}'\left(M\right)$ 
\end{rem}

\begin{cBoxB}{}
\begin{prop}[Basic properties of PDO]
\label{prop:We-havewhere-}If $a\in L^{\infty}(T^{*}M)$, the operator
$\mathrm{Op}\left(a\right)$ extends to a bounded operator on $L^{2}(M)$
and we have
\begin{equation}
\left\Vert \mathrm{Op}\left(a\right)\right\Vert _{L^{2}\left(M\right)}\underset{}{\leq}\left|a\right|_{L^{\infty}\left(T^{*}M\right)}.\label{eq:norme_symbol}
\end{equation}
If $a\in L^{1}(T^{*}M)$, the operator $\mathrm{Op}\left(a\right)$
extends to a trace class operator on $L^{2}(M)$ whose trace is given
by 
\begin{equation}
\mathrm{Tr}\left(\mathrm{Op}\left(a\right)\right)\underset{(\ref{eq:def_operator})}{=}\int a\left(\rho\right)\mathrm{Tr}\left(\Pi\left(\rho\right)\right)\frac{d\rho}{\left(2\pi\right)^{n+1}}\label{eq:Tr_Opa}
\end{equation}
and whose trace norm is bounded as $\forall\epsilon>0$,
\begin{equation}
\left\Vert \mathrm{Op}\left(a\right)\right\Vert _{\mathrm{Tr}}\underset{(\ref{eq:Tr(Pi_rho)})}{\leq}\left\Vert a\left(\rho\right)\left(1+O\left(\Delta^{1-\epsilon}\left(\rho\right)+\delta^{\parallel}\left(\rho\right)\right)\right)\right\Vert _{L^{1}\left(T^{*}M\right)}.\label{eq:norm_Tr_Op}
\end{equation}
\end{prop}

\end{cBoxB}

\begin{proof}
The former claim on the operator norm follows from Lemma \ref{lem:Bargmann projector}
and the expression \eqref{eq:Prop_Opa}. (See the proof of Theorem
\ref{thm:Composition-of-PDO.} where we detail this and prove a more
general statement.) The latter claims on the trace norm are consequences
of the expression \eqref{eq:def_operator} and the estimate \eqref{eq:Tr(Pi_rho)}.
\end{proof}
\begin{rem}
From (\ref{eq:norme_symbol}) it is natural to consider the norm $L^{\infty}\left(T^{*}M\right)$
on symbols space so that the linear operator $\mathrm{Op}:L^{\infty}\left(T^{*}M\right)\rightarrow\mathcal{L}\left(L^{2}\left(M\right),L^{2}\left(M\right)\right)$
is bounded. We will use this to derive expressions as in (\ref{eq:res1}).
\end{rem}

~
\begin{rem}
We can write the PDO $\mathrm{Op}\left(a\right)$ more explicitly
by using local charts as 
\begin{equation}
\mathrm{Op}\left(a\right)=\left(\mathcal{B}I\right)^{\dagger}\mathcal{M}_{\tilde{a}}\left(\mathcal{B}I\right)\label{eq:Prop_Opa}
\end{equation}
where $\mathcal{M}_{\tilde{a}}$ denotes the component-wise multiplication
on $\oplus_{j=1}^{J}\mathbb{R}_{\varrho}^{2\left(n+1\right)}$ by
functions $\tilde{a}=\left(\tilde{a}_{j}\right)_{1\le j\le J}$ where
$\tilde{a}_{j}:T^{*}\mathbb{R}^{n+1}\to\mathbb{C}$ is the function
$a$ viewed in the local chart $\tilde{\psi}\circ\tilde{\kappa}_{j}:T^{*}U_{j}\to T^{*}\mathbb{R}^{n+1}$:
\begin{equation}
\tilde{a}_{j}\left(\varrho'\right)=a\left(\tilde{\kappa}_{j}^{-1}\left(\tilde{\psi}^{-1}\left(\varrho'\right)\right)\right)\quad\text{for }\varrho'\in T^{*}\mathbb{R}^{n+1}.\label{eq:def_a_tilde_j}
\end{equation}
Indeed, recalling the definitions of the operator $\mathcal{B}$ in
\eqref{eq:def_Bargamnn_B}, we have
\end{rem}

\begin{align}
\mathrm{Op}\left(a\right) & =\int_{T^{*}M}a\left(\rho\right)\Pi\left(\rho\right)\frac{d\rho}{\left(2\pi\right)^{n+1}}\label{eq:Op_a}\\
 & \underset{(\ref{eq:Pi_(rho)})}{=}\sum_{j}\int a\left(\tilde{\kappa}_{j}^{-1}\left(\varrho\right)\right)\left|\mathrm{det}\left(d\psi\left(y\right)\right)\right|I_{j}^{\dagger}\pi\left(\tilde{\psi}\left(\varrho\right)\right)I_{j}\frac{d\varrho}{\left(2\pi\right)^{n+1}}\nonumber \\
 & \underset{(\ref{eq:def_psi_tilde})}{=}\sum_{j}\int_{\mathbb{R}^{2n+2}}a\left(\tilde{\kappa}_{j}^{-1}\left(\tilde{\psi}^{-1}\left(\varrho'\right)\right)\right)I_{j}^{\dagger}\pi\left(\varrho'\right)I_{j}\frac{dy'd\eta}{\left(2\pi\right)^{n+1}}\nonumber \\
 & =\left(\mathcal{B}I\right)^{\dagger}\mathcal{M}_{\tilde{a}}\left(\mathcal{B}I\right)\nonumber 
\end{align}
where $\varrho=(y,\eta)\in T^{*}V_{j}$ and $\varrho'=(y',\eta)=\tilde{\psi}\left(\varrho\right)\in T^{*}\mathbb{R}^{n+1}$.

\subsubsection{\label{subsec:Sobolev-space}The Sobolev spaces $\mathcal{H}_{W}\left(M\right)$}

Next we define generalized Sobolev spaces $\mathcal{H}_{W}\left(M\right)$
from a weight function $W\in C\left(T^{*}M;\mathbb{R}^{+}\right)$
for which we put the following assumption.

\begin{cBoxA}{}
\begin{defn}[Temperate weight function]
\label{def:temperate_W-1}We consider a family of strictly positive
continuous functions $W_{h_{0}}\in C\left(T^{*}M;\mathbb{R}^{+}\right)$
that depends on parameter $h_{0}>0$. This family is said to be \textbf{$\left(C_{W},N_{W}\right)$-temperate}
for some constants $C_{W}\geq1,N_{W}>0$ if for every $h_{0}>0$ small
enough, we have

\begin{equation}
\frac{W_{h_{0}}\left(\rho'\right)}{W_{h_{0}}\left(\rho\right)}\leq C_{W}\left\langle h_{0}\,\mathrm{dist}_{g}\left(\rho',\rho\right)\right\rangle ^{N_{W}}\quad\text{for every }\rho,\rho'\in T^{*}U_{j}\text{ and }1\le j\le J.\label{eq:temperate_property-W-2}
\end{equation}
For simplicity we will write $W=W_{h_{0}}$. Let us write the unit
co-sphere bundle for the metric $g_{M}$ on $M$ by $S^{*}M:=\left\{ \rho\in T^{*}M;\,\left\Vert \rho\right\Vert _{g_{M}}=1\right\} $.
Then the \textbf{order function} $r_{W}:S^{*}M\to\mathbb{R}$ of $W$
is the function that expresses the (upper) growth rate of the function
$W$ in each direction and defined precisely by
\begin{equation}
r_{W}\left(\rho\right):=\inf\left\{ r\in\mathbb{R}\:\mid\exists C>0,\forall\alpha>1,W\left(\alpha\rho\right)\leq C\alpha^{r}\right\} .\label{eq:def_order-2}
\end{equation}
\end{defn}

\end{cBoxA}

\begin{rem}
The temperate property of $W$ implies that the function $r_{W}$
is bounded.
\end{rem}

~
\begin{rem}
\label{rem:h0-model}The role of the small parameter $h_{0}$ in (\ref{eq:temperate_property-W-2})
is somehow to assume that the function $W$ is flat with respect to
the metric (or with respect to the size of wave packets). Here is
a simple toy model to show how we will use this parameter $h_{0}$
technically. Let $x\in\mathbb{R}$, $A\left(x\right)=e^{-x^{2}}$,
$W_{h_{0}}\left(x\right):=\left\langle h_{0}x\right\rangle ^{N_{W}}$
and 
\[
I\left(h_{0},N_{W}\right):=\int_{\mathbb{R}}A\left(x\right)W_{h_{0}}\left(x\right)dx.
\]
We have for fixed $h_{0}>0$ that $I\left(h_{0},N_{W}\right)\underset{N_{W}\rightarrow+\infty}{\rightarrow}+\infty$.
However, from the estimate $\forall N,\exists C_{N},\forall x\left|A\left(x\right)\right|\leq C_{N}\left\langle x\right\rangle ^{-N}$
and $\left\Vert A\left(x\right)\right\Vert _{L^{1}}\leq C$, we deduce
that
\begin{equation}
\exists C>0,\forall N_{W}>0,\exists h_{0}>0,\quad\left|I\left(h_{0},N_{W}\right)\right|\leq C.\label{eq:est}
\end{equation}
\begin{proof}
To show (\ref{eq:est}), we split the integral in two parts: (1) $\left|x\right|\leq h_{0}^{-1}$
where $W_{h_{0}}\left(x\right)\asymp1$ and $\int_{\left|x\right|\leq h_{0}^{-1}}A\left(x\right)W\left(x\right)dx\leq C$
independent on $N_{W},h_{0}$ and (2) $\left|x\right|\geq h_{0}^{-1}$,
where 
\begin{align*}
\int_{\left|x\right|\geq h_{0}^{-1}}A\left(x\right)W\left(x\right)dx & \leq C_{N}h_{0}^{N_{W}}\int_{\left|x\right|\geq h_{0}^{-1}}\left|x\right|^{N_{W}-N}dx\\
 & \leq\frac{C_{N}h_{0}^{N_{W}}}{N-N_{W}-1}h_{0}^{-N_{W}+N-1}=\frac{C_{N}}{N-N_{W}-1}h_{0}^{N-1},
\end{align*}
for some large $N>N_{W}+1$. We finally take $h_{0}$ small enough.
\end{proof}
\end{rem}

~

We define the Sobolev spaces $\mathcal{H}_{W}\left(M\right)$ as follows.
This definition is similar to the definition of Sobolev spaces in
micro-local analysis given in \cite[Def.2.6.1]{lerner2011metrics}
or \cite[Section 1.7.4]{nicola_rodino_livre_11}.

\begin{cBoxA}{}
\begin{defn}[The Sobolev space $\mathcal{H}_{W}\left(M\right)$]
\label{def:Anisotropic_Sob_space}Let $W\in C\left(T^{*}M;\mathbb{R}^{+}\right)$
be a temperate function (\ref{eq:temperate_property-W-2}) called
weight function. For $u,v\in C^{\infty}\left(M\right)$, we define
the $\mathcal{H}_{W}$-scalar product by
\begin{align}
\left\langle u|v\right\rangle _{\mathcal{H}_{W}}: & =\left\langle u|\mathrm{Op}\left(W^{2}\right)v\right\rangle _{L^{2}\left(M,dm\right)}\underset{(\ref{eq:Op_a})}{=}\langle\mathcal{M}_{W}\mathcal{T}u|\mathcal{M}_{W}\mathcal{T}v\rangle_{L^{2}\left(T^{*}M,\frac{d\rho}{\left(2\pi\right)^{n+1}}\right)}.\label{eq:def_H_W_M}
\end{align}
The associated $\mathcal{H}_{W}$-norm is defined by
\begin{align}
\left\Vert u\right\Vert _{\mathcal{H}_{W}}^{2}:=\langle u|u\rangle_{\mathcal{H}_{W}} & =\left\Vert \mathcal{M}_{W}\mathcal{T}u\right\Vert _{L^{2}\left(T^{*}M\right)}^{2}.\label{eq:def_Sobolev_norm}
\end{align}
The Sobolev space $\mathcal{H}_{W}\left(M\right)$ is defined as the
Hilbert space obtained by completion of $C^{\infty}\left(M\right)$
with respect to the norm (\ref{eq:def_Sobolev_norm}): 
\begin{equation}
\mathcal{H}_{W}\left(M\right):=\overline{C^{\infty}\left(M\right)}^{\left\Vert .\right\Vert _{\mathcal{H}_{W}}}.\label{eq:def_H_W}
\end{equation}
\end{defn}

\end{cBoxA}

By definition we have the isometric embedding
\begin{equation}
\mathcal{M}_{W}\mathcal{T}\quad:\mathcal{H}_{W}\left(M\right)\rightarrow L^{2}\left(T^{*}M;\mathbb{C}^{J}\right).\label{eq:isometry}
\end{equation}
We have
\[
H^{r_{\mathrm{max}}}\left(M\right)\subset\mathcal{H}_{W}\left(M\right)\subset H^{r_{\mathrm{min}}}\left(M\right)
\]
where $H^{r}\left(M\right)$ denotes the usual Sobolev space of constant
order $r\in\mathbb{R}$ \cite[chap.3]{taylor_tome1} and $r_{\mathrm{max}}:=\max_{S^{*}M}r_{W}$,
$r_{\mathrm{min}}:=-\max_{S^{*}M}r_{W^{-1}}$ with the order function
$r_{W}$ defined in Definition \ref{def:temperate_W-1}. In particular,
if $W\equiv1$, we have $\left\Vert u\right\Vert _{W}=\left\Vert u\right\Vert _{L^{2}\left(M\right)}$
and $\mathcal{H}_{W}\left(M\right)=L^{2}\left(M\right)$ from (\ref{eq:res_id-1}).

\subsubsection{How to estimate the operator norm using Schur Lemma}

We give here a general remark about a way (very common in microlocal
analysis) to estimate the operator norm $\left\Vert B\right\Vert _{\mathcal{H}_{W}}$
of a given operator $B:\mathcal{H}_{W}\left(M\right)\rightarrow\mathcal{H}_{W}\left(M\right)$.
We consider the commutative diagram:
\begin{equation}
\begin{CD}\mathcal{H}_{W}\left(M\right)@>{B}>>\mathcal{H}_{W}\left(M\right)\\
@V{\mathcal{M}_{W}\mathcal{T}}VV@V{\mathcal{M}_{W}\mathcal{T}}VV\\
L^{2}\left(T^{*}M;\mathbb{C}^{J}\right)@>{B_{W}}>>L^{2}\left(T^{*}M;\mathbb{C}^{J}\right)
\end{CD}\label{eq:lift_diagram-1}
\end{equation}
where $B_{W}$ denotes the lifted operator defined by
\begin{equation}
B_{W}:=\mathcal{M}_{W}\mathcal{T}B\mathcal{T}^{\dagger}\mathcal{M}_{W^{-1}}.\label{eq:def_BW}
\end{equation}
From (\ref{eq:isometry}), $\mathcal{M}_{W}\mathcal{T}$ is an isometric
embedding and $\mathcal{M}_{W}^{\dagger}=\mathcal{M}_{W^{-1}}$, hence
$\left\Vert \mathcal{M}_{W}\mathcal{T}\right\Vert =1$, $\left\Vert \mathcal{T}^{\dagger}\mathcal{M}_{W^{-1}}\right\Vert =\left\Vert \left(\mathcal{M}_{W}\mathcal{T}\right)^{\dagger}\right\Vert =1$
and the operator norm of $B$ on $\mathcal{H}_{W}\left(M\right)$
is equal to that of $B_{W}$ on $L^{2}\left(T^{*}M;\mathbb{C}^{J}\right)$
(recall Remark \ref{rem:Projection}.):
\begin{equation}
\left\Vert B\right\Vert _{\mathcal{H}_{W}}=\left\Vert B_{W}\right\Vert _{L^{2}\left(T^{*}M;\mathbb{C}^{J}\right)}.\label{eq:compare_norms}
\end{equation}
We can estimate $\left\Vert B_{W}\right\Vert _{L^{2}\left(T^{*}M;\mathbb{C}^{J}\right)}$
from its Schwartz kernel
\[
\langle\delta_{\rho',j'}|B_{W}\delta_{\rho,j}\rangle=\langle\delta_{\rho',j'}|\mathcal{T}B\mathcal{T}^{\dagger}\delta_{\rho,j}\rangle\frac{W\left(\rho'\right)}{W\left(\rho\right)}
\]
 and using Schur Lemma:

\begin{cBoxB}{}
\begin{lem}[Schur Lemma]
\label{lem:Schur-Lemma-.}\cite[Lemma 2.8.4 p.50]{martinez-01}.
\begin{align}
\left\Vert B_{W}\right\Vert _{L^{2}\left(T^{*}M;\mathbb{C}^{J}\right)} & \leq\label{eq:schur_lemma}\\
 & \left(\sup_{\rho',j'}\sum_{j}\int\left|\langle\delta_{\rho',j'}|B_{W}\delta_{\rho,j}\rangle\right|d\rho\right)^{1/2}\left(\sup_{\rho,j}\sum_{j'}\int\left|\langle\delta_{\rho',j'}|B_{W}\delta_{\rho,j}\rangle\right|d\rho'\right)^{1/2}.\nonumber 
\end{align}
\end{lem}

\end{cBoxB}

\begin{rem}
\label{rem:To-estimate-the}To estimate the trace norm $\left\Vert B\right\Vert _{\mathrm{Tr}_{\mathcal{H}_{W}}}=\left\Vert B_{W}\right\Vert _{\mathrm{Tr}_{L^{2}}}$
we note that the operator $B$ is written as an integral of rank one
operators: 
\begin{align}
B\eq{\ref{eq:def_BW}} & \mathcal{T}^{\dagger}W^{-1}B_{W}W\mathcal{T}\label{eq:expression_integral_of_rankone}\\
= & \sum_{j,j'}\int\langle\delta_{\rho',j'}|B_{W}\delta_{\rho,j}\rangle\cdot\frac{W\left(\rho\right)}{W\left(\rho'\right)}\cdot|\Phi_{\rho',j'}\rangle\langle\Phi_{\rho,j}|.\rangle_{L^{2}}\frac{d\rho'}{\left(2\pi\right)^{n+1}}\frac{\text{d\ensuremath{\rho}}}{\left(2\pi\right)^{n+1}}\nonumber 
\end{align}
and $\left\Vert |\Phi_{\rho',j'}\rangle\langle\Phi_{\rho,j}|.\rangle_{L^{2}}\right\Vert _{\mathrm{Tr_{\mathcal{H}_{W}}}}=\left\Vert \Phi_{\rho,j}\right\Vert _{\mathcal{H}_{W}}^{2}\leq C$
uniformly.
\end{rem}

\subsubsection{\label{subsec:Description-of-the}Description of the vector field
$X$ by a PDO}

We first introduce a definition that will be used to express that
some operator $R$ is ``under control'' or ``negligible'' in our
analysis. It will means that the Schwartz kernel of $\mathcal{T}R\mathcal{T}^{\dagger}$
decays very fast outside the diagonal graph of $\tilde{\phi}^{t}$
defined in (\ref{eq:lifted_flow}) and moreover that on the graph,
the Schwartz kernel is bounded by a given positive function $h\left(\rho\right)$.

\begin{cBoxA}{}
\begin{defn}
\label{def:Let--be-1}Let $t\in\mathbb{R}$ and a positive function
$h:\rho\in T^{*}M\rightarrow h\left(\rho\right)\in\mathbb{R}^{+}$.
We define $\Psi_{\tilde{\phi}^{t}}\left(h\right)$ as the set of operators
$R:\mathcal{S}\left(M\right)\rightarrow\mathcal{S}'\left(M\right)$
such that for any $N>0$, there exists a constant $C_{N,t}>0$ such
that for any $\rho,\rho'\in T^{*}M$
\begin{align}
\left|\langle\delta_{\rho'}|\mathcal{T}R\mathcal{T}^{\dagger}\delta_{\rho}\rangle_{L^{2}\left(T^{*}M\right)}\right| & \leq C_{N,t}\left\langle \mathrm{dist}_{g}\left(\rho',\tilde{\phi}^{t}\left(\rho\right)\right)\right\rangle ^{-N}h\left(\rho\right).\label{eq:microl_estimate-2}
\end{align}
In particular for $t=0$, we simply write $\Psi\left(h\right):=\Psi_{\tilde{\phi}^{t}}\left(h\right)$.
\end{defn}

\end{cBoxA}

\begin{rem}
$h$ is defined up to constant scaling i.e. for any $C>0,$$\Psi_{\tilde{\phi}^{t}}\left(h\right)=\Psi_{\tilde{\phi}^{t}}\left(Ch\right).$
\end{rem}

\begin{cBoxB}{}
\begin{prop}[Principal symbol of the vector field $X$]
\label{prop:We-havewith-a}We have
\begin{equation}
-X=\mathrm{Op}\left(i\boldsymbol{\omega}\right)+R\label{eq:symbol_X}
\end{equation}
with a remainder $R\in\Psi\left(\left\langle \left|\rho\right|\right\rangle ^{\alpha^{\parallel}}\right)$,
where $\boldsymbol{\omega}$ denotes the frequency function (\ref{eq:omega_function})
and $0\leq\alpha^{\parallel}<1$ is a parameter for the metric $g$
along the flow direction in (\ref{eq:conditions}).
\end{prop}

\end{cBoxB}

\begin{proof}
For $\rho'\in T^{*}M$, we put $\omega'=\boldsymbol{\omega}\left(\rho'\right)$.
Since $\left(-X\right)=\frac{\partial}{\partial z}$ in flow box coordinates,
we check from (\ref{eq:def_wave_packet_Phi_j_rho}) and (\ref{eq:wave_packet})
(or or the asymptotic expression (\ref{eq:wave_packet_1})) that $\forall N>0,\exists C_{N}>0$,$\forall\rho,\rho'$,$\forall j,j'$,
\begin{equation}
\left|\langle\Phi_{\rho',j'}|\left(-X-i\omega'\right)\Phi_{\rho,j}\rangle\right|=\left|\langle\left(-X-i\omega'\right)\Phi_{\rho',j'}|\Phi_{\rho,j}\rangle\right|\leq\delta^{\parallel}\left(\rho\right)^{-1}C_{N}\left\langle \mathrm{dist}_{g}\left(\rho',\rho\right)\right\rangle ^{-N},\label{eq:X_omega}
\end{equation}
where $\delta^{\parallel}\left(\rho\right)^{-1}\eq{\ref{eq:def_delta}}\left\langle \left|\rho\right|\right\rangle ^{\alpha^{\parallel}}$.
\end{proof}
\begin{rem}
The error term $R$ in (\ref{eq:symbol_X}) dominates the principal
term on the domain $\rho=\left(\xi,\omega\right)\in T^{*}M\text{ s.t. }\left|\xi\right|^{\alpha^{\parallel}}\geq\omega$. 
\end{rem}

\subsubsection{Continuity theorem for PDO}

The next theorems are a variant of standard theorem for PDO, called
continuity theorem and composition theorem. We state and prove them
here for the PDO's that we have defined in (\ref{eq:def_operator}).

Recall that in the definition (\ref{eq:temperate_property-W-2}) of
temperate property of the weight $W$, there are two parameters $N_{W},h_{0}$.
In the next Theorem, we obtain a bound for the operator norm $\left\Vert \mathrm{Op}\left(a\right)\right\Vert _{\mathcal{H}_{W}\left(M\right)\rightarrow\mathcal{H}_{W}\left(M\right)}$.
This bound can be uniform with respect to $N_{W}$ if $h_{0}>0$ is
chosen small enough accordingly. 

\begin{cBoxB}{}
\begin{thm}[Continuity theorem]
\label{thm:Continuity-of-PDO}Suppose that $W$ is $\left(C_{W},N_{W}\right)$-temperate
according to property (\ref{eq:temperate_property-W-2}). Then there
exists a constant $C>0$ such that for any bounded measurable function
$a\in L^{\infty}\left(T^{*}M\right)$, the PDO $\mathrm{Op}\left(a\right)$
extends to a bounded operator on $\mathcal{H}_{W}\left(M\right)$
with the following estimate
\[
\left\Vert \mathrm{Op}\left(a\right)\right\Vert _{\mathcal{H}_{W}\left(M\right)\rightarrow\mathcal{H}_{W}\left(M\right)}\leq C\|a\|_{L^{\infty}}.
\]
Moreover, $C$ may depend on $C_{W}$ but not on $N_{W}$ if $h_{0}$
is taken small enough.
\end{thm}

\end{cBoxB}

\begin{proof}
We write
\[
\left(\mathrm{Op}\left(a\right)\right)_{W}\eq{\ref{eq:def_BW}}W\mathcal{P}\mathcal{M}_{a}\mathcal{P}W^{-1}
\]
So
\[
\langle\delta_{\rho'}|\left(\mathrm{Op}\left(a\right)\right)_{W}\delta_{\rho}\rangle=\frac{W\left(\rho'\right)}{W\left(\rho\right)}\int\langle\delta_{\rho'}|\mathcal{P}\delta_{\rho_{1}}\rangle a\left(\rho_{1}\right)\langle\delta_{\rho_{1}}|\mathcal{P}\delta_{\rho}\rangle\frac{d\rho_{1}}{\left(2\pi\right)^{n+1}},
\]
and $\forall N>0,\exists C_{N}>0,$
\begin{align*}
\left|\langle\delta_{\rho'}|\left(\mathrm{Op}\left(a\right)\right)_{W}\delta_{\rho}\rangle\right|\underset{(\ref{eq:temperate_property-W-2},\ref{eq:estimate_Bergman_kernel})}{\leq} & C_{W}\left\langle h_{0}\,\mathrm{dist}_{g}\left(\rho',\rho\right)\right\rangle ^{N_{W}}C_{N}\\
 & \qquad\int\left\langle \mathrm{dist}_{g}\left(\rho',\rho_{1}\right)\right\rangle ^{-N}\left|a\left(\rho_{1}\right)\right|\left\langle \mathrm{dist}_{g}\left(\rho_{1},\rho\right)\right\rangle ^{-N}\frac{d\rho_{1}}{\left(2\pi\right)^{n+1}}\\
\leq & C_{W}\left\langle h_{0}\,\mathrm{dist}_{g}\left(\rho',\rho\right)\right\rangle ^{N_{W}}C_{N}\|a\|_{L^{\infty}}\left\langle \mathrm{dist}_{g}\left(\rho',\rho\right)\right\rangle ^{-N}
\end{align*}
 If $h_{0}\mathrm{dist}_{g}\left(\rho',\rho\right)\leq1$ then $\forall N>0,\exists C_{N}>0,$

\[
\left|\langle\delta_{\rho'}|\left(\mathrm{Op}\left(a\right)\right)_{W}\delta_{\rho}\rangle\right|\leq C_{W}C_{N}\|a\|_{L^{\infty}}\left\langle \mathrm{dist}_{g}\left(\rho',\rho\right)\right\rangle ^{-N}
\]
with $C_{N}$ independent on $N_{W}$. If $h_{0}\mathrm{dist}_{g}\left(\rho',\rho\right)\geq1$
then
\begin{align}
\left\langle h_{0}\,\mathrm{dist}_{g}\left(\rho',\rho\right)\right\rangle ^{N_{W}}C_{N}\left\langle \mathrm{dist}_{g}\left(\rho',\rho\right)\right\rangle ^{-N} & \asymp C_{N}h_{0}^{N_{W}}\left\langle \mathrm{dist}_{g}\left(\rho',\rho\right)\right\rangle ^{-N+N_{W}}\label{eq:aide}\\
 & =C_{N'}\left\langle \mathrm{dist}_{g}\left(\rho',\rho\right)\right\rangle ^{-N'},\nonumber 
\end{align}
with $N'=N-N_{W}$ and $C_{N'}=C_{N}h_{0}^{N_{W}}$. If we take $h_{0}$
small enough with respect to $N_{W}$ we get again $\forall N'>0,\exists C_{N'}>0,$

\[
\left|\langle\delta_{\rho'}|\left(\mathrm{Op}\left(a\right)\right)_{W}\delta_{\rho}\rangle\right|\leq C_{W}C_{N'}\|a\|_{L^{\infty}}\left\langle \mathrm{dist}_{g}\left(\rho',\rho\right)\right\rangle ^{-N'}
\]
with $C_{N'}$ independent on $N_{W}$. Finally, by Schur Lemma \ref{lem:Schur-Lemma-.}
we conclude that the operator norm of $\mathrm{Op}\left(a\right)_{W}$
in $L^{2}$ is bounded by $C\|a\|_{L^{\infty}}$ and the same holds
true for $\mathrm{Op}\left(a\right)$ in $\mathcal{H}_{W}\left(M\right)$,
where $C$ depends on $C_{W}$ but does not depend on $N_{W}$, and
$h_{0}>0$ has been taken small enough depending on $N_{W}$.
\end{proof}

\subsubsection{Composition theorem for PDO}

\begin{cBoxA}{}
\begin{defn}
\label{def:Let--be}Let $h\in C\left(T^{*}M;\mathbb{R}^{+}\right)$
be some continuous function, let $N_{0}\geq0$ and $0<h_{0}<1$. A
bounded measurable function $b\in L^{\infty}\left(T^{*}M\right)$
is  said to be \textbf{$\left(h,N_{0},h_{0}\right)$-slowly varying
symbol with respect to the metric $g$} if
\begin{equation}
\left|b\left(\rho'\right)-b\left(\rho\right)\right|\leq h\left(\rho\right)\,\left\langle h_{0}\,\mathrm{dist}_{g}\left(\rho',\rho\right)\right\rangle ^{N_{0}}\quad\text{for all }\rho,\rho'\in T^{*}M.\label{eq:slow_variations}
\end{equation}
\end{defn}

\end{cBoxA}

\begin{rem}
The small parameter $h_{0}$ will be used in Lemma \ref{lem:For-the-weight}.
Recall Remark \ref{rem:h0-model} about the usefulness of a similar
small parameter $h_{0}$.
\end{rem}

The following lemma will be useful in the proof of the composition
theorem \ref{thm:Composition-of-PDO.} below. Recall $\mathcal{P}=\mathcal{T}\mathcal{T}^{\dagger}$
defined in (\ref{eq:BargmannProjector}). 

\begin{cBoxB}{}
\begin{lem}[Basic Lemma for slow varying symbols]
\label{lem:Basic_Lemma}Assume that $b\in C\left(T^{*}M\right)$
satisfies the slow variation property (\ref{eq:slow_variations})
with a function $h$ and some parameters $N_{0},h_{0}>0$. Then the
Schwartz kernel of the operator $\left[\mathcal{M}_{b},\mathcal{P}\right]=\mathcal{M}_{b}\mathcal{P}-\mathcal{P}\mathcal{M}_{b}$
satisfies $\forall N,\exists C_{N}>0,\exists h_{0}>0$, $\forall\rho',\rho\in T^{*}M,\forall j,j'\in\left\{ 1,\ldots J\right\} $,
\begin{equation}
\left|\langle\delta_{\rho',j'}|\left[\mathcal{M}_{b},\mathcal{P}\right]\delta_{\rho,j}\rangle\right|\leq C_{N}h\left(\rho\right)\left\langle \mathrm{dist}_{g}\left(\rho',\rho\right)\right\rangle ^{-N}.\label{eq:commut_P}
\end{equation}
Using the notation of Definition \ref{def:Let--be-1}, we can write
that $\left[\mathcal{M}_{b},\mathcal{P}\right]\in\Psi\left(h\right).$
\end{lem}

\end{cBoxB}

\begin{proof}
We have
\[
\langle\delta_{\rho',j'}|\left[\mathcal{M}_{b},\mathcal{P}\right]\delta_{\rho,j}\rangle=\langle\delta_{\rho',j'}|\mathcal{P}\delta_{\rho,j}\rangle\left(b\left(\rho'\right)-b\left(\rho\right)\right)
\]
Hence $\forall N>0,\exists C_{N}>0$,
\begin{align*}
\left|\langle\delta_{\rho',j'}|\left[\mathcal{M}_{b},\mathcal{P}\right]\delta_{\rho,j}\rangle\right|\underset{(\ref{eq:estimate_Bergman_kernel},\ref{eq:slow_variations})}{\leq} & C_{N}h\left(\rho\right)\left\langle \mathrm{dist}_{g}\left(\rho',\rho\right)\right\rangle ^{-N}\left\langle h_{0}\,\mathrm{dist}_{g}\left(\rho',\rho\right)\right\rangle ^{N_{0}}\\
\leq & C_{N}h\left(\rho\right)\left\langle \mathrm{dist}_{g}\left(\rho',\rho\right)\right\rangle ^{-N}.
\end{align*}
For the last line we proceed as in (\ref{eq:aide}) with $h_{0}$
small enough.
\end{proof}
\begin{rem}
In following Lemma and Theorems we will still have ``uniform estimates
w.r.t. $N_{W}$,$N_{0}$'' in the sense of letting $h_{0}>0$ be
small depending on $N_{W},N_{0}>0$ (as in Theorem \ref{thm:Continuity-of-PDO}),
and for this, we will use the shorter notation ``$\exists C>0,\forall N_{W},N_{0}>0,\exists h_{0}>0$''.
\end{rem}

\begin{cBoxA}{}
\begin{thm}[Composition theorem]
\label{thm:Composition-of-PDO.}

For a bounded measurable function $a\in L^{\infty}\left(T^{*}M\right)$
and slowly varying symbol $b\in L^{\infty}\left(T^{*}M\right)$ as
(\ref{eq:slow_variations}) with function $h$ and parameters $N_{0},h_{0}>0$,
let us consider the operator
\[
B:=\mathrm{Op}\left(a\right)\mathrm{Op}\left(b\right)-\mathrm{Op}\left(ab\right).
\]
For the Schwartz kernel of the corresponding lifted operator $B_{W}=\mathcal{M}_{W}\mathcal{T}B\mathcal{T}^{\dagger}\mathcal{M}_{W^{-1}}$
as defined in (\ref{eq:def_BW}) with parameter $N_{W}$ in (\ref{eq:temperate_property-W-2}),
we have $\forall N>0,\exists C_{N}>0,\forall N_{W},N_{0}>0,\exists h_{0}>0$,
$\forall\rho',\rho\in T^{*}M,\forall j,j'\in\left\{ 1,\ldots J\right\} $,
\[
\left|\langle\delta_{\rho',j'}|B_{W}\delta_{\rho,j}\rangle\right|\leq C_{N}\left|a\left(\rho\right)\right|h\left(\rho\right)\left\langle \mathrm{dist}_{g}\left(\rho',\rho\right)\right\rangle ^{-N}.
\]
Using the notation of Definition \ref{def:Let--be-1}, we can write
that $B_{W}\in\Psi\left(\left|a\right|h\right).$ Consequently $\exists C>0,\forall N_{W},N_{0}>0,\exists h_{0}>0$,
\begin{enumerate}
\item If $\left\Vert ah\right\Vert _{L^{\infty}}<\infty$, we have
\begin{equation}
\left\Vert B\right\Vert _{\mathcal{H}_{W}}\leq C\left\Vert ah\right\Vert _{L^{\infty}}.\label{eq:norm_Compos_PDO}
\end{equation}
\item If $\left\Vert ah\right\Vert _{L^{1}}<\infty$, the difference $\mathrm{Op}\left(a\right)\circ\mathrm{Op}\left(b\right)-\mathrm{Op}\left(ab\right)$
is a trace class operator on $\mathcal{H}_{W}\left(M\right)$ and
we have 
\[
\left\Vert B\right\Vert _{\mathrm{Tr}_{\mathcal{H}_{W}}}\leq C\left\Vert ah\right\Vert _{L^{1}}.
\]
\end{enumerate}
These claims hold if we exchange $\mathrm{Op}\left(a\right)$ and
$\mathrm{Op}\left(b\right)$ in the definition of $B$.
\end{thm}

\end{cBoxA}

\begin{proof}
We write
\[
B_{W}\eq{\ref{eq:def_BW}}W\mathcal{P}\mathcal{M}_{a}\left(\mathcal{P}\mathcal{M}_{b}-\mathcal{M}_{b}\mathcal{P}\right)\mathcal{P}W^{-1}
\]
So (we omit chart indices $j,j'$ for simplicity)
\begin{align*}
\langle\delta_{\rho'}|B_{W}\delta_{\rho}\rangle & =\frac{W\left(\rho'\right)}{W\left(\rho\right)}\int\langle\delta_{\rho'}|\mathcal{P}\delta_{\rho_{1}}\rangle a\left(\rho_{1}\right)\langle\delta_{\rho_{1}}|\left[\mathcal{P},\mathcal{M}_{b}\right]\delta_{\rho_{2}}\rangle\\
 & \qquad\qquad\qquad\langle\delta_{\rho_{2}}|\mathcal{P}\delta_{\rho}\rangle\frac{d\rho_{1}}{\left(2\pi\right)^{n+1}}\frac{d\rho_{2}}{\left(2\pi\right)^{n+1}},
\end{align*}
and
\begin{align*}
\left|\langle\delta_{\rho'}|B_{W}\delta_{\rho}\rangle\right|\underset{(\ref{eq:temperate_property-W-2},\ref{eq:estimate_Bergman_kernel},\ref{eq:slow_variations},\ref{eq:commut_P})}{\leq} & C_{W}\left\langle h_{0}\,\mathrm{dist}_{g}\left(\rho',\rho\right)\right\rangle ^{N_{W}}C_{N}\int\left\langle \mathrm{dist}_{g}\left(\rho',\rho_{1}\right)\right\rangle ^{-N}\\
 & \left|a\left(\rho_{1}\right)\right|h\left(\rho_{1}\right)\left\langle \mathrm{dist}_{g}\left(\rho_{1},\rho_{2}\right)\right\rangle ^{-N}\left\langle \mathrm{dist}_{g}\left(\rho_{2},\rho\right)\right\rangle ^{-N}\frac{d\rho_{1}}{\left(2\pi\right)^{n+1}}\frac{d\rho_{2}}{\left(2\pi\right)^{n+1}},\\
\leq & C_{W}C_{N}\left|a\left(\rho\right)\right|h\left(\rho\right)\left\langle \mathrm{dist}_{g}\left(\rho',\rho\right)\right\rangle ^{-N}
\end{align*}
In the last line, $h_{0}$ is chosen small enough with respect to
$N_{W},N_{0}$. We obtain Item 1 by Schur Lemma \ref{lem:Schur-Lemma-.}
and Item 2 using Remark \ref{rem:To-estimate-the}.
\end{proof}
The following corollary will be used in Section \ref{par:Contribution}
and can be skipped for the moment.

\begin{cBoxB}{}
\begin{cor}
\label{cor:B3}Let $\Omega_{R}$ be a compact subset on $T^{*}M$
which depends on a parameter $R\ge1$. Let $\boldsymbol{1}_{\Omega_{R}}:T^{*}M\rightarrow\left\{ 0,1\right\} $
the characteristic function of $\Omega_{R}$. Assume that $b_{R}:T^{*}M\rightarrow\mathbb{C}$
is a measurable function that takes constant value $1$ on the $R-$neighborhood
of $\Omega_{R}$, that is, 
\[
b_{R}\left(\rho'\right)=1\quad\text{whenever \ensuremath{\mathrm{dist}_{g}\left(\rho',\rho\right)\leq}R for \ensuremath{\rho\in\Omega_{R}} and \ensuremath{\rho'\in T^{*}M}.}
\]
Assume further that the growth of $b_{R}$ on the outside of $\Omega_{R}$
is temperate uniformly in $R$ in the sense that there exists $N_{0}>0$
and $C_{0}>0$ independent of the parameter $R$ such that $\forall\rho',\rho\in T^{*}M$,
\[
\left|b_{R}\left(\rho'\right)-b_{R}\left(\rho\right)\right|\leq C_{0}\left\langle \mathrm{dist}_{g}\left(\rho',\rho\right)\right\rangle ^{N_{0}}.
\]
Then, for arbitrarily large $N>0$, there exists a constant $C_{N,W}>0$
such that for any $R\geq1$,
\begin{equation}
\left\Vert \mathrm{Op}\left(b_{R}\right)\circ\mathrm{Op}\left(\boldsymbol{1}_{\Omega_{R}}\right)-\mathrm{Op}\left(\boldsymbol{1}_{\Omega_{R}}\right)\right\Vert _{\mathcal{H}_{W}}\leq C_{N,W}R^{-N}.\label{eq:corollary_compos}
\end{equation}
\end{cor}

\end{cBoxB}

\begin{proof}
Let us check first that, for arbitrarily large $N>0$, we have 
\[
\left|b_{R}\left(\rho'\right)-b_{R}\left(\rho\right)\right|\leq C_{0}R^{-N}\mathrm{dist}_{g}\left(\rho',\rho\right)^{N_{0}+N}\quad\text{for all \ensuremath{\rho\in\Omega_{R}} and }\rho'\in T^{*}M.
\]
This is trivial if $\mathrm{dist}_{g}\left(\rho',\rho\right)\leq R$
because $b_{R}\left(\rho'\right)=b_{R}\left(\rho\right)=1$. If $\mathrm{dist}_{g}\left(\rho',\rho\right)\geq R\geq1$,
we have 
\[
\left|b_{R}\left(\rho'\right)-b_{R}\left(\rho\right)\right|\leq C_{0}\mathrm{dist}_{g}\left(\rho',\rho\right)^{N_{0}}\leq C_{0}R^{-N}\mathrm{dist}_{g}\left(\rho',\rho\right)^{N+N_{0}},
\]
so this is true again. Now we apply Theorem \ref{thm:Composition-of-PDO.}
with setting $a=\boldsymbol{1}_{\Omega_{R}}$, $b=b_{R}$ and $h\equiv R^{-N}$.
Since $\mathrm{Op}\left(ab\right)=\mathrm{Op}\left(\boldsymbol{1}_{\Omega_{R}}b_{R}\right)=\mathrm{Op}\left(\boldsymbol{1}_{\Omega_{R}}\right)$
and $\left\Vert ah\right\Vert _{L^{\infty}}=\left\Vert \boldsymbol{1}_{\Omega_{R}}h\right\Vert _{L^{\infty}}=R^{-N}$,
we obtain the conclusion (\ref{eq:corollary_compos}).
\end{proof}

\subsection{Properties of the transfer operator}

In this section we consider the transfer operator $\mathcal{L}^{t}=\exp\left(tA\right)$
with $A=-X+V$ defined in (\ref{eq:def_Transfer_operator}). As we
noted at the beginning of this section, we will assume only that $X$
is a smooth non-singular vector field on $M$ and $V$ is some smooth
complex valued function on $M$.

\subsubsection{\label{subsec:A-matrix-elements}The matrix elements of the transfer
operator between wave packets}

The next theorem is rather a direct consequence of Lemma \ref{lem:Description-of-evolving}.
It shows that the Schwartz kernel of the lifted operator $\mathcal{T}\mathcal{L}^{t}\mathcal{T}^{\dagger}$
(i.e. the matrix elements of the transfer operator $\mathcal{L}^{t}$
expressing transformation between wave-packets)
\begin{equation}
\langle\delta_{\rho',j'}|\mathcal{T}\mathcal{L}^{t}\mathcal{T}^{\dagger}\delta_{\rho,j}\rangle_{L^{2}\left(T^{*}M\right)}\underset{(\ref{eq:def_T_*})}{=}\langle\Phi_{\rho',j'}|\mathcal{L}^{t}\Phi_{\rho,j}\rangle_{L^{2}\left(M\right)}\label{eq:Schwartz_kernel_transfer_op}
\end{equation}
is (micro-)localized around the graph of the flow map $\tilde{\phi}^{t}$
in $T^{*}M$, Eq.(\ref{eq:lifted_flow}), with respect to the distance
$\mathrm{dist}_{g}\left(\cdot,\cdot\right)$.

\begin{cBoxB}{}
\begin{thm}[Propagation of singularities by the transfer operator $\mathcal{L}^{t}$]
\label{thm:Microlocality-of-the_TO} For any $t\in\mathbb{R}$ and
any $N>0$, there exists a constant $C_{N,t}>0$ such that
\begin{align}
\left|\langle\delta_{\rho',j'}|\mathcal{T}\mathcal{L}^{t}\mathcal{T}^{\dagger}\delta_{\rho,j}\rangle_{L^{2}\left(T^{*}M\right)}\right| & \leq C_{N,t}\left\langle \mathrm{dist}_{g}\left(\rho',\tilde{\phi}^{t}\left(\rho\right)\right)\right\rangle ^{-N}\label{eq:microl_estimate}
\end{align}
for any $\rho,\rho'\in T^{*}M$ and $0\leq j,j'\leq J$. Using the
notation of Definition \ref{def:Let--be-1}, we can write that $\mathcal{L}^{t}\in\Psi_{\tilde{\phi}^{t}}\left(1\right).$
\end{thm}

\end{cBoxB}

\begin{proof}
Let $t\geq0$ and $\rho,\rho'\in T^{*}M$. We take $0\leq j,j'\leq J$
such that $m=\pi\left(\rho\right)\in U_{j}$ and $m'=\pi\left(\rho'\right)\in U_{j'}$
respectively. In the following we will sometimes omit the indices
$j$ and $j'$ from the notation. Set 
\[
\varrho=\left(y,\eta\right)=\tilde{\psi}\left(\tilde{\kappa}_{j}\left(\rho\right)\right),\quad\varrho'=\left(y',\eta'\right)=\tilde{\psi}\left(\tilde{\kappa}_{j'}\left(\rho'\right)\right)\quad\in T^{*}\mathbb{R}^{n+1}
\]
where $\tilde{\kappa}_{j}$ is defined in (\ref{eq:def_k_tilde_j})
and $\tilde{\psi}$ is defined in (\ref{eq:def_psi_tilde}). We write
the Schwartz kernel \eqref{eq:Schwartz_kernel_transfer_op} as 
\begin{align}
\langle\delta_{\rho',j'}|\mathcal{T}\mathcal{L}^{t}\mathcal{T}^{\dagger}\delta_{\rho,j}\rangle_{L^{2}\left(T^{*}M\right)} & \underset{(\ref{eq:def_wave_packet_Phi_j_rho})}{=}\left|\mathrm{det}\left(d\psi\left(\kappa_{j'}\left(m'\right)\right)\right)\right|^{1/2}\left|\mathrm{det}\left(d\psi\left(\kappa_{j}\left(m\right)\right)\right)\right|^{1/2}\langle\varphi_{\varrho'}|\left(I_{j'}\mathcal{L}^{t}I_{j}^{\dagger}\right)\varphi_{\varrho}\rangle_{L^{2}}.\label{eq:kernel}
\end{align}
Below we estimate the last term
\begin{equation}
K\left(\varrho',\varrho\right):=\langle\varphi_{\varrho'}|\left(I_{j'}\mathcal{L}^{t}I_{j}^{\dagger}\right)\varphi_{\varrho}\rangle_{L^{2}}.\label{eq:def_K}
\end{equation}
We assume that $\phi^{t}\left(U_{j}\right)\cap U_{j'}\neq\emptyset$
because the Schwartz kernel vanishes otherwise. Let
\[
\phi:=\kappa_{j'}\circ\phi^{t}\circ\kappa_{j}^{-1}\quad:\kappa_{j}\left(U_{j}\cap\phi^{-t}\left(U_{j'}\right)\right)\rightarrow\kappa_{j'}\left(\phi^{t}\left(U_{j}\right)\cap U_{j'}\right).
\]
For $u\in\mathcal{S}\left(\mathbb{R}_{y}^{n+1}\right)$, we have for
$y''\in\mathbb{R}^{n+1}$,
\begin{equation}
\left(I_{j'}\mathcal{L}^{t}I_{j}^{\dagger}u\right)\left(y''\right)\underset{(\ref{eq:def_Transfer_operator})}{=}Y\left(y''\right)u\left(\phi^{-1}\left(y''\right)\right)\label{eq:def_Ltilde-1}
\end{equation}
where $Y:\mathbb{R}^{n+1}\to\mathbb{R}$ is a $C^{\infty}$ function
written precisely as
\[
Y\left(y''\right)=e^{V_{\left[-t,0\right]}\left(\kappa_{j'}^{-1}\left(y''\right)\right)}\chi_{j'}\left(y''\right)\chi_{j}\left(\phi^{-1}\left(y''\right)\right).
\]
The definition of the function $Y\left(.\right)$ is rather complicated
but we only need the property that it is $C^{\infty}$ function supported
on a compact subset $\mathrm{supp}\chi_{j'}$. We have
\begin{align}
K\left(\varrho',\varrho\right) & \eq{\ref{eq:def_K},\ref{eq:def_Ltilde-1}}\int_{\mathbb{R}^{n+1}}\overline{\varphi_{\varrho'}\left(y''\right)}\varphi_{\varrho}\left(\phi^{-1}\left(y''\right)\right)\,Y\left(y''\right)dy''.\label{eq:integrale}
\end{align}
We first bound this integral ``in space'' and use (\ref{eq:size_wave_packet}),
getting that $\forall N>0,\exists C_{N,t}>0,\forall\varrho,\varrho'\in T^{*}\mathbb{R}^{n+1}$,
\begin{align}
\left|K\left(\varrho',\varrho\right)\right| & \ineq{\ref{eq:integrale}}\int_{\mathbb{R}^{n+1}}\left|\overline{\varphi_{\varrho'}\left(y''\right)}\right|\left|\varphi_{\varrho}\left(\phi^{-1}\left(y''\right)\right)\right|\left|Y\left(y''\right)\right|dy''\label{eq:integrale-1}\\
 & \ineq{\ref{eq:size_wave_packet}}\left(\mathrm{det}\delta\left(\eta'\right)\right)^{-1/2}\left(\mathrm{det}\delta\left(\eta\right)\right)^{-1/2}C_{N,t}\nonumber \\
 & \qquad\int_{\mathbb{R}^{n+1}}\left\langle \delta\left(\eta'\right)^{-1}\left(y'-y''\right)\right\rangle ^{-N}\left\langle \delta\left(\eta\right)^{-1}\left(\phi^{t}\left(y\right)-y''\right)\right\rangle ^{-N}dy''\nonumber 
\end{align}
Let $\delta_{\mathrm{max}}:=\max\left\{ \delta\left(\eta\right),\delta\left(\eta'\right)\right\} $
and
\[
q:=\delta_{\mathrm{max}}^{-1}y'',a:=\delta_{\mathrm{max}}^{-1}y',b:=\delta_{\mathrm{max}}^{-1}\phi^{t}\left(y\right)
\]

giving
\[
\left|K\left(\varrho',\varrho\right)\right|\leq\left(\mathrm{det}\delta\left(\eta'\right)\right)^{-1/2}\left(\mathrm{det}\delta\left(\eta\right)\right)^{-1/2}\left(\mathrm{det}\delta_{\mathrm{max}}\right)C_{N,t}\int_{\mathbb{R}^{n+1}}\left\langle \left|q-a\right|\right\rangle ^{-N}\left\langle \left|q-b\right|\right\rangle ^{-N}dq
\]
We split the integral:
\begin{align*}
\int_{\mathbb{R}^{n+1}}\left\langle \left|q-a\right|\right\rangle ^{-N}\left\langle \left|q-b\right|\right\rangle ^{-N}dq\\
\leq\int_{q\in\mathbb{R}^{n+1},\left|q-a\right|\leq\left|q-b\right|} & \left\langle \left|q-b\right|\right\rangle ^{-N}dq+\int_{q\in\mathbb{R}^{n+1},\left|q-a\right|\geq\left|q-b\right|}\left\langle \left|q-a\right|\right\rangle ^{-N}dq.
\end{align*}
Let $D:=\frac{1}{2}\left|a-b\right|$. Since $\left\{ q\in\mathbb{R}^{n+1},\left|q-a\right|\leq\left|q-b\right|\right\} \subset\left\{ q\in\mathbb{R}^{n+1},\left|q-b\right|\geq D\right\} $,
we have for $D\geq1$, $\exists C,C',C''>0$$\forall N$,
\begin{align*}
\int_{q\in\mathbb{R}^{n+1},\left|q-a\right|\leq\left|q-b\right|}\left\langle \left|q-b\right|\right\rangle ^{-N}dq & \leq\int_{q\in\mathbb{R}^{n+1},\left|q-b\right|\geq D}\left\langle \left|q-b\right|\right\rangle ^{-N}dq=\int_{q\in\mathbb{R}^{n+1},\left|q\right|\geq D}\left\langle \left|q\right|\right\rangle ^{-N}dq\\
 & \leq C\int_{r>D}r^{-N+n}dr\leq C'D^{-N+n+1}\leq C''\left\langle \left|a-b\right|\right\rangle ^{-N+n+1}.
\end{align*}

Hence 
\begin{align*}
\int_{\mathbb{R}^{n+1}}\left\langle \left|q-a\right|\right\rangle ^{-N}\left\langle \left|q-b\right|\right\rangle ^{-N}dq & \leq C\left\langle \left|a-b\right|\right\rangle ^{-N+n+1}\\
 & \leq C\left\langle \left|\delta{}_{\mathrm{max}}^{-1}\left(y'-\phi^{t}\left(y\right)\right)\right|\right\rangle ^{-N+n+1}.
\end{align*}
Let $\left(y_{t},\eta_{t}\right):=\tilde{\phi}^{t}\left(\varrho\right)$.
Using (\ref{eq:equivalence_distance}) we deduce that $\forall N>0,\exists C_{N,t}>0,\forall\varrho,\varrho'\in T^{*}\mathbb{R}^{n+1}$,
\begin{equation}
\left|K\left(\varrho',\varrho\right)\right|\leq C_{N,t}\left\langle \left\Vert \tilde{\phi}^{t}\left(\varrho\right)-\left(y',\eta_{t}\right)\right\Vert _{g_{\varrho}}\right\rangle ^{-N}.\label{eq:K1}
\end{equation}

Secondly we can write and bound the same integral (\ref{eq:integrale})
in Fourier space:
\[
K\left(\varrho',\varrho\right)\eq{\ref{eq:def_K},\ref{eq:def_Ltilde-1}}\int_{\mathbb{R}^{n+1}}\overline{\left(\mathcal{F}\varphi_{\varrho'}\right)\left(\eta''\right)}\mathcal{F}\left(Y.\varphi_{\varrho}\circ\phi^{-1}\right)\left(\eta''\right)d\eta''
\]
and using (\ref{eq:F_phi_t}), we deduce similarly that $\forall N>0,\exists C_{N,t}>0,\forall\varrho,\varrho'\in T^{*}\mathbb{R}^{n+1}$,
\begin{equation}
\left|K\left(\varrho',\varrho\right)\right|\leq C_{N,t}\left\langle \left\Vert \tilde{\phi}^{t}\left(\varrho\right)-\left(y_{t},\eta'\right)\right\Vert _{g_{\varrho}}\right\rangle ^{-N}.\label{eq:K2}
\end{equation}
Finally we write 
\begin{align*}
\tilde{\phi}^{t}\left(\varrho\right)-\varrho' & =\tilde{\phi}^{t}\left(\varrho\right)-\left(y',\eta_{t}\right)+\left(y',\eta_{t}\right)-\left(y',\eta'\right)\\
 & =\tilde{\phi}^{t}\left(\varrho\right)-\left(y',\eta_{t}\right)+\left(y_{t},\eta_{t}\right)-\left(y_{t},\eta'\right)\\
 & =\left(\tilde{\phi}^{t}\left(\varrho\right)-\left(y',\eta_{t}\right)\right)+\left(\tilde{\phi}^{t}\left(\varrho\right)-\left(y_{t},\eta'\right)\right)
\end{align*}
and from (\ref{eq:K1}), (\ref{eq:K2}) we deduce that $\forall N>0,\exists C_{N,t}>0,\forall\varrho,\varrho'\in T^{*}\mathbb{R}^{n+1}$,
\[
\left|K\left(\varrho',\varrho\right)\right|\leq C_{N,t}\left\langle \left\Vert \tilde{\phi}^{t}\left(\varrho\right)-\varrho'\right\Vert _{g_{\varrho}}\right\rangle ^{-N}.
\]
Finally, from (\ref{eq:kernel}), properties of the map $\tilde{\psi}$
in (\ref{eq:def_psi_tilde}) and using (\ref{eq:log_log}), we obtain
(\ref{eq:microl_estimate}).
\end{proof}

\subsubsection{Egorov's Theorem on evolution of PDO.}

The results of this section are direct consequences of Theorem \ref{thm:Microlocality-of-the_TO}.
This first basic lemma will be used later.

\begin{cBoxB}{}
\begin{lem}[Basic Egorov's Lemma for slow varying symbols]
\label{lem:Basic_Egorov_lemma}Assume that $a\in C\left(T^{*}M;\mathbb{C}\right)$
satisfies the slow variation property (\ref{eq:slow_variations})
with function $h$ and parameters $N_{0},h_{0}>0$. Then the Schwartz
kernel of the operator
\[
R:=\mathcal{T}\mathcal{L}^{t}\mathcal{T}^{\dagger}\mathcal{M}_{a\circ\tilde{\phi}^{t}}-\mathcal{M}_{a}\mathcal{T}\mathcal{L}^{t}\mathcal{T}^{\dagger}
\]
satisfies $\forall N,\exists C_{N,t}>0,\forall N_{0}>0,\exists h_{0}>0,\forall\rho',\rho\in T^{*}M,\forall j,j'\in\left\{ 1,\ldots J\right\} ,$
\begin{equation}
\left|\langle\delta_{\rho',j'}|R\delta_{\rho,j}\rangle\right|\leq h\left(\rho'\right)C_{N,t}\left\langle \mathrm{dist}_{g}\left(\rho',\tilde{\phi}^{t}\left(\rho\right)\right)\right\rangle ^{-N}.\label{eq:Egorov_Lemma}
\end{equation}
\end{lem}

\end{cBoxB}

\begin{proof}
We have
\begin{align*}
\langle\delta_{\rho',j'}|R\delta_{\rho,j}\rangle & =\langle\delta_{\rho',j'}|\mathcal{T}\mathcal{L}^{t}\mathcal{T}^{\dagger}\delta_{\rho,j}\rangle\left(a\left(\tilde{\phi}^{t}\left(\rho\right)\right)-a\left(\rho'\right)\right),
\end{align*}
hence
\begin{align*}
\left|\langle\delta_{\rho',j'}|R\delta_{\rho,j}\rangle\right| & \underset{(\ref{eq:microl_estimate},\ref{eq:slow_variations})}{\leq}h\left(\rho'\right)C_{N,t}\left\langle \mathrm{dist}_{g}\left(\rho',\tilde{\phi}^{t}\left(\rho\right)\right)\right\rangle ^{-N}\left\langle h_{0}\mathrm{dist}_{g}\left(\rho',\tilde{\phi}^{t}\left(\rho\right)\right)\right\rangle ^{N_{0}}\\
 & \leq h\left(\rho'\right)C'_{N,t}\left\langle \mathrm{dist}_{g}\left(\rho',\tilde{\phi}^{t}\left(\rho\right)\right)\right\rangle ^{-N}
\end{align*}
For the last line we proceed as in (\ref{eq:aide}) with $h_{0}$
small enough.
\end{proof}
The next theorem concerns the operator $e^{-tX}$ that is the transport
part of the transfer operator $\mathcal{L}^{t}=e^{t\left(-X+V\right)}$.
We will use this theorem later in the proof of Lemma \ref{lem:bounded_resolvent_of_A'}.

\begin{cBoxB}{}
\begin{thm}[Egorov's Theorem]
\label{thm:Egorov.-For-any} Assume that $a\in C\left(T^{*}M;\mathbb{C}\right)$
is a slowly varying symbol (Eq. (\ref{eq:slow_variations}) with parameters
$N_{0},h_{0}>0$). Let us consider the operator
\[
B:=e^{-tX}\mathrm{Op}\left(a\circ\tilde{\phi}^{t}\right)-\mathrm{Op}\left(a\right)e^{-tX}.
\]
For the Schwartz kernel of the corresponding lifted operator $B_{W}=\mathcal{M}_{W}\mathcal{T}B\mathcal{T}^{\dagger}\mathcal{M}_{W^{-1}}$
as defined in (\ref{eq:def_BW}) with parameter $N_{W}$ in (\ref{eq:temperate_property-W-2}),
we have $\forall t\in\mathbb{R},\forall N>0,\exists C_{N,t}>0,\forall N_{W},N_{0}>0,\exists h_{0}>0$,
$\forall\rho',\rho\in T^{*}M,\forall j,j'\in\left\{ 1,\ldots J\right\} $,
\[
\left|\langle\delta_{\rho',j'}|B_{W}\delta_{\rho,j}\rangle\right|\leq C_{N,t}\frac{W\left(\tilde{\phi}^{t}\left(\rho\right)\right)}{W\left(\rho\right)}h\left(\tilde{\phi}^{t}\left(\rho\right)\right)\left\langle \mathrm{dist}_{g}\left(\rho',\tilde{\phi}^{t}\left(\rho\right)\right)\right\rangle ^{-N}
\]
Using the notation of Definition \ref{def:Let--be-1}, we can write
that $B\in\Psi_{\tilde{\phi}^{t}}\left(h\right).$ Consequently $\forall t\in\mathbb{R},\exists C_{t}>0,\forall N_{W},N_{0}>0,\exists h_{0}>0$,
\begin{equation}
\left\Vert B\right\Vert _{\mathcal{H}_{W}}\leq C_{t}\left\Vert \frac{W\circ\tilde{\phi}^{t}}{W}\cdot h\right\Vert _{L^{\infty}}\label{eq:norm_Egorov_PDO}
\end{equation}
and
\begin{equation}
\left\Vert B\right\Vert _{\mathrm{Tr}_{\mathcal{H}_{W}}}\leq C_{t}\left\Vert \frac{W\circ\tilde{\phi}^{t}}{W}\cdot h\right\Vert _{L^{1}}\label{eq:Egorov_Trace}
\end{equation}
provided that the norm on the right-hand side is bounded in the respective
inequality.
\end{thm}

\end{cBoxB}

\begin{proof}
We set
\[
R:=\mathcal{T}e^{-tX}\mathcal{T}^{\dagger}\mathcal{M}_{a\circ\tilde{\phi}^{t}}-\mathcal{M}_{a}\mathcal{T}e^{-tX}\mathcal{T}^{\dagger},
\]
and
\[
B_{W}:=\left(e^{-tX}\mathrm{Op}\left(a\circ\tilde{\phi}^{t}\right)-\mathrm{Op}\left(a\right)e^{-tX}\right)_{W}\eq{\ref{eq:def_BW}}W\mathcal{P}R\mathcal{P}W^{-1}
\]
We will omit chart indices $j,j'$ for simplicity. We have
\begin{align*}
\langle\delta_{\rho'}|B_{W}\delta_{\rho}\rangle & =\frac{W\left(\rho'\right)}{W\left(\tilde{\phi}^{t}\left(\rho\right)\right)}\frac{W\left(\tilde{\phi}^{t}\left(\rho\right)\right)}{W\left(\rho\right)}\int\langle\delta_{\rho'}|\mathcal{P}\delta_{\rho_{1}}\rangle\langle\delta_{\rho_{1}}|R\delta_{\rho_{2}}\rangle\langle\delta_{\rho_{2}}|\mathcal{P}\delta_{\rho}\rangle\frac{d\rho_{1}}{\left(2\pi\right)^{n+1}}\frac{d\rho_{2}}{\left(2\pi\right)^{n+1}}
\end{align*}
and
\begin{align*}
\left|\langle\delta_{\rho'}|B_{W}\delta_{\rho}\rangle\right|\underset{(\ref{eq:temperate_property-W-2},\ref{eq:estimate_Bergman_kernel},\ref{eq:slow_variations},\ref{eq:Egorov_Lemma})}{\leq} & \frac{W\left(\tilde{\phi}^{t}\left(\rho\right)\right)}{W\left(\rho\right)}C_{W}\left\langle h_{0}\,\mathrm{dist}_{g}\left(\rho',\tilde{\phi}^{t}\left(\rho\right)\right)\right\rangle ^{N_{W}}C_{N,t}\int\left\langle \mathrm{dist}_{g}\left(\rho',\rho_{1}\right)\right\rangle ^{-N}\\
 & h\left(\rho_{1}\right)\left\langle \mathrm{dist}_{g}\left(\rho_{1},\tilde{\phi}^{t}\left(\rho_{2}\right)\right)\right\rangle ^{-N}\left\langle h_{0}\,\mathrm{dist}_{g}\left(\rho_{1},\tilde{\phi}^{t}\left(\rho_{2}\right)\right)\right\rangle ^{N_{0}}\\
 & \left\langle \mathrm{dist}_{g}\left(\rho_{2},\rho\right)\right\rangle ^{-N}\frac{d\rho_{1}}{\left(2\pi\right)^{n+1}}\frac{d\rho_{2}}{\left(2\pi\right)^{n+1}}\\
\leq C_{W} & C_{N,t}\frac{W\left(\tilde{\phi}^{t}\left(\rho\right)\right)}{W\left(\rho\right)}h\left(\tilde{\phi}^{t}\left(\rho\right)\right)\left\langle \mathrm{dist}_{g}\left(\rho',\tilde{\phi}^{t}\left(\rho\right)\right)\right\rangle ^{-N}.
\end{align*}
In the last line, $h_{0}$ is chosen small enough with respect to
$N_{W},N_{0}$. We deduce \eqref{eq:norm_Egorov_PDO} using Schur
Lemma \ref{lem:Schur-Lemma-.} and also \eqref{eq:Egorov_Trace} from
Remark \ref{rem:To-estimate-the}.
\end{proof}

\subsubsection{\label{subsec:Strongly-continuous-semi-group}Strong continuity of
the one-parameter group of transfer operators}

In the next Lemma we show strong continuity of the transfer operator
(this is version 1). Later in Lemma \ref{lem:For-the-weight}, with
additional properties on the escape function $W$ we will get some
improved estimates (version 2).

\begin{cBoxB}{}
\begin{lem}[Strong continuity of transfer operator $\mathcal{L}^{t}$ (version
1)]
\label{lem:semi_group}Assume that $W$ is a temperate weight (according
to Definition \ref{def:temperate_W-1}) and assume that, for some
constant $C>0$ and $\tau>0$, 
\begin{equation}
\frac{W\left(\tilde{\phi}^{t}\left(\rho\right)\right)}{W\left(\rho\right)}\leq C\quad\text{for any }\rho\in T^{*}M\text{ and }t\in\left[0,\tau\right].\label{eq:bound_C}
\end{equation}
Then the transfer operator $\mathcal{L}^{t}$ extends to a strongly
continuous semi-group of bounded operators $\mathcal{L}^{t}:\mathcal{H}_{W}\left(M\right)\rightarrow\mathcal{H}_{W}\left(M\right)$
with generator $A=-X+V$ given in (\ref{eq:generator_A}). Moreover,
we have the uniform estimate w.r.t. $N_{W}$ that writes $\exists C',C''>0,\forall N_{W},\exists h_{0}>0,\forall t\geq0,\left\Vert \mathcal{L}^{t}\right\Vert _{\mathcal{H}_{W}}\leq C'e^{tC''}$.

If in addition, we have that
\begin{equation}
\frac{W\left(\tilde{\phi}^{t}\left(\rho\right)\right)}{W\left(\rho\right)}\geq\frac{1}{C}\quad\text{for any }\rho\in T^{*}M\text{ and }t\in\left[0,\tau\right]\label{eq:estim2}
\end{equation}
then $\mathcal{L}^{t}:\mathcal{H}_{W}\left(M\right)\rightarrow\mathcal{H}_{W}\left(M\right)$
with $t\in\mathbb{R}$ form a strongly continuous group.
\end{lem}

\end{cBoxB}

\begin{proof}
Assume that the condition \eqref{eq:bound_C} holds. Let $t\geq0$
and as in (\ref{eq:def_BW}), consider the lifted operator $\mathcal{L}_{W}^{t}:=\mathcal{M}_{W}\mathcal{T}\mathcal{L}^{t}\mathcal{T}^{\dagger}\mathcal{M}_{W^{-1}}.$
For its Schwartz kernel, we have
\begin{align*}
\left|\langle\delta_{\rho',j'}|\mathcal{L}_{W}^{t}\delta_{\rho,j}\rangle\right| & \eq{\ref{eq:Schwartz_kernel_transfer_op}}\frac{W\left(\rho'\right)}{W\left(\rho\right)}\left|\langle\Phi_{\rho',j'}|\mathcal{L}^{t}\Phi_{\rho,j}\rangle\right|\\
 & \underset{(\ref{eq:microl_estimate})}{\leq}\frac{W\left(\rho'\right)}{W\left(\rho\right)}C_{N,t}\left\langle \mathrm{dist}_{g}\left(\rho',\tilde{\phi}^{t}\left(\rho\right)\right)\right\rangle ^{-N}\\
 & =C_{N,t}\left(\frac{W\left(\tilde{\phi}^{t}\left(\rho\right)\right)}{W\left(\rho\right)}\right)\left(\frac{W\left(\rho'\right)}{W\left(\tilde{\phi}^{t}\left(\rho\right)\right)}\right)\left\langle \mathrm{dist}_{g}\left(\rho',\tilde{\phi}^{t}\left(\rho\right)\right)\right\rangle ^{-N}\\
 & \underset{(\ref{eq:bound_C},\ref{eq:temperate_property-W-2})}{\leq}C_{N,t}CC_{W}\left\langle h_{0}\,\mathrm{dist}_{g}\left(\rho',\tilde{\phi}^{t}\left(\rho\right)\right)\right\rangle ^{N_{W}}\left\langle \mathrm{dist}_{g}\left(\rho',\tilde{\phi}^{t}\left(\rho\right)\right)\right\rangle ^{-N}\\
 & \leq C_{W}C'_{N,t}\left\langle \mathrm{dist}_{g}\left(\rho',\tilde{\phi}^{t}\left(\rho\right)\right)\right\rangle ^{-N}
\end{align*}
For the last line we proceed as in (\ref{eq:aide}) with $h_{0}$
small enough with respect to $N_{W}$. If we let $N$ be large enough
then, by Schur Lemma \ref{lem:Schur-Lemma-.}, we obtain that $\mathcal{L}_{W}^{t}:L^{2}\left(T^{*}M\right)\rightarrow L^{2}\left(T^{*}M\right)$
is bounded for any $t\geq0$ and therefore $\mathcal{L}^{t}:\mathcal{H}_{W}\left(M\right)\rightarrow\mathcal{H}_{W}\left(M\right)$
is bounded. For $u\in C^{\infty}\left(M\right)$, we have $\|\mathcal{L}^{t}u-u\|_{C^{\infty}}\to0$
as $t\to0$ and consequently $\left\Vert \mathcal{L}^{t}u-u\right\Vert _{\mathcal{H}_{W}}\underset{}{\rightarrow}0$.
Since $C^{\infty}\left(M\right)$ is dense in $\mathcal{H}_{W}$ this
is also true for $u\in\mathcal{H}_{W}$. Therefore $\mathcal{L}^{t}$
is strongly continuous. This finishes the proof of the former part.
For the latter part of the theorem, we consider $\mathcal{L}^{-t}=e^{t\left(-A\right)}$
with $t\geq0$ instead of $\mathcal{L}^{t}$. Then the condition (\ref{eq:estim2})
reads $\frac{W\left(\tilde{\phi}^{-t}\left(\rho\right)\right)}{W\left(\rho\right)}\leq C$
and the same procedure as above yields the conclusion.
\end{proof}

\section{\label{sec:Proof-of-Theorem_discrete_spectrum}Proof of Theorems
\ref{thm:Discrete_spectrum} and \ref{thm:Weyl law} on discrete spectrum
and Weyl law}

In this section we will assume that $X$ is a general Anosov vector
field on $M$. We first describe the geometry of the lifted flow $\tilde{\phi}^{t}:T^{*}M\rightarrow T^{*}M$
in Section \ref{subsec:Dynamics-lifted-in}. Then we explain how to
design a suitable escape (or weight) function $W:T^{*}M\rightarrow\mathbb{R}^{+,*}$
in order to obtain Theorems \ref{thm:Discrete_spectrum} and \ref{thm:Weyl law}
for the Sobolev space $\mathcal{H}_{W}\left(M\right)$ in Section
\ref{subsec:Escape-function}. Finally in Section \ref{subsec:Discrete-Spectrum-and}
we provide the proofs of Theorem \ref{thm:Discrete_spectrum} and
\ref{thm:Weyl law}, which gives discrete spectrum and the Fractal
Weyl law.

\subsection{\label{subsec:Dynamics-lifted-in}The lifted flow $\tilde{\phi}^{t}$
in the cotangent bundle $T^{*}M$ and the trapped set $E_{0}^{*}$}

We define the frequency function $\boldsymbol{\omega}\in C^{\infty}\left(T^{*}M;\mathbb{R}\right)$
by\footnote{In micro-local analysis, the function $\boldsymbol{\omega}$ is the
principal symbol of the operator $-iX$, see \cite[footnote on  page 332.]{fred_flow_09}.}
\begin{equation}
\boldsymbol{\omega}\left(\rho\right):=X\left(\rho\right)\quad\text{for }\rho\in T^{*}M.\label{eq:omega_function}
\end{equation}
The dual decomposition of (\ref{eq:decomp_TM}) is\footnote{Beware that the notations of $E_{u}^{*},E_{s}^{*}$ are interchanged
with respect to the natural definition from linear algebra conventions
of duality. The advantage of this choice is that the dynamics is contracting
(stable) on the space $E_{s}^{*}$ and expanding (unstable) on the
space $E_{u}^{*}$, see (\ref{hyperbolicity_xi}).}
\begin{equation}
T_{m}^{*}M=E_{u}^{*}\left(m\right)\oplus E_{s}^{*}\left(m\right)\oplus E_{0}^{*}\left(m\right)\label{eq:decomp*}
\end{equation}
for $m\in M$ with
\begin{equation}
E_{0}^{*}\left(E_{u}\oplus E_{s}\right)=0,\qquad E_{s}^{*}\left(E_{s}\oplus E_{0}\right)=0,\quad E_{u}^{*}\left(E_{u}\oplus E_{0}\right)=0.\label{eq:def_E_u*}
\end{equation}
Since $E_{0}=\mathbb{R}X$, we have $E_{u}^{*}\oplus E_{s}^{*}\eq{\ref{eq:def_E_u*}}\mathrm{Ker}\left(X\right)\eq{\ref{eq:omega_function}}\boldsymbol{\omega}^{-1}\left(0\right)$
and therefore the correspondence $m\mapsto E_{u}^{*}\left(m\right)\oplus E_{s}^{*}\left(m\right)$
is smooth. We have 
\begin{equation}
E_{0}^{*}\underset{(\ref{eq:one_form})}{=}\mathbb{R}\mathscr{A}=\left\{ \omega\mathscr{A}\left(m\right),\quad m\in M,\omega\in\mathbb{R}\right\} \label{eq:def_E_0*}
\end{equation}
and the map $m\mapsto E_{0}^{*}\left(m\right)$ is Hölder continuous
with exponent $\beta_{0}$.

A cotangent vector $\rho\in T_{m}^{*}M$ is decomposed accordingly
to the dual decomposition (\ref{eq:decomp*}) 
\begin{equation}
\rho=\rho_{u}+\rho_{s}+\underbrace{\rho_{0}}_{\omega\mathscr{A}\left(m\right)}\label{eq:decomp_Xi}
\end{equation}
 with components 
\[
\rho_{u}\in E_{u}^{*},\quad\rho_{s}\in E_{s}^{*},\quad\rho_{0}=\omega\mathscr{A}\left(m\right)\in E_{0}^{*},\quad\omega=\boldsymbol{\omega}\left(\rho\right)\in\mathbb{R}.
\]

Recall the lifted flow $\tilde{\phi}^{t}:T^{*}M\rightarrow T^{*}M$
introduced in (\ref{eq:lifted_flow}). If we denote $\rho\left(t\right)=\tilde{\phi}^{t}\left(\rho\right)$
and $m\left(t\right)=\phi^{t}\left(m\right)$, we write $\rho\left(t\right)=\rho_{u}\left(t\right)+\rho_{s}\left(t\right)+\omega\left(t\right)\mathscr{A}\left(m\left(t\right)\right)$
as in \eqref{eq:decomp_Xi}, and the hyperbolicity assumption (\ref{eq:hyperbolicity})
gives that
\begin{equation}
\left|\rho_{u}\left(t\right)\right|\geq\frac{1}{C}e^{\lambda_{\mathrm{min}}t}\left|\rho_{u}\left(0\right)\right|,\qquad\left|\rho_{s}\left(t\right)\right|\leq Ce^{-\lambda_{\mathrm{min}}t}\left|\rho_{s}\left(0\right)\right|,\quad\omega\left(t\right)=\omega\left(0\right)\quad\text{ for }t\geq0,\label{hyperbolicity_xi}
\end{equation}
where $\left|\rho_{u}\left(t\right)\right|=\left\Vert \rho_{u}\left(t\right)\right\Vert _{g_{M}}$
is the norm measured with the dual metric $g_{M}$ on the cotangent
bundle $T^{*}M$ induced by $g_{M}$ on $TM$. See Figure \ref{fig:The-Anosov-flow-on_T*M}.

From (\ref{hyperbolicity_xi}), we see that the set $E_{0}^{*}=\mathbb{R}\mathscr{A}$
defined in (\ref{eq:def_E_0*}) is the trapped set of the flow $\tilde{\phi}^{t}$
in the sense that
\[
E_{0}^{*}=\left\{ \rho\in T^{*}M\,\mid\,\mbox{ the orbit }\{\tilde{\phi}^{t}\left(\rho\right)\in C\mid t\in\mathbb{R}\}\text{ is relatively compact in }T^{*}M\right\} .
\]
$E_{0}^{*}$ is a (Hölder continuous) rank one sub-bundle of $T^{*}M$,
with $\mathrm{dim}E_{0}^{*}=n+2$ and, in terms of dynamical system
theory, $E_{0}^{*}$ is just the non-wandering set of the flow $\tilde{\phi}^{t}:T^{*}M\rightarrow T^{*}M$.

\begin{figure}[h]
\centering{}\scalebox{0.8}[0.8]{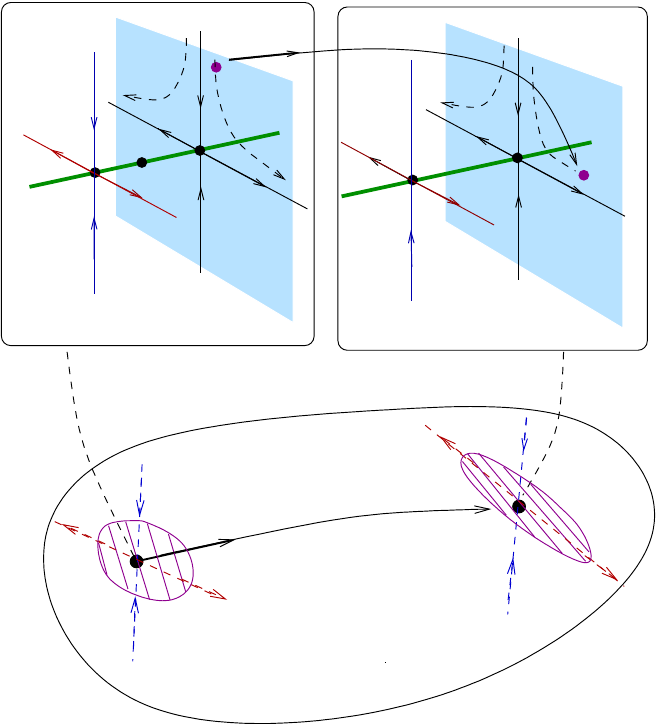}\caption{\label{fig:The-Anosov-flow-on_T*M}The Anosov flow $\phi^{t}=\exp\left(tX\right)$
on $M$ is lifted to a Hamiltonian flow $\tilde{\phi}^{t}=\exp\left(t\tilde{X}\right)$
in the cotangent bundle $T^{*}M$. The lines on the base (in magenta)
represent \textquotedblleft internal oscillations\textquotedblright{}
of a function $u$ at point $m\in M$. These oscillations correspond
to a cotangent vector $\rho\in T_{m}^{*}M$. Transported by the flow
$\phi^{t}$, these oscillations become parallel to $E_{u}$, i.e.
the direction of $\tilde{\phi}^{t}\left(\rho\right)$ converges to
$E_{u}^{*}\subset T^{*}M$ and $\tilde{\phi}^{t}\left(\rho\right)$
remains in the frequency level $\Sigma_{\omega}:=\boldsymbol{\omega}^{-1}\left(\omega\right)$
(in blue). The trapped set $E_{0}^{*}$ (in green) of the lifted flow
$\tilde{\phi^{t}}$ is the Hölder continuous rank one vector bundle
$E_{0}^{*}=\mathbb{R}\mathscr{A}$ where $\mathscr{A}$ is the Anosov
one form.}
\end{figure}

\subsection{\label{subsec:Escape-function-}Escape function $W$}

In Lemma \ref{lem:semi_group}, in order to show that the transfer
operators $\mathcal{L}^{t}$ acts on the generalized Sobolev space
$\mathcal{H}_{W}\left(M\right)$ as a strongly continuous (semi-)group,
we required two properties of the weight function $W$, namely the
temperate property (\ref{eq:temperate_property-W-2}) and the boundedness
property (\ref{eq:bound_C}) with respect to the lifted flow $\tilde{\phi}^{t}$.

In the next section, we will show that the generator $A$ has discrete
Ruelle spectrum. For this we need to reinforce the properties of the
weight function $W$ by the slowly varying and temperate property
in Definition \ref{def:temperate_W} and the decay property in Definition
\ref{def:decay_W} below. In Section \ref{subsec:Escape-function}
we also provide an example of a weight function $W$ on $T^{*}M$
satisfying such conditions. Another example of weight function similar
to the weight function used in \cite{fred-roy-sjostrand-07,fred_flow_09}
is given in Appendix \ref{sec:Second-example-of}. (However this latter
example does not satisfies (\ref{eq:estim2}) that provides the group
property of $\mathcal{L}^{t}$).

\subsubsection{Requirements for an escape function $W$ on $T^{*}M$}

For a point $m\in M$, recall the decomposition of a cotangent vector
$\rho=\rho_{u}+\rho_{s}+\rho_{0}\in T_{m}^{*}M$, in Eq.(\ref{eq:decomp_Xi}).
For $0\leq\gamma<1$, we define the continuous function $h_{\gamma}\in C\left(T^{*}M;\mathbb{R}^{+}\right)$
by
\begin{equation}
h_{\gamma}^{\perp}\left(\rho\right):=\left\langle \left\Vert \rho_{u}+\rho_{s}\right\Vert _{g_{\rho}}\right\rangle ^{-\gamma}\label{eq:h_gamma}
\end{equation}
where the vertical vector $\rho_{u}+\rho_{s}\in T_{m}^{*}M$ is naturally
identified with a vector in the tangent space $T_{\rho}\left(T^{*}M\right)$
of $T^{*}M$ at $\rho\in T^{*}M$ in order to get its norm. Remark
that if $\gamma>0$ then $h_{\gamma}^{\perp}\left(\rho\right)$ decays
as $\rho$ gets far from the trapped set $E_{0}^{*}$. But for some
arguments we will sometimes need $\gamma=0$. See Section \ref{subsec:Remarks-about-the}
where we comment on the use of function $h_{\gamma}^{\perp}$ and
parameter $h_{0}$.

\begin{cBoxA}{}
\begin{defn}[slowly varying and $h_{\gamma}^{\perp}$-temperate property for a
weight function $W$]
\label{def:temperate_W}A (family of) continuous function $W\in C\left(T^{*}M;\mathbb{R}^{+}\right)$
that depends on $h_{0}>0$, is said to be slowly varying and $h_{\gamma}^{\perp}$
-temperate, if there exists $0<\mu<1$, $C_{W}>0$ and $N_{W}\geq0$
independent of $h_{0}$ such that
\begin{equation}
\frac{W\left(\rho'\right)}{W\left(\rho\right)}\leq1+C_{W}h_{0}^{\mu}\left\langle h_{\gamma}^{\perp}\left(\rho\right)\,\mathrm{dist}_{g}\left(\rho',\rho\right)\right\rangle ^{N_{W}}\qquad\text{for all }\rho,\rho'\in T^{*}M.\label{eq:slow_variation_W}
\end{equation}
\end{defn}

\end{cBoxA}

\begin{rem}
Definition (\ref{eq:slow_variation_W}) expresses two properties.
First that for small $h_{0}$ and for any two points $\rho,\rho'$
such that $\mathrm{dist}_{g}\left(\rho',\rho\right)\leq h_{\gamma}^{\perp}\left(\rho\right)^{-1}$
where $h_{\gamma}^{\perp}\left(\rho\right)^{-1}$ can be large, then
$W$ has ratio close to 1, i.e. slow variations. Second that, at larger
distances, $W$ has temperate variations which grow at most polynomial
rate in $\left(h_{\gamma}^{\perp}\left(\rho\right)\mathrm{dist}_{g}\left(\rho',\rho\right)\right)$. 
\end{rem}

~
\begin{rem}
Eq.(\ref{eq:slow_variation_W}) can be written as
\begin{equation}
\left|W\left(\rho'\right)-W\left(\rho\right)\right|\leq W\left(\rho\right)C_{W}h_{0}^{\mu}\left\langle h_{\gamma}^{\perp}\left(\rho\right)\mathrm{dist}_{g}\left(\rho',\rho\right)\right\rangle ^{N_{W}},\label{eq:slow_variation_W-1}
\end{equation}
and implies $\forall\rho,\rho'\in T^{*}M$,
\begin{equation}
\frac{W\left(\rho'\right)}{W\left(\rho\right)}\leq C'_{W}\left\langle h_{0}^{\mu/N_{W}}h_{\gamma}^{\perp}\left(\rho\right)\mathrm{dist}_{g}\left(\rho',\rho\right)\right\rangle ^{N_{W}}\leq C'_{W}\left\langle h_{0}^{\mu/N_{W}}\mathrm{dist}_{g}\left(\rho',\rho\right)\right\rangle ^{N_{W}},\label{eq:temperate_property}
\end{equation}
with $C'_{W}>0$ independent of $h_{0}$ and $N_{W}$. In particular
Eq.(\ref{eq:slow_variation_W-1}) shows that according to Definition
\ref{def:Let--be}, the function $W$ is $h,N_{0},h_{0}$-slowly varying
with function $h=C_{W}h_{0}^{\mu}W$. We will use this later in the
proof of Lemma \ref{lem:For-the-weight}. Also, Eq.(\ref{eq:temperate_property})
shows that according to Definition \ref{def:temperate_W-1}, the function
$W$ is temperate so the results of Section \ref{sec:Semiclassical-analysis-with}
apply. We will use (\ref{eq:temperate_property}) later in the proof
of Lemma \ref{lem:For-the-weight} and of Theorem \ref{thm:decay_outside_trapped_set}. 
\end{rem}

~

\begin{cBoxA}{}
\begin{defn}[Decay property for a weight function $W$]
\label{def:decay_W}A continuous function $W\in C\left(T^{*}M;\mathbb{R}^{+}\right)$
has time decay property with rate $\Lambda>0$ with respect to the
flow $\left(\tilde{\phi}^{t}\right)_{t\geq0}$, if there exists a
constant $C>1$ such that, for any $t\geq0$, one can take $C_{W,t}>0$
so that
\begin{align}
\frac{W\left(\tilde{\phi}^{t}\left(\rho\right)\right)}{W\left(\rho\right)}\leq\begin{cases}
C\\
Ce^{-\Lambda t} & \text{ if }\left\Vert \rho_{u}+\rho_{s}\right\Vert _{g\left(\rho\right)}>C_{W,t}\quad\text{for any \ensuremath{\rho\in T^{*}M}.}
\end{cases}\label{eq:decay_property}
\end{align}
Moreover $W$ has a decay controlled from below with rate $\Lambda'\geq0$
if there exists a constant $C>1$ such that 
\begin{align}
\frac{W\left(\tilde{\phi}^{t}\left(\rho\right)\right)}{W\left(\rho\right)}\geq\frac{1}{C}e^{-\Lambda't} & \quad\text{for any \ensuremath{\rho\in T^{*}M}\text{ and }\ensuremath{t\ge0}}.\label{eq:decay_property-1}
\end{align}
\end{defn}

\end{cBoxA}

\begin{rem}
The general bound $W\left(\tilde{\phi}^{t}\left(\rho\right)\right)/W\left(\rho\right)\leq C$
given by the first line of (\ref{eq:decay_property}) together with
the temperate property has been used in the proof of Lemma \ref{lem:semi_group}
in order to show that the transfer operator $\mathcal{L}^{t}:\mathcal{H}_{W}\left(M\right)\rightarrow\mathcal{H}_{W}\left(M\right)$
is bounded and forms a strongly continuous semi-group. The bound from
below (\ref{eq:decay_property-1}) has been used in Lemma \ref{lem:semi_group}
to show that $\mathcal{L}^{t}:\mathcal{H}_{W}\left(M\right)\rightarrow\mathcal{H}_{W}\left(M\right),t\in\mathbb{R}$
forms a strongly continuous group.

The more precise bound given in the second line of (\ref{eq:decay_property})
shows that we have time decay outside the trapped set $E_{0}^{*}$.
It will be used in the proof of Lemma \ref{lem:bounded_resolvent_of_A'}
to show that the transfer operator $\mathcal{L}^{t}$ has small norm
outside the trapped set and to deduce discrete spectrum of its generator.
The fact that the constant $C_{W,t}$ depends on time $t$ in (\ref{eq:decay_property})
is due to the fact that a trajectory close to the stable manifold
of the trapped set $E_{0}^{*}$ can spend a long time close to the
trapped set where there is no decay. During that time, a wave packet
spreads exponentially.
\end{rem}

\subsubsection{\label{subsec:Escape-function}Example of weight function $W$}

We provide here an example of escape function $W$ and we prove that
$W$ satisfies the properties of being temperate, slowly varying and
decay property given above in Definitions \ref{def:temperate_W} and
\ref{def:decay_W}.

Recall from (\ref{eq:def_beta_*}) that the Hölder exponents satisfy
$\min\left(\beta_{u},\beta_{s}\right)\underset{}{\leq}\beta_{0}$
and recall that the parameters $\alpha^{\perp}$,$\alpha^{\parallel}$
that enter in the definition of the metric $g$ in (\ref{eq:def_delta}).

Before giving the definition of $W$ we need, for technical reasons\footnote{The use of this modified metric $\tilde{g}$ instead of $g$ will
permit to provide a simpler proof, compared to a previous proof given
in the \href{https://arxiv.org/pdf/1706.09307v2.pdf\#page=52}{version v2}
of this paper on ArXiv, that uses the metric $g$ itself.}, to use a slightly different norm on $TT_{m}^{*}M$ for $m\in M$,
defined as follows. Let $g_{M}$ be a global smooth metric on $M$.
Let $\tilde{g}_{M}$ be the metric on $T_{m}M,m\in M$, that equals
$g_{M}$ on $E_{\sigma}$ for $\sigma=u,s,0$ and such that the sum
$E_{u}\oplus E_{s}\oplus E_{0}$ is orthogonal. As a consequence $g_{M}$
is Hölder continuous on $M$ and the dual sum $E_{u}^{*}\oplus E_{s}^{*}\oplus E_{0}^{*}$
is also orthogonal for the induced metric on $T_{m}^{*}M$, $m\in M$,
that we still denote by $\tilde{g}_{M}$. Now we define the metric
$\tilde{g}$ on $TT_{m}^{*}M,m\in M,$ as follows. For $\tilde{v}=\left(\tilde{v}_{\xi},\tilde{v}_{\omega}\right)\in T_{\rho}T_{m}^{*}M$,
with $m\in M,\rho\in T_{m}^{*}M$, $\tilde{v}_{\xi}\in E_{u}^{*}\oplus E_{s}^{*}$,
$\tilde{v}_{\omega}\in E_{0}^{*}$, we set
\begin{equation}
\left\Vert \tilde{v}\right\Vert _{\tilde{g}\left(\rho\right)}^{2}=\left(\left\langle \left\Vert \rho\right\Vert _{\tilde{g}_{M}}\right\rangle ^{-\alpha^{\perp}}\left\Vert \tilde{v}_{\xi}\right\Vert _{\tilde{g}_{M}}\right)^{2}+\left(\left\langle \left\Vert \rho\right\Vert _{\tilde{g}_{M}}\right\rangle ^{-\alpha^{\parallel}}\left\Vert \tilde{v}_{\omega}\right\Vert _{\tilde{g}_{M}}\right)^{2}.\label{eq:norm_g-1}
\end{equation}
Let $0<\gamma<1$ . We define the continuous function $\tilde{h}_{\gamma}^{\perp}:T^{*}M\rightarrow\mathbb{R}^{+}$
as follows. For $\rho=\rho_{u}+\rho_{s}+\rho_{0}\in T^{*}M$, we set
\begin{equation}
\tilde{h}_{\gamma}^{\perp}\left(\rho\right)=\left\langle \left\Vert \rho_{u}+\rho_{s}\right\Vert _{\tilde{g}\left(\rho\right)}\right\rangle ^{-\gamma}.\label{eq:h_tilde-1}
\end{equation}
So notice that the norms $g_{M}$ and $\tilde{g}_{M}$ are equivalent,
that the norm of a vertical vector given by $\tilde{g}$ in (\ref{eq:norm_g-1})
is equivalent to the norm given by the metric $g$ defined in (\ref{eq:metric_g_in_coordinates}),
and finally that the function $\tilde{h}_{\gamma}^{\perp}$ defined
in (\ref{eq:h_tilde-1}) looks like the function $h_{\gamma}^{\perp}$
defined in (\ref{eq:h_gamma}), except that we use now the metric
$\tilde{g}$ instead of $g$.

\begin{cBoxA}{}
\begin{defn}[Example of a weight function $W$ with good properties]
\label{def:def_W1}Assume that $\alpha^{\perp}$ and $\alpha^{\parallel}$
satisfy
\begin{equation}
\frac{1}{1+\beta_{0}}\leq\alpha^{\perp}<1,\quad0\leq\alpha^{\parallel}\leq\alpha^{\perp}.\label{eq:ineq1}
\end{equation}
For $\gamma\in[0,1[$ such that
\begin{equation}
1-\frac{\alpha^{\perp}\min\left(\beta_{u},\beta_{s}\right)}{1-\alpha^{\perp}}\leq\gamma<1,\label{eq:gamma_interval}
\end{equation}
and $R_{u}>0$, $R_{s}>0$, $h_{0}>0$, we define the \textbf{escape
function} $W:T^{*}M\rightarrow\mathbb{R}^{+}$ by
\begin{equation}
W\left(\rho\right):=\frac{\left\langle h_{0}\tilde{h}_{\gamma}^{\perp}\left(\rho\right)\left\Vert \rho_{s}\right\Vert _{\tilde{g}\left(\rho\right)}\right\rangle ^{R_{s}}}{\left\langle h_{0}\tilde{h}_{\gamma}^{\perp}\left(\rho\right)\left\Vert \rho_{u}\right\Vert _{\tilde{g}\left(\rho\right)}\right\rangle ^{R_{u}}}.\label{eq:def_W1}
\end{equation}
\end{defn}

\end{cBoxA}

\begin{rem}
In the expression (\ref{eq:def_W1}) for $W$, we can write in a more
concise way $h_{0}h_{\gamma}^{\perp}\left(\rho\right)\left\Vert \rho_{s}\right\Vert _{\tilde{g}\left(\rho\right)}=\left\Vert \rho_{s}\right\Vert _{\tilde{g}_{\gamma,h_{0}}\left(\rho\right)}$
using the metric $\tilde{g}_{\gamma,h_{0}}:=\left(h_{0}h_{\gamma}^{\perp}\right)^{2}\tilde{g}$
that is conformal to $\tilde{g}$.
\end{rem}

~
\begin{rem}
The meaning of Definition \ref{def:def_W1} will be discussed in Section
\ref{rem:Due-to-the}. The way to chose optimally the different parameters
that enter in the construction of $W$ for the analysis of Anosov
flow will be discussed in Section \ref{subsec:Optimal-values-of}.
We first give the following theorem that shows that $W$ has the required
properties.
\end{rem}

\begin{cBoxB}{}
\begin{thm}
\label{thm:W}The weight function $W$ in (\ref{eq:def_W1}) has the
following properties
\begin{enumerate}
\item \label{enu:-satisfies--temperate-1}$W$ satisfies the slowly varying
and $h_{\gamma}^{\perp}$ -temperate property (\ref{eq:slow_variation_W}).
\item \label{enu:-satisfies-decay-1}$W$ satisfies the decay property (\ref{eq:decay_property})
with rate
\begin{equation}
\Lambda=\lambda_{\mathrm{min}}\left(1-\gamma\right)\left(1-\alpha^{\perp}\right)\min\left(R_{s},R_{u}\right).\label{eq:def_Lambda}
\end{equation}
\item \label{enu:-satisfies-decay-below}$W$ satisfies the decay property
controlled from below (\ref{eq:decay_property}) with rate
\begin{equation}
\Lambda'=\lambda_{\mathrm{max}}\left(1-\gamma\right)\left(1-\alpha^{\perp}\right)\left(R_{s}+R_{u}\right).\label{eq:lAMBDA_P}
\end{equation}
\item \label{enu:Its-order-is-1}Its order, defined in (\ref{eq:def_order-2}),
is
\begin{equation}
r\left(\left[\rho\right]\right)=\begin{cases}
0, & \text{along \ensuremath{E_{0}^{*}}; }\\
-\left(1-\gamma\right)\left(1-\alpha^{\perp}\right)R_{u} & \text{along \ensuremath{E_{u}^{*}};}\\
\left(1-\gamma\right)\left(1-\alpha^{\perp}\right)R_{s} & \text{along \ensuremath{E_{s}^{*}};}\\
\left(1-\gamma\right)\left(1-\alpha^{\perp}\right)\left(R_{s}-R_{u}\right) & \text{along all other directions.}
\end{cases}\label{eq:order_of_W1}
\end{equation}
\end{enumerate}
\end{thm}

\end{cBoxB}

The proof of Theorem \ref{thm:W} is given in Section \ref{subsec:Proof-of-Thm}.
From the weight function $W$ in \eqref{eq:def_W1}, we can define
the Sobolev space $\mathcal{H}_{W}\left(M\right)$ from Definition
\ref{def:Anisotropic_Sob_space}. We will use this (anisotropic) Sobolev
space in the next section.
\begin{rem}
One can consider other weight functions that satisfy the required
properties. For example in Appendix \ref{sec:Second-example-of} one
gives a second example of weight function $W_{2}$ that has been introduced
in \cite{fred-roy-sjostrand-07,fred_flow_09} and used many times
since in the literature. We also compare the (dis)advantages of these
two weight functions in Appendix \ref{sec:Second-example-of}.
\end{rem}

\subsubsection{\label{rem:Due-to-the}Remarks on the weight function $W$}

We discuss here the construction of $W$ in (\ref{eq:def_W1}) and
the meaning of the inequalities (\ref{eq:ineq1}) and (\ref{eq:gamma_interval}).
First observe one main feature of (\ref{eq:def_W1}) : $W\left(\rho\right)$
increases with respect to $\left\Vert \rho_{s}\right\Vert _{g_{M}}$
and decreases with respect to $\left\Vert \rho_{u}\right\Vert _{g_{M}}$.
Due to (\ref{hyperbolicity_xi}) this will give the decay property
(\ref{eq:decay_property}) of $W$. However, instead of using $\left\Vert \rho_{s}\right\Vert _{g_{M}}$
directly we use $h_{\gamma}^{\perp}\left(\rho\right)\left\Vert \rho_{s}\right\Vert _{g}$
that measures $\rho_{s}$ at the scale of wave packets or greater
and this will give the slow varying and temperate property of $W$.
In Section \ref{subsec:Heuristic-explanation-of} and in Figure \ref{fig:fractal_bound}
we have explained heuristically that the choice of parameter $\alpha^{\perp}\geq\frac{1}{1+\beta_{0}}$
for the metric $g$ is necessary in order that unit box of the metric
$g$ (i.e. the size or uncertainty of a wave-packet) is greater than
the Hölder fluctuations of the trapped set $E_{0}^{*}$ and therefore
absorb or hide them. This explains (\ref{eq:ineq1}). We should have
similar properties for the directions $E_{s}^{*}$ and $E_{u}^{*}$
that are also Hölder continue and a natural construction would have
been to define a metric $g$ with exponents $\alpha_{s}^{\perp},\alpha_{u}^{\perp}$
that satisfy $\alpha_{u}^{\perp}\geq\frac{1}{1+\beta_{u}},\alpha_{s}^{\perp}\geq\frac{1}{1+\beta_{s}}$.
We are not been able to do this because our specific construction
of the metric $g$ needs the same exponent $\alpha^{\perp}$ in every
transverse directions. In order to overcome this technical problem
we have introduced instead the scaling (or conformal) factor $h_{\gamma}^{\perp}\left(\rho\right)$
 in the construction of $W$, (\ref{eq:def_W1}), with some exponent
$\gamma$ that plays a role similar to $\alpha_{u}^{\perp},\alpha_{s}^{\perp}$.
However this scaling factor in the fibers of $T^{*}M$ is not equivalent
(it is weaker in fact) to choosing a different metric on $T^{*}M$
compatible with the symplectic form.

Let us explain now the range of $\gamma$ in (\ref{eq:gamma_interval}).
At a point $\rho\in T^{*}M$ we consider a unit box of the metric
$g$ that has horizontal transverse size $\delta m\asymp\delta^{\perp}\left(\rho\right)\asymp\left|\rho\right|^{-\alpha^{\perp}}$.
The Hölder exponents $\beta_{s}$ implies that fluctuations of $E_{s}^{*}$
are smaller than $\left|\rho_{s}\right|\left(\delta m\right)^{\beta_{s}}\asymp\left|\rho_{s}\right|\left|\rho\right|^{-\alpha^{\perp}\beta_{s}}$.
Near this direction we have $h_{\gamma}^{\perp}\left(\rho\right)^{-1}\asymp\left(\left|\rho\right|^{-\alpha^{\perp}}\left|\rho_{s}\right|\right)^{\gamma}\asymp\left|\rho_{s}\right|^{\gamma}\left|\rho\right|^{-\alpha^{\perp}\gamma}$.
The unit boxes for the re-scaled metric $\left(h_{\gamma}^{\perp}\right)^{2}g$
have size $\Delta\rho\asymp\left(h_{\gamma}^{\perp}\left(\rho\right)\delta^{\perp}\left(\rho\right)\right)^{-1}\asymp\left|\rho_{s}\right|^{\gamma}\left|\rho\right|^{\alpha^{\perp}\left(1-\gamma\right)}$.
In order that these unit boxes absorb the fluctuations of $E_{s}^{*}$,
we require the condition
\begin{align*}
\Delta\rho & \geq\left|\rho_{s}\right|\left(\delta m\right)^{\beta_{*}}\;\Longleftrightarrow\;\left|\rho_{s}\right|^{\gamma}\left|\rho\right|^{\alpha^{\perp}\left(1-\gamma\right)}\geq\left|\rho_{s}\right|\left|\rho\right|^{-\alpha^{\perp}\beta_{s}}\;\Longleftrightarrow\;\left|\rho\right|^{\alpha^{\perp}\left(1-\gamma\right)+\alpha^{\perp}\beta_{s}}\geq\left|\rho_{s}\right|^{1-\gamma}.
\end{align*}
Since $\left|\rho_{s}\right|\leq\left|\rho\right|$, this follows
if we assume
\[
\alpha^{\perp}\left(1-\gamma\right)+\alpha^{\perp}\beta_{s}\geq1-\gamma\quad\Longleftrightarrow\quad\gamma\geq1-\frac{\alpha^{\perp}\beta_{s}}{1-\alpha^{\perp}}.
\]
With the same consideration for $E_{u}^{*}$, we are led to the condition
on $\gamma$ that appears in (\ref{eq:gamma_interval}). See Figure
\ref{fig:zones-W}.

\begin{figure}[h]
\centering{}\scalebox{0.9}[0.9]{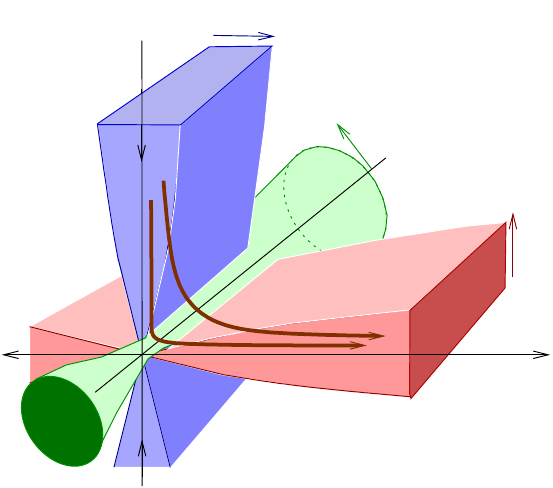}\caption{\label{fig:zones-W}Representation of domains associated to the weight
function $W$ defined in (\ref{eq:def_W1}): $\mathcal{V}_{s}$ is
the (blue) the \textquotedblleft parabolic\textquotedblright{} neighborhood
of $E_{s}^{*}$ given by $\mathcal{V}_{s}:=\left\{ \rho;\,\left\langle h_{\gamma}^{\perp}\left(\rho\right)\left\Vert \rho_{u}\right\Vert _{g_{\rho}}\right\rangle \asymp1\right\} =\left\{ \rho;\,\left|\rho_{u}\right|\apprle\min\left(\left|\omega\right|,\left|\rho_{s}\right|\right)^{\alpha^{\perp}+\gamma\left(1-\alpha^{\perp}\right)}\right\} $,
similarly for $\mathcal{V}_{u}$, and $\mathcal{V}_{0}:=\mathcal{V}_{u}\cap\mathcal{V}_{s}=\left\{ \rho;\,\left|\rho_{u}+\rho_{s}\right|\apprle\omega^{\alpha^{\perp}}\right\} $.
Outside this domain $\mathcal{V}_{0}$ (in green), the weight function
$W$ decays along the flow $\tilde{\phi}^{t}$.}
\end{figure}

\subsubsection{\label{subsec:Optimal-values-of}Remarks about optimal values of
parameters for the weight function $W$}

In this section we comment the optimal choice of weight function $W$
in order to get the theorems of this paper. Recall that, in Section
\ref{subsec:A-metric}, the metric $g$ on $T^{*}M$ is defined depending
on the parameters $\alpha^{\perp},\alpha^{\parallel}$ which satisfy
$0\leq\alpha^{\parallel}<\alpha^{\perp}<1,\frac{1}{2}\leq\alpha^{\perp}<1$
from (\ref{eq:conditions}).

The weight function $W$ depends on additional parameters $h_{0}>0$,
$R_{u},R_{s}>0,\gamma\geq0$ and Theorem \ref{thm:W} give their admissible
range of values. Also the admissible range of $\alpha^{\perp}$ is
restricted to $\frac{1}{1+\beta_{0}}\leq\alpha^{\perp}<1$. (See Lemma
\ref{thm:W} and Remark \ref{rem:Due-to-the}.) According to these
constraints we will consider optimal values of these parameters depending
on the purpose.

\paragraph{Choice 1 of $W$ to estimate the density of eigenvalues (Weyl law):}

In order to prove Theorem \ref{thm:Weyl law} about the Weyl law for
the density of eigenvalues, we will take, after Lemma \ref{prop:discrete_spectrum},
the minimal allowed value for $\alpha^{\perp},\alpha^{\parallel}$
and $\gamma$ that are

\[
\alpha^{\perp}=\frac{1}{1+\beta_{0}},\quad\alpha^{\parallel}=0,\quad\gamma\eq{\ref{eq:gamma_interval}}1-\frac{\beta_{*}}{\beta_{0}}\quad\text{with setting }\beta_{*}=\min\left(\beta_{u},\beta_{s}\right)\le\beta_{0}.
\]
that satisfy the assumptions (\ref{eq:ineq1}) and (\ref{eq:gamma_interval}).
The choice of these optimal values were explained in Section \ref{subsec:Heuristic-explanation-of}.
See also Figure \ref{fig:zones-W} where they give that the transverse
size for the (green) region $\mathcal{V}_{0}$ is
\[
\Delta_{0}^{\left(1\right)}=\omega^{\frac{1}{1+\beta_{0}}}.
\]
However, for this choice, the transverse size of the region $\mathcal{V}_{u}$
on Figure \ref{fig:zones-W} is
\begin{equation}
\Delta_{*}^{\left(1\right)}:=\left|\rho_{u}\right|^{\alpha^{\perp}+\gamma\left(1-\alpha^{\perp}\right)}=\left|\rho_{u}\right|^{1-\frac{\beta_{*}}{1+\beta_{0}}}.\label{eq:Delta_1*}
\end{equation}
This choice is not optimal to study the concentration of the wave
front set and one can improve with Choice 2 below.

\paragraph{Choice 2 of $W$ to estimate the concentration of the wave front
set:}

In Corollary \ref{cor:Let-us-choose} that concerns the width exponent
of the parabolic wave front set of Ruelle resonances, we choose the
following values

\begin{equation}
\alpha^{\perp}=\frac{1}{1+\beta_{*}},\quad\alpha^{\parallel}=0,\quad\gamma\eq{\ref{eq:gamma_interval}}0,\quad\text{with setting }\beta_{*}=\min\left(\beta_{u},\beta_{s}\right)\le\beta_{0}.\label{eq:choice_2}
\end{equation}
that satisfy the assumptions (\ref{eq:ineq1}) and (\ref{eq:gamma_interval}).
The reason for this choice is that the transverse size $\Delta_{*}^{\left(2\right)}$
of the region $\mathcal{V}_{u}$ on Figure \ref{fig:zones-W} is 
\begin{equation}
\Delta_{*}^{\left(2\right)}:=\left|\rho_{u}\right|^{\alpha^{\perp}+\gamma\left(1-\alpha^{\perp}\right)}=\left|\rho_{u}\right|^{\frac{1}{1+\beta_{*}}}\label{eq:size_Delta_2*}
\end{equation}
and smaller than $\Delta_{*}^{\left(1\right)}$ in (\ref{eq:Delta_1*}),
(except for the case $\beta_{*}=\beta_{0}$ )
\[
\Delta_{*}^{\left(2\right)}\leq\Delta_{*}^{\left(1\right)}
\]
because 
\[
\frac{1}{1+\beta_{*}}<1-\frac{\beta_{*}}{1+\beta_{0}}\quad\Longleftrightarrow\quad1+\beta_{0}<(1+\beta_{*})(1+\beta_{0}-\beta_{*})=1+\beta_{0}+\beta_{*}(\beta_{0}-\beta_{*}).
\]
However this choice gives the transverse size $\Delta_{0}^{\left(2\right)}$
of the green region $\mathcal{V}_{0}$ that is greater than $\Delta_{0}^{\left(1\right)}$:
\[
\Delta_{0}^{\left(2\right)}=\omega^{\frac{1}{1+\beta_{*}}}\geq\Delta_{0}^{\left(1\right)}.
\]
Therefore this choice is not adequate to deduce Theorem \ref{thm:Weyl law}
about the Weyl law.

\subsubsection{\label{subsec:Remarks-about-the}Remarks about the use of parameter
$h_{0}$ and function $h_{\gamma}^{\perp}\left(\rho\right)$}

The definition (\ref{eq:def_W1}) for the function $W$ makes use
of the product $h_{0}\tilde{h}_{\gamma}^{\perp}\left(\rho\right)$
with a parameter $h_{0}$ and a function $\tilde{h}_{\gamma}^{\perp}\left(\rho\right)=\left\langle \left\Vert \rho_{u}+\rho_{s}\right\Vert _{\tilde{g}\left(\rho\right)}\right\rangle ^{-\gamma}$
that depends on the parameter $\gamma\geq0$. We have already comment
about the meaning and utility of the parameter $\gamma$, in the previous
section \ref{subsec:Optimal-values-of}.

One effect of the product $h_{0}\tilde{h}_{\gamma}^{\perp}\left(\rho\right)\ll1$
when it is small is to assure that the ratio of variations of the
function $W$ is close to one on a large scale $\left(h_{0}h_{\gamma}^{\perp}\right)^{-1}\gg1$
and is used as follows in our micro-local analysis: if we consider
a wave packet (with inital size of order $1$ measured by the metric
$g$) that evolves and get size $e^{\lambda t}$ after a given long
time $t$ where $\lambda$ is the Lyapounov exponent, then the wave
packet still feels the function $W$ as almost constant provided $e^{\lambda t}\ll\left(h_{0}h_{\gamma}^{\perp}\right)^{-1}$.
We use crucially this argument at the end of proof of Lemma \ref{lem:For-the-weight}
and Theorem \ref{thm:decay_outside_trapped_set}, where we take first
$t$ large to get effect of hyperbolicity and after we take $h_{0}$
small. Notice in Theorem \ref{thm:decay_outside_trapped_set} we consider
the outside of the trapped set, where $\left\Vert \rho_{u}+\rho_{s}\right\Vert _{\tilde{g}\left(\rho\right)}$
is large, so we may have used $h_{\gamma}^{\perp}\left(\rho\right)\ll1$
for this purpose but only if we assume $\gamma>0$. However there
are some cases where need to take $\gamma=0$, see (\ref{eq:choice_2}),
giving $h_{\gamma}^{\perp}\left(\rho\right)=1$ that is not small.
So for simplicity we always use $h_{0}$ as the small parameter. 

\subsection{\label{subsec:Discrete-Spectrum-and}Discrete Spectrum and Weyl law
upper bound}

In this section, we prove Theorem \ref{thm:grey-band} about strong
continuity of the transfer operators $\mathcal{L}^{t}$ and Theorem
\ref{thm:Discrete_spectrum} and \ref{thm:Weyl law} about discrete
spectrum and an analogue of the Weyl law. We henceforth assume that
the weight function $W$ is chosen from the family (\ref{eq:def_W1})
and we consider the anisotropic Sobolev space $\mathcal{H}_{W}=\mathcal{H}_{W}\left(M\right)$
defined in (\ref{eq:def_H_W}). Recall that the definition of $W$
depends on the parameters $R_{u}$, $R_{s}$, $\gamma$ and $h_{0}$
besides the parameters $\alpha^{\perp},\alpha^{\parallel}$ used in
the definition of the metric $g$ on $T^{*}M$. From (\ref{eq:order_of_W1}),
the parameters $R_{u}$ and $R_{s}$ determine the order of the function
$W$ at infinity, while the parameters $\gamma$ and $h_{0}$ are
related to the variation of $W$ in smaller scales.

\subsubsection{Strong continuity}

From Lemma \ref{lem:semi_group}, it follows immediately that $\left\Vert \mathcal{L}^{t}\right\Vert _{\mathcal{H}_{W}}\leq Ce^{tC_{X,V,W}}$
for $t\ge0$ with a constant $C_{X,V,W}$ depending on $X,V$ and
$W$. The next lemma shows that, if we choose the parameter $h_{0}$
(that enters in $W$ in (\ref{eq:h_gamma})) depending on the choice
of $R_{u}$, $R_{s}$ and $\gamma$, then we may assume that the constant
$C_{X,V,W}$ is uniform i.e. does not depend on $W$. Let us recall
the constants $C_{X,V}$ and $C'_{X,V}$ defined in (\ref{eq:C_XV})
and (\ref{eq:C_XV-1}).

\begin{cBoxB}{}
\begin{lem}[Strong continuity of transfer operators $\mathcal{L}^{t}$ (version
2)]
\label{lem:For-the-weight}For any $\epsilon>0$, there exists $C>0$,
for any parameters $R_{u},R_{s}>0$, we can choose $h_{0}>0$ in (\ref{eq:h_gamma})
small enough so that, for any $t\geq0$, we have
\begin{equation}
\left\Vert \mathcal{L}^{t}\right\Vert _{\mathcal{H}_{W}}\leq Ce^{t\left(C_{X,V}+\epsilon\right)},\label{eq:bound_C_X,V}
\end{equation}
and
\begin{equation}
\left\Vert \mathcal{L}^{-t}\right\Vert _{\mathcal{H}_{W}}\leq Ce^{-t\left(C'_{X,V}-\Lambda'-\epsilon\right)},\label{eq:bound_C_X,V-1}
\end{equation}
where $\Lambda'$ is given in (\eqref{eq:lAMBDA_P}).
\end{lem}

\end{cBoxB}

\begin{rem}
From Hille-Yosida-Feller-Miyadera-Phillips's Theorem \cite[p.77]{engel_1999},
Eq.(\ref{eq:bound_C_X,V}) implies estimates on the norm of the resolvent
$\left(z-A\right)^{-1}$ of the generator $A=-X+V$ for $\mathrm{Re}\left(z\right)>C_{X,V}$.
\end{rem}

\begin{proof}
We prove (\ref{eq:bound_C_X,V}). We have that
\[
\left\Vert \mathcal{L}^{t}\right\Vert _{\mathcal{H}_{W}}\eq{\ref{eq:compare_norms}}\left\Vert \mathcal{M}_{W}\text{\ensuremath{\mathcal{T}}\ensuremath{\mathcal{L}^{t}}\ensuremath{\mathcal{T}^{\dagger}}}\mathcal{M}_{W^{-1}}\right\Vert _{L^{2}}.
\]
We write
\begin{equation}
\mathcal{M}_{W}\text{\ensuremath{\mathcal{T}}\ensuremath{\mathcal{L}^{t}}\ensuremath{\mathcal{T}^{\dagger}}}\mathcal{M}_{W^{-1}}=\left(\ensuremath{\mathcal{T}}\ensuremath{\mathcal{L}^{t}}\ensuremath{\mathcal{T}^{\dagger}}\right)\mathcal{M}_{\left(W\circ\tilde{\phi}^{t}\right)/W}+R\label{eq:decomp}
\end{equation}
with remainder operator
\begin{align}
R & :=\mathcal{M}_{W}\text{\ensuremath{\mathcal{T}}\ensuremath{\mathcal{L}^{t}}\ensuremath{\mathcal{T}^{\dagger}}}\mathcal{M}_{W^{-1}}-\text{\ensuremath{\mathcal{T}}\ensuremath{\mathcal{L}^{t}}\ensuremath{\mathcal{T}^{\dagger}}}\mathcal{M}_{\left(W\circ\tilde{\phi}^{t}\right)/W}\label{eq:def_R-1}\\
 & =\left(\mathcal{M}_{W}\text{\ensuremath{\mathcal{T}}\ensuremath{\mathcal{L}^{t}}\ensuremath{\mathcal{T}^{\dagger}}}-\text{\ensuremath{\mathcal{T}}\ensuremath{\mathcal{L}^{t}}\ensuremath{\mathcal{T}^{\dagger}}}\mathcal{M}_{W\circ\tilde{\phi}^{t}}\right)\mathcal{M}_{W^{-1}}.\nonumber 
\end{align}
We apply Theorem \ref{thm:Microlocality-of-the_TO} and Theorem \ref{thm:W}
that gives (\ref{eq:slow_variation_W-1}) and we get

\begin{align}
\left|\langle\delta_{\rho',j'}|R\delta_{\rho,j}\rangle\right|\eq{\ref{eq:def_R-1}} & \left|\langle\delta_{\rho',j'}|\text{\ensuremath{\mathcal{T}}\ensuremath{\mathcal{L}^{t}}\ensuremath{\mathcal{T}^{\dagger}}}\delta_{\rho,j}\rangle\right|\left|W\left(\rho'\right)-W\left(\tilde{\phi}^{t}\left(\rho\right)\right)\right|W^{-1}\left(\rho\right)\nonumber \\
\underset{(\ref{eq:microl_estimate},\ref{eq:slow_variation_W-1})}{\leq} & \left(W\left(\rho'\right)C_{W}h_{0}^{\mu}\left\langle h_{\gamma}^{\perp}\left(\tilde{\phi}^{t}\left(\rho\right)\right)\mathrm{dist}_{g}\left(\rho',\tilde{\phi}^{t}\left(\rho\right)\right)\right\rangle ^{N_{W}}\right)\\
 & \left(C_{N,t}\left\langle \mathrm{dist}_{g}\left(\rho',\tilde{\phi}^{t}\left(\rho\right)\right)\right\rangle ^{-N}\right)W^{-1}\left(\rho\right)\\
\leq & C_{N,t}h_{0}^{\mu}C_{W}\left(\frac{W\left(\rho'\right)}{W\left(\tilde{\phi}^{t}\left(\rho\right)\right)}\right)\left(\frac{W\left(\tilde{\phi}^{t}\left(\rho\right)\right)}{W\left(\rho\right)}\right)\left\langle \mathrm{dist}_{g}\left(\rho',\tilde{\phi}^{t}\left(\rho\right)\right)\right\rangle ^{-N+N_{W}}\nonumber \\
\underset{(\ref{eq:decay_property},\ref{eq:temperate_property})}{\leq} & C_{N,t}h_{0}^{\mu}C_{W}C'_{W}\left\langle h_{0}^{\mu/N_{W}}h_{\gamma}^{\perp}\left(\tilde{\phi}^{t}\left(\rho\right)\right)\mathrm{dist}_{g}\left(\rho',\tilde{\phi}^{t}\left(\rho\right)\right)\right\rangle ^{N_{W}}\left\langle \mathrm{dist}_{g}\left(\rho',\tilde{\phi}^{t}\left(\rho\right)\right)\right\rangle ^{-N+N_{W}}\nonumber \\
\leq & C_{N,t}h_{0}^{\mu}C_{W}\left\langle \mathrm{dist}_{g}\left(\rho',\tilde{\phi}^{t}\left(\rho\right)\right)\right\rangle ^{-N+2N_{W}}.\nonumber 
\end{align}
Then by Schur Lemma \ref{lem:Schur-Lemma-.}, we get
\begin{equation}
\left\Vert R\right\Vert _{L^{2}}\le C_{W,N_{W},t}h_{0}^{\mu},\label{eq:bound_R}
\end{equation}
where the constant $C_{W,N_{W},t}$ depends on $t$ and also on the
parameters $R_{u}$, $R_{s},N_{W}$ and $\gamma$ but not on $h_{0}$.
We also have

\[
\left\Vert \mathcal{M}_{W\circ\tilde{\phi}^{t}/W}\right\Vert _{L^{2}}\ineq{\ref{eq:decay_property}}C
\]
with $C$ independent of $t$ and $h_{0}$. Let $0<\epsilon'<\epsilon$.
From the choice of $C_{X,V}$ in \eqref{eq:C_XV}, we have
\[
\left\Vert \ensuremath{\mathcal{T}}\ensuremath{\mathcal{L}^{t}}\ensuremath{\mathcal{T}^{\dagger}}\right\Vert _{L^{2}}\eq{\ref{eq:resol_ident_Pi_rho}}\ensuremath{\left\Vert \mathcal{L}^{t}\right\Vert _{L^{2}}}\underset{(\ref{eq:L*t*Lt})}{\leq}Ce^{\left(C_{X,V}+\epsilon'\right)t}\quad\text{for }t>0
\]
with $C$ independent of $t$. We obtain
\[
\left\Vert \mathcal{M}_{W}\text{\ensuremath{\mathcal{T}}\ensuremath{\mathcal{L}^{t}}\ensuremath{\mathcal{T}^{\dagger}}}\mathcal{M}_{W^{-1}}\right\Vert _{L^{2}}\ineq{\ref{eq:decomp},\ref{eq:bound_R}}Ce^{(C_{X,V}+\epsilon')t}+C_{W,N_{W},t}h_{0}^{\mu}\le e^{(C_{X,V}+\epsilon)t}
\]
where the last inequality is obtained by taking $t$ large enough
and then $h_{0}$ small enough. By iteration, the last inequality
implies (\ref{eq:bound_C_X,V}). To prove \eqref{eq:bound_C_X,V-1},
we proceed similarly for the negative time $t<0$ and use (\ref{eq:decay_property-1})
in the place of \eqref{eq:decay_property}.
\end{proof}
For the next Lemma, let $\sigma>0$ that will be chosen large enough
later. Define the following characteristic function $\Lambda_{\sigma}:T^{*}M\rightarrow\mathbb{R}^{+}$,
such that for every $\rho\in T^{*}M$
\begin{align*}
\Lambda_{\sigma}\left(\rho\right) & :=\begin{cases}
1 & \text{ if }\mathrm{dist}_{g}\left(\rho,E_{0}^{*}\right)\leq\sigma\\
0 & \text{ if }\mathrm{dist}_{g}\left(\rho,E_{0}^{*}\right)>\sigma
\end{cases}.
\end{align*}
Let 
\begin{equation}
\mathrm{Op}\left(1-\Lambda_{\sigma}\right):=\mathcal{T}^{\dagger}\left(1-\Lambda_{\sigma}\right)\mathcal{T}\quad:C^{\infty}\left(M\right)\rightarrow C^{\infty}\left(M\right)\label{eq:def_op_Lambda}
\end{equation}
be the corresponding PDO operator, that removes components near the
trapped set. The next Theorem shows that the transfer operator decays
very fast on the outer part of the trapped set.

\begin{cBoxB}{}
\begin{thm}[Decay outside the trapped set]
\label{thm:decay_outside_trapped_set}For any $\epsilon>0$, there
exists $C>0$, for any parameters $R_{u},R_{s}>0$, we can choose
$h_{0}>0$ in (\ref{eq:h_gamma}) small enough so that, for any $t\geq0$,
we can choose $\sigma_{t}>0$ large enough and we have
\begin{equation}
\left\Vert \mathcal{L}^{t}\mathrm{Op}\left(1-\Lambda_{\sigma_{t}}\right)\right\Vert _{\mathcal{H}_{W}}\leq Ce^{t\left(C_{X,V}-\Lambda+\epsilon\right)},\label{eq:bound_C_X,V-2}
\end{equation}
where $\Lambda$ is given in (\ref{eq:def_Lambda}) and can be arbitrarily
large (if $R_{s},R_{u}$ are taken large).
\end{thm}

\end{cBoxB}

\begin{proof}
We proceed as in Lemma \ref{lem:For-the-weight}. For simplicity we
write $\Lambda_{\sigma}^{c}=1-\Lambda_{\sigma}$. The operator $\ensuremath{\mathcal{L}^{t}}\textrm{Op}\left(\Lambda_{\sigma}^{c}\right)$
on $\mathcal{H}_{W}\left(M\right)$ lifted on $L^{2}\left(T^{*}M;\mathbb{C}^{J}\right)$
is
\begin{align*}
\mathcal{M}_{W}\text{\ensuremath{\mathcal{T}}\ensuremath{\mathcal{L}^{t}}\textrm{Op}\ensuremath{\left(\Lambda_{\sigma}^{c}\right)}\ensuremath{\ensuremath{\mathcal{T}^{\dagger}}}}\mathcal{M}_{W^{-1}} & =\ensuremath{\mathcal{T}}\ensuremath{\mathcal{L}^{t}}\ensuremath{\mathcal{T}^{\dagger}}\mathcal{M}_{\Lambda_{\sigma}^{c}\left(W\circ\tilde{\phi}^{t}\right)/W}Q_{1}+R_{2}
\end{align*}
with
\[
Q_{1}=W\mathcal{P}W^{-1},\quad Q_{2}=W\mathcal{M}_{\Lambda_{\sigma}^{c}}\mathcal{P}W^{-1},
\]
\[
R_{2}=RQ_{2},
\]
and $R=\left(\mathcal{M}_{W}\text{\ensuremath{\mathcal{T}}\ensuremath{\mathcal{L}^{t}}\ensuremath{\mathcal{T}^{\dagger}}}-\text{\ensuremath{\mathcal{T}}\ensuremath{\mathcal{L}^{t}}\ensuremath{\mathcal{T}^{\dagger}}}\mathcal{M}_{W\circ\tilde{\phi}^{t}}\right)\mathcal{M}_{W^{-1}}$
as in (\ref{eq:def_R-1}). We have
\[
\left\Vert \ensuremath{\mathcal{T}}\ensuremath{\mathcal{L}^{t}}\ensuremath{\mathcal{T}^{\dagger}}\right\Vert _{L^{2}}\eq{\ref{eq:resol_ident_Pi_rho}}\ensuremath{\left\Vert \mathcal{L}^{t}\right\Vert _{L^{2}}}\underset{(\ref{eq:L*t*Lt})}{\leq}Ce^{\left(C_{X,V}+\epsilon\right)t}\quad\text{for }t>0.
\]
From the decay property \eqref{eq:decay_property} of the weight function
$W$, we have that if $\sigma$ is large enough depending on $t$
and $W$,

\[
\left\Vert \mathcal{M}_{\Lambda_{\sigma}^{c}\left(W\circ\tilde{\phi}^{t}\right)/W}\right\Vert _{L^{2}}\ineq{\ref{eq:decay_property}}Ce^{-\Lambda t}.
\]
We have
\begin{align*}
\left|\langle\delta_{\rho',j'}|Q_{1}\delta_{\rho,j}\rangle\right| & =\frac{W\left(\rho'\right)}{W\left(\rho\right)}\left|\langle\delta_{\rho',j'}|\mathcal{P}\delta_{\rho,j}\rangle\right|\\
 & \ineq{\ref{eq:temperate_property},\ref{eq:estimate_Bergman_kernel}}C_{W}\left\langle h_{0}^{\mu/N_{W}}\mathrm{dist}_{g}\left(\rho',\rho\right)\right\rangle ^{N_{W}}C_{N}\left\langle \mathrm{dist}_{g}\left(\rho',\rho\right)\right\rangle ^{-N}.
\end{align*}
with $C_{W}$ independent on $N_{W}$, but with $h_{0}$ small enough
w.r.t. $N_{W}$, so that we proceed as in (\ref{eq:aide}), and by
Schur Lemma \ref{lem:Schur-Lemma-.}, we get that
\[
\left\Vert Q_{1}\right\Vert _{L^{2}}\le C_{W}.
\]
Similarly $\left\Vert Q_{2}\right\Vert _{L^{2}}\le C_{W}$. Then
\[
\left\Vert R_{2}\right\Vert _{L^{2}}\le\left\Vert R\right\Vert \left\Vert Q_{2}\right\Vert \ineq{\ref{eq:bound_R}}C_{W,N_{W},t}h_{0}^{\mu}.
\]
We obtain
\[
\left\Vert \mathcal{M}_{W}\text{\ensuremath{\mathcal{T}}\ensuremath{\mathcal{L}^{t}}\textrm{Op}\ensuremath{\left(\Lambda_{\sigma}^{c}\right)}\ensuremath{\ensuremath{\mathcal{T}^{\dagger}}}}\mathcal{M}_{W^{-1}}\right\Vert _{L^{2}}\le C_{W}e^{(C_{X,V}+\epsilon-\Lambda)t}+C_{W,N_{W},t}h_{0}^{\mu},
\]
hence for some large $t>0$, then taking $h_{0}$ small enough and
$\sigma$ large enough, we get
\[
\left\Vert \mathcal{L}^{t}\mathrm{Op}\left(1-\Lambda_{\sigma}\right)\right\Vert _{\mathcal{H}_{W}}=\left\Vert \mathcal{M}_{W}\text{\ensuremath{\mathcal{T}}\ensuremath{\mathcal{L}^{t}}\textrm{Op}\ensuremath{\left(\Lambda_{\sigma}^{c}\right)}\ensuremath{\ensuremath{\mathcal{T}^{\dagger}}}}\mathcal{M}_{W^{-1}}\right\Vert _{L^{2}}\le Ce^{(C_{X,V}-\Lambda+\epsilon)t}.
\]
\end{proof}

\subsubsection{Meromorphic extension of the resolvent}

\begin{cBoxB}{}
\begin{prop}[Discrete spectrum of the generator and upper bound of the density]
\label{prop:discrete_spectrum}The generator $A=-X+V:\mathcal{H}_{W}\rightarrow\mathcal{H}_{W}$
of the semi-group $\left(\mathcal{L}^{t}\right)_{t\geq0}$ has discrete
spectrum on the domain $\mathrm{Re}\left(z\right)\geq C_{X,V}-\Lambda$
with $\Lambda$ given in (\eqref{eq:def_Lambda}). For any $\epsilon>0$,
there exists $C>0$ such that, for any $\omega\in\mathbb{R}$, the
number of the discrete eigenvalues (counted with multiplicities) in
the spectral region 
\begin{equation}
R_{\omega}:=\left\{ z\in\mathbb{C}\;\left|\;\mathrm{Re}\left(z\right)\geq C_{X,V}-\Lambda+\epsilon,\;\mathrm{Im}(z)\in\left[\omega,\omega+1\right]\right.\right\} \label{eq:def_region_R}
\end{equation}
is bounded by $C\left\langle \omega\right\rangle ^{n/(1+\beta_{0})}$.
\end{prop}

\end{cBoxB}

Lemma \ref{lem:For-the-weight} and Proposition \ref{prop:discrete_spectrum}
give Theorem \ref{thm:Discrete_spectrum} and Theorem \ref{thm:Weyl law}.
\begin{proof}
We fix the parameter $\alpha^{\parallel}=0$ that enters in the metric
(\ref{eq:conditions}). Let $\omega\in\mathbb{R}$ and set
\begin{equation}
J_{\omega}:=\left[\omega-1,\omega+1\right].\label{eq:def_J_omega}
\end{equation}
We want to show that $A$ has discrete spectrum on domain $R_{\omega}$.
See Figure \ref{fig:zones}. We take a constant $\sigma\geq1$ and
let it be large enough independently of $\omega$ in the course of
the proof below. Let us consider a continuous function with compact
support $\varpi\in C_{c}\left(T^{*}M;\left[0,1\right]\right)$ such
that, for $\rho\in T^{*}M$, $m=\pi\left(\rho\right)\in M$, 
\begin{align}
\varpi\left(\rho\right) & =\begin{cases}
1 & \text{ if }\left\Vert \rho-\omega\alpha\left(m\right)\right\Vert _{g_{\rho}}\leq\sigma\\
0 & \text{ if }\left\Vert \rho-\omega\alpha\left(m\right)\right\Vert _{g_{\rho}}\geq2\sigma
\end{cases}.\label{eq:def_Omega_0'}
\end{align}
We can and will assume that $\varpi$ is slowly varying, as in Definition
\ref{def:Let--be}, in the sense that 
\begin{equation}
\left|\varpi\left(\rho'\right)-\varpi\left(\rho\right)\right|\leq\frac{1}{\sigma}\left\langle \mathrm{dist}_{g}\left(\rho',\rho\right)\right\rangle \quad\text{for all }\rho,\rho'\in T^{*}M.\label{eq:slow_variation_varpi}
\end{equation}
There exists $C_{\sigma}>0$ such that 
\[
C_{\sigma}^{-1}\left\langle \omega\right\rangle ^{n\alpha^{\perp}}\le\mathrm{Vol}\left(\mathrm{supp}\left(\varpi\right)\right)\le C_{\sigma}\left\langle \omega\right\rangle ^{n\alpha^{\perp}}
\]
and hence, from (\ref{eq:norm_Tr_Op}), it follows 
\begin{equation}
\left\Vert \mathrm{Op}\left(\varpi\right)\right\Vert _{\mathrm{Tr}}\leq C_{\sigma}\mathrm{Vol}\left(\mathrm{supp}\left(\varpi\right)\right)\leq C\left\langle \omega\right\rangle ^{n\alpha^{\perp}}.\label{eq:trace_Omega_0}
\end{equation}
with $C$ independent of $\omega$. Let us consider the modified operator

\begin{equation}
A':=A-\Lambda\cdot\mathrm{Op}\left(\varpi\right)\qquad:\mathcal{D}\left(A\right)\rightarrow\mathcal{H}_{W}\label{eq:def_A'}
\end{equation}
where $\Lambda>0$ is the decay rate of the escape function $W$ in
(\ref{eq:decay_property}).
\begin{rem}
The special role played by $\Lambda$ will appear in (\ref{eq:res6}).
The added term $-\Lambda\cdot\mathrm{Op}\left(\varpi\right)$ in (\ref{eq:def_A'})
is sometimes called ``absorbing potential'' in micro-local analysis.
Its action is somehow to select the component of $A$ that acts micro-locally
on the support of $\varpi$ and push it on the left in the spectral
domain.
\end{rem}

\begin{cBoxB}{}
\begin{lem}
\label{lem:bounded_resolvent_of_A'}If we let $\sigma\geq1$ in the
definition (\ref{eq:def_Omega_0'}) be sufficiently large (independently
on $\omega$), the resolvent $\left(z-A'\right)^{-1}$ extends holomorphically
to the region $z\in R_{\omega}$ defined in (\ref{eq:def_region_R})
and further satisfies $\left\Vert \left(z-A'\right)^{-1}\right\Vert \leq C$
for $z\in R_{\omega}$ with some constant $C$ independent of $\omega\in\mathbb{R}$.
\end{lem}

\end{cBoxB}

The proof of Lemma \ref{lem:bounded_resolvent_of_A'} will be given
in Subsection \ref{subsec:Proof-of-Lemma-1}. We continue the proof
of Proposition \ref{prop:discrete_spectrum}. For $z\in R_{\omega}$,
we write
\[
\left(z-A\right)=\left(z-A'-\Lambda\mathrm{Op}\left(\varpi\right)\right)=\left(z-A'\right)\left(1-\left(z-A'\right)^{-1}\Lambda\mathrm{Op}\left(\varpi\right)\right).
\]
From Lemma \ref{lem:bounded_resolvent_of_A'} and (\ref{eq:trace_Omega_0}),
$\left(z-A'\right)^{-1}\Lambda\mathrm{Op}\left(\varpi\right)$ is
a holomorphic family of trace class operators that depends on $z\in R_{\omega}$,
hence $\left(1-\left(z-A'\right)^{-1}\Lambda\mathrm{Op}\left(\varpi\right)\right)$
is a holomorphic family of Fredholm operators of index $0$. We deduce
that
\[
\left(z-A\right)^{-1}=\left(1-\left(z-A'\right)^{-1}\Lambda\mathrm{Op}\left(\varpi\right)\right)^{-1}\left(z-A'\right)^{-1}
\]
except for a discrete set of points $z\in R_{\omega}$ where $\left(z-A\right)^{-1}$
has a pole of finite order.

To prove the latter claim on the number of eigenvalues of $A$ in
$R_{\omega}$, we use a method due to Sjöstrand \cite[Section 4]{sjostrand_2000}.
(See also \cite[Section 8.2]{faure-tsujii_prequantum_maps_12}). First
of all, the operator $\mathcal{K}(z):=\left(z-A'\right)^{-1}\Lambda\mathrm{Op}\left(\varpi\right)$
is a trace class operator for $z\in R_{\omega}$ and its trace norm
is bounded by $C\left\langle \omega\right\rangle ^{n\alpha^{\perp}}$
from \eqref{eq:trace_Omega_0}. We consider the holomorphic function
\[
\text{\ensuremath{\mathcal{D}}}:z\in R_{\omega}\mapsto\det\left(1-\mathcal{K}(z)\right)\in\mathbb{C}.
\]
We have 
\[
\log|\ensuremath{\mathcal{D}}(z)|\le C\left\Vert \mathcal{K}(z)\right\Vert _{\mathrm{Tr}}\le C\left\langle \omega\right\rangle ^{n\alpha^{\perp}}.
\]
If $\text{Re}(z)>C_{0}$ for sufficiently large $C_{0}$, we have
$\left\Vert \left(z-A\right)^{-1}\right\Vert \le C$ and
\[
\left(1-\mathcal{K}(z)\right)^{-1}=\left(z-A\right)^{-1}\left(z-A'\right)=1-\left(z-A\right)^{-1}\Lambda\mathrm{Op}\left(\varpi\right),
\]
and $\left\Vert \left(z-A\right)^{-1}\Lambda\mathrm{Op}\left(\varpi\right)\right\Vert _{\mathrm{Tr}}\le C'\left\langle \omega\right\rangle ^{n\alpha^{\perp}}$
from \eqref{eq:trace_Omega_0}, so
\[
-\log\left|\ensuremath{\mathcal{D}}\left(z\right)\right|=\log\left|\mathrm{det}\left(1-\mathcal{K}(z)\right)^{-1}\right|\le C'\left\langle \omega\right\rangle ^{n\alpha^{\perp}}.
\]
Now we take a slightly larger scaled neighborhood $\widetilde{R}_{\omega}\Supset R_{\omega}$
and a Riemann mapping $\psi:\widetilde{R}_{\omega}\to\mathbb{D}:=\{|z|<1\}$
so that $\psi(w)=0$ for a point $w\in R_{\omega}$ with $\text{Re}(z)>C_{0}$.
If we apply Jensen's formula to the sub-harmonic function
\[
w\in\mathbb{D}\mapsto\log|\ensuremath{\mathcal{D}}(\psi^{-1}(w))|,
\]
we obtain that the number of eigenvalues of $A$ in the region $R_{\omega}$
is bounded by $C\left\langle \omega\right\rangle ^{n\alpha^{\perp}}$.
This finishes the proof of Proposition \ref{prop:discrete_spectrum}.
\end{proof}

\subsubsection{\label{subsec:Proof-of-Lemma-1}Proof of Lemma \ref{lem:bounded_resolvent_of_A'}}

Let $\omega\in\mathbb{R}$ and $z\in R_{\omega}\subset\mathbb{C}$
where $R_{\omega}$ is defined in (\ref{eq:def_region_R}). We want
to show that $\left(z-A'\right)$ is invertible with uniformly bounded
inverse w.r.t. $z,\omega$. For this, we will construct operators
$R_{r}\left(z\right),R_{l}\left(z\right)$ so that they are uniformly
bounded and are approximate right and left inverses in the sense that
\begin{align}
\left\Vert \left(z-A'\right)R_{r}\left(z\right)-\mathrm{Id}\right\Vert _{\mathcal{H}_{W}} & \leq\frac{1}{2}\quad\text{and }\label{eq:almost_inverse}\\
\left\Vert R_{l}\left(z\right)\left(z-A'\right)-\mathrm{Id}\right\Vert _{\mathcal{H}_{W}} & \leq\frac{1}{2}.
\end{align}
Then the Neumann series
\begin{align*}
\tilde{R}_{r}\left(z\right):= & R_{r}\left(z\right)\sum_{k\geq0}\left(\mathrm{Id}-\left(z-A'\right)R_{r}\left(z\right)\right)^{k}\quad\text{and }\\
\tilde{R}_{l}\left(z\right):= & \sum_{k\geq0}\left(\mathrm{Id}-\left(z-A'\right)R_{r}\left(z\right)\right)^{k}R_{l}\left(z\right)
\end{align*}
give\footnote{Formally if $\left|1-XR\right|\leq1/2$ then $\tilde{R}:=R\sum_{k\geq0}\left(1-XR\right)^{k}$
is convergent and 
\[
X\tilde{R}=\left(XR-1+1\right)\sum_{k\geq0}\left(1-XR\right)^{k}=-\sum_{k\geq1}\left(1-XR\right)^{k}+\sum_{k\geq0}\left(1-XR\right)^{k}=1.
\]
} the exact right and left inverses satisfying $\left(z-A'\right)\tilde{R}_{r}\left(z\right)=\tilde{R}_{l}\left(z\right)\left(z-A'\right)=\mathrm{Id}$,
which are uniformly bounded. Therefore we obtain that the resolvent
$\left(z-A'\right)^{-1}=\tilde{R}_{r}\left(z\right)=\tilde{R}_{r}\left(z\right)\left(z-A'\right)\tilde{R}_{l}\left(z\right)=\tilde{R}_{l}\left(z\right)$
is bounded uniformly.

\begin{figure}[h]
\centering{}\scalebox{0.8}[0.8]{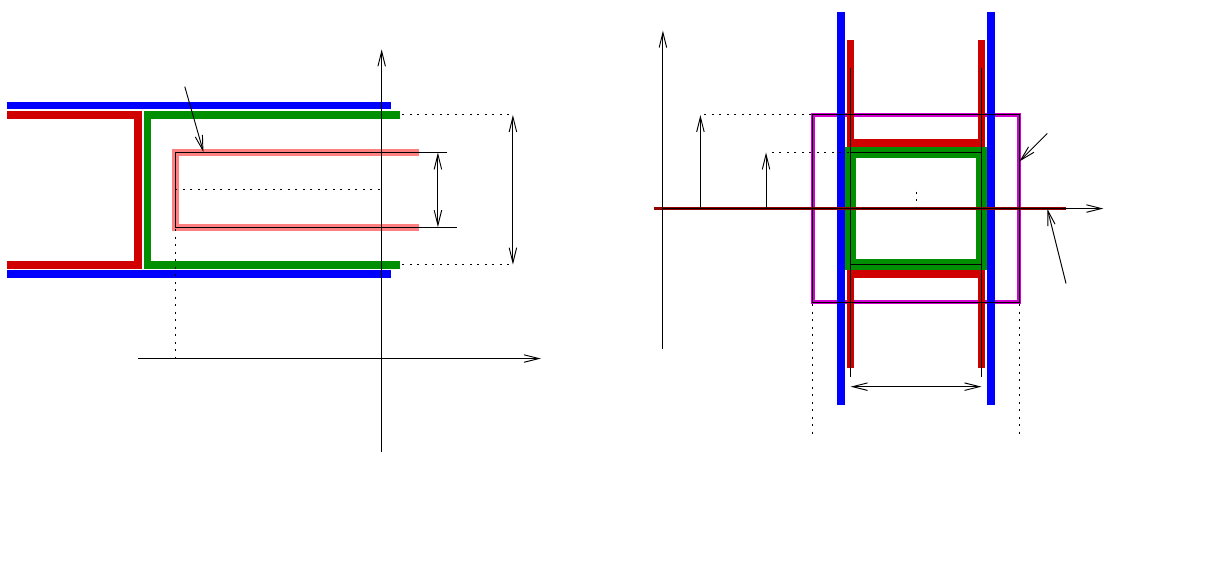}\caption{\textbf{\label{fig:zones}Figure (a):} In the proof of Lemma \ref{prop:discrete_spectrum},
we show that a perturbation $A'=A-\Lambda\mathrm{Op}\left(\varpi\right)$
of the generator $A$ has no spectrum on the spectral domain $R_{\omega}\subset\mathbb{C}$
with frequency range $J_{\omega}=\left[\omega-1,\omega+1\right]$.\textbf{}\protect \\
\textbf{Figure (b):} The operator $\mathrm{Op}\left(\varpi\right)$
is constructed from a compact vicinity $\mathrm{supp}(\varpi)$ of
the trapped set $E_{0}^{*}\subset T^{*}M$ for a larger range of frequencies
and with transverse size $\sigma\left\langle \omega\right\rangle ^{\alpha^{\perp}}$.
We introduce a partition $T^{*}M=\Omega_{0}\cup\Omega_{1}\cup\Omega_{2}$
where $\Omega_{0}$ is a compact vicinity of the trapped set $E_{0}^{*}$
for an intermediate range of frequencies $J'_{\omega}$ and with transverse
size $\sigma'\left\langle \omega\right\rangle ^{\alpha{}^{\perp}}$
so that $\Omega_{0}\subset\mathrm{supp}(\varpi)$. The set $\Omega_{1}$
is the outside of $\Omega_{0}$ (away from the trapped set) with the
same frequencies range $J'_{\omega}$ and $\Omega_{2}$ are all the
other points with other frequencies. During the proof, we take $\sigma\gg\sigma'\gg1$.}
\end{figure}

In order to construct $R_{r}\left(z\right)$ and $R_{l}\left(z\right)$,
we take a real number $c>0$ and set 
\begin{equation}
J'_{\omega}:=\left[\omega-\left(1+c\right),\omega+\left(1+c\right)\right]\:\underset{(\ref{eq:def_J_omega})}{\supset}\:J{}_{\omega}.\label{eq:def_J'}
\end{equation}
We take another constant $\sigma'>0$ and set $\Omega_{0}$ as a $K-$ball
around the trapped set at frequency $\omega$:
\begin{equation}
\Omega_{0}:=\left\{ \rho=\rho_{u}+\rho_{s}+\omega\alpha\left(m\right)\in T^{*}M,\,m=\pi\left(\rho\right)\in M,\omega\in J'_{\omega},\,\left\Vert \rho_{u}+\rho_{s}\right\Vert _{g_{\rho}}\leq K\right\} ,\label{eq:def_Omega_0}
\end{equation}
\begin{equation}
\Omega_{1}:=\left\{ \rho=\rho_{u}+\rho_{s}+\omega\alpha\left(m\right)\in T^{*}M,\,\omega\in J'_{\omega},\,\left\Vert \rho_{u}+\rho_{s}\right\Vert _{g_{\rho}}\geq K\right\} \label{eq:def_Omega_1}
\end{equation}
and 
\begin{equation}
\Omega_{2}:=\left\{ \rho\in T^{*}M,\quad\omega\left(\rho\right)\in\mathbb{R}\backslash J'_{\omega}\right\} \label{eq:def_Omega_2}
\end{equation}
so that $T^{*}M=\Omega_{0}\cup\Omega_{1}\cup\Omega_{2}$. These regions
are drawn on Figure \ref{fig:zones} and explained in the caption.
For $j=0,1,2$, let $\chi_{\Omega_{j}}$ be the characteristic function
of $\Omega_{j}$. We have $\sum_{j}\chi_{\Omega_{j}}=1$ and hence
\begin{equation}
\mathrm{Id}\underset{(\ref{eq:resol_ident_Pi_rho})}{=}\mathrm{Op}\left(1\right)=\mathrm{Op}\left(\chi_{\Omega_{0}}\right)+\mathrm{Op}\left(\chi_{\Omega_{1}}\right)+\mathrm{Op}\left(\chi_{\Omega_{2}}\right).\label{eq:ident}
\end{equation}
Correspondingly to the decomposition (\ref{eq:ident}), we will construct
an approximate resolvent 
\[
R_{r}\left(z\right)=R_{r}^{\left(0\right)}\left(z\right)+R_{r}^{\left(1\right)}\left(z\right)+R_{r}^{\left(2\right)}\left(z\right)
\]
with three contributions $R_{r}^{\left(j\right)}\left(z\right)$,
$j=0,1,2$, so that each of them are uniformly bounded for $z\in R_{\omega}$
(and $\omega\in\mathbb{R}$) and satisfies
\begin{equation}
\left\Vert \left(z-A'\right)R_{r}^{\left(j\right)}\left(z\right)-\mathrm{Op}\left(\chi_{\Omega_{j}}\right)\right\Vert _{\mathcal{H}_{W}}\leq\frac{1}{6}.\label{eq:approximate_R_j}
\end{equation}
The estimate (\ref{eq:approximate_R_j}) will give (\ref{eq:almost_inverse})
and finish the proof of Lemma \ref{lem:bounded_resolvent_of_A'}.

\paragraph{\label{par:Contribution}Contribution $R_{r}^{\left(0\right)}\left(z\right)$}

For the function $\varpi$ in (\ref{eq:def_Omega_0'}) and $t\geq0$,
we define
\begin{equation}
\varpi_{\left[0,t\right]}:=\int_{0}^{t}\varpi\circ\tilde{\phi}^{s}ds.\label{eq:def_varpi_0-t}
\end{equation}
Let $T>0$ and\footnote{Let us explain why the expression (\ref{eq:expression_R(0)}) is a
natural guess. Formally a good guess for the resolvent would be $R_{r}\left(z\right):=\int_{0}^{\infty}e^{-t\left(z-A'\right)}dt$
with $A'=A-\Lambda\mathrm{Op}\left(\varpi\right)$. However the operator
$e^{tA'}$ has no clear mathematical sense but using Egorov's formula
(\ref{eq:norm_Egorov_PDO}), for bounded time, the operator $e^{tA'}$
is approximated by $e^{tA}e^{-\mathrm{Op}\left(\varpi_{\left[0,t\right]}\right)}$.
Finally we truncate the integral up to a finite (but large) $T$ and
compose with a final truncation operator $\mathrm{Op}\left(\chi_{\Omega_{0}}\right)$
to get (\ref{eq:expression_R(0)}).}
\begin{equation}
R_{r}^{\left(0\right)}\left(z\right):=\int_{0}^{T}e^{-t\left(z-A\right)}\mathrm{Op}\left(e^{-\Lambda\varpi_{\left[0,t\right]}}\right)\mathrm{Op}\left(\chi_{\Omega_{0}}\right)dt.\label{eq:expression_R(0)}
\end{equation}
We assume that the constant $\sigma$ in (\ref{eq:def_Omega_0'})
is large enough so that we have
\[
\varpi\left(\tilde{\phi}^{t}\left(\rho\right)\right)=1\quad\text{and}\quad\varpi_{\left[0,t\right]}\left(\rho\right)\underset{(\ref{eq:def_varpi_0-t})}{=}t\quad\text{for all }t\in\left[0,T\right]\text{ and }\rho\in\Omega_{0}.
\]
We apply Corollary \ref{cor:B3}, regarding $\sigma\geq1$ as the
parameter, let $t\in\left[0,T\right]$, and see that for every $N>0$,
there exists $C_{N,T}>0$ such that
\[
\left\Vert \mathrm{Op}\left(e^{-\Lambda\varpi_{\left[0,t\right]}}\right)\mathrm{Op}\left(\chi_{\Omega_{0}}\right)-e^{-t\Lambda}\mathrm{Op}\left(\chi_{\Omega_{0}}\right)\right\Vert _{\mathcal{H}_{W}}\underset{(\ref{eq:corollary_compos})}{\leq}C_{N,T}\sigma^{-N}.
\]
Since $\left\Vert \mathrm{Op}\left(\chi_{\Omega_{0}}\right)\right\Vert _{\mathcal{H}_{W}}\leq1$
from (\ref{eq:norme_symbol}), it follows that
\[
\left\Vert \mathrm{Op}\left(e^{-\Lambda\varpi_{\left[0,t\right]}}\right)\mathrm{Op}\left(\chi_{\Omega_{0}}\right)\right\Vert _{\mathcal{H}_{W}}\leq e^{-t\Lambda}+C_{N,T}\sigma^{-N}.
\]
Therefore we have
\begin{align}
\left\Vert e^{-t\left(z-A\right)}\mathrm{Op}\left(e^{-\Lambda\varpi_{\left[0,t\right]}}\right)\mathrm{Op}\left(\chi_{\Omega_{0}}\right)\right\Vert _{\mathcal{H}_{W}} & \underset{(\ref{eq:bound_C_X,V})}{\leq}Ce^{-t\left(\mathrm{Re}\left(z\right)-C_{X,V}+\Lambda\right)}+C_{N,T}\sigma^{-N}\nonumber \\
 & \le Ce^{-t\epsilon}+C_{N,T}\sigma^{-N}\label{eq:res5}
\end{align}
for $t\in[0,T]$ and hence we deduce that, for $z\in R_{\omega}$,
\[
\left\Vert R_{r}^{\left(0\right)}\left(z\right)\right\Vert \underset{(\ref{eq:expression_R(0)})}{\leq}\int_{0}^{T}\left\Vert e^{-t\left(z-A\right)}\mathrm{Op}\left(e^{-\Lambda\varpi_{\left[0,t\right]}}\right)\mathrm{Op}\left(\chi_{\Omega_{0}}\right)\right\Vert _{\mathcal{H}_{W}}dt\underset{(\ref{eq:res5})}{\leq}\frac{C}{\epsilon}+C{}_{N,T}\sigma{}^{-N}.
\]
Since the constant $C_{N,T}$ does not depend on $\sigma$, we can
let the constant $\sigma$ be large for given $T>0$ so that the last
term $C_{N,T}\sigma{}^{-N}$ is small relative to the former term
$C/\epsilon$ and that the operator norm $\left\Vert R_{r}^{\left(0\right)}\left(z\right)\right\Vert $
is uniformly bounded for $z\in R_{\omega}$ (and $\omega\in\mathbb{R}$).

Now we prove the required estimate \eqref{eq:approximate_R_j}. Observe
that
\begin{align}
\frac{d}{dt}\left(e^{-t\left(z-A\right)}\mathrm{Op}\left(e^{-\Lambda\varpi_{\left[0,t\right]}}\right)\right) & =-\left(z-A\right)e^{-t\left(z-A\right)}\mathrm{Op}\left(e^{-\Lambda\varpi_{\left[0,t\right]}}\right)\label{eq:res1}\\
 & \quad\quad-e^{-t\left(z-A\right)}\mathrm{Op}\left(\Lambda\varpi\circ\tilde{\phi}^{t}\cdot e^{-\Lambda\varpi_{\left[0,t\right]}}\right).\nonumber 
\end{align}
For the last term of (\ref{eq:res1}), we apply Theorem \ref{thm:Composition-of-PDO.}
on composition of PDO, Egorov's Theorem \ref{thm:Egorov.-For-any}
and also the slow variation property of $\varpi$ in (\ref{eq:slow_variation_varpi})
and obtain
\begin{align*}
e^{-t\left(z-A\right)} & \mathrm{Op}\left(\Lambda\varpi\circ\tilde{\phi}^{t}\cdot e^{-\Lambda\varpi_{\left[0,t\right]}}\right)\\
\underset{(\ref{eq:norm_Compos_PDO})}{=} & e^{-t\left(z-A\right)}\mathrm{Op}\left(\Lambda\varpi\circ\tilde{\phi}^{t}\right)\circ\mathrm{Op}\left(e^{-\Lambda\varpi_{\left[0,t\right]}}\right)+O_{\mathcal{H}_{W}}\left(C_{t}\sigma^{-1}\right)\\
\underset{(\ref{eq:norm_Egorov_PDO})}{=} & \Lambda\mathrm{Op}\left(\varpi\right)\circ e^{-t\left(z-A\right)}\circ\mathrm{Op}\left(e^{-\Lambda\varpi_{\left[0,t\right]}}\right)+O_{\mathcal{H}_{W}}\left(C_{t}\sigma^{-1}\right)
\end{align*}
where $C_{t}>0$ depends on $t$ and the error term $O_{\mathcal{H}_{W}}\left(C_{t}\sigma^{-1}\right)$
denote an operator whose operator norm on $\mathcal{H}_{W}$ is bounded
by $C_{t}\sigma^{-1}$. Putting this estimate in \eqref{eq:res1},
we get
\begin{align}
\frac{d}{dt}\left(e^{-t\left(z-A\right)}\mathrm{Op}\left(e^{-\Lambda\varpi_{\left[0,t\right]}}\right)\right)= & -\left(z-A+\Lambda\mathrm{Op}\left(\varpi\right)\right)e^{-t\left(z-A\right)}\mathrm{Op}\left(e^{-\Lambda\varpi_{\left[0,t\right]}}\right)\label{eq:res3}\\
 & \quad+O_{\mathcal{H}_{W}}\left(C_{t}\sigma^{-1}\right)\nonumber 
\end{align}
Then we deduce that
\begin{align}
 & \left(z-A'\right)\circ R_{r}^{\left(0\right)}\left(z\right)-\mathrm{Op}\left(\chi_{\Omega_{0}}\right)\nonumber \\
 & \quad\underset{(\ref{eq:def_A'}),(\ref{eq:expression_R(0)})}{=}\left(z-A+\Lambda\mathrm{Op}\left(\varpi\right)\right)\int_{0}^{T}e^{-t\left(z-A\right)}\mathrm{Op}\left(e^{-\Lambda\varpi_{\left[0,t\right]}}\right)\mathrm{Op}\left(\chi_{\Omega_{0}}\right)dt-\mathrm{Op}\left(\chi_{\Omega_{0}}\right)\nonumber \\
 & \qquad\underset{(\ref{eq:res3})}{=}-\int_{0}^{T}\frac{d}{dt}\left(e^{-t\left(z-A\right)}\mathrm{Op}\left(e^{-\Lambda\varpi_{\left[0,t\right]}}\right)\right)\mathrm{Op}\left(\chi_{\Omega_{0}}\right)dt+O_{\mathcal{H}_{W}}\left(C_{T}\sigma^{-1}\right)-\mathrm{Op}\left(\chi_{\Omega_{0}}\right)\nonumber \\
 & \qquad=-e^{-T\left(z-A\right)}\mathrm{Op}\left(e^{-\Lambda\varpi_{\left[0,T\right]}}\right)\mathrm{Op}\left(\chi_{\Omega_{0}}\right)+O_{\mathcal{H}_{W}}\left(C_{T}\sigma^{-1}\right)\label{eq:res}
\end{align}
and hence that
\[
\left\Vert \left(z-A'\right)\circ R_{r}^{\left(0\right)}\left(z\right)-\mathrm{Op}\left(\chi_{\Omega_{0}}\right)\right\Vert _{\mathcal{H}_{W}}\underset{(\ref{eq:res}),(\ref{eq:res5})}{\leq}Ce^{-T\epsilon}+C_{T}\sigma^{-1}.
\]
For a fixed $\epsilon>0$ and some given $N\geq1$ we take $T$ large
enough and then $\sigma$ large enough depending on $T$ so that $\left\Vert \left(z-A'\right)\circ R_{r}^{\left(0\right)}\left(z\right)-\mathrm{Op}\left(\chi_{\Omega_{0}}\right)\right\Vert _{\mathcal{H}_{W}}\leq1/6$.
This proves (\ref{eq:approximate_R_j}) for $j=0$.

\paragraph{Contribution $R_{r}^{\left(1\right)}\left(z\right)$}

We follow a construction similar to what we did for $R_{r}^{\left(0\right)}\left(z\right)$.
Let $T>0$ and
\begin{equation}
R_{r}^{\left(1\right)}\left(z\right):=\int_{0}^{T}e^{-t\left(z-A\right)}\mathrm{Op}\left(e^{-\Lambda\varpi_{\left[0,t\right]}}\right)\mathrm{Op}\left(\chi_{\Omega_{1}}\right)dt.\label{eq:expression_R(1)}
\end{equation}
By Theorem \eqref{thm:decay_outside_trapped_set} applied to function
$\chi_{\Omega_{1}}$ instead of $\left(1-\Lambda_{\sigma}\right)$
(the proof is the same) we get that there exists $C>0$, for any $T>0$,
$t\in\left[0,T\right]$, we can choose $h_{0}$ small enough and $\sigma$
and $K$ large enough so that
\begin{equation}
\left\Vert e^{tA}\mathrm{Op}\left(\chi_{\Omega_{1}}\right)\right\Vert _{\mathcal{H}_{W}}\leq Ce^{\left(C_{X,V}-\Lambda+\epsilon'\right)t}.\label{eq:res6}
\end{equation}
By using Egorov Theorem to permute the operators, we get
\begin{equation}
e^{-t\left(z-A\right)}\mathrm{Op}\left(e^{-\Lambda\varpi_{\left[0,t\right]}}\right)\underset{(\ref{eq:norm_Egorov_PDO})}{=}\mathrm{Op}\left(e^{-\Lambda\varpi_{\left[-t,0\right]}}\right)e^{-t\left(z-A\right)}+O_{\mathcal{H}_{W}}\left(C_{T}\sigma^{-1}\right).\label{eq:egorov_R1}
\end{equation}
for $t\in\left[0,T\right]$. Since $\left\Vert \mathrm{Op}\left(e^{-\Lambda\varpi_{\left[-T,0\right]}}\right)\right\Vert _{\mathcal{H}_{W}}\leq C$,
we have
\begin{align}
\left\Vert \mathrm{Op}\left(e^{-\Lambda\varpi_{\left[-t,0\right]}}\right)e^{-t\left(z-A\right)}\mathrm{Op}\left(\chi_{\Omega_{1}}\right)\right\Vert _{\mathcal{H}_{W}} & \underset{(\ref{eq:res6})}{\leq}Ce^{-t\mathrm{Re}\left(z\right)}e^{\left(C_{X,V}-\Lambda+\epsilon'\right)t}\nonumber \\
 & \underset{(\ref{eq:def_region_R})}{\leq}Ce^{-t\epsilon}.\label{eq:res55}
\end{align}
We deduce 
\begin{align*}
\left\Vert R_{r}^{\left(1\right)}\left(z\right)\right\Vert  & \underset{(\ref{eq:expression_R(1)}),(\ref{eq:egorov_R1})}{\leq}\int_{0}^{T}\left\Vert \mathrm{Op}\left(e^{-\Lambda\varpi_{\left[-t,0\right]}}\right)e^{-t\left(z-A\right)}\mathrm{Op}\left(\chi_{\Omega_{1}}\right)\right\Vert dt+C_{T}\sigma^{-1}\\
 & \leq\frac{C}{\epsilon}+C_{T}\sigma^{-1}.
\end{align*}
Therefore $R_{r}^{\left(1\right)}\left(z\right)$ is bounded uniformly
$z\in R_{\omega}$. Similarly to the case of $R_{r}^{\left(0\right)}\left(z\right)$,
we have 
\begin{align}
\left(z-A'\right)\circ R_{r}^{\left(1\right)}\left(z\right)-\mathrm{Op}\left(\chi_{\Omega_{1}}\right)\underset{}{=} & -e^{-T\left(z-A\right)}\mathrm{Op}\left(e^{-\Lambda\varpi_{\left[0,T\right]}}\right)\mathrm{Op}\left(\chi_{\Omega_{1}}\right)+O_{\mathcal{H}_{W}}\left(C_{T}\sigma^{-1}\right)\nonumber \\
\underset{}{=} & -\mathrm{Op}\left(e^{-\Lambda\varpi_{\left[-T,0\right]}}\right)e^{-T\left(z-A\right)}\mathrm{Op}\left(\chi_{\Omega_{1}}\right)+O_{\mathcal{H}_{W}}\left(C_{T}\sigma^{-1}\right)\label{eq:ress}
\end{align}
and deduce
\[
\left\Vert \left(z-A'\right)\circ R_{r}^{\left(1\right)}\left(z\right)-\mathrm{Op}\left(\chi_{\Omega_{1}}\right)\right\Vert _{\mathcal{H}_{W}}\underset{(\ref{eq:ress}),(\ref{eq:res55})}{\leq}Ce^{-T\epsilon}+C_{T}\sigma^{-1}.
\]
For a fixed $\epsilon>0$, we may take large $T$ and then take $h_{0}$
small and $\sigma$ large enough depending on $T$ so that $\left\Vert \left(z-A'\right)\circ R_{r}^{\left(1\right)}\left(z\right)-\mathrm{Op}\left(\chi_{\Omega_{0}}\right)\right\Vert _{\mathcal{H}_{W}}\leq\frac{1}{6}$.
We have obtained (\ref{eq:approximate_R_j}) for $j=1$.

\paragraph{Contribution $R_{r}^{\left(2\right)}\left(z\right)$}

On the domain $\Omega_{2}$ we will use ``elliptic estimate'' techniques.
The symplectic volume $m$ on $T^{*}M$ naturally induces the conditional
measure $m_{\omega}$ on each of the level set $\boldsymbol{\omega}{}^{-1}\left(\omega\right)$
so that $m=\int m_{\omega}d\omega$. For each $\omega'\in\mathbb{R}$,
we define
\begin{equation}
\Pi\left(\omega'\right):=\int\Pi\left(\rho'\right)\frac{dm_{\omega}}{\left(2\pi\right)^{n+1}}.\label{eq:def_Pi_omega}
\end{equation}
Let
\begin{align*}
R_{r}^{\left(2\right)}\left(z\right): & =\int_{\omega'\notin J'}\frac{1}{z-i\omega'}\Pi\left(\omega'\right)d\omega'\\
 & \eq{\ref{eq:def_Pi_omega},\ref{eq:def_Omega_2}}\int_{\Omega_{2}}\frac{1}{z-i\omega\left(\rho'\right)}\Pi\left(\rho'\right)\frac{d\rho'}{\left(2\pi\right)^{n+1}}\\
 & \eq{\ref{eq:def_operator}}\mathrm{Op}\left(\frac{1}{z-i\omega\left(.\right)}\mathrm{1}_{\Omega_{2}}\right)
\end{align*}

Since $\max_{\omega\notin J'}\left|\frac{1}{z-i\omega}\right|\leq c^{-1}$,
we deduce from Theorem \ref{thm:Continuity-of-PDO} that 
\begin{equation}
\left\Vert R_{r}^{\left(2\right)}\left(z\right)\right\Vert _{\mathcal{H}_{W}}\le Cc^{-1},\label{eq:norm_B}
\end{equation}
where $C$ does not depend on $\omega,z$.
\begin{lem}
\label{lem:bound_norm_D}We have
\begin{equation}
\left\Vert \left(z-A'\right)R_{r}^{\left(2\right)}\left(z\right)-\mathrm{Op}\left(\chi_{\Omega_{2}}\right)\right\Vert _{\mathcal{H}_{W}}^{2}\le Cc^{-1}.\label{eq:norm_D}
\end{equation}
where the constant $C$ depends on $\Lambda$ and $V$.
\end{lem}

Then taking $c$ large enough proves (\ref{eq:approximate_R_j}) for
$j=2$.
\begin{proof}
We write first
\begin{align}
\left(z-A'\right)R_{r}^{\left(2\right)}\left(z\right)-\mathrm{Op}\left(\chi_{\Omega_{2}}\right) & =\int_{\Omega_{2}}\frac{z-A'}{z-i\omega\left(\rho\right)}\Pi\left(\rho\right)\frac{d\rho}{\left(2\pi\right)^{n+1}}-\int_{\Omega_{2}}\Pi\left(\rho\right)\frac{d\rho}{\left(2\pi\right)^{n+1}}\nonumber \\
 & =\int_{\Omega_{2}}\frac{z-A'-\left(z-i\omega\left(\rho\right)\right)}{z-i\omega\left(\rho\right)}\Pi\left(\rho\right)\frac{d\rho}{\left(2\pi\right)^{n+1}}\nonumber \\
 & =-\int_{\Omega_{2}}\frac{A'-i\omega\left(\rho\right)}{z-i\omega\left(\rho\right)}\Pi\left(\rho\right)\frac{d\rho}{\left(2\pi\right)^{n+1}}\nonumber \\
 & =-\int_{\omega'\notin J'}\frac{A'-i\omega'}{z-i\omega'}\Pi\left(\omega'\right)d\omega'.\label{eq:A'}
\end{align}
We have
\[
A'\eq{\ref{eq:def_A'}}A-\Lambda\cdot\mathrm{Op}\left(\varpi\right)\eq{\ref{eq:generator_A}}-X+V-\Lambda\cdot\mathrm{Op}\left(\varpi\right).
\]
We have seen in (\ref{eq:symbol_X}) that $-X=\mathrm{Op}\left(i\omega\right)+R$
with a remainder $R\in\Psi\left(\langle\left|\rho\right|\rangle^{\alpha^{\parallel}}\right)$,
here $\alpha^{\parallel}=0$. Then $R=A'-\mathrm{Op}\left(i\omega\right)\in\Psi\left(1\right)$
as well, since $V,\Lambda\varpi$ are bounded. We write
\[
\int_{\omega'\notin J'}\frac{A'-i\omega'}{z-i\omega'}\Pi\left(\omega'\right)d\omega'=\left(\mathrm{Op}\left(i\omega\left(.\right)\right)+R\right)\mathrm{Op}\left(\frac{1}{z-i\omega\left(.\right)}\mathrm{1}_{\Omega_{2}}\right)+\mathrm{Op}\left(\frac{-i\omega\left(.\right)}{z-i\omega\left(.\right)}\mathrm{1}_{\Omega_{2}}\right).
\]
By the composition theorem \ref{thm:Composition-of-PDO.}, using that
$\omega\left(.\right)$ is slowly varying with function $h\left(\rho\right)=\left(\delta^{\parallel}\left(\rho\right)\right)^{-1}=1$,
and since $\max_{\omega\notin J'}\left|\frac{1}{z-i\omega}\right|\leq c^{-1}$,
we have
\[
\mathrm{Op}\left(i\omega\left(.\right)\right)\mathrm{Op}\left(\frac{1}{z-i\omega\left(.\right)}\mathrm{1}_{\Omega_{2}}\right)=\mathrm{Op}\left(\frac{i\omega\left(.\right)}{z-i\omega\left(.\right)}\mathrm{1}_{\Omega_{2}}\right)+\Psi\left(r_{2}\right)
\]
with $r_{2}=h\left(\rho\right)c^{-1}=c^{-1}.$ Also
\[
R\mathrm{Op}\left(\frac{1}{z-i\omega\left(.\right)}\mathrm{1}_{\Omega_{2}}\right)\in\Psi\left(c^{-1}\right),
\]
hence
\[
\int_{\omega'\notin J'}\frac{A'-i\omega'}{z-i\omega'}\Pi\left(\omega'\right)d\omega'\in\Psi\left(c^{-1}\right).
\]
From Theorem \ref{thm:Continuity-of-PDO} we deduce that $\left\Vert \int_{\omega'\notin J'}\frac{A'-i\omega'}{z-i\omega'}\Pi\left(\omega'\right)d\omega'\right\Vert _{\mathcal{H}_{W}}\leq Cc^{-1}$.
With (\ref{eq:A'}), we deduce (\ref{eq:norm_D}).
\end{proof}

\section{\label{sec:Proof-of-Theorem_WF}Proof of Theorem \ref{thm:WF} and
Corollary \ref{cor:Let-us-choose} about the wave front set}

\subsection{Proof of Theorem \ref{thm:WF}}

Let $g$ be an admissible metric\emph{. }$\mathcal{T}:C^{\infty}\left(M;\mathbb{C}\right)\rightarrow\mathcal{S}\left(T^{*}M;\mathbb{C}^{J}\right)$
denotes the wave-packet transform \eqref{eq:def_T_in-1} constructed
from that metric $g$. Let $W$ be an escape function as defined in
Lemma \ref{thm:W} and $\mathcal{H}_{W}\left(M\right)$ be the anisotropic
Sobolev space defined from $W$ in (\ref{eq:def_H_W}).

Consider a discrete eigenvalue $z_{0}\in\mathbb{C}$ of the generator
$A=-X+V$ with $\text{Re}(z_{0})>C_{X,V}-\Lambda$ and we set $\omega_{0}=\mathrm{Im}\left(z_{0}\right)$.
Assume that $u\in\mathcal{H}_{W}\left(M\right)$ is a generalized
eigenvector for the eigenvalue $z_{0}$, that is to say
\begin{equation}
\left(A-z_{0}\right)^{k}u=0\label{eq:Aku}
\end{equation}
for some $k\in\mathbb{N}$, $k\geq1$. We assume $\mathrm{Re}z_{0}>a$,
for some $a\in\mathbb{R}$. The estimates will be uniform in $\omega_{0}$
but non uniform in $a$.

Let us define a function $f\in C^{\infty}\left(T^{*}M;\mathbb{C}\right)$
\[
f\left(\rho\right)=\left(i\left(\omega-\omega_{0}\right)\right)^{-k},
\]
with $\omega=\boldsymbol{\omega}\left(\rho\right)$, when $\left|\omega-\omega_{0}\right|\geq1$
and bounded by $1$ when $\left|\omega-\omega_{0}\right|\leq1$. Let
$R$ be the operator defined by
\begin{equation}
R:=\mathrm{Id}-\mathrm{Op}\left(f\right)\left(A-z_{0}\right)^{k}.\label{eq:def_R-2}
\end{equation}
Using the notation of Definition \ref{def:Let--be-1} for $\Psi$,
we have the following estimate.
\begin{lem}
For any $m>0$, 
\begin{equation}
R^{m}\in\Psi\left(\left\langle \boldsymbol{\omega}-\omega_{0}\right\rangle ^{-m}\left\langle \left|\rho\right|\right\rangle ^{m\alpha^{\parallel}}\right).\label{eq:Rm}
\end{equation}
\end{lem}

\begin{proof}
Recall from (\ref{eq:symbol_X}) that
\[
\left(A-\mathrm{Op}\left(i\boldsymbol{\omega}\right)\right)\in\Psi\left(\left\langle \left|\rho\right|\right\rangle ^{\alpha^{\parallel}}\right)
\]
Also uniformly in $\omega_{0}$ we have that
\[
\left(\left(A-z_{0}\right)^{k}-\mathrm{Op}\left(\left(i\boldsymbol{\omega}-z_{0}\right)^{k}\right)\right)\in\Psi\left(\left|\boldsymbol{\omega}-\omega_{0}\right|^{k-1}\left\langle \left|\rho\right|\right\rangle ^{\alpha^{\parallel}}\right)
\]
\[
\left(\left(A-z_{0}\right)^{k}-\mathrm{Op}\left(f^{-1}\right)\right)\in\Psi\left(\left|\boldsymbol{\omega}-\omega_{0}\right|^{k-1}\left\langle \left|\rho\right|\right\rangle ^{\alpha^{\parallel}}\right)
\]
From composition theorem:
\[
\mathrm{Op}\left(f\right)\mathrm{Op}\left(f^{-1}\right)-\mathrm{Id}\in\Psi\left(\left|\boldsymbol{\omega}-\omega_{0}\right|^{-1}\left\langle \left|\rho\right|\right\rangle ^{\alpha^{\parallel}}\right)
\]
Thus
\begin{align*}
R & \eq{\ref{eq:def_R-2}}\mathrm{Id}-\mathrm{Op}\left(f\right)\mathrm{Op}\left(f^{-1}\right)+\mathrm{Op}\left(f\right)\left(\mathrm{Op}\left(f^{-1}\right)-\left(A-z_{0}\right)^{k}\right)\\
 & \in\Psi\left(\left|\boldsymbol{\omega}-\omega_{0}\right|^{-1}\left\langle \left|\rho\right|\right\rangle ^{\alpha^{\parallel}}\right).
\end{align*}
and (\ref{eq:Rm}) follows.
\end{proof}
We have $Ru\underset{(\ref{eq:Aku},\ref{eq:def_R-2})}{=}u$ hence
for any $m\geq1$,
\begin{equation}
\left(\mathcal{T}u\right)\left(\rho'\right)=\left(\mathcal{T}R^{m}\mathcal{T}^{\dagger}\mathcal{T}u\right)\left(\rho'\right)=\int\langle\delta_{\rho',j'}|\mathcal{T}R^{m}\mathcal{T}^{\dagger}\delta_{\rho,j}\rangle\left(\mathcal{T}u\right)\left(\rho\right)\frac{d\rho}{\left(2\pi\right)^{n+1}}\label{eq:Tu}
\end{equation}
Therefore for $\rho\in T^{*}M$,
\begin{align*}
W\left(\rho'\right)\left|\left(\mathcal{T}u\right)\left(\rho'\right)\right| & \ineq{\ref{eq:Tu}}W\left(\rho'\right)\sum_{j,j'}\int\left|\langle\delta_{\rho',j'}|\mathcal{T}R^{m}\mathcal{T}^{\dagger}\delta_{\rho,j}\rangle\right|\left|\left(\mathcal{T}u\right)\left(\rho\right)\right|\frac{d\rho}{\left(2\pi\right)^{n+1}}\\
 & \underset{(\ref{eq:Rm},\ref{eq:microl_estimate-2})}{\leq}W\left(\rho'\right)\frac{C_{N,m}}{\left\langle \omega'-\omega_{0}\right\rangle ^{m}}\left\langle \left|\rho'\right|\right\rangle ^{m\alpha^{\parallel}}\int\left\langle \mathrm{dist}_{g}\left(\rho',\rho\right)\right\rangle ^{-N}\left|\left(\mathcal{T}u\right)\left(\rho\right)\right|\frac{d\rho}{\left(2\pi\right)^{n+1}}\\
 & \ineq{\ref{eq:slow_variation_W}}\frac{C_{N,m}\left\langle \left|\rho'\right|\right\rangle ^{m\alpha^{\parallel}}}{\left\langle \omega'-\omega_{0}\right\rangle ^{m}}\int\left\langle \mathrm{dist}_{g}\left(\rho',\rho\right)\right\rangle ^{-N}W\left(\rho\right)\left|\left(\mathcal{T}u\right)\left(\rho\right)\right|\frac{d\rho}{\left(2\pi\right)^{n+1}}\\
 & \ineq{\mathrm{C.S.}}\frac{C_{N,m}\left\langle \left|\rho'\right|\right\rangle ^{m\alpha^{\parallel}}}{\left\langle \omega'-\omega_{0}\right\rangle ^{m}}\left(\int\left\langle \mathrm{dist}_{g}\left(\rho',\rho\right)\right\rangle ^{-2N}\frac{d\rho}{\left(2\pi\right)^{n+1}}\right)^{1/2}\\
 & \qquad\left(\int\left(W\left(\rho\right)\right)^{2}\left|\left(\mathcal{T}u\right)\left(\rho\right)\right|^{2}\frac{d\rho}{\left(2\pi\right)^{n+1}}\right)^{1/2}\\
 & \underset{(\ref{eq:def_Sobolev_norm})}{\leq}\frac{C_{m}\left\langle \left|\rho'\right|\right\rangle ^{m\alpha^{\parallel}}}{\left\langle \omega'-\omega_{0}\right\rangle ^{m}}\left\Vert u\right\Vert _{\mathcal{H}_{W}\left(M\right)}
\end{align*}
or
\begin{equation}
\left|\left(\mathcal{T}u\right)\left(\rho'\right)\right|\leq\frac{C_{m}\left\langle \left|\rho'\right|\right\rangle ^{m\alpha^{\parallel}}}{W\left(\rho'\right)\left\langle \omega'-\omega_{0}\right\rangle ^{m}}\left\Vert u\right\Vert _{\mathcal{H}_{W}\left(M\right)}.\label{eq:res-1-1}
\end{equation}
We have obtained the claim (\ref{eq:WF-1-1}) of the theorem.

\subsection{\label{subsec:Proof-of-Corollary}Proof of Corollary \ref{cor:Let-us-choose}}

We choose a phase space metric $g$ as defined in (\ref{eq:metric_g_in_coordinates})
with the following parameters
\[
\alpha^{\perp}=\frac{1}{1+\min\left(\beta_{u},\beta_{s}\right)},\quad\alpha^{\parallel}=0.
\]
Let $C>0$. We choose a weight function $W$ as in Lemma \ref{thm:W},
with parameter $R_{s},R_{u}>0$ large enough to reveal the discrete
spectrum on $\mathrm{Re}\left(z\right)\geq-C$ and $R_{s}$ will be
taken larger later. We choose parameters $\gamma=0$ . For $\rho\in T^{*}M$,
we write $\rho=\omega\mathscr{A}+\rho_{s}+\rho_{u}$ with $\omega\in\mathbb{R}$,
$\rho_{u}\in E_{u}^{*}$, $\rho_{s}\in E_{s}^{*}$. Then
\begin{align*}
W\left(\rho\right)\left\langle \omega-\omega_{0}\right\rangle ^{m} & \underset{(\ref{eq:def_W1})}{=}\frac{\left\langle h_{0}\left\Vert \rho_{s}\right\Vert _{g_{\rho}}\right\rangle ^{R_{s}}}{\left\langle h_{0}\left\Vert \rho_{u}\right\Vert _{g_{\rho}}\right\rangle ^{R_{u}}}\left\langle \omega-\omega_{0}\right\rangle ^{m}\underset{(\ref{eq:def_norm_g})}{=}\frac{\left\langle h_{0}\left|\rho\right|^{-\alpha^{\perp}}\left|\rho_{s}\right|\right\rangle ^{R_{s}}}{\left\langle h_{0}\left|\rho\right|^{-\alpha^{\perp}}\left|\rho_{u}\right|\right\rangle ^{R_{u}}}\left\langle \omega-\omega_{0}\right\rangle ^{m}\\
 & \geq\left\langle \left|\rho\right|\right\rangle ^{-R_{u}\left(1-\alpha^{\perp}\right)}\left\langle h_{0}\left|\rho\right|^{-\alpha^{\perp}}\left|\rho_{s}\right|\right\rangle ^{R_{s}}\left\langle \omega-\omega_{0}\right\rangle ^{m}.
\end{align*}
Let $\epsilon>0$ and assume $\rho\in T^{*}M\backslash\mathcal{V}_{\omega_{0},\epsilon}$
where $\mathcal{V}_{\omega_{0},\epsilon}$ is defined in (\ref{eq:Vu}).
This means that
\[
\left\langle \omega-\omega_{0}\right\rangle \geq\left|\rho\right|^{\epsilon}\quad\text{ or }\quad\left\langle \left|\rho\right|^{-\alpha^{\perp}}\left|\rho_{s}\right|\right\rangle \geq\left|\rho\right|^{\epsilon}.
\]
Let $N\geq1$. If $\left\langle \omega-\omega_{0}\right\rangle \geq\left|\rho\right|^{\epsilon}$,
we choose $m$ large enough so that $m\epsilon-R_{u}\left(1-\alpha^{\perp}\right)\geq N$
and hence that
\[
W\left(\rho\right)\left\langle \omega-\omega_{0}\right\rangle ^{m}\geq\left\langle \left|\rho\right|\right\rangle ^{-R_{u}\left(1-\alpha^{\perp}\right)}\left\langle \omega-\omega_{0}\right\rangle ^{m}\geq C_{N,\epsilon}\left\langle \left|\rho\right|\right\rangle ^{m\epsilon-R_{u}\left(1-\alpha^{\perp}\right)}\geq C_{N,\epsilon}\left\langle \rho\right\rangle ^{N}.
\]
 If $\left\langle \left|\rho\right|^{-\alpha^{\perp}}\left|\rho_{s}\right|\right\rangle \geq\left|\rho\right|^{\epsilon}$,
we choose $R_{s}$ large enough so that $R_{s}\epsilon-R_{u}\left(1-\alpha^{\perp}\right)\geq N$
and hence that
\[
W\left(\rho\right)\left\langle \omega-\omega_{0}\right\rangle ^{m}\geq\left\langle \left|\rho\right|\right\rangle ^{-R_{u}\left(1-\alpha^{\perp}\right)}\left\langle \left|\rho\right|^{-\alpha^{\perp}}\left|\rho_{s}\right|\right\rangle ^{R_{s}}\geq C_{N,\epsilon}\left\langle \left|\rho\right|\right\rangle ^{R_{s}\epsilon-R_{u}\left(1-\alpha^{\perp}\right)}\geq C_{N,\epsilon}\left\langle \rho\right\rangle ^{N}.
\]
 In both cases, we obtain that
\[
\left|\left(\mathcal{T}u\right)\left(\rho\right)\right|\underset{(\ref{eq:res-1-1})}{\leq}\frac{C_{m}}{W\left(\rho\right)\left\langle \omega-\omega_{0}\right\rangle ^{m}}\left\Vert u\right\Vert _{\mathcal{H}_{W}\left(M\right)}\leq\frac{C_{N,\epsilon}}{\left\langle \rho\right\rangle ^{N}}\left\Vert u\right\Vert _{\mathcal{H}_{W}\left(M\right)}.
\]
This finishes the proof of Corollary \ref{cor:Let-us-choose}.

\appendix

\section{\label{subsec:Proof-of-Thm}Proof of Theorem \ref{thm:W} about properties
of $W$}

\subsection{Definition of the escape function $W$}

The escape function $W:T^{*}M\rightarrow\mathbb{R}^{+}$ has been
defined in (\ref{eq:def_W1}). Here, we will first give an other equivalent
definition of $W$ in (\ref{eq:def_W}) that will be more convenient
for the proof of Theorem \ref{thm:W}. 

\subparagraph{Definition of map $\Phi_{j}$:}

For a point $m\in M$, we have defined the decomposition of a cotangent
vector $\rho\eq{\ref{eq:decomp_Xi}}\rho_{u}+\rho_{s}+\rho_{0}\in T_{m}^{*}M$,
where $\rho_{u}\in E_{u}^{*}\left(m\right),\rho_{s}\in E_{s}^{*}\left(m\right),\rho_{0}\in E_{0}^{*}\left(m\right)$.
Recall from (\ref{eq:def_E_u*}) and (\ref{eq:Holder_exp}) that the
map $m\in M\mapsto E_{\sigma}^{*}\left(m\right)$, $\sigma=u,s,0$,
are Hölder continuous with respective exponent $\beta_{u},\beta_{s},\beta_{0}$.
Consequently the decomposition $\rho=\rho_{u}+\rho_{s}+\rho_{0}$
is continuous but not smooth.

We have taken a global smooth metric $g_{M}$ on $M$ and we denote
$\left\Vert .\right\Vert _{g_{M}}$ the induced norm on $T^{*}M$.
From $g_{M}$ we define a new metric $\tilde{g}_{M}$ on $M$ as follows.
The metric $\tilde{g}_{M}$ equals $g_{M}$ on $E_{\sigma}^{*}$ for
$\sigma=u,s,0$ and the sum $E_{u}^{*}\oplus E_{s}^{*}\oplus E_{0}^{*}$
is orthogonal for $\tilde{g}_{M}$. As a consequence $\tilde{g}_{M}$
is Hölder continuous on $M$.

Let $\kappa_{j}:U_{j}\subset M\rightarrow V_{j}\subset\mathbb{R}_{x,z}^{n+1}$
a local chart diffeomorphism on $M$ as defined in (\ref{eq:def_kappa_j})
and $\tilde{\kappa}_{j}:T^{*}U_{j}\subset T^{*}M\rightarrow T^{*}V_{j}\subset T^{*}\mathbb{R}^{n+1}$
the lifted map on cotangent space defined in (\ref{eq:def_k_tilde_j}).
There exists a continuous linear bundle map over $\mathrm{Id}$ of
the form 
\begin{equation}
\Phi_{j}:\begin{cases}
T^{*}V_{j} & \rightarrow T^{*}\left(\mathbb{R}^{n}\times\mathbb{R}\right)\\
\left(\left(x,z\right),\left(\xi,\omega\right)\right) & \mapsto\left(\left(x,z\right),\left(\tilde{\xi},\omega\right)\right)
\end{cases}\label{eq:def_Phi}
\end{equation}
with
\begin{equation}
\tilde{\xi}=A\left(x\right)\xi+\omega B\left(x\right)\in\mathbb{R}^{n}\label{eq:xi_tilde}
\end{equation}
such that:
\begin{enumerate}
\item $\Phi_{j}\circ\tilde{\kappa}_{j}$ maps $E_{u}^{*}\oplus E_{s}^{*}\oplus E_{0}^{*}$
to $\mathbb{R}^{d_{u}}\oplus\mathbb{R}^{d_{s}}\oplus\mathbb{R}$,
with $d_{u}+d_{s}=n$. For this, $A\left(x\right):\mathbb{R}^{n}\rightarrow\mathbb{R}^{n}$
is a linear invertible map that depends continuously on $x\in\mathbb{R}^{n}$
with Hölder exponent $\beta_{*}=\min\{\beta_{s},\beta_{u}\}$ and
$B\left(x\right)\in\mathbb{R}^{n}$ depends continuously on $x\in\mathbb{R}^{n}$
with Hölder exponent $\beta_{0}$.
\item For any $\rho=\rho_{u}+\rho_{s}+\rho_{0}\in E_{u}^{*}\oplus E_{s}^{*}\oplus E_{0}^{*}=T^{*}M$,
\begin{equation}
\left|\left(\Phi_{j}\circ\tilde{\kappa}_{j}\right)\left(\rho_{\sigma}\right)\right|=\left\Vert \rho_{\sigma}\right\Vert _{\tilde{g}_{M}},\quad\text{for }\sigma=u,s,0,\label{eq:prop_Phi}
\end{equation}
where $\left|.\right|$ denotes the Euclidean canonical norm on $\mathbb{R}^{n+1}$.
\end{enumerate}
From $\tilde{g}_{M}$-orthogonality of the decomposition $E_{u}^{*}\oplus E_{s}^{*}\oplus E_{0}^{*}$
and orthogonality of $\mathbb{R}^{d_{u}}\oplus\mathbb{R}^{d_{s}}\oplus\mathbb{R}$
for the Euclidean metric, we have that
\begin{equation}
\forall\rho\in T^{*}M,\quad\left|\left(\Phi_{j}\circ\tilde{\kappa}_{j}\right)\left(\rho\right)\right|=\left\Vert \rho\right\Vert _{\tilde{g}_{M}}.\label{eq:prop_Phi_1}
\end{equation}

\subparagraph{Definition of the function $\tilde{W}$:}

We define the metric $\tilde{g}$ on $\mathbb{R}^{n+1}$ as follows.
Let $\tilde{\eta}=\left(\tilde{\xi},\omega\right)\in\mathbb{R}^{n}\times\mathbb{R}$.
For $\tilde{v}=\left(\tilde{v}_{\xi},\tilde{v}_{\omega}\right)\in T_{\tilde{\eta}}\left(\mathbb{R}^{n}\times\mathbb{R}\right)$,
\begin{equation}
\left\Vert \tilde{v}\right\Vert _{\tilde{g}\left(\tilde{\eta}\right)}^{2}=\left(\left\langle \left|\tilde{\eta}\right|\right\rangle ^{-\alpha^{\perp}}\left|\tilde{v}_{\xi}\right|\right)^{2}+\left(\left\langle \left|\tilde{\eta}\right|\right\rangle ^{-\alpha^{\parallel}}\left|\tilde{v}_{\omega}\right|\right)^{2}.\label{eq:norm_g}
\end{equation}
Let $0<\gamma<1$ . We define the function $\tilde{h}_{\gamma}^{\perp}:\mathbb{R}^{n+1}\rightarrow\mathbb{R}^{+}$
as follows.
\begin{equation}
\tilde{h}_{\gamma}^{\perp}\left(\tilde{\eta}\right)=\left\langle \left\Vert \tilde{\xi}\right\Vert _{\tilde{g}\left(\tilde{\eta}\right)}\right\rangle ^{-\gamma}\underset{(\ref{eq:norm_g})}{=}\left\langle \left\langle \left|\tilde{\eta}\right|\right\rangle ^{-\alpha^{\perp}}\left|\tilde{\xi}\right|\right\rangle ^{-\gamma}.\label{eq:h_tilde}
\end{equation}
We define the function $\tilde{W}:\mathbb{R}^{n+1}\rightarrow\mathbb{R}^{+}$
as follows. With $\tilde{\xi}=\left(\tilde{\xi}{}_{u},\tilde{\xi}{}_{s}\right)\in\mathbb{R}^{d_{u}}\times\mathbb{R}^{d_{s}}=\mathbb{R}^{n}$,
\begin{equation}
\tilde{W}\left(\tilde{\eta}\right):=\frac{\left\langle h_{0}\tilde{h}_{\gamma}^{\perp}\left(\tilde{\eta}\right)\left\Vert \tilde{\xi}{}_{s}\right\Vert _{\tilde{g}\left(\tilde{\eta}\right)}\right\rangle ^{R_{s}}}{\left\langle h_{0}\tilde{h}_{\gamma}^{\perp}\left(\tilde{\eta}\right)\left\Vert \tilde{\xi}{}_{u}\right\Vert _{\tilde{g}\left(\tilde{\eta}\right)}\right\rangle ^{R_{u}}}.\label{eq:def_W_tilde}
\end{equation}

\begin{rem}
Let us observe from (\ref{eq:prop_Phi_1}) that for $\rho\in T^{*}U_{j}\cap T^{*}U_{j'}$
we have $\tilde{W}\circ\Phi_{j}\circ\tilde{\kappa}_{j}\left(\rho\right)=\tilde{W}\circ\Phi_{j'}\circ\tilde{\kappa}_{j'}\left(\rho\right)$,
i.e. this expression is independent on chart.
\end{rem}

\subparagraph{The global function $W$:}

~

\begin{cBoxB}{}
\begin{lem}[Alternative expression of the escape function $W$]
The function $W:T^{*}U_{j}\subset T^{*}M\rightarrow\mathbb{R}^{+}$
defined in (\ref{eq:def_W1}) can be expressed using the function
$\tilde{W}$ in (\ref{eq:def_W_tilde}) as follows:
\begin{equation}
W=\left(\boldsymbol{1}_{y}\otimes\tilde{W}\right)\circ\Phi_{j}\circ\tilde{\kappa}_{j},\label{eq:def_W}
\end{equation}
where $\boldsymbol{1}_{y}:y=\left(x,z\right)\in\mathbb{R}^{n+1}\rightarrow1\in\mathbb{R}$
is the constant unit function on the variables of position.
\end{lem}

\end{cBoxB}

\begin{proof}
We check that $W$ in (\ref{eq:def_W}) coincides with the function
$W$ given in (\ref{eq:def_W1}). Essentially this is due to Definition
(\ref{eq:norm_g-1}) that we have used for the metric $\tilde{g}\left(\rho\right)$.
\end{proof}

\subsection{The slowly varying and temperate property \eqref{enu:-satisfies--temperate-1}}

Here, we will prove that the function $W$ satisfies (\ref{eq:slow_variation_W}).
For this, we use the expression (\ref{eq:def_W}) that expresses the
function $W$ as a smooth function $\boldsymbol{1}_{y}\otimes\tilde{W}$
constant on the base, composed with the function $\Phi_{j}$. In the
first step, we will prove that the smooth function $\tilde{W}$ satisfies
the slowly varying and temperate property. In a second step, we will
show that $\Phi_{j}$ has the property of being Lipschitz at scale
greater than $1$. We will then deduce that the slowly varying and
temperate property holds true for $W$.

\subsubsection{The slowly varying and temperate property for $\tilde{W}$}

The next proposition gives the temperate property and slow variation
property of $\tilde{W}$ defined in (\ref{eq:def_W_tilde}), that
corresponds to the claim \ref{enu:-satisfies--temperate-1} of Theorem
\ref{thm:W} (but for $\tilde{W}$).

\begin{cBoxB}{}
\begin{prop}
Assume that $\alpha^{\perp},\alpha^{\parallel}$ and $0\leq\gamma<1$
are given and satisfy the relations (\ref{eq:ineq1}) and (\ref{eq:gamma_interval}).
Then there exist constants $N_{\tilde{W}}>0$, $C_{\tilde{W}}$ and
$0<\mu<1$ such that for every $h_{0}>0$,
\begin{align}
\frac{\tilde{W}(\tilde{\eta}')}{\tilde{W}(\tilde{\eta})}\le1+C_{\tilde{W}}h_{0}^{\mu}\langle\tilde{h}_{\gamma}^{\perp}\left(\tilde{\eta}\right)\|\tilde{\eta}'-\tilde{\eta}\|_{\tilde{g}\left(\tilde{\eta}\right)}\rangle^{N_{\tilde{W}}}\qquad\text{for all }\tilde{\eta},\tilde{\eta}'\in\mathbb{R}^{n+1}.\label{eq:w2}
\end{align}
\end{prop}

\end{cBoxB}

\begin{proof}
For simplicity of notation, we define the conformal metric $\tilde{g}_{\gamma}$
on $T\mathbb{R}^{n+1}$ as follows. We write for $\tilde{v}\in T_{\tilde{\eta}}\mathbb{R}^{n+1}$
\begin{equation}
\left\Vert \tilde{v}\right\Vert _{\tilde{g}_{\gamma}\left(\tilde{\eta}\right)}:=\tilde{h}_{\gamma}^{\perp}\left(\tilde{\eta}\right)\left\Vert \tilde{v}\right\Vert _{\tilde{g}\left(\tilde{\eta}\right)}.\label{eq:def_gtilde}
\end{equation}
Consider expression (\ref{eq:def_W_tilde}) of $\tilde{W}$. For the
proof of the proposition, it is enough to show the estimates
\begin{align}
\max\left\{ \frac{\langle h_{0}\|\tilde{\xi}'_{s}\|_{\tilde{g}_{\gamma}\left(\tilde{\eta}'\right)}\rangle}{\langle h_{0}\|\tilde{\xi}_{s}\|_{\tilde{g}_{\gamma}\left(\tilde{\eta}\right)}\rangle},\frac{\langle h_{0}\|\tilde{\xi}_{s}\|_{\tilde{g}_{\gamma}\left(\tilde{\eta}\right)}\rangle}{\langle h_{0}\|\tilde{\xi}'_{s}\|_{\tilde{g}_{\gamma}\left(\tilde{\eta}'\right)}\rangle}\right\} \le1+Ch_{0}^{\mu}\langle\|\tilde{\eta}'-\tilde{\eta}\|_{\tilde{g}_{\gamma}\left(\tilde{\eta}\right)}\rangle^{N}\label{eq:reducedclaim1}
\end{align}
and the corresponding claims with $\tilde{\xi}_{s}$ and $\tilde{\xi}'_{s}$
replaced by $\tilde{\xi}_{u}$ and $\tilde{\xi}'_{u}$ respectively.
We give a proof of the claim \eqref{eq:reducedclaim1}. The proof
of the other claim is parallel. We will choose the small constant
$\mu>0$ in the due course and the choice may be implicit in the following
argument.

Note that, if $\|\tilde{\xi}\|_{\tilde{g}_{\gamma}\left(\tilde{\eta}\right)}\le h_{0}^{-\delta}$
and $\|\tilde{\xi}'\|_{\tilde{g}_{\gamma}\left(\tilde{\eta}'\right)}\le h_{0}^{-\delta}$
with some $0<\delta<1$, then 
\[
h_{0}\|\tilde{\xi}_{s}\|_{\tilde{g}_{\gamma}\left(\tilde{\eta}\right)}\le h_{0}^{1-\delta}<1
\]
and the same estimate for $\|\tilde{\xi}'_{s}\|_{\tilde{g}_{\gamma}\left(\tilde{\eta}\right)}$.
Then the required estimate (\ref{eq:reducedclaim1}) is trivial because
the function $x\mapsto\langle x\rangle$ is almost flat around the
origin $0$. Therefore we may assume either $\|\tilde{\xi}\|_{\tilde{g}_{\gamma}\left(\tilde{\eta}\right)}>h_{0}^{-\delta}$
or $\|\tilde{\xi}'\|_{\tilde{g}_{\gamma}\left(\tilde{\eta}'\right)}>h_{0}^{-\delta}$.
Below we assume 
\begin{equation}
\|\tilde{\xi}\|_{\tilde{g}_{\gamma}\left(\tilde{\eta}\right)}>h_{0}^{-\delta}.\label{eq:hdelta}
\end{equation}
As we will explain at the end, the claim in the other case follows
immediately if we are done with this case. Note that 
\begin{equation}
|\tilde{\eta}|\ge\|\tilde{\eta}\|_{\tilde{g}\left(\tilde{\eta}\right)}\ge\|\tilde{\xi}\|_{\tilde{g}\left(\tilde{\eta}\right)}\ge\|\tilde{\xi}\|_{\tilde{g}_{\gamma}\left(\tilde{\eta}\right)}\underset{(\ref{eq:hdelta})}{>}h_{0}^{-\delta}.\label{eq:eta}
\end{equation}
We have 
\begin{equation}
\max\left\{ \frac{\langle h_{0}\|\tilde{\xi}'_{s}\|_{\tilde{g}_{\gamma}\left(\tilde{\eta}\right)}\rangle}{\langle h_{0}\|\tilde{\xi}_{s}\|_{\tilde{g}_{\gamma}\left(\tilde{\eta}\right)}\rangle},\frac{\langle h_{0}\|\tilde{\xi}_{s}\|_{\tilde{g}_{\gamma}\left(\tilde{\eta}\right)}\rangle}{\langle h_{0}\|\tilde{\xi}'_{s}\|_{\tilde{g}_{\gamma}\left(\tilde{\eta}\right)}\rangle}\right\} \ineq{\ref{eq:D1p}}1+h_{0}\left|\|\tilde{\xi}'_{s}\|_{\tilde{g}_{\gamma}\left(\tilde{\eta}\right)}-\|\tilde{\xi}{}_{s}\|_{\tilde{g}_{\gamma}\left(\tilde{\eta}\right)}\right|\le1+h_{0}\|\tilde{\xi}'_{s}-\tilde{\xi}_{s}\|_{\tilde{g}_{\gamma}\left(\tilde{\eta}\right)}.\label{eq:basic}
\end{equation}
Hence, for the proof of the claim \eqref{eq:reducedclaim1}, we compare
the metric $\tilde{g}_{\gamma}\left(\tilde{\eta}\right)$ and $\tilde{g}_{\gamma}\left(\tilde{\eta}'\right)$.

Recall from Lemma \ref{lem:temperate_metric} that the metric $g$
hence $\tilde{g}$ also, is $\Delta^{\gamma}$-moderate that is, with
the distortion function $\Delta(\tilde{\eta})\eq{\ref{eq:def_distortion_function_D-1}}\langle|\tilde{\eta}|\rangle^{-(1-\alpha^{\perp})}$,
we have that, for arbitrary $0\neq v\in\real^{n+1}$, 
\begin{equation}
\max\left\{ \frac{\|v\|_{\tilde{g}\left(\tilde{\eta}'\right)}}{\|v\|_{\tilde{g}\left(\tilde{\eta}\right)}},\frac{\|v\|_{\tilde{g}\left(\tilde{\eta}\right)}}{\|v\|_{\tilde{g}\left(\tilde{\eta}'\right)}}\right\} \le1+C\Delta(\tilde{\eta})^{1-\gamma}\left\langle \Delta(\tilde{\eta})^{\gamma}\cdot\|\tilde{\eta}'-\tilde{\eta}\|_{\tilde{g}\left(\tilde{\eta}\right)}\right\rangle ^{N}.\label{claim1}
\end{equation}
For $0<\lambda<1$, $x\in\mathbb{R}$, we have
\begin{equation}
\lambda\left\langle x\right\rangle \leq\left\langle \lambda x\right\rangle .\label{eq:ineq-1}
\end{equation}
We have
\begin{equation}
\tilde{h}_{\gamma}^{\perp}(\tilde{\eta})\underset{(\ref{eq:h_tilde})}{\geq}\langle\langle|\tilde{\eta}|\rangle^{-\alpha^{\perp}}|\tilde{\eta}|\rangle^{-\gamma}\underset{(\ref{eq:ineq-1})}{\geq}\langle|\tilde{\eta}|\rangle^{-\gamma\left(1-\alpha^{\perp}\right)}\eq{\ref{eq:def_distortion_function_D-1}}\Delta(\tilde{\eta})^{\gamma},\label{eq:*3}
\end{equation}
and 
\begin{equation}
\Delta(\tilde{\eta})^{1-\gamma}=\langle|\tilde{\eta}|\rangle^{-\left(1-\gamma\right)\left(1-\alpha^{\perp}\right)}\ineq{\ref{eq:eta}}Ch_{0}^{\delta\left(1-\gamma\right)\left(1-\alpha^{\perp}\right)}=Ch_{0}^{\mu}\label{eq:*3b}
\end{equation}
with $\mu=\left(1-\gamma\right)\left(1-\alpha^{\perp}\right)<1$.
We get that, for arbitrary $0\neq v\in\real^{n+1}$, 
\begin{align}
\max\left\{ \frac{\|v\|_{\tilde{g}\left(\tilde{\eta}\right)}}{\|v\|_{\tilde{g}\left(\tilde{\eta}'\right)}},\frac{\|v\|_{\tilde{g}\left(\tilde{\eta}'\right)}}{\|v\|_{\tilde{g}\left(\tilde{\eta}\right)}}\right\} \underset{(\ref{claim1}),(\ref{eq:*3}),(\ref{eq:*3b})}{\leq} & 1+Ch_{0}^{\mu}\cdot\left\langle \|\tilde{\eta}'-\tilde{\eta}\|_{\tilde{g}_{\gamma}\left(\tilde{\eta}\right)}\right\rangle ^{N}.\label{claim3}
\end{align}

To get (\ref{eq:reducedclaim1}), we want to extend the estimate \eqref{claim3}
to the case where the metric $\tilde{g}$ on the left hand side is
replaced by the metric $\tilde{g}_{\gamma}$. To this end, we have
to estimate the ratio between $\tilde{h}_{\gamma}^{\perp}(\tilde{\eta})$
and $\tilde{h}_{\gamma}^{\perp}(\tilde{\eta}')$. We have 
\begin{align}
\max & \left\{ \frac{\tilde{h}_{\gamma}^{\perp}(\tilde{\eta}')}{\tilde{h}_{\gamma}^{\perp}(\tilde{\eta})},\frac{\tilde{h}_{\gamma}^{\perp}(\tilde{\eta})}{\tilde{h}_{\gamma}^{\perp}(\tilde{\eta}')}\right\} \eq{\ref{eq:h_tilde}}\max\left\{ \frac{\langle\|\tilde{\xi}'\|_{\tilde{g}\left(\tilde{\eta}'\right)}\rangle}{\langle\|\tilde{\xi}\|_{\tilde{g}\left(\tilde{\eta}\right)}\rangle},\frac{\langle\|\tilde{\xi}\|_{\tilde{g}\left(\tilde{\eta}\right)}\rangle}{\langle\|\tilde{\xi}'\|_{\tilde{g}\left(\tilde{\eta}'\right)}\rangle}\right\} ^{\gamma}\nonumber \\
 & \le\max\left\{ \frac{\langle\|\tilde{\xi}'\|_{\tilde{g}\left(\tilde{\eta}\right)}\rangle}{\langle\|\tilde{\xi}\|_{\tilde{g}\left(\tilde{\eta}\right)}\rangle}\cdot\frac{\langle\|\tilde{\xi}'\|_{\tilde{g}\left(\tilde{\eta}'\right)}\rangle}{\langle\|\tilde{\xi}'\|_{\tilde{g}\left(\tilde{\eta}\right)}\rangle},\frac{\langle\|\tilde{\xi}\|_{\tilde{g}\left(\tilde{\eta}\right)}\rangle}{\langle\|\tilde{\xi}'\|_{\tilde{g}\left(\tilde{\eta}\right)}\rangle}\cdot\frac{\langle\|\tilde{\xi}'\|_{\tilde{g}\left(\tilde{\eta}\right)}\rangle}{\langle\|\tilde{\xi}'\|_{\tilde{g}\left(\tilde{\eta}'\right)}\rangle}\right\} ^{\gamma}\nonumber \\
 & \le\max\left\{ \frac{\langle\|\tilde{\xi}'\|_{\tilde{g}\left(\tilde{\eta}\right)}\rangle}{\langle\|\tilde{\xi}\|_{\tilde{g}\left(\tilde{\eta}\right)}\rangle},\frac{\langle\|\tilde{\xi}\|_{\tilde{g}\left(\tilde{\eta}\right)}\rangle}{\langle\|\tilde{\xi}'\|_{\tilde{g}\left(\tilde{\eta}\right)}\rangle}\right\} ^{\gamma}\cdot\max\left\{ \frac{\langle\|\tilde{\xi}'\|_{\tilde{g}\left(\tilde{\eta}'\right)}\rangle}{\langle\|\tilde{\xi}'\|_{\tilde{g}\left(\tilde{\eta}\right)}\rangle},\frac{\langle\|\tilde{\xi}'\|_{\tilde{g}\left(\tilde{\eta}\right)}\rangle}{\langle\|\tilde{\xi}'\|_{\tilde{g}\left(\tilde{\eta}'\right)}\rangle}\right\} ^{\gamma}.\label{eq:vg}
\end{align}
And, for the second term on the right-hand side of \eqref{eq:vg},
we apply \eqref{claim3} and see 
\begin{align}
\max\left\{ \frac{\langle\|\tilde{\xi}'\|_{\tilde{g}\left(\tilde{\eta}'\right)}\rangle}{\langle\|\tilde{\xi}'\|_{\tilde{g}\left(\tilde{\eta}\right)}\rangle},\frac{\langle\|\tilde{\xi}'\|_{\tilde{g}\left(\tilde{\eta}\right)}\rangle}{\langle\|\tilde{\xi}'\|_{\tilde{g}\left(\tilde{\eta}'\right)}\rangle}\right\}  & \le1+Ch_{0}^{\mu}\cdot\left\langle \|\tilde{\eta}'-\tilde{\eta}\|_{\tilde{g}_{\gamma}\left(\tilde{\eta}\right)}\right\rangle ^{N}.\label{hgamma1}
\end{align}
For the first term on the right-hand side of \eqref{eq:vg}, we consider
two cases separately:
\begin{enumerate}
\item In the case where $\|\tilde{\xi}'-\tilde{\xi}\|_{\tilde{g}\left(\tilde{\eta}\right)}\leq\frac{1}{2}\|\tilde{\xi}\|_{\tilde{g}\left(\tilde{\eta}\right)}$
we have
\begin{equation}
\|\tilde{\xi}'\|_{\tilde{g}\left(\tilde{\eta}\right)}\ge\|\tilde{\xi}\|_{\tilde{g}\left(\tilde{\eta}\right)}-\|\tilde{\xi}'-\tilde{\xi}\|_{\tilde{g}\left(\tilde{\eta}\right)}\ge\|\tilde{\xi}\|_{\tilde{g}\left(\tilde{\eta}\right)}/2\underset{(\ref{eq:eta})}{>}h_{0}^{-\delta}/2\label{eq:*4}
\end{equation}
and hence, using (D.1), we see 
\begin{align}
\max\left\{ \frac{\langle\|\tilde{\xi}'\|_{\tilde{g}\left(\tilde{\eta}\right)}\rangle}{\langle\|\tilde{\xi}\|_{\tilde{g}\left(\tilde{\eta}\right)}\rangle},\frac{\langle\|\tilde{\xi}\|_{\tilde{g}\left(\tilde{\eta}\right)}\rangle}{\langle\|\tilde{\xi}'\|_{\tilde{g}\left(\tilde{\eta}\right)}\rangle}\right\}  & \ineq{\ref{eq:D1p}}1+\frac{\|\tilde{\xi}'-\tilde{\xi}\|_{\tilde{g}\left(\tilde{\eta}\right)}}{\min\{\langle\|\tilde{\xi}\|_{\tilde{g}\left(\tilde{\eta}\right)}\rangle,\langle\|\tilde{\xi}'\|_{\tilde{g}\left(\tilde{\eta}\right)}\rangle\}}\label{eq:log}\\
 & \ineq{\ref{eq:*4}}1+\frac{2\|\tilde{\xi}'-\tilde{\xi}\|_{\tilde{g}\left(\tilde{\eta}\right)}}{\langle\|\tilde{\xi}\|_{\tilde{g}\left(\tilde{\eta}\right)}\rangle}\label{eq:*10}
\end{align}
Since $\tilde{h}_{\gamma}^{\perp}(\tilde{\eta})=\langle\|\tilde{\xi}\|_{\tilde{g}\left(\tilde{\eta}\right)}\rangle^{-\gamma}$,
we continue 
\begin{align}
\max\left\{ \frac{\langle\|\tilde{\xi}'\|_{\tilde{g}\left(\tilde{\eta}\right)}\rangle}{\langle\|\tilde{\xi}\|_{\tilde{g}\left(\tilde{\eta}\right)}\rangle},\frac{\langle\|\tilde{\xi}\|_{\tilde{g}\left(\tilde{\eta}\right)}\rangle}{\langle\|\tilde{\xi}'\|_{\tilde{g}\left(\tilde{\eta}\right)}\rangle}\right\}  & \ineq{\ref{eq:*10}}1+2\langle\|\tilde{\xi}\|_{\tilde{g}\left(\tilde{\eta}\right)}\rangle^{-(1-\gamma)}\|\tilde{\xi}'-\tilde{\xi}\|_{\tilde{g}_{\gamma}\left(\tilde{\eta}\right)}\label{eq:log3}\\
 & \le1+Ch_{0}^{\mu}\langle\|\tilde{\xi}'-\tilde{\xi}\|_{\tilde{g}_{\gamma}\left(\tilde{\eta}\right)}\rangle.\nonumber 
\end{align}
where $\mu=\delta\left(1-\gamma\right)<1$ because $\langle\|\tilde{\xi}\|_{\tilde{g}\left(\tilde{\eta}\right)}\rangle^{-(1-\gamma)}\ineq{\ref{eq:eta}}h_{0}^{\delta\left(1-\gamma\right)}$.
\item In the case where $\|\tilde{\xi}'-\tilde{\xi}\|_{\tilde{g}\left(\tilde{\eta}\right)}>\frac{1}{2}\|\tilde{\xi}\|_{\tilde{g}\left(\tilde{\eta}\right)}$
we have
\begin{align}
\|\tilde{\xi}'-\tilde{\xi}\|_{\tilde{g}_{\gamma}\left(\tilde{\eta}\right)} & \eq{\ref{eq:def_gtilde}}\tilde{h}_{\gamma}^{\perp}(\tilde{\eta})\|\tilde{\xi}'-\tilde{\xi}\|_{\tilde{g}\left(\tilde{\eta}\right)}=\langle\|\tilde{\xi}\|_{g_{\tilde{\eta}}}\rangle^{-\gamma}\cdot\|\tilde{\xi}'-\tilde{\xi}\|_{\tilde{g}\left(\tilde{\eta}\right)}\nonumber \\
 & \underset{(\mathrm{Hyp.})}{\geq}C^{-1}\langle\|\tilde{\xi}'-\tilde{\xi}\|_{\tilde{g}\left(\tilde{\eta}\right)}\rangle^{1-\gamma}.\label{eq:*5}
\end{align}
Note also that the assumption \eqref{eq:eta} gives 
\begin{equation}
\|\tilde{\xi}'-\tilde{\xi}\|_{\tilde{g}\left(\tilde{\eta}\right)}>\|\tilde{\xi}\|_{\tilde{g}\left(\tilde{\eta}\right)}/2>h_{0}^{-\delta}/2\quad\text{and hence}\quad2h_{0}^{\delta}\|\tilde{\xi}'-\tilde{\xi}\|_{\tilde{g}\left(\tilde{\eta}\right)}>1\label{eq:a}
\end{equation}
From this estimate, we obtain 
\begin{align}
\max\left\{ \frac{\langle\|\tilde{\xi}'\|_{\tilde{g}\left(\tilde{\eta}\right)}\rangle}{\langle\|\tilde{\xi}\|_{\tilde{g}\left(\tilde{\eta}\right)}\rangle},\frac{\langle\|\tilde{\xi}\|_{\tilde{g}\left(\tilde{\eta}\right)}\rangle}{\langle\|\tilde{\xi}'\|_{\tilde{g}\left(\tilde{\eta}\right)}\rangle}\right\}  & \ineq{\ref{eq:D1p}}1+\|\tilde{\xi}'-\tilde{\xi}\|_{\tilde{g}\left(\tilde{\eta}\right)}\label{log2}\\
 & \ineq{\ref{eq:a}}1+2h_{0}^{\delta}\|\tilde{\xi}'-\tilde{\xi}\|_{\tilde{g}\left(\tilde{\eta}\right)}^{2}\nonumber \\
 & \ineq{\ref{eq:*5}}1+2h_{0}^{\delta}\cdot\langle\|\tilde{\xi}'-\tilde{\xi}\|_{\tilde{g}_{\gamma}\left(\tilde{\eta}\right)}\rangle^{2/(1-\gamma)}\nonumber \\
 & \le1+Ch_{0}^{\delta}\langle\|\tilde{\xi}'-\tilde{\xi}\|_{\tilde{g}\left(\tilde{\eta}\right)^{\gamma}}\rangle^{N}\nonumber 
\end{align}
where, in the last inequality, we let $N>2/(1-\gamma)$.
\end{enumerate}
From \eqref{eq:log3} and \eqref{log2} in the two (exhaustive) cases,
we always have 
\[
\max\left\{ \frac{\langle\|\tilde{\xi}'\|_{\tilde{g}\left(\tilde{\eta}\right)}\rangle}{\langle\|\tilde{\xi}\|_{\tilde{g}\left(\tilde{\eta}\right)}\rangle},\frac{\langle\|\tilde{\xi}\|_{\tilde{g}\left(\tilde{\eta}\right)}\rangle}{\langle\|\tilde{\xi}'\|_{\tilde{g}\left(\tilde{\eta}\right)}\rangle}\right\} \le1+Ch_{0}^{\mu}\left\langle \|\tilde{\eta}'-\tilde{\eta}\|_{\tilde{g}_{\gamma}\left(\tilde{\eta}\right)}\right\rangle ^{N}
\]
for some constant $C>0$ and $N>0$. Plugging this inequality and
\eqref{hgamma1} in \eqref{eq:vg}, we obtain 
\[
\max\left\{ \frac{\tilde{h}_{\gamma}^{\perp}(\tilde{\eta}')}{\tilde{h}_{\gamma}^{\perp}(\tilde{\eta})},\frac{\tilde{h}_{\gamma}^{\perp}(\tilde{\eta})}{\tilde{h}_{\gamma}^{\perp}(\tilde{\eta}')}\right\} \le1+C'h_{0}^{\mu}\left\langle \|\tilde{\eta}'-\tilde{\eta}\|_{\tilde{g}_{\gamma}\left(\tilde{\eta}\right)}\right\rangle ^{N'}
\]
for some constant $C'>0$ and $N'>0$. From the last inequality and
\eqref{claim3}, we obtain, for any $0\neq v\in\real^{n+1}$, that
\begin{align}
\max\left\{ \frac{\|v\|_{\tilde{g}_{\gamma}\left(\tilde{\eta}\right)}}{\|v\|_{\tilde{g}_{\gamma}\left(\tilde{\eta}'\right)}},\frac{\|v\|_{\tilde{g}_{\gamma}\left(\tilde{\eta}'\right)}}{\|v\|_{\tilde{g}_{\gamma}\left(\tilde{\eta}\right)}}\right\}  & \le1+C''h_{0}^{\mu}\cdot\left\langle \|\tilde{\eta}'-\tilde{\eta}\|_{\tilde{g}_{\gamma}\left(\tilde{\eta}\right)}\right\rangle ^{N''}\label{claim4}
\end{align}
for some $C''>1$ and $N''>1$.

Now we can conclude the claim \eqref{eq:reducedclaim1}. Indeed, we
have 
\begin{align*}
\frac{\langle h_{0}\|\tilde{\xi}'_{s}\|_{\tilde{g}_{\gamma}\left(\tilde{\eta}'\right)}\rangle}{\langle h_{0}\|\tilde{\xi}_{s}\|_{\tilde{g}_{\gamma}\left(\tilde{\eta}\right)}\rangle} & =\frac{\langle h_{0}\|\tilde{\xi}'_{s}\|_{\tilde{g}_{\gamma}\left(\tilde{\eta}'\right)}\rangle}{\langle h_{0}\|\tilde{\xi}'_{s}\|_{\tilde{g}_{\gamma}\left(\tilde{\eta}\right)}\rangle}\frac{\langle h_{0}\|\tilde{\xi}'_{s}\|_{\tilde{g}_{\gamma}\left(\tilde{\eta}\right)}\rangle}{\langle h_{0}\|\tilde{\xi}_{s}\|_{\tilde{g}_{\gamma}\left(\tilde{\eta}\right)}\rangle}\\
 & \le\frac{\langle\|\tilde{\xi}'_{s}\|_{\tilde{g}_{\gamma}\left(\tilde{\eta}'\right)}\rangle}{\langle\|\tilde{\xi}'_{s}\|_{\tilde{g}_{\gamma}\left(\tilde{\eta}\right)}\rangle}\frac{\langle h_{0}\|\tilde{\xi}'_{s}\|_{\tilde{g}_{\gamma}\left(\tilde{\eta}\right)}\rangle}{\langle h_{0}\|\tilde{\xi}_{s}\|_{\tilde{g}_{\gamma}\left(\tilde{\eta}\right)}\rangle}
\end{align*}
and we can apply \eqref{claim4} for the former term on the right
hand side and \eqref{eq:basic} to the latter to get 
\begin{align*}
\frac{\langle h_{0}\|\tilde{\xi}'_{s}\|_{\tilde{g}_{\gamma}\left(\tilde{\eta}'\right)}\rangle}{\langle h_{0}\|\tilde{\xi}_{s}\|_{\tilde{g}_{\gamma}\left(\tilde{\eta}\right)}\rangle} & \le(1+C''h_{0}^{\mu}\cdot\left\langle \|\tilde{\eta}'-\tilde{\eta}\|_{\tilde{g}_{\gamma}\left(\tilde{\eta}\right)}\right\rangle ^{N''})\cdot(1+h_{0}\|\tilde{\xi}'_{s}-\tilde{\xi}_{s}\|_{\tilde{g}_{\gamma}\left(\tilde{\eta}\right)})\\
 & \le1+C_{*}h_{0}^{\mu}\cdot\left\langle \|\tilde{\eta}'-\tilde{\eta}\|_{\tilde{g}_{\gamma}\left(\tilde{\eta}\right)}\right\rangle ^{N_{*}}
\end{align*}
for some $C_{*},N_{*}>1$. Likewise, we obtain the same inequality
for the reciprocal of the left-hand side, concluding the claim \eqref{eq:reducedclaim1}.

Consider the remaining case where we have $\|\tilde{\xi}'\|_{\tilde{g}_{\gamma}\left(\tilde{\eta}'\right)}>h_{0}^{-\delta}$
instead of $\|\tilde{\xi}\|_{\tilde{g}_{\gamma}\left(\tilde{\eta}\right)}>h_{0}^{-\delta}$
in (\ref{eq:eta}). Then we can follow the argument above with $\tilde{\eta}$
and $\tilde{\eta}'$ exchanged and obtain
\begin{align}
\max\left\{ \frac{\langle h_{0}\|\tilde{\xi}'_{s}\|_{\tilde{g}_{\gamma}\left(\tilde{\eta}'\right)}\rangle}{\langle h_{0}\|\tilde{\xi}_{s}\|_{\tilde{g}_{\gamma}\left(\tilde{\eta}\right)}\rangle},\frac{\langle h_{0}\|\tilde{\xi}_{s}\|_{\tilde{g}_{\gamma}\left(\tilde{\eta}\right)}\rangle}{\langle h_{0}\|\tilde{\xi}'_{s}\|_{\tilde{g}_{\gamma}\left(\tilde{\eta}'\right)}\rangle}\right\} \le1+Ch_{0}^{\mu}\langle\|\tilde{\eta}'-\tilde{\eta}\|_{\tilde{g}_{\gamma}\left(\tilde{\eta}'\right)^{\gamma}}\rangle^{N}.\label{eq:reducedclaim2}
\end{align}
(The difference from \eqref{eq:reducedclaim1} is only that the metric
$\tilde{g}_{\gamma}\left(\tilde{\eta}\right)$ is replaced by $\tilde{g}_{\gamma}\left(\tilde{\eta}'\right)$
on the right-hand side.) Then from the temperate property (\ref{eq:g_moderate and temperate})
in Lemma \ref{lem:temperate_metric}, we obtain the claim \eqref{eq:reducedclaim1}. 
\end{proof}

\subsubsection{The straightening coordinates are Lipschitz at scale greater than
$1$}

In (\ref{eq:norm_g}) we have defined an Euclidean scalar product
$\tilde{g}$ on $T_{\tilde{\eta}}\mathbb{R}^{n+1}\equiv\mathbb{R}^{n+1}$
with coordinates $\tilde{\eta}=\left(\tilde{\xi},\omega\right)\in\mathbb{R}^{n+1}$.
Here we extend this metric $\tilde{g}$ to a metric denoted $\overset{\thickapprox}{g}$
on $\mathbb{R}^{2\left(n+1\right)}$ with coordinates 
\[
\tilde{\varrho}=\left(y,\tilde{\eta}\right)=\left(x,z,\tilde{\xi},\omega\right)\in\mathbb{R}^{2\left(n+1\right)},
\]
by simply adding the same Euclidean scalar product along $\left(x,z\right)$
as in (\ref{eq:metric_g_in_coordinates}):
\begin{equation}
\overset{\thickapprox}{g}:=\left(\left\langle \tilde{\eta}\right\rangle ^{\alpha^{\perp}}dx\right)^{2}+\left(\left\langle \tilde{\eta}\right\rangle ^{\alpha^{\parallel}}dz\right)^{2}+\tilde{g},\label{eq:def_g_double_tilde}
\end{equation}
so that in coordinates $\overset{\thickapprox}{g}$ coincides with
the metric $g$ in (\ref{eq:metric_g_in_coordinates}). As in (\ref{eq:def_gtilde})
we define a conformal (or re-scaled) metric $\overset{\thickapprox}{g}_{\gamma}$
as follows. For $\tilde{v}\in T_{\tilde{\varrho}}\mathbb{R}^{2\left(n+1\right)}$,
\begin{equation}
\left\Vert \tilde{v}\right\Vert _{\overset{\thickapprox}{g}_{\gamma}\left(\tilde{\varrho}\right)}=\tilde{h}_{\gamma}^{\perp}\left(\tilde{\eta}\right)\left\Vert \tilde{v}\right\Vert _{\overset{\thickapprox}{g}\left(\tilde{\varrho}\right)}.\label{eq:def_gtilde-1}
\end{equation}
The following proposition shows that the map $\Phi_{j}$ is Hölder
continuous and also Lipschitz in the scale greater than $1$ for the
metric $\overset{\thickapprox}{g}_{\gamma}$. 

\begin{cBoxB}{}
\begin{prop}
With the assumption (\ref{eq:gamma_interval}) on $0\leq\gamma<1$,
the map $\Phi_{j}$ defined in (\ref{eq:def_Phi}) satisfies the following
estimate: $\exists C>0$, $\forall\varrho,\varrho'\in T^{*}V_{j}\subset\mathbb{R}^{2\left(n+1\right)}$
\begin{align}
\|\Phi_{j}(\varrho')-\Phi_{j}(\varrho)\|_{\overset{\thickapprox}{g}_{\gamma}\left(\Phi_{j}\left(\varrho\right)\right)} & \le C\max\left\{ \|\varrho'-\varrho\|_{\overset{\thickapprox}{g}_{\gamma}\left(\varrho\right)},\|\varrho'-\varrho\|_{\overset{\thickapprox}{g}_{\gamma}\left(\varrho\right)}^{\beta_{*}}\right\} \label{eq:Phi_Lipchitz}\\
 & \le C\langle\|\varrho'-\varrho\|_{\overset{\thickapprox}{g}_{\gamma}\left(\varrho\right)}\rangle.\nonumber 
\end{align}
Consequently, over a chart $U_{j}\subset M$, $\exists C>0$, $\forall\rho,\rho'\in T^{*}U_{j}$
we have
\begin{align}
\|\Phi_{j}\circ\tilde{\kappa}_{j}\left(\rho'\right)-\Phi_{j}\circ\tilde{\kappa}_{j}\left(\rho\right)\|_{\overset{\thickapprox}{g}_{\gamma}\left(\Phi_{j}\circ\tilde{\kappa}_{j}\left(\rho\right)\right)} & \le C\langle h_{\gamma}^{\perp}\left(\rho\right)\|\rho'-\rho\|_{g\left(\rho\right)}\rangle.\label{eq:Phi_Lipchitz-1}
\end{align}
\end{prop}

\end{cBoxB}

\begin{proof}
Take two points
\[
\varrho=(y,\eta)=(x,z,\xi,\omega),\quad\varrho'=(y',\eta')=(x',z',\xi',\omega')\in\mathbb{R}^{2\left(n+1\right)}
\]
that satisfy
\[
d:=\|\varrho'-\varrho\|_{\overset{\thickapprox}{g}_{\gamma}\left(\varrho\right)}>0.
\]
From the definition of $\overset{\thickapprox}{g}_{\gamma}$ in (\ref{eq:def_gtilde-1}),(\ref{eq:def_g_double_tilde})
and (\ref{eq:norm_g}), this implies that 
\begin{align*}
|x'-x|\le d\langle\eta\rangle^{-\alpha^{\perp}},\quad|z'-z|\le d\langle\eta\rangle^{-\alpha^{\parallel}},\\
|\xi'-\xi|\le d\langle\eta\rangle^{\alpha^{\perp}}\cdot h_{\gamma}^{\perp}(\eta)^{-1},\quad|\omega'-\omega|\le d\langle\eta\rangle^{\alpha^{\parallel}}.
\end{align*}
We set
\[
\Phi_{j}(\varrho)=(y,\tilde{\eta})=(x,z,\tilde{\xi},\omega),\quad\Phi_{j}(\varrho')=(y',\tilde{\eta}')=(x',z',\tilde{\xi}',\omega').
\]
with $\tilde{\xi},\tilde{\xi}'$ given as in (\ref{eq:xi_tilde}).
Then we have
\[
|\tilde{\xi}'-\tilde{\xi}|\le C|\xi'-\xi|+C|\xi||x'-x|^{\beta_{*}}+C|\omega||x'-x|^{\beta_{0}}
\]
and also 
\[
C^{-1}\langle\eta\rangle\le\langle\tilde{\eta}\rangle\le C\langle\eta\rangle.
\]
We will use that
\begin{equation}
h_{\gamma}^{\perp}(\eta)\eq{\ref{eq:h_gamma}}\left\langle \left\Vert \xi\right\Vert _{g\left(\eta\right)}\right\rangle ^{-\gamma}=\left\langle \left\langle \left|\eta\right|\right\rangle ^{-\alpha^{\perp}}\left|\xi\right|\right\rangle ^{-\gamma}\leq\left\langle \left|\xi\right|^{1-\alpha^{\perp}}\right\rangle ^{-\gamma}\leq C\left\langle \left|\xi\right|\right\rangle ^{-\gamma\left(1-\alpha^{\perp}\right)}\label{eq:*7}
\end{equation}
and $\gamma<1<\frac{1}{1-\alpha^{\perp}}$ to get that $1-\gamma\left(1-\alpha^{\perp}\right)>0$
hence
\begin{equation}
\left\langle \xi\right\rangle ^{1-\gamma\left(1-\alpha^{\perp}\right)}\leq\left\langle \eta\right\rangle ^{1-\gamma\left(1-\alpha^{\perp}\right)}\label{eq:*8}
\end{equation}

Hence
\begin{align*}
\|\Phi_{j}(\varrho')-\Phi_{j}(\varrho)\|_{\overset{\thickapprox}{g}_{\gamma}\left(\Phi_{j}\left(\varrho\right)\right)} & \le\langle\tilde{\eta}\rangle^{\alpha^{\perp}}|x'-x|+\langle\tilde{\eta}\rangle^{\alpha^{\parallel}}|z'-z|\\
 & \quad+h_{\gamma}^{\perp}(\xi)\cdot\langle\tilde{\eta}\rangle^{-\alpha^{\perp}}\cdot|\tilde{\xi}'-\tilde{\xi}|+\langle\eta\rangle^{-\alpha^{\parallel}}|\omega'-\omega|\\
 & \le Cd+Ch_{\gamma}^{\perp}(\xi)\cdot\langle\eta\rangle^{-\alpha^{\perp}}|\xi|\cdot(d\langle\eta\rangle^{-\alpha^{\perp}})^{\beta_{*}}\\
 & \quad+C\langle\eta\rangle^{-\alpha^{\perp}}|\omega|(d\langle\eta\rangle^{-\alpha^{\perp}})^{\beta_{0}}\\
 & \ineq{\ref{eq:*7}}Cd+C\langle\xi\rangle^{1-\gamma(1-\alpha^{\perp})}\langle\eta\rangle^{-(1+\beta_{*})\alpha^{\perp}}d^{\beta_{*}}\\
 & \quad+C\langle\eta\rangle^{1-\alpha^{\perp}-\alpha^{\perp}\beta_{0}}d^{\beta_{0}}\\
 & \ineq{\ref{eq:*8}}Cd+C\langle\eta\rangle^{1-\gamma(1-\alpha^{\perp})-(1+\beta_{*})\alpha^{\perp}}d^{\beta_{*}}\\
 & \quad+C\langle\eta\rangle^{1-\alpha^{\perp}-\alpha^{\perp}\beta_{0}}d^{\beta_{0}}.
\end{align*}
We have from (\ref{eq:gamma_interval})
\begin{align*}
 & \gamma\underset{(\ref{eq:gamma_interval})}{\geq}\frac{1-\alpha^{\perp}\left(1+\beta_{*}\right)}{1-\alpha^{\perp}}\\
\Leftrightarrow & 1-\gamma\left(1-\alpha^{\perp}\right)-\left(1+\beta_{*}\right)\alpha^{\perp}\leq0,
\end{align*}
and also $1-\alpha^{\perp}(1+\beta_{0})<0$ from (\ref{eq:ineq1}).
We obtain that 
\[
\|\Phi_{j}(\varrho')-\Phi_{j}(\varrho)\|_{\overset{\thickapprox}{g}_{\gamma}\left(\Phi_{j}\left(\varrho\right)\right)}\le C(d+d^{\beta_{0}}+d^{\beta_{*}}).
\]
Then (\ref{eq:def_beta_*}) gives (\ref{eq:Phi_Lipchitz}). Since
the metric $\overset{\thickapprox}{g}$ in (\ref{eq:def_g_double_tilde})
coincides with the metric $g$ in (\ref{eq:metric_g_in_coordinates}),
we deduce (\ref{eq:Phi_Lipchitz-1}), by writing $\varrho=\tilde{\kappa}_{j}\left(\rho\right),\varrho'=\tilde{\kappa}_{j}\left(\rho'\right)\in\mathbb{R}^{2\left(n+1\right)}$.
\end{proof}

\subsubsection{Final step for the proof}

Let $\rho,\rho'\in T^{*}U_{j}$ over a chart $U_{j}\subset M$ and
write $\varrho=\tilde{\kappa}_{j}\left(\rho\right),\varrho'=\tilde{\kappa}_{j}\left(\rho'\right)\in\mathbb{R}^{2\left(n+1\right)}$
and $\tilde{\varrho}=\Phi_{j}\left(\varrho\right)$, $\tilde{\varrho}'=\Phi_{j}\left(\varrho'\right)$.
We have
\[
\frac{W\left(\rho'\right)}{W\left(\rho\right)}\eq{\ref{eq:def_W}}\frac{\tilde{W}\left(\tilde{\eta}'\right)}{\tilde{W}\left(\tilde{\eta}\right)}\ineq{\ref{eq:w2}}1+C_{\tilde{W}}h_{0}^{\mu}\langle\|\tilde{\eta}'-\tilde{\eta}\|_{\tilde{g}_{\gamma}\left(\tilde{\eta}\right)}\rangle^{N_{\tilde{W}}}.
\]
But
\begin{align*}
\|\tilde{\eta}'-\tilde{\eta}\|_{\tilde{g}_{\gamma}\left(\tilde{\eta}\right)}\ineq{\ref{eq:def_g_double_tilde}}\|\tilde{\varrho}'-\tilde{\varrho}\|_{\overset{\thickapprox}{g}_{\gamma}\left(\tilde{\varrho}\right)}\ineq{\ref{eq:Phi_Lipchitz-1}}C\langle h_{\gamma}^{\perp}\left(\rho\right)\|\rho'-\rho\|_{g\left(\rho\right)}\rangle & \ineq{\ref{eq:log_log}}C'\left\langle h_{\gamma}^{\perp}\left(\rho\right)\mathrm{dist}_{g}\left(\rho',\rho\right)\right\rangle ^{C}.
\end{align*}
We deduce that
\[
\frac{W\left(\rho'\right)}{W\left(\rho\right)}\leq1+C_{W}h_{0}^{\mu}\left\langle h_{\gamma}^{\perp}\left(\rho\right)\mathrm{dist}_{g}\left(\rho',\rho\right)\right\rangle ^{N_{W}}
\]
This is property (\ref{eq:slow_variation_W}) for $W$ as claimed
in (\ref{enu:-satisfies--temperate-1}) in Theorem \ref{thm:W}.

\subsection{The decay property \eqref{enu:-satisfies-decay-1}}

We will write $A\preceq B$ if $A\le CB$ for some constant $C>0$
independently of $t\geq0$, $\rho\in T^{*}M$ and write $A\asymp B$
if $A\preceq B$ and $A\succeq B$ simultaneously. Let $\rho\in T^{*}M$
and $t\geq0$. As in \eqref{eq:decomp_Xi} and \eqref{hyperbolicity_xi},
we write
\[
\rho\left(t\right):=\tilde{\phi}^{t}\left(\rho\right)\quad\text{ and }\quad\rho(t)=\rho_{*}(t)+\omega(t)\cdot\mathcal{A}(m(t)),\quad\rho_{*}(t):=\rho_{s}(t)+\rho_{u}(t).
\]
Let us express the weight function $W$ in \eqref{eq:def_W1} as
\begin{equation}
W\left(\rho\right)\underset{(\ref{eq:def_W1})}{=}\frac{\left\langle A\left(\rho\right)\left|\rho_{s}\right|\right\rangle ^{R_{s}}}{\left\langle A\left(\rho\right)\left|\rho_{u}\right|\right\rangle ^{R_{u}}}\label{eq:W2}
\end{equation}
with setting
\begin{equation}
A\left(\rho\right)\underset{(\ref{eq:metric_g_in_coordinates}),(\ref{eq:h_gamma})}{=}h_{0}\left\langle \left\langle \left|\rho\right|\right\rangle ^{-\alpha^{\perp}}\left|\rho_{*}\right|\right\rangle ^{-\gamma}\left\langle \left|\rho\right|\right\rangle ^{-\alpha^{\perp}}.\label{eq:def_A}
\end{equation}
We consider two cases $\left|\rho_{*}(t)\right|\leq\left|\rho_{*}\right|$
and $\left|\rho_{*}(t)\right|>\left|\rho_{*}\right|$ separately.
Let us assume 
\begin{equation}
\left|\rho_{*}(t)\right|\leq\left|\rho_{*}\right|.\label{eq:hyp*}
\end{equation}
 
\begin{lem}
We have
\begin{equation}
\frac{|\rho_{s}(t)|}{|\rho_{s}|}\preceq\frac{|\rho_{*}(t)|}{|\rho_{*}|}\preceq\frac{|\rho(t)|}{|\rho|}\preceq1\preceq\frac{|\rho_{u}(t)|}{|\rho_{u}|}.\label{eq:rationXi}
\end{equation}
\end{lem}

\begin{proof}
Since $\left|\rho_{u}\right|\preceq\left|\rho_{u}\left(t\right)\right|$,
$\left|\rho_{s}\left(t\right)\right|\preceq\left|\rho_{s}\right|$
and $\omega(t)=\omega$, we have
\[
\frac{|\rho_{s}(t)|}{|\rho_{s}|}=\frac{\left|\rho_{s}\left(t\right)\right|+\left|\rho_{u}\right|\frac{\left|\rho_{s}\left(t\right)\right|}{\left|\rho_{s}\right|}}{\left|\rho_{s}\right|+\left|\rho_{u}\right|}\preceq\frac{|\rho_{*}(t)|}{|\rho_{*}|}.
\]
Similarly
\[
\frac{|\rho_{*}(t)|}{|\rho_{*}|}=\frac{\left|\rho_{*}\left(t\right)\right|+\left|\omega\right|\frac{\left|\rho_{*}\left(t\right)\right|}{\left|\rho_{*}\right|}}{\left|\rho_{*}\right|+\left|\omega\right|}\preceq\frac{|\rho(t)|}{|\rho|},
\]
\[
\frac{|\rho(t)|}{|\rho|}\leq\frac{\left|\rho_{*}\left(t\right)\right|+\left|\omega\right|}{\left|\rho_{*}\right|+\left|\omega\right|}\ineq{\ref{eq:hyp*}}1.
\]
\end{proof}
We have $\left\langle \frac{|\rho|}{|\rho(t)|}\right\rangle \underset{\ref{eq:rationXi}}{\asymp}\frac{|\rho|}{|\rho(t)|}\succeq1$
and $\alpha^{\perp}<1$ hence 
\begin{equation}
\left\langle \frac{|\rho|}{|\rho(t)|}\right\rangle ^{\alpha^{\perp}}\preceq\frac{|\rho|}{|\rho(t)|}.\label{eq:*12}
\end{equation}
We get
\begin{equation}
\frac{\left\langle \left\langle \left|\rho(t)\right|\right\rangle ^{-\alpha^{\perp}}\left|\rho_{*}(t)\right|\right\rangle }{\left\langle \left\langle \left|\rho\right|\right\rangle ^{-\alpha^{\perp}}\left|\rho_{*}\right|\right\rangle }\ineq{\ref{eq:prod2}}\left\langle \frac{\left\langle \left|\rho\right|\right\rangle ^{\alpha^{\perp}}\left|\rho_{*}(t)\right|}{\left\langle \left|\rho(t)\right|\right\rangle ^{\alpha^{\perp}}\left|\rho_{*}\right|}\right\rangle \ineq{\ref{eq:prod2}}\left\langle \left\langle \frac{|\rho|}{|\rho(t)|}\right\rangle ^{\alpha^{\perp}}\frac{\left|\rho_{*}(t)\right|}{\left|\rho_{*}\right|}\right\rangle \underset{(\ref{eq:*12}),(\ref{eq:rationXi})}{\preceq}1.\label{eq:*13}
\end{equation}
This implies 
\[
\frac{A\left(\rho(t)\right)\left|\rho{}_{u}(t)\right|}{A\left(\rho\right)\left|\rho{}_{u}\right|}\eq{\ref{eq:def_A}}\frac{\left\langle \left\langle \left|\rho(t)\right|\right\rangle ^{-\alpha^{\perp}}\left|\rho_{*}(t)\right|\right\rangle ^{-\gamma}\left\langle \left|\rho(t)\right|\right\rangle ^{-\alpha^{\perp}}\left|\rho{}_{u}(t)\right|}{\left\langle \left\langle \left|\rho\right|\right\rangle ^{-\alpha^{\perp}}\left|\rho_{*}\right|\right\rangle ^{-\gamma}\left\langle \left|\rho\right|\right\rangle ^{-\alpha^{\perp}}\left|\rho{}_{u}\right|}\underset{(\ref{eq:*13})}{\succeq}\frac{\left\langle \left|\rho\right|\right\rangle ^{\alpha^{\perp}}\left|\rho{}_{u}(t)\right|}{\left\langle \left|\rho(t)\right|\right\rangle ^{\alpha^{\perp}}\left|\rho{}_{u}\right|}\underset{(\ref{eq:rationXi})}{\succeq}1,
\]
and hence 
\begin{equation}
\left\langle A\left(\rho(t)\right)\left|\rho{}_{u}(t)\right|\right\rangle \succeq\left\langle A\left(\rho\right)\left|\rho_{u}\right|\right\rangle .\label{eq:Au}
\end{equation}
Again from \eqref{eq:rationXi}, we get
\[
\frac{\left\langle \rho(t)\right\rangle ^{-\alpha^{\perp}}\left|\rho_{*}(t)\right|}{\left\langle \rho\right\rangle ^{-\alpha^{\perp}}\left|\rho_{*}\right|}\succeq\frac{\left|\rho_{*}(t)\right|}{\left|\rho_{*}\right|}\ge1
\]
and therefore 
\[
\frac{A\left(\rho(t)\right)}{A\left(\rho\right)}\eq{\ref{eq:def_A}}\frac{\left\langle \left\langle \left|\rho(t)\right|\right\rangle ^{-\alpha^{\perp}}\left|\rho_{*}(t)\right|\right\rangle ^{-\gamma}\left\langle \left|\rho(t)\right|\right\rangle ^{-\alpha^{\perp}}}{\left\langle \left\langle \left|\rho\right|\right\rangle ^{-\alpha^{\perp}}\left|\rho_{*}\right|\right\rangle ^{-\gamma}\left\langle \left|\rho\right|\right\rangle ^{-\alpha^{\perp}}}\preceq\left\langle \frac{\left|\rho\right|}{\left|\rho(t)\right|}\right\rangle ^{\alpha^{\perp}}\underset{(\ref{eq:rationXi})}{\preceq}\left\langle \frac{\left|\rho_{*}\right|}{\left|\rho_{*}(t)\right|}\right\rangle ^{\alpha^{\perp}}.
\]
This implies
\[
\frac{A\left(\rho(t)\right)\left|\rho{}_{s}(t)\right|}{A\left(\rho\right)\left|\rho{}_{s}\right|}=\frac{\left\langle \left\langle \left|\rho(t)\right|\right\rangle ^{-\alpha^{\perp}}\left|\rho_{*}(t)\right|\right\rangle ^{-\gamma}\left\langle \left|\rho(t)\right|\right\rangle ^{-\alpha^{\perp}}\left|\rho{}_{s}(t)\right|}{\left\langle \left\langle \left|\rho\right|\right\rangle ^{-\alpha^{\perp}}\left|\rho_{*}\right|\right\rangle ^{-\gamma}\left\langle \left|\rho\right|\right\rangle ^{-\alpha^{\perp}}\left|\rho{}_{s}\right|}\preceq\left\langle \frac{\left|\rho_{*}\right|}{\left|\rho_{*}(t)\right|}\right\rangle ^{\alpha^{\perp}}\frac{\left|\rho{}_{s}(t)\right|}{\left|\rho{}_{s}\right|}\preceq1
\]
and hence 
\begin{equation}
\left\langle A\left(\rho(t)\right)\left|\rho{}_{s}(t)\right|\right\rangle \preceq\left\langle A\left(\rho\right)\left|\rho_{s}\right|\right\rangle .\label{eq:As}
\end{equation}
From (\ref{eq:W2}), \eqref{eq:Au} and \eqref{eq:As}, we see that
\[
W\left(\rho(t)\right)\preceq W\left(\rho\right)
\]
provided $\left|\rho_{*}(t)\right|\leq\left|\rho_{*}\right|$. Further
we can choose a large constant $C_{t}>1$ depending on $t$ so that,
if $\left\Vert \rho_{*}\right\Vert _{g_{\rho}}>C_{t}$, we have for
any $t'\in[0,t]$ that $\left\langle \rho(t')\right\rangle ^{-\alpha^{\perp}}\left|\rho_{*}(t')\right|\succeq1$
and $\left|\rho(t')\right|\succeq1$ and 
\[
\left\langle \left\langle \rho(t')\right\rangle ^{-\alpha^{\perp}}\left|\rho_{*}(t')\right|\right\rangle \asymp\left\langle \rho(t')\right\rangle ^{-\alpha^{\perp}}\left|\rho_{*}(t')\right|\quad\text{and}\quad\left\langle \left|\rho(t')\right|\right\rangle \asymp\left|\rho(t')\right|\quad\text{for }t'\in\left[0,t\right].
\]
We therefore obtain
\[
\frac{A\left(\rho(t)\right)\left|\rho{}_{s}(t)\right|}{A\left(\rho\right)\left|\rho{}_{s}\right|}\preceq\frac{|\rho_{*}(t)|^{-\gamma}\cdot\left|\rho(t)\right|^{-\alpha^{\perp}(1-\gamma)}\left|\rho{}_{s}(t)\right|}{|\rho_{*}|^{-\gamma}\cdot\left|\rho\right|^{-\alpha^{\perp}(1-\gamma)}\left|\rho{}_{s}\right|}\underset{(\ref{eq:rationXi})}{\preceq}\left(\frac{\left|\rho{}_{s}(t)\right|}{\left|\rho{}_{s}\right|}\right)^{(1-\alpha^{\perp})(1-\gamma)}\underset{(\ref{hyperbolicity_xi})}{\preceq}e^{-\lambda_{\mathrm{min}}t(1-\alpha^{\perp})(1-\gamma)}
\]
and

\begin{align*}
\frac{W\left(\rho(t)\right)}{W\left(\rho\right)} & \underset{(\ref{eq:W2},\ref{eq:Au})}{\preceq}e^{-\lambda_{\mathrm{min}}t\left(1-\gamma\right)\left(1-\alpha^{\perp}\right)R_{s}}.
\end{align*}
We have proved the conclusion of the Theorem \ref{thm:W}, claim \ref{enu:-satisfies-decay-1},
in the case where $\left|\rho_{*}\left(t\right)\right|\le\left|\rho_{*}\right|$.
For the other case where $\left|\rho_{*}\left(t\right)\right|\ge\left|\rho_{*}\right|$,
we can argue in a similar manner.

\subsection{The order property \eqref{enu:Its-order-is-1}}

It remains to prove the claim (\ref{enu:Its-order-is-1}) on the order
of the weight function $W$. Let us assume that $\rho_{u}=0$ and
$\rho_{0}=0$, that is, consider the directions in $E_{s}^{*}$. Then
$\left|\rho\right|\asymp\left|\rho_{*}\right|\asymp\left|\rho_{s}\right|$
and 
\begin{align*}
W\left(\rho\right) & \underset{(\ref{eq:W2})}{\asymp}\left|\rho_{s}\right|^{\left(1-\gamma\right)\left(1-\alpha^{\perp}\right)R_{s}}
\end{align*}
hence $W$ has order $r\left(\left[\rho\right]\right)=\left(1-\gamma\right)\left(1-\alpha^{\perp}\right)R_{s}$.
We proceed similarly in other directions.

We have finished the proof of Theorem \ref{thm:W}.

\section{\label{sec:Second-example-of}Second example of escape function}

In this Section we provide another example $W_{2}:T^{*}M\to\mathbb{R}$
of escape function that satisfies the temperate property (\ref{eq:slow_variation_W})
and the decay property (\ref{eq:decay_property}). For the construction
we use the projective space $\mathbb{P}\left(E_{u}^{*}\oplus E_{s}^{*}\right)$
and this weight function has therefore some ``conical shape'' in
opposite to the first example in (\ref{eq:W2-2-1}) or (\ref{eq:def_W1})
that has some ``parabolic shape'' (compare the red and blue domains
in Figure \ref{fig:zones-W} and Figure \ref{fig:zones-1}).

We propose this escape function, because some escape function similar
to $W_{2}$ has been constructed in \cite{fred-roy-sjostrand-07}
for Anosov diffeomorphisms and extended in \cite{fred_flow_09} for
Anosov flows. Later it has been used in the study of Ruelle resonances
in different settings \cite{dyatlov_Ruelle_resonances_2012,dyatlov_zworski_zeta_2013,dyatlov_guillarmou_2014,dyatlov_faure_guillarmou_2014,jin_zworski_local_trace_14,guillarmou_weich_resonances_16,dang_riviere_morse_smale_16}.
In particular this escape function $W_{2}$ is useful to construct
space $\mathcal{H}_{W}\left(M\right)$ that are fixed with respect
to small perturbations of $X$, see \cite{bonthonneau2018flow}. We
could have use it in this paper, except for the proof of Theorem \ref{thm:WF}
that needs a parabolic neighborhood of $E_{u}^{*}$ and not only conical
ans for the proof of Theorem \ref{thm:grey-band} that needs the decay
controlled from below (\ref{eq:decay_property-1}).

Let us consider the bundle $\mathbb{P}\left(\mathcal{E}^{*}\right)\rightarrow M$
where the fiber over $m\in M$ is the real projective space $\mathbb{RP}\left(\mathcal{E}^{*}\left(m\right)\right)$.
Notice that $\tilde{\phi}^{t}:T^{*}M\rightarrow T^{*}M$ in (\ref{eq:lifted_flow})
induces a flow on $\mathbb{P}\left(\mathcal{E}^{*}\right)$ denoted
by $\mathbb{P}\tilde{\phi}^{t}:\mathbb{P}\left(\mathcal{E}^{*}\right)\to\mathbb{P}\left(\mathcal{E}^{*}\right)$,
because $\tilde{\phi}^{t}$ keeps the sub-bundle $\mathcal{E}^{*}$
invariant and is linear in the fibers. For this flow $\mathbb{P}\tilde{\phi}^{t}$,
the unstable direction $\left[E_{u}^{*}\right]\subset\mathbb{P}\left(\mathcal{E}^{*}\right)$
is an attractor and the stable direction $\left[E_{s}^{*}\right]\subset\mathbb{P}\left(\mathcal{E}^{*}\right)$
is a repeller. Let $a_{0}\in C^{\infty}\left(\mathbb{P}\left(\mathcal{E}^{*}\right);\left[-1,+1\right]\right)$
be a smooth function such that
\begin{itemize}
\item $a_{0}\equiv-1$ on a vicinity $\left[\mathcal{V}_{u}\right]\subset\mathbb{P}\left(\mathcal{E}^{*}\right)$
of the unstable direction $\left[E_{u}^{*}\right]\subset\mathbb{P}\left(\mathcal{E}^{*}\right)$
and
\item $a_{0}\equiv+1$ on a vicinity $\left[\mathcal{V}_{s}\right]\subset\mathbb{P}\left(\mathcal{E}^{*}\right)$
of the stable direction $\left[E_{s}^{*}\right]\subset\mathbb{P}\left(\mathcal{E}^{*}\right)$.
\end{itemize}
From the function $a_{0}$ thus defined, we construct a smooth function
$a\in C^{\infty}\left(\mathcal{E}^{*};\left[-1,+1\right]\right)$
by averaging it along finite orbits of $\mathbb{P}\tilde{\phi}^{t}$
as follows. Let $T>0$ be a constant and put
\begin{equation}
a\left(\rho_{*}\right)=\frac{1}{2T}\int_{-T}^{T}a_{0}\left(\mathbb{P}\tilde{\phi}^{t}\rho_{*}\right)dt\label{eq:def_a_T}
\end{equation}
 for $\rho_{*}=\rho_{u}+\rho_{s}\in\mathcal{E}^{*}$ . In the next
lemma, we will assume that $T>0$ is sufficiently large and denote
by $\left\Vert v\right\Vert _{_{g\left(\rho\right)}}$ the norm of
$v\in T_{\rho}\left(T^{*}M\right)$ with respect to the metric $g$
at $\rho\in T^{*}M$ defined in (\ref{eq:def_norm_g}).

\begin{cBoxA}{}
\begin{defn}[Example 2 of a weight function $W_{2}$ with good properties]
\label{lem:W1}Assume that the parameters $\alpha^{\perp}$ and $\alpha^{\parallel}$
satisfy
\[
\frac{1}{1+\beta_{0}}\leq\alpha^{\perp}<1\quad\text{and}\quad0<\alpha^{\parallel}\leq\alpha^{\perp}.
\]
For $r>0$, we define the escape function $W_{2}:T^{*}M\rightarrow\mathbb{R}^{+}$
by
\begin{equation}
W_{2}\left(\rho\right):=\left\langle \left\Vert \rho_{u}+\rho_{s}\right\Vert _{g_{\rho}}\right\rangle ^{\frac{r}{\left(1-\alpha^{\perp}\right)}a\left(\rho_{u}+\rho_{s}\right)}\label{eq:def_W-1}
\end{equation}
where  $\rho=\omega\mathscr{A}+\rho_{s}+\rho_{u}$ with $\omega\in\mathbb{R}$,
$\rho_{u}\in E_{u}^{*}$, $\rho_{s}\in E_{s}^{*}$.
\end{defn}

\end{cBoxA}

\begin{cBoxB}{}
\begin{lem}
The function $W_{2}$ in (\ref{eq:def_W-1}) has the following properties
\begin{enumerate}
\item \label{enu:-satisfies--temperate}$W_{2}$ satisfies $h_{\gamma}^{\perp}$-temperate
property (\ref{eq:slow_variation_W}) for any $0\leq\gamma<1$.
\item \label{enu:-satisfies-decay}$W_{2}$ satisfies decay property (\ref{eq:decay_property})
with rate $\Lambda=\lambda_{\mathrm{min}}r$.
\item \label{enu:Its-order-is}Its order is $r\left(\left[\rho\right]\right)=ra\left(\rho_{u}+\rho_{s}\right)$
along $\mathcal{E}_{*}$ in particular $r\left(\left[\rho\right]\right)=r$
along $E_{s}^{*}$ and $r\left(\left[\rho\right]\right)=-r$ along
$E_{u}^{*}$.
\end{enumerate}
\end{lem}

\end{cBoxB}

See Figure \ref{fig:zones-1}. We can use the escape function $W_{2}$
to get Theorem \ref{thm:Weyl law} about the density of eigenvalues.
For this we choose the optimal values
\[
\alpha^{\perp}=\frac{1}{1+\beta_{0}},\quad\alpha^{\parallel}=0
\]
that give a transverse size $\Delta_{0}=\omega^{\frac{1}{1+\beta_{0}}}$
for the green region $\mathcal{V}_{0}$ on Figure \ref{fig:zones-1}.

\begin{figure}[h]
\centering{}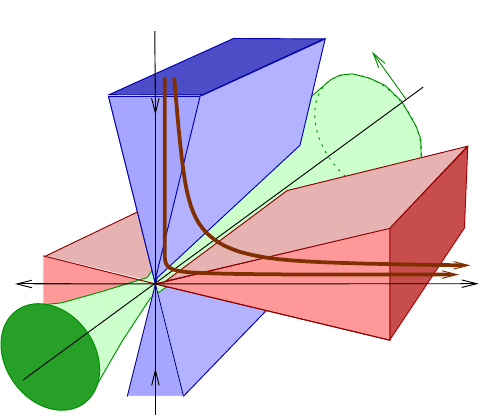\caption{\label{fig:zones-1}Representation of domains associated to the escape
function $W_{2}$ defined in (\ref{eq:def_W-1}): the domain \emph{$\mathcal{V}_{s}$}
in blue (respect. $\mathcal{V}_{u}$ in red) is a conical neighborhood
of $E_{s}^{*}$ (respect. of $E_{u}^{*}$) where the exponent is $a\left(\rho_{u}+\rho_{s}\right)=\pm1$.
Outside the domain $\mathcal{V}_{0}:=\left\{ \rho;\,\left\langle \left\Vert \rho_{u}+\rho_{s}\right\Vert _{g_{\rho}}\right\rangle \asymp1\right\} =\left\{ \rho;\,\left|\rho_{u}+\rho_{s}\right|\apprle\omega^{\alpha^{\perp}}\right\} $
in green, the weight functions $W_{2}$ decays along the flow $\tilde{\phi}^{t}$.}
\end{figure}

\section{\label{sec:Hormander}Relation with the class of symbols $S_{\tilde{\rho},\tilde{\delta}}^{m}$
of Hörmander}

Let us give the relation between, on one side the metric $g$ on $T^{*}M$
with parameter $\alpha^{\perp}$, the conformal metric $g_{\gamma}=h_{\gamma}^{2}g$
with parameter $\gamma$ used in this paper and on the other side
the (traditional) class of symbols $S_{\tilde{\rho},\tilde{\delta}}^{m}$
of Hörmander \cite[chap.18]{hormander_3} characterized by some parameters
$m\in\mathbb{R}$, $0\leq\tilde{\delta}\leq\tilde{\rho}\leq1$, $\tilde{\delta}+\tilde{\rho}\geq1$.
In this paper we have a metric in (\ref{eq:metric_g_in_coordinates})
similar to\footnote{The Weyl-Hörmander calculus in principle enables to use metrics of
a more general form than (\ref{eq:g}), \cite{hormander1979weyl}.}
\begin{equation}
g:=\left(\frac{dx}{\delta^{\perp}\left(\xi\right)}\right)^{2}+\left(\delta^{\perp}\left(\xi\right)d\xi\right)^{2}\label{eq:g}
\end{equation}
on $\left(x,\xi\right)\in\mathbb{R}^{2n}$ (variable transverse to
the flow direction) with
\begin{equation}
\delta^{\perp}\left(\xi\right)=\left\langle \left|\xi\right|\right\rangle ^{-\alpha^{\perp}}\label{eq:delta_ex}
\end{equation}
 as defined in (\ref{eq:def_delta}) and parameter $\frac{1}{2}\leq\alpha^{\perp}<1$.
We also have in (\ref{eq:h_gamma}), a function of the form
\begin{equation}
h_{\gamma}\left(\xi\right):=\left\langle \left|\xi\right|\right\rangle ^{-\gamma},\label{eq:def_h-1}
\end{equation}
with $0\leq\gamma<1$. This function $h_{\gamma}\left(\xi\right)$
plays the role of a ``small Planck parameter'' and is associated
to the re-scaled or conformal metric introduced by Hörmander that
is used to measure the variations of symbols on phase space \cite{hormander1979weyl}\cite[chap.18]{hormander_3}\cite[p.22,p.68]{lerner2011metrics}\cite{nicola_rodino_livre_11}
\begin{align*}
g_{\gamma}\left(\xi\right) & :=\left(h_{\gamma}\left(\xi\right)\right)^{2}g\left(\xi\right)\\
 & \underset{(\ref{eq:g})}{\asymp}\left(\frac{h_{\gamma}\left(\xi\right)dx}{\delta^{\perp}\left(\xi\right)}\right)^{2}+\left(h_{\gamma}\left(\xi\right)\delta^{\perp}\left(\xi\right)d\xi\right)^{2}\\
 & \underset{(\ref{eq:delta_ex}),(\ref{eq:def_h-1})}{=}\left(\frac{dx}{\left\langle \left|\xi\right|\right\rangle ^{-\tilde{\delta}}}\right)^{2}+\left(\frac{d\xi}{\left\langle \left|\xi\right|\right\rangle ^{\tilde{\rho}}}\right)^{2}
\end{align*}
with $\tilde{\delta}=\alpha^{\perp}-\gamma$ and $\tilde{\rho}=\alpha^{\perp}+\gamma$
that gives the relations
\begin{equation}
\alpha^{\perp}:=\frac{1}{2}\left(\tilde{\rho}+\tilde{\delta}\right)\geq\frac{1}{2},\quad\gamma:=\frac{1}{2}\left(\tilde{\rho}-\tilde{\delta}\right)\geq0.\label{eq:relations-1}
\end{equation}

\section{\label{sec:A-simple-model-of resonances}How to reveal intrinsic
discrete spectrum (resonances) on a simple model}

In this section we present an elementary model that explains why we
observe discrete spectrum (resonances) in an appropriate Sobolev space
though this is not the case for the usual $L^{2}$ space. At the end
of the section, we will see the analogy between this simple model
and the hyperbolic dynamics considered in this paper.

\subsection{The model}

Let us consider the following bi-infinite matrix $\mathcal{L}=\left(\mathcal{L}_{i,j}\right)_{i,j\in\mathbb{Z}}$
\begin{equation}
\mathcal{L}:=\left(\begin{array}{ccccc}
\ddots\\
\ddots & 0\\
 & 1 & w_{0}\\
 &  & 1 & 0\\
0 &  & -w_{1}^{-1} & 1 & \ddots
\end{array}\right)\label{eq:Matrix_L}
\end{equation}
whose non-vanishing elements are only $\mathcal{L}_{0,0}=w_{0}\in\mathbb{C}\backslash\left\{ 0\right\} $,
$\mathcal{L}_{2,0}=-w_{1}^{-1}\in\mathbb{C}$ and $\mathcal{L}_{j+1,j}=1$
for every $j\in\mathbb{Z}$. Considering these non vanishing elements,
we can associate to $\mathcal{L}$ a ``Markov graph'' that contains
the dynamics of a shift $\tilde{\phi}:j\mapsto j+1$, perturbed by
a finite rank matrix, see Figure \ref{fig:markov}.

\begin{figure}[h]
\centering{}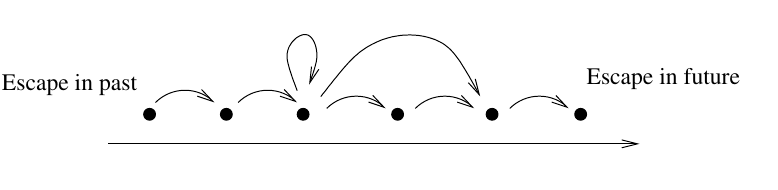\caption{\label{fig:markov}Markov graph associated to the infinite matrix
(\ref{eq:Matrix_L}).}
\end{figure}

\begin{rem}
\label{rem:The-inverse-of}The inverse matrix of $\mathcal{L}$ is
\begin{equation}
\mathcal{L}^{-1}:=\left(\begin{array}{ccccc}
\\
\ddots & 1 &  & 0\\
 & 0 & 1 & -w_{0}\\
 &  & 0 & 1\\
0 &  & 0 & w_{1}^{-1} & \ddots\\
 &  &  &  & \ddots
\end{array}\right)\label{eq:Matrix_L-1}
\end{equation}
where $\left(\mathcal{L}^{-1}\right)_{1,1}=w_{1}^{-1}$, $\left(\mathcal{L}^{-1}\right)_{-1,1}=-w_{0}$
and $\left(\mathcal{L}^{-1}\right)_{j,j+1}=1$ for $j\in\mathbb{Z}$.
The matrix $\mathcal{L}$ has the following eigenvectors $U$ and
$V$:
\begin{itemize}
\item $\mathcal{L}U=w_{0}U$ with vector $U=\left(\ldots,0,U_{0},\frac{1}{w_{0}}U_{0},\ldots,\frac{1}{w_{0}^{j}}\left(1-\frac{w_{0}}{w_{1}}\right)U_{0},\ldots\right)\in\mathbb{C}^{\mathbb{Z}}$
with components $U_{j}=0$ for $j<0$, $U_{0}\in\mathbb{C}$, $U_{1}=\frac{1}{w_{0}}U_{0}$
and $U_{j}=\frac{1}{w_{0}^{j}}\left(1-\frac{w_{0}}{w_{1}}\right)U_{0}$
for $j\geq2$. Hence $U\in l^{2}\left(\mathbb{Z}\right)$ if and only
if $\left|w_{0}\right|>1$.
\item $\mathcal{L}V=w_{1}V$ with vector $V=\left(\ldots,w_{1}^{1-j}\left(1-\frac{w_{0}}{w_{1}}\right)V_{1},\ldots,w_{1}V_{1},V_{1},0,\ldots\right)\in\mathbb{C}^{\mathbb{Z}}$
with components $V_{1}\in\mathbb{C}$, $V_{0}=w_{1}V_{1}$, $V_{j}=w_{1}^{1-j}\left(1-\frac{w_{0}}{w_{1}}\right)V_{1}$
for $j\leq-1$ and $V_{j}=0$ for $j\geq2$. Hence $V\in l^{2}\left(\mathbb{Z}\right)$
if and only if $\left|w_{1}\right|<1$.
\end{itemize}
\end{rem}

For some $r\in\mathbb{R}$ called the order, let us consider the following
function $W$ on $\mathbb{Z}$,
\[
W\left(j\right):=e^{-rj},\quad j\in\mathbb{Z}.
\]
For $j\neq0,2$ we have $\frac{W\left(\tilde{\phi}\left(j\right)\right)}{W\left(j\right)}=e^{-r}$
hence if $r>0$, $W$ decreases along the trajectories of $\tilde{\phi}$
and we call $W$ an ``escape function''. Let $\mathrm{Diag}\left(W\right)$
be the diagonal matrix with diagonal elements $W\left(j\right)$.
Let
\begin{align}
\tilde{\mathcal{L}}_{W}: & =\mathrm{Diag}\left(W\right)\circ\mathcal{L}\circ\mathrm{Diag}\left(W\right)^{-1}\underset{(\ref{eq:Matrix_L})}{=}\left(\begin{array}{ccccc}
\\
\ddots & 0 &  & 0\\
 & e^{-r} & w_{0}\\
 &  & e^{-r} & 0\\
0 &  & -e^{-2r}w_{1}^{-1} & e^{-r} & \ddots\\
 &  &  &  & \ddots
\end{array}\right)\label{eq:def_L_tilde_W}
\end{align}
We refer to \cite[chap. 1]{bottcher2012introduction}\cite[p.51]{trefethen_book_05}
for the spectrum of Toeplitz operators.
\begin{lem}
\label{lem:spectrum}The operator $\tilde{\mathcal{L}}_{W}:l^{2}\left(\mathbb{Z}\right)\rightarrow l^{2}\left(\mathbb{Z}\right)$
has essential spectrum on the circle of radius $e^{-r}$ and the following
discrete spectrum elsewhere: $w_{0}$ is an eigenvalue if and only
if $\left|w_{0}\right|>e^{-r}$. $w_{1}$ is an eigenvalue if and
only if $\left|w_{1}\right|<e^{-r}$.
\end{lem}

We define
\[
\mathcal{H}_{W}\left(\mathbb{Z}\right):=\mathrm{Diag}\left(W\right)^{-1}\left(l^{2}\left(\mathbb{Z}\right)\right)
\]
meaning that the norm of a vector $u\in\mathcal{H}_{W}\left(\mathbb{Z}\right)$
is
\begin{equation}
\left\Vert u\right\Vert _{\mathcal{H}_{W}\left(\mathbb{Z}\right)}^{2}:=\left\Vert \mathrm{Diag}\left(W\right)u\right\Vert _{l^{2}\left(\mathbb{Z}\right)}^{2}=\sum_{j\in\mathbb{Z}}\left|e^{-rj}u_{j}\right|^{2}\label{eq:norm_W}
\end{equation}
$\mathcal{H}_{W}\left(\mathbb{Z}\right)$ is similar to a anisotropic
Sobolev space with weight $W\left(j\right)=e^{-rj}$. We have the
following commutative diagram
\[
\begin{CD}\mathcal{H}_{W}\left(\mathbb{Z}\right)@>{\mathcal{L}}>>\mathcal{H}_{W}\left(\mathbb{Z}\right)\\
@V{\mathrm{Diag}\left(W\right)}VV@V{\mathrm{Diag}\left(W\right)}VV\\
l^{2}\left(\mathbb{Z}\right)@>{\tilde{\mathcal{L}}_{W}:=\mathrm{Diag}\left(W\right)\circ\mathcal{L}\circ\mathrm{Diag}\left(W\right)^{-1}}>>l^{2}\left(\mathbb{Z}\right)
\end{CD}
\]
where $\mathrm{Diag}\left(W\right):\mathcal{H}_{W}\left(\mathbb{Z}\right)\rightarrow l^{2}\left(\mathbb{Z}\right)$
is an isometry (by definition) and hence Lemma \ref{lem:spectrum}
gives
\begin{lem}
The operator $\mathcal{L}:\mathcal{H}_{W}\left(\mathbb{Z}\right)\rightarrow\mathcal{H}_{W}\left(\mathbb{Z}\right)$
has essential spectrum on the circle of radius $e^{-r}$  and the
following discrete spectrum elsewhere: $w_{0}$ is an eigenvalue with
eigenvector $U\in\mathcal{H}_{W}\left(\mathbb{Z}\right)$ if and only
if $\left|w_{0}\right|>e^{-r}$. $w_{1}$ is an eigenvalue with eigenvector
$V\in\mathcal{H}_{W}\left(\mathbb{Z}\right)$ if and only $\left|w_{1}\right|<e^{-r}$.
\end{lem}

See Figure \ref{fig:spectrum_matrix}.

\begin{figure}[h]
\centering{}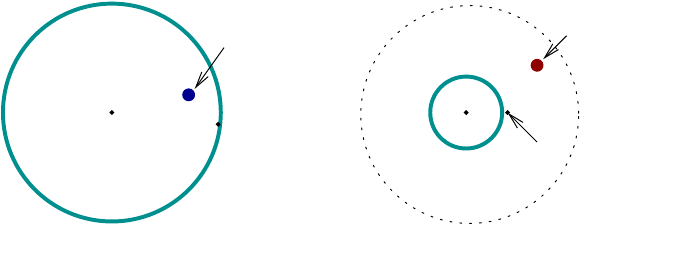\caption{\label{fig:spectrum_matrix}In this picture we assume $\left|w_{0}\right|=\left|w_{1}\right|<1$.
The circle of radius $e^{-r}$ (in green) is the essential spectrum
of $\mathcal{L}$ in the space $\mathcal{H}_{W}\left(\mathbb{Z}\right)$
that depends on $r\in\mathbb{R}$. As $r\rightarrow+\infty$ this
circle shrinks to zero and we reveal the intrinsic \textquotedblleft future
discrete spectrum\textquotedblright{} of $\mathcal{L}$, here this
is the eigenvalue $w_{0}$ (in red), as soon as $e^{-r}<\left|w_{0}\right|$.
As $r\rightarrow-\infty$ this circle goes to infinity and we reveal
the intrinsic \textquotedblleft past discrete spectrum\textquotedblright{}
of $\mathcal{L}$, here this is the eigenvalue $w_{1}$ (in blue),
as soon as $\left|w_{1}\right|<e^{-r}$.}
\end{figure}

\paragraph{The conclusion of this simple model is that:}
\begin{enumerate}
\item We observe that the given matrix $\mathcal{L}$ corresponds to a simple
dynamics $\tilde{\phi}$ and outside a compact region, the dynamics
of $\tilde{\phi}$ escapes to/from infinity.
\item We construct an ``escape function'' $W$ for that dynamic $\tilde{\phi}$
that decays with rate $e^{-r}$ and define an anisotropic Sobolev
space $\mathcal{H}_{W}\left(\mathbb{Z}\right):=\mathrm{Diag}\left(W\right)^{-1}\left(l^{2}\left(\mathbb{Z}\right)\right)$.
\item It appears that the matrix $\mathcal{L}$ has essential spectral radius
$r_{\mathrm{ess.}}=e^{-r}$ in $\mathcal{H}_{W}\left(\mathbb{Z}\right)$.
Moreover, by increasing the parameter $r$ we get $r_{\mathrm{ess.}}=e^{-r}\underset{r\rightarrow+\infty}{\rightarrow0}$
and this may reveal new eigenvalues and eigenspaces of $\mathcal{L}$
that do not depend on $W$ (here we have only $w_{0}\in\mathbb{C}$)
that we call ``future discrete spectrum''. If we do $r_{\mathrm{ess.}}=e^{-r}\underset{r\rightarrow-\infty}{\rightarrow+\infty}$
this may reveal new eigenvalues and eigenspaces of $\mathcal{L}$
that do not depend on $W$ (here we have only $w_{1}\in\mathbb{C}$)
that we call ``past discrete spectrum''. The past discrete spectrum
is the future discrete spectrum for $\mathcal{L}^{-1}$ and conversely.
See Figure \ref{fig:spectrum_matrix}.
\end{enumerate}
\begin{rem}
For $j\neq0$, the dynamics $\tilde{\phi}:j\mapsto j+1$ is a translation.
Observe that
\begin{itemize}
\item For $j>0$, if we set $\xi_{u}:=e^{j}$, we get an expanding dynamics
$\tilde{\phi}:\xi_{u}\mapsto e^{1}\xi_{u}$. The escape function is
$W\left(j\right)=e^{-jr}=\xi_{u}^{-r}$ hence the anisotropic Sobolev
space $\mathcal{H}_{W}\left(\mathbb{Z}\right)$ has negative order
$-r$.
\item For $j<0$, if we set $\xi_{s}:=e^{-j}$, we get a contracting dynamics
$\tilde{\phi}:\xi_{s}\mapsto e^{-1}\xi_{s}$. The escape function
is $W\left(j\right)=e^{-jr}=\xi_{s}^{r}$ hence the anisotropic Sobolev
space $\mathcal{H}_{W}\left(\mathbb{Z}\right)$ has positive order
$r$.
\end{itemize}
The fact that the order depends on the sign of $j$  explains the
term ``anisotropic''.
\end{rem}

~
\begin{rem}
Starting from the semi infinite matrix 
\[
\mathcal{L}:=\left(\begin{array}{cccc}
w_{0} &  &  & 0\\
1 & 0\\
 & 1 & 0\\
0 &  & \ddots & \ddots
\end{array}\right)
\]
we would have a similar analysis with the difference that the dynamics
is $\tilde{\phi}\left(j\right)=j+1$ is a semi-shift and the spectrum
in $\mathcal{H}_{W}\left(\mathbb{N}\right)$ is essential on the circle
of radius $e^{-r}$ and residual inside. If one chooses the escape
function $W\left(j\right)=e^{-j^{\alpha}}=e^{-\left(\log\xi\right)^{\alpha}}$
with some $\alpha>1$, we get that $\frac{W\left(j+1\right)}{W\left(j\right)}\rightarrow0$
as $j\rightarrow\infty$ and that $\mathcal{L}$ is Trace class in
$\mathcal{H}_{W}\left(\mathbb{Z}\right)$ with Trace obtained as the
sum over the discrete spectrum (the essential spectrum has shrunk
to $0$ immediately). This model corresponds to the so-called Gevrey
class in $\xi$ variable (and their dual). Even stronger, if one choose
the escape function $W\left(j\right)=e^{-re^{j}}=e^{-r\xi}$ with
some $r>0$, we get that $\frac{W\left(j+1\right)}{W\left(j\right)}=e^{-r\left(e-1\right)e^{j}}\rightarrow0$
as $j\rightarrow\infty$. This model corresponds to ``analytic class''
in $\xi$ variable (and their dual of hyper-functions).
\end{rem}

~
\begin{rem}
We can replace the single element $w_{0}\in\mathbb{C}$ on the diagonal
of $\mathcal{L}$ by a finite rank (or compact) matrix with eigenvalues
$w_{0},w_{1},\ldots$. From the discrete spectrum of $\mathcal{L}$
in $\mathcal{H}_{W}\left(\mathbb{Z}\right)$, we deduce decay of correlations
for $u,v\in\mathcal{H}_{W}\left(\mathbb{Z}\right)$:
\[
\langle u|\mathcal{L}^{t}v\rangle_{\mathcal{H}_{W}\left(\mathbb{Z}\right)}=\sum_{k}w_{k}^{t}\langle u|\mathcal{L}^{t}\pi_{k}v\rangle_{\mathcal{H}_{W}\left(\mathbb{Z}\right)}+O\left(e^{-\lambda rt}\right)\underset{t\rightarrow+\infty}{\sim}w_{0}^{t}\langle u|\mathcal{L}^{t}\pi_{0}v\rangle
\]
where $\pi_{k}$ is the spectral projector associated to the eigenvalue
$w_{k}$ and we have assumed that $\left|w_{0}\right|>\left|w_{j}\right|$
for $j\geq1$.
\end{rem}

\subsection{Analogy with Ruelle resonances for hyperbolic dynamics}

In the following table we put in correspondence the properties of
the matrix $\mathcal{L}$, (\ref{eq:Matrix_L}) and the transfer operator
$\mathcal{L}^{t}=e^{tA}$ studied in this paper that is defined from
an Anosov flow $\phi^{t}$ on $M$.
\begin{center}
\begin{tabular}{|l|l|}
\hline 
Matrix model $\mathcal{L}$ & Hyperbolic dynamics\tabularnewline
\hline 
\hline 
Orthonormal basis & Almost orthogonal basis of wave-packets\tabularnewline
 $\varphi_{j}=\left(0\ldots,\underset{(j)}{1},0\ldots\right)\in l^{2}\left(\mathbb{Z}\right)$,
$j\in\mathbb{Z}$. & \noun{$\varphi_{\rho}\in L^{2}\left(M\right)$, }with $\rho\in T^{*}M$.\tabularnewline
\hline 
The action of $\mathcal{L}$ is described by & The action of $\mathcal{L}^{t}$ is described by\tabularnewline
a dynamics $\tilde{\phi}$ on $\mathbb{Z}$, & the lifted flow $\tilde{\phi}$ on $\rho=\left(x,\xi\right)\in T^{*}M$,\tabularnewline
hyperbolic in $\xi_{u},\xi_{s}=e^{\pm j}$ & hyperbolic in $\xi$.\tabularnewline
\hline 
Escape function $W\left(j\right)$, & Escape function $W$ on $T^{*}M$\tabularnewline
with decay property $\frac{W\left(\tilde{\phi}\left(j\right)\right)}{W\left(j\right)}=e^{-r}$. & with temperate and decay property\tabularnewline
 &  $\frac{W\left(\tilde{\phi}^{t}\left(\rho\right)\right)}{W\left(\rho\right)}\leq Ce^{-\Lambda t}$
outside the trapped set $E_{0}^{*}$.\tabularnewline
\hline 
$\mathcal{H}_{W}\left(\mathbb{Z}\right):=\mathrm{Diag}\left(W\right)^{-1}\left(l^{2}\left(\mathbb{Z}\right)\right)$ & $\mathcal{H}_{W}\left(M\right):=\mathrm{Op}\left(W\right)^{-1}\left(L^{2}\left(M\right)\right)$\tabularnewline
 & with $\mathrm{Op}\left(W\right)$ a PDO.\tabularnewline
 & i.e. almost diagonal in wave-packet basis.\tabularnewline
\hline 
$\mathcal{S}_{\xi}\subset\mathcal{H}_{W}\left(\mathbb{Z}\right)\subset\mathcal{S}_{\xi}'$ & $\mathcal{S}\left(M\right)\subset\mathcal{H}_{W}\left(M\right)\subset\mathcal{S}'\left(M\right)$\tabularnewline
\hline 
Discrete spectrum $w=e^{z}$ on & Discrete spectrum $z$ of the generator $A$\tabularnewline
$\left|w\right|>e^{-1.r}\Leftrightarrow\mathrm{Re}\left(z\right)>-1.r$. &  on $\mathrm{Re}\left(z\right)>-\lambda_{\mathrm{min}}r+\mathrm{Cste}$.\tabularnewline
\hline 
\end{tabular}
\par\end{center}

\section{Relations for the Japanese bracket $\left\langle .\right\rangle $}

For $s\in\mathbb{R}$, we set 
\[
\langle s\rangle:=(1+s^{2})^{1/2}.
\]
Clearly we have 
\[
\max\{1,|s|\}\le\langle s\rangle\le\langle\langle s\rangle\rangle\le2\max\{1,|s|\}
\]
and 
\[
\frac{d}{ds}\langle s\rangle=\frac{s}{\sqrt{1+s^{2}}}\in(-1,1)\quad\text{for any \ensuremath{s\in\mathbb{R}}.}
\]
From the second estimate, it follows 
\begin{equation}
\langle s+t\rangle\le\langle s\rangle+|t|\le\langle s\rangle+\langle t\rangle\label{eq:sum}
\end{equation}
that implies
\begin{equation}
\frac{\left\langle t\right\rangle }{\left\langle s\right\rangle }\leq1+\frac{\left|t-s\right|}{\left\langle s\right\rangle }\leq1+\left|t-s\right|\label{eq:D1p}
\end{equation}

Also, by considering the cases $|s|\le1$ and $|s|>1$ separately,
we can check
\begin{equation}
\langle s\cdot t\rangle\le\max\{1,|s|\}\cdot\langle t\rangle\le\langle s\rangle\cdot\langle t\rangle\label{eq:prod}
\end{equation}
Note that \eqref{eq:prod} implies that, if $s\neq0$, 
\[
\langle t\rangle=\langle s\cdot s^{-1}t\rangle\le\max\{1,|s|\}\cdot\langle s^{-1}t\rangle
\]
and hence 
\begin{equation}
\langle\frac{t}{s}\rangle\ge\max\{1,|s|\}^{-1}\langle t\rangle\ge\frac{\left\langle t\right\rangle }{\left\langle s\right\rangle }.\label{eq:prod2}
\end{equation}

\bibliographystyle{plain}
\bibliography{/home/faure/articles/articles}

\begin{thebibliography}{10}

\bibitem{https://doi.org/10.48550/arxiv.1809.04062}
A.~Adam and V.~Baladi.
\newblock Horocycle averages on closed manifolds and transfer operators, 2018.

\bibitem{Baladi05}
V.~Baladi.
\newblock Anisotropic {S}obolev spaces and dynamical transfer operators:
  {$C^\infty$} foliations.
\newblock In {\em Algebraic and topological dynamics}, volume 385 of {\em
  Contemp. Math.}, pages 123--135. Amer. Math. Soc., Providence, RI, 2005.

\bibitem{baladi_05}
V.~Baladi and M.~Tsujii.
\newblock {Anisotropic H\"older and Sobolev spaces for hyperbolic
  diffeomorphisms}.
\newblock {\em Ann. Inst. Fourier}, 57:127--154, 2007.

\bibitem{liverani_02}
M.~Blank, G.~Keller, and C.~Liverani.
\newblock {{R}uelle-Perron-Frobenius spectrum for {A}nosov maps}.
\newblock {\em Nonlinearity}, 15:1905--1973, 2002.

\bibitem{bonthonneau2018flow}
Yannick~Guedes Bonthonneau.
\newblock Flow-independent anisotropic space, and perturbation of resonances.
\newblock {\em arXiv preprint arXiv:1806.08125}, 2018.

\bibitem{bonthonneau2020fbi}
Yannick~Guedes Bonthonneau and Malo J{\'e}z{\'e}quel.
\newblock Fbi transform in gevrey classes and anosov flows.
\newblock {\em arXiv preprint arXiv:2001.03610}, 2020.

\bibitem{bottcher2012introduction}
A.~B{\"o}ttcher and B.~Silbermann.
\newblock {\em Introduction to large truncated Toeplitz matrices}.
\newblock Springer Science \& Business Media, 2012.

\bibitem{liverani_butterley_07}
O.~Butterley and C.~Liverani.
\newblock Smooth {A}nosov flows: correlation spectra and stability.
\newblock {\em J. Mod. Dyn.}, 1(2):301--322, 2007.

\bibitem{da_silva_01}
A.~Cannas Da~Silva.
\newblock {\em Lectures on Symplectic Geometry}.
\newblock Springer, 2001.

\bibitem{dang_riviere_morse_smale_16}
N.~V. Dang and G.~Riviere.
\newblock Spectral analysis of morse-smale gradient flows.
\newblock {\em arXiv preprint arXiv:1605.05516}, 2016.

\bibitem{dyatlov_Ruelle_resonances_2012}
K.~Datchev, S.~Dyatlov, and M.~Zworski.
\newblock Sharp polynomial bounds on the number of {P}ollicott-{R}uelle
  resonances.
\newblock {\em Ergodic Theory and Dynamical Systems, arXiv:1208.4330}, pages
  1--16, 2012.

\bibitem{dyatlov_2021_pollicott_ruelle_sobolev}
S.~Dyatlov.
\newblock Pollicott-ruelle resolvent and sobolev regularity, 2021.

\bibitem{dyatlov_faure_guillarmou_2014}
S.~Dyatlov, F.~Faure, and C.~Guillarmou.
\newblock Power spectrum of the geodesic flow on hyperbolic manifolds.
\newblock {\em Analysis and PDE
  \href{https://arxiv.org/abs/1403.0256v1}{link}}, 8:923--1000, 2015.

\bibitem{dyatlov_guillarmou_2014}
S.~Dyatlov and C.~Guillarmou.
\newblock Pollicott-ruelle resonances for open systems.
\newblock {\em arXiv preprint arXiv:1410.5516}, 2014.

\bibitem{dyatlov_zworski_zeta_2013}
S.~Dyatlov and M.~Zworski.
\newblock Dynamical zeta functions for {A}nosov flows via microlocal analysis.
\newblock {\em arXiv preprint arXiv:1306.4203}, 2013.

\bibitem{engel_1999}
K.J. Engel and R.~Nagel.
\newblock {\em One-parameter semigroups for linear evolution equations}, volume
  194.
\newblock Springer, 1999.

\bibitem{falconer_03_book}
K.J. Falconer.
\newblock {\em Fractal geometry: mathematical foundations and applications}.
\newblock John Wiley \& Sons Inc, 2003.

\bibitem{fred-RP-06}
F.~Faure and N.~Roy.
\newblock {R}uelle-{P}ollicott resonances for real analytic hyperbolic map.
\newblock {\em Nonlinearity.
  \href{https://arxiv.org/abs/nlin.CD/0601010}{link}}, 19:1233--1252, 2006.

\bibitem{fred-roy-sjostrand-07}
F.~Faure, N.~Roy, and J.~Sj{\"o}strand.
\newblock A semiclassical approach for {A}nosov diffeomorphisms and {R}uelle
  resonances.
\newblock {\em Open Math. Journal.
  \href{https://arxiv.org/abs/0802.1780}{link}}, 1:35--81, 2008.

\bibitem{fred_flow_09}
F.~Faure and J.~Sj{\"o}strand.
\newblock Upper bound on the density of {R}uelle resonances for {A}nosov flows.
  a semiclassical approach.
\newblock {\em Comm. in Math. Physics, Issue 2.
  \href{https://fr.arxiv.org/abs/1003.0513}{link}}, 308:325--364, 2011.

\bibitem{faure_tsujii_band_CRAS_2013}
F.~Faure and M.~Tsujii.
\newblock Band structure of the {R}uelle spectrum of contact {A}nosov flows.
\newblock {\em Comptes rendus - Math\'ematique 351 , 385-391, (2013)
  \href{https://arxiv.org/abs/1301.5525}{link}}, 2013.

\bibitem{faure-tsujii_prequantum_maps_12}
F.~Faure and M.~Tsujii.
\newblock Prequantum transfer operator for symplectic {A}nosov diffeomorphism.
\newblock {\em Asterisque 375 (2015),
  \href{https://fr.arxiv.org/abs/1206.0282}{link}}, pages ix+222 pages, 2015.

\bibitem{faure-tsujii_anosov_flows_13}
F.~Faure and M.~Tsujii.
\newblock The semiclassical zeta function for geodesic flows on negatively
  curved manifolds.
\newblock {\em Inventiones mathematicae.
  \href{https://arxiv.org/abs/1311.4932}{link}}, 2016.

\bibitem{faure-tsujii_anosov_flows_16}
F.~Faure and M.~Tsujii.
\newblock Microlocal analysis and band structure of contact {A}nosov flows.
\newblock {\em arxiv2102.11196 \href{https://arxiv.org/abs/2102.11196}{link}},
  2021.

\bibitem{folland-88}
G.~Folland.
\newblock {\em Harmonic Analysis in phase space}.
\newblock Princeton University Press, 1988.

\bibitem{liverani_04}
S.~Gou{\"e}zel and C.~Liverani.
\newblock Banach spaces adapted to {A}nosov systems.
\newblock {\em Ergodic Theory and dynamical systems}, 26:189--217, 2005.

\bibitem{grigis_sjostrand}
A.~Grigis and J.~Sj{\"o}strand.
\newblock {\em Microlocal analysis for differential operators}, volume 196 of
  {\em London Mathematical Society Lecture Note Series}.
\newblock Cambridge University Press, Cambridge, 1994.
\newblock An introduction.

\bibitem{guillarmou_weich_resonances_16}
C.~Guillarmou, J.~Hilgert, and T.~Weich.
\newblock Classical and quantum resonances for hyperbolic surfaces.
\newblock {\em arXiv preprint arXiv:1605.08801}, 2016.

\bibitem{hasselblatt1994regularity}
B.~Hasselblatt.
\newblock Regularity of the {A}nosov splitting and of horospheric foliations.
\newblock {\em Ergodic Theory and Dynamical Systems}, 14(04):645--666, 1994.

\bibitem{sjostrand_87}
B.~Helffer and J.~Sj{\"o}strand.
\newblock R\'esonances en limite semi-classique. (resonances in semi-classical
  limit).
\newblock {\em Memoires de la S.M.F.}, 24/25, 1986.

\bibitem{sjostrand_hitrick_07}
M.~Hitrik and J.~Sj\"ostrand.
\newblock Rational invariant tori, phase space tunneling, and spectra for
  non-selfadjoint operators in dimension 2.
\newblock {\em Ann. Scient. de l'\'ecole normale sup\'erieure.
  arXiv:math/0703394v1 [math.SP]}, 2008.

\bibitem{hormander1979weyl}
L.~H{\"o}rmander.
\newblock The weyl calculus of pseudo-differential operators.
\newblock {\em Communications on Pure and Applied Mathematics}, 32(3):359--443,
  1979.

\bibitem{hormander_3}
L.~H{\"o}rmander.
\newblock {\em The analysis of linear partial differential operators III},
  volume 257.
\newblock Springer, 1983.

\bibitem{hormander_1}
L.~Hormander.
\newblock {\em The Analysis of the Linear Partial Differential Operators I:
  Distribution Theory and Fourier Analysis. Classics in Mathematics}.
\newblock Springer: Berlin, 2003.

\bibitem{hurder-90}
S.~Hurder and A.~Katok.
\newblock {Differentiability, rigidity and Godbillon-Vey classes for {A}nosov
  flows.}
\newblock {\em Publ. Math., Inst. Hautes \'etud. Sci.}, 72:5--61, 1990.

\bibitem{jenkinson_2018}
O.~Jenkinson.
\newblock Ergodic optimization in dynamical systems.
\newblock {\em Ergodic Theory and Dynamical Systems}, pages 1--26, 2018.

\bibitem{jezequel2020spectral_thesis}
Malo J{\'e}z{\'e}quel.
\newblock {\em Spectral theory for ultradifferentiable hyperbolic dynamics}.
\newblock PhD thesis, Sorbonne-Universit{\'e}, 2020.

\bibitem{jin_zworski_local_trace_14}
Long Jin and Maciej Zworski.
\newblock A local trace formula for anosov flows.
\newblock In {\em Annales Henri Poincar{\'e}}, volume~18, pages 1--35.
  Springer, 2017.

\bibitem{lerner2011metrics}
N.~Lerner.
\newblock {\em Metrics on the phase space and non-selfadjoint
  pseudo-differential operators}, volume~3.
\newblock Springer, 2011.

\bibitem{martinez-01}
A.~Martinez.
\newblock {\em An Introduction to Semiclassical and Microlocal Analysis}.
\newblock Universitext. New York, NY: Springer, 2002.

\bibitem{mac_duff_98}
D~McDuff and D~Salamon.
\newblock {\em Introduction to symplectic topology, 2nd edition}.
\newblock clarendon press, Oxford, 1998.

\bibitem{https://doi.org/10.48550/arxiv.2107.08875}
A.~Meddane.
\newblock A morse complex for axiom a flows, 2021.

\bibitem{nicola_rodino_livre_11}
F.~Nicola and L.~Rodino.
\newblock {\em Global pseudo-differential calculus on Euclidean spaces},
  volume~4.
\newblock Springer Science \& Business Media, 2011.

\bibitem{Nonnenmacher_Zworski_14}
S.~Nonnenmacher, J.~Sj\"ostrand, and M.~Zworski.
\newblock {Fractal Weyl law for open quantum chaotic maps.}
\newblock {\em {Ann. Math. (2)}}, 179(1):179--251, 2014.

\bibitem{pazy_semigroups_83}
A.~Pazy.
\newblock {\em Semigroups of linear operators and applications to partial
  differential equations}, volume 198.
\newblock Springer New York, 1983.

\bibitem{sjostrand_90}
J.~Sj{\"o}strand.
\newblock Geometric bounds on the density of resonances for semiclassical
  problems.
\newblock {\em Duke Math. J.}, 60(1):1--57, 1990.

\bibitem{sjostrand_density_resonances_96}
J.~Sj{\"o}strand.
\newblock Density of resonances for strictly convex analytic obstacles.
\newblock {\em Canad. J. Math.}, 48(2):397--447, 1996.
\newblock With an appendix by M. Zworski.

\bibitem{sjostrand_2000}
J.~Sj{\"o}strand.
\newblock Asymptotic distribution of eigenfrequencies for damped wave
  equations.
\newblock {\em Publ. Res. Inst. Math. Sci}, 36(5):573--611, 2000.

\bibitem{taylor_tome1}
M.~Taylor.
\newblock {\em Partial differential equations, Vol I}.
\newblock Springer, 1996.

\bibitem{taylor_tome2}
M.~Taylor.
\newblock {\em Partial differential equations, Vol II}.
\newblock Springer, 1996.

\bibitem{trefethen_book_05}
L.N. Trefethen and M.~Embree.
\newblock {\em Spectra and pseudospectra}.
\newblock Princeton University Pr., 2005.

\bibitem{tsujii_08}
M.~Tsujii.
\newblock Quasi-compactness of transfer operators for contact {A}nosov flows.
\newblock {\em Nonlinearity, arXiv:0806.0732v2 [math.DS]}, 23(7):1495--1545,
  2010.

\bibitem{tsujii_FBI_10}
M.~Tsujii.
\newblock Contact {A}nosov flows and the fourier--bros--iagolnitzer transform.
\newblock {\em Ergodic theory and dynamical systems}, 32(06):2083--2118, 2012.

\bibitem{wunsh-zworski_01}
J.~Wunsch and M.~Zworski.
\newblock {The FBI transform on compact ${\cal{C}^\infty}$ manifolds.}
\newblock {\em Trans. Am. Math. Soc.}, 353(3):1151--1167, 2001.

\bibitem{zworski_book_2012}
M.~Zworski.
\newblock {\em Semiclassical Analysis}.
\newblock Graduate Studies in Mathematics Series. Amer Mathematical Society,
  2012.

\end{thebibliography}

\end{document}